\documentclass[12pt]{article}
\usepackage[letterpaper, margin=1in]{geometry}

\usepackage{amsmath, amssymb, amsthm, color, mathtools}
\usepackage{caption, graphicx}
\usepackage{subcaption}
\usepackage{enumitem}
\usepackage{hyperref}
\usepackage{cleveref}
\newtheorem{theorem}{Theorem}
\newtheorem{corollary}{Corollary}[theorem]

\usepackage[longnamesfirst, authoryear]{natbib}
\usepackage{authblk}
\newtheorem{definition}{Definition}
\newtheorem{lemma}[theorem]{Lemma}
\newtheorem{example}{Example}
\newtheorem{remark}{Remark}

\usepackage{changepage}
\usepackage[utf8]{inputenc}
\usepackage{setspace}
\usepackage{booktabs}
\usepackage{makecell}
\hypersetup{hidelinks,colorlinks=true,linkcolor=blue,citecolor=blue,linktocpage=true}

\usepackage{algorithm}
\usepackage{algpseudocode}

\providecommand{\keywords}[1]
{
  {\textit{Keywords---} #1}
}
\renewcommand{\longrightarrow}{\to}

\usepackage[title, titletoc]{appendix}
\title{Honest Inference for Stochastic Optimization}
\author{Kenta Takatsu}
\author{Arun Kumar Kuchibhotla}
\affil{Department of Statistics and Data Science, Carnegie Mellon University}
\date{}

\DeclareMathOperator{\E}{\mathbb{E}}
\DeclareMathOperator{\M}{\mathbb{M}}
\DeclareMathOperator*{\argmax}{arg\,max}
\DeclareMathOperator*{\argmin}{arg\,min}
\newcommand{\ratio}{\widehat{\Delta}_2}

\usepackage[longnamesfirst, authoryear]{natbib}
\usepackage{algorithm}
\usepackage{algpseudocode}
\usepackage{tocloft}
\cftsetindents{section}{0em}{3em}
\cftsetindents{subsection}{3em}{3em}
\begin{document}
\maketitle

\begin{abstract}
This manuscript studies a general approach to construct confidence sets for the solution of stochastic optimization, rendering empirical risk minimization as special cases. Statistical inference for stochastic optimization poses significant challenges due to the non-standard limiting behaviors of the corresponding estimator, which arise in settings with increasing dimension of parameters, non-smooth objectives, or constraints. We propose a simple and unified method that guarantees validity in both regular and irregular cases. We provide a unified treatment of validity, conservativeness, and the size of the resulting confidence sets. In particular, the presented width analysis demonstrates the adaptive behavior of the confidence set to the unknown degree of instance-specific regularity. We apply the proposed method to several high-dimensional and irregular statistical problems. Numerical results for all statistical applications are provided.
\end{abstract}

\tableofcontents
% \clearpage
% \singlespacing
% \doublespacing
\section{Introduction}\label{sec:introduction}
The present study examines inference for parameters defined as solutions to stochastic minimization (or maximization) problems, which arise in broad statistical applications. Let $(Z_1, \ldots, Z_N)$ be random variables defined on a common probability space, and let $\mathcal{L}(\cdot)$ denote the law of the corresponding random variables. Define 
\begin{equation*}
    P^N := \mathcal{L}(Z_1, \ldots, Z_N) \quad \textrm{and} \quad P_i := \mathcal{L}(Z_i) \quad \textrm{for}\quad 1\le i \le N.
\end{equation*}
Let $\mathcal{P}^N$ be a collection of joint distributions $P^N$. Given a metric space $(\Theta, \|\cdot\|)$ and a criterion function $\M: \Theta \times \mathcal{P}^N \mapsto \mathbb{R}$, the parameter of interest is
\begin{equation}\label{eq:def-m-functional}
    \theta(P^N):= \argmin_{\theta \in \Theta}\, \M(\theta, P^N).
\end{equation}
We assume that the population minimizer is common across all marginals in the sense that for any $P^N \in \mathcal{P}^N$, we have $\theta(P^N) = \theta(P_i)$ for all $1 \le i \le N$. This formulation allows observations $Z_1, \ldots, Z_N$ to be neither independent nor identically distributed. The main statistical challenge is that the criterion $\M(\theta, P^N)$ must be estimated from data.

The primary objective of this manuscript is the construction of an \emph{honest} confidence set \citep{Li1989, Ptscher2002Lower} for the $P^N$-dependent minimizer such that 
\begin{equation}\label{eq:asymptotic-validity}
    \liminf_{N\to \infty}\,\inf_{P^N \in \mathcal{P}^N}\,\inf_{\theta^*\in\theta(P^N)}\, \mathbb{P}_{P^N} (\theta^* \in \widehat{\mathrm{CI}}_{N,\alpha}) \ge 1-\alpha
\end{equation}
where $\mathbb{P}_{P^N}(\cdot)$ denotes probability under $P^N$. It is known that a finite-sample guarantee for fixed $N$ is generally impossible without imposing strong restrictions on the class of distributions $\mathcal{P}^N$ \citep{bahadur1956nonexistence}.
Although the asymptotic validity guarantee in \eqref{eq:asymptotic-validity} may appear less informative in the settings with high-dimensional parameters, the methods studied in this manuscript satisfies~\eqref{eq:asymptotic-validity} and retains validity regardless of the dimension/complexity of the space $\Theta$, including the cases where the dimension is comparable to the sample size. This property has recently been termed \emph{dimension-agnostic} \citep{kim2020dimension}, though closely related ideas also appear in \citet{Robins2006}.

Finally, the guarantee \eqref{eq:asymptotic-validity} does not require uniqueness of the minimizer. To simplify notation, the following convention will be used:
\begin{equation}\label{eq:miscoverage-over-set}
    \mathbb{P}_{P^N}(\theta(P^N)\in\widehat{\mathrm{CI}}_{N,\alpha}) ~:=~ \inf_{\theta^*\in\theta(P^N)}\,\mathbb{P}_P(\theta^*\in\widehat{\mathrm{CI}}_{N,\alpha}).
\end{equation}

\paragraph{Motivation: Irregular and High-dimensional Inference}
Common approaches for constructing confidence sets include (i) Wald methods based on the limiting distribution of an estimator, and (ii) resampling methods. Both require an estimator $\widehat{\theta}_N$ such that, for a diverging sequence $r_N$, the scaled quantity $r_N(\widehat{\theta}_N - \theta(P^N))$ converges in distribution. Wald methods posit a parametric form for this limit and construct confidence sets from its quantiles, whereas resampling methods estimate the limiting distribution nonparametrically. In many settings, the honest validity guarantee~\eqref{eq:asymptotic-validity} can fail if the weak convergence of $r_N(\widehat{\theta}_N - \theta(P^N))$ is not continuous in $P^N$~\citep{Andrews2000, andrews2010asymptotic, cattaneo2020bootstrap, cattaneo2023bootstrap}. A classical example of such ``continuity'' is the regularity of the estimator~\citep[Sec. 8.5]{van2000asymptotic}. We refer to the settings where such condition fails (or equivalently, the estimator is ill-behaved) as irregular problems.

% Problems where this fails are referred to as \emph{irregular}.

% Commonly used confidence set procedures include (1) the Wald methods based on the limiting distribution of a studied estimator, and (2) the resampling approaches. Both methods require an estimator $\widehat{\theta}_N$ such that for a suitable rate of convergence $r_N$, diverging to $\infty$, the random object $r_N(\widehat{\theta}_N - \theta(P))$ converges in distribution. The Wald methods assume a parametric limiting distribution and use the estimated quantiles to construct the confidence sets while the second (resampling) approach non-parametrically estimates the limiting distribution. From the extensive study of both approaches in the literature, we know of numerous settings in which the corresponding confidence sets might not satisfy the guarantee of honest inference~\eqref{eq:asymptotic-validity}. In particular, if the weak convergence of the normalized estimator is not ``continuous'' in $P$, the honest validity guarantee may fail for both Wald and resampling techniques~\citep{Andrews2000, andrews2010asymptotic, cattaneo2020bootstrap, cattaneo2023bootstrap}. An example of such ``continuity'' condition is the regularity of an estimator~\citep[Sec. 8.5]{van2000asymptotic}. We refer to the settings where such ``continuity'' condition fails (or equivalently, the estimator is ill-behaved) as irregular problems. It may be helpful to clarify that we consider cases where the functional $\theta(P)$ or estimator $\widehat{\theta}_n$ considered is ill-behaved to be irregular.

Inference for stochastic optimization under non-standard or irregular conditions has been extensively studied in statistics, econometrics, operations research, and other related fields \citep{geyer1994asymptotics, Ketz2018,Horowitz2019, Hsieh2022, Li2024}. Existing inference methods are typically tailored to specific regularity conditions that are required for validity. For instance, the general approach by \citet{vogel2008universal} is valid under regularity conditions on the objective function and parameter space. The framework by \citet{li2024inference} requires knowledge of the estimator’s convergence rate. Procedures based on sample-splitting by \citet{dey2024anytime} and \citet{park2023robust} are proved under stronger            distributional assumptions to ensure valid coverage.

This manuscript proposes a simple, general-purpose approach to inference for stochastic optimization problems, building on robust procedures based on sample-splitting~\citep{Robins2006, chakravarti2019gaussian, wasserman2020universal, park2023robust, kim2020dimension, dey2024anytime}. The resulting confidence set is valid under weak distributional assumptions, applicable to irregular or high-dimensional settings, and accommodates constraints or regularization. The corresponding inference tasks have been difficult without knowledge of the estimator’s convergence rate, the existence of a limiting distribution, or specific structure in the parameter space. In this regard, we provide a significant improvement, as the convergence rate in irregular problems can be impossible to estimate uniformly, and the corresponding limiting distributions can be highly complex~\citep{wang1996asymptotics}.
 
The following statistical problems are a few examples in which inference remains difficult to date and the proposed confidence set provides a simple solution.
\begin{enumerate}
    \item \textbf{High-dimensional Linear Regression}: Inference for ordinary least squares (OLS) remains challenging when the dimension $d$ increases with the sample size $N$, due to bias of $d/N^{1/2}$  \citep{mammen1993bootstrap, cattaneo2018inference}. Existing bias-corrected methods regain validity in some regimes $d \gg N^{1/2}$~\citep{cattaneo2019two, chang2023inference}, but typically require growth restrictions on $d$. The proposed method is valid regardless of the dimension (Sections~\ref{sec:mean} and \ref{sec:ols}). 
    \item \textbf{Cube-root Estimators}: \cite{kim1990cube} identify 
    a class of problems exhibiting \emph{cube-root asymptotics}, where the minimizer converges at $N^{-1/3}$. Examples include Manski’s maximum score estimator \citep{manski1975maximum, manski1985semiparametric,horowitz1992smoothed,delgado2001subsampling}, the Grenander estimator \citep{grenander1956theory, sen2010inconsistency, westling2020unified, cattaneo2023bootstrap}, and classification in machine learning \citep{mohammadi2005asymptotics}. Empirical bootstrap is known to be inconsistent \citep{sen2010inconsistency, patra2018consistent}, and modified resampling procedures have been proposed \citep{cattaneo2020bootstrap, cattaneo2023bootstrap}. We provide new inferential results for a prototypical example (Section~\ref{sec:manski}).
    \item \textbf{Non-smooth Objective}: Many criterion functions can be written as $\M(\theta, P^N) \equiv \E_{P^N}[m_{\theta}(Z)]$, where $m_{\theta}(Z)$ is a ``loss" function. When $\theta \mapsto m_{\theta}$ is non-smooth, the limiting distribution of an estimator can be non-standard~\citep{smirnov1952limit, knight1998limiting}. A canonical example is quantile estimation whose limiting distribution depends on the smoothness of the cumulative distribution function (CDF). While distribution-free finite sample valid confidence intervals exist \citep{lanke1974interval}, we study the behavior of the proposed confidence set in this setting (Section~\ref{sec:quantile}).  
    \item \textbf{Constrained Optimization}: The parameter space $\Theta$ can incorporate structural constraints, such as sparsity or shape \citep{ wang1996asymptotics, candes2007dantzig, li2015geometric, royset2020variational}. Confidence sets under such constraints have been studied \citep{geyer1994asymptotics,vogel2008confidence, vogel2008universal, vogel2017confidence, vogel2019universal}. The proposed confidence set remains valid under such structural constraints. 
\end{enumerate}

\paragraph{Summary of Methods and Contributions}

A key conceptual idea is that, instead of relying on the behavior of an estimator, we can exploit the defining property of the functional \eqref{eq:def-m-functional}. In particular, the following holds from the definition \eqref{eq:def-m-functional}:
\begin{equation}\label{eq:zeroth-order}
    \M(\theta(P^N), P^N) \le \M(\theta, P^N) \quad \text{for any non-random} \quad \theta \in \Theta.
\end{equation}
At an intuitive level, the proposed confidence set can be motivated in two steps: (1) Since $\theta(P^N)$ minimizes $\theta\mapsto \mathbb{M}(\theta, P^N)$, it must belong to the set $\{\theta:\, \mathbb{M}(\theta, P^N) \le \mathbb{M}(\theta', P^N)\}$ for any $\theta' \in \Theta$; (2) if $\theta\mapsto\widehat{\mathbb{M}}_N(\theta)$ is an estimator of $\mathbb{M}(\theta, P^N)$, then it is natural to expect that $\theta(P^N)$ will also belong to $\{\theta:\, \widehat{\mathbb{M}}_N(\theta) \le \widehat{\mathbb{M}}_N(\theta') + \gamma_N\}$ for an appropriate tolerance level $\gamma_N$. In particular, the reference point $\theta'$ can be replaced by a data-dependent estimator, provided that $\widehat{\M}_N(\cdot)$ and $\theta'$ are (approximately) independent; one way to achieve this independence is through sample-splitting. This idea is not new and can be found as early as \citet{Stein1981}, and more recently in \cite{Robins2006} and \cite{vogel2008universal}; see Section~\ref{appsec:references-history} for further historical discussion.

The main contributions of this manuscript are as follows: (1) we provide a systematic way to obtain a dimension- and complexity-agnostic validity guarantee; (2) we analyze the width (or diameter) of the proposed confidence set under mild conditions; and (3) we use the general result to establish rate adaptivity in several examples. The conditions used in the width analysis are comparable to those commonly employed in studying the convergence rates of M-estimators. These results contribute to the development of honest and adaptive inference procedures for stochastic optimization problems.

This flexibility and generality come with a price. First, although we provide conditions under which the proposed confidence set shrinks to a singleton at the optimal rate adaptively, it can be larger than traditional confidence sets due to the use of sample splitting. In practice, data efficiency can be partially recovered by swapping the roles of the splits and aggregating the resulting sets via a majority-vote procedure~\citep{gasparin2024merging}, though such efficiency improvements are not the focus of this work. Importantly, the role of sample splitting here is fundamentally different from that in double machine learning (DML) ~\citep{chernozhukov2018double, foster2023orthogonal}: at its core, DML relies on influence-function expansions and regularity conditions to establish asymptotic normality, whereas the proposed framework is precisely designed to avoid this type of assumptions. Second, unlike traditional methods that control the shape of the confidence set by considering an appropriate statistic, the proposed confidence set can be non-convex or even disconnected depending on the (estimated) objective function. It might be worth pointing out that, in regular cases, the proposed confidence set will approximately be an ellipsoid in similarity to the likelihood ratio confidence set. Furthermore, we note that the universal inference procedure of~\cite{wasserman2020universal} also shares the same drawbacks.

\paragraph{Organization.} The remainder of this manuscript is organized as follows. \Cref{sec:general} provides the most general construction for stochastic optimization and  \Cref{sec:general-M-estimation} discusses more refined results for M-estimation; \Cref{tab:base-confidence-sets} at the end of \Cref{sec:general-M-estimation} summarizes the validity results from both sections. \Cref{sec:coverage-slpha} develops methods for any prescribed significance level with corresponding validity guarantees. \Cref{sec:convergence-rates} establishes non-asymptotic diameter bounds for the proposed confidence sets. \Cref{sec:improvements,sec:computation} address practical considerations, covering power improvements and computational aspects respectively. \Cref{sec:application} provides an analysis of the confidence set proposed in statistical applications whose inference has been considered challenging. \Cref{sec:num} presents numerical results, with additional experiments available in \Cref{supp:num}. \Cref{sec:conclusions} concludes with remarks on open problems and future directions.

\paragraph{Notation.} We adopt the following convention. For $x \in \mathbb{R}^d$, we write $\|x\|_2 = \sqrt{x^\top x}$. In particular, we define the unit sphere with respect to $\|\cdot\|_2$ such that $\mathbb{S}^{d-1} = \{u \in \mathbb{R}^d \, : \, \|u\|_2 =1\}$. Given a square matrix $A \in \mathbb{R}^{d\times d}$, its trace, the smallest and the largest eigenvalues are denoted by $\mathrm{tr}(A)$, $\lambda_{\min}(A)$ and $\lambda_{\max}(A)$ respectively. A standard indicator function is denoted by $\mathbf{1}\{\cdot\}$, i.e., $\mathbf{1}\{x\in A\} = 1$ if $x\in A$ and $0$ if $x\notin A$. For any deterministic sequences $\{x_n\}_{n \ge 1}$ and $\{r_n\}_{n \ge 1}$, we denote $x_n = O(r_n)$ if there exists a universal constant $C>0$ such that $|x_n| \le C|r_n|$ for all $n$ larger than some $N$. Similarly, we denote $x_n = O_P(r_n)$ if, for any $\varepsilon>0$, there exists a constant $C_\varepsilon>0$ such that $\mathbb{P}(|x_n| \le C_\varepsilon|r_n|) \le \varepsilon$ for all $n$ larger than some $N_\varepsilon$.
We denote $x_n = o(r_n)$ if $x_n/r_n \to 0$ and $x_n = o_p(r_n)$ if $x_n/r_n \overset{p}{\to} 0$ where $\overset{p}{\to}$ denotes convergence in probability.

\section{Construction for General Stochastic Optimization}\label{sec:general}
We begin with the construction for general stochastic optimization problems. Let $N \ge 1$ denote the total sample size, and partition the index set as: 
\begin{equation}\label{eq:k-partition}
    I_1 = \{1, \ldots, n_1\} \quad \textrm{and} \quad I_2 = \{n_1+r+1, \ldots, N\},
\end{equation}
where $n_1, r \ge 0$ and $n_2 := |I_2|\ge1$ such that $N = n_1 + n_2 + r$. Write $D_\ell = \{Z_i : i \in I_\ell\}$ for $\ell=1,2$, with induced marginal laws $P^{\ell} := \mathcal L(D_\ell)$. From $D_1$, we construct an estimator $\widehat \theta_1 := \widehat \theta_1(D_1) \in \Theta$ of $\theta(P^N)$, without imposing any restrictions on its choice. The population objective evaluated at $P^2$ and its estimator based on $D_2$ as
\begin{equation}
    \M_2(\theta) := \M(\theta, P^2) \quad \textrm{and} \quad\widehat\M_2(\theta) := \widehat\M_2(\theta; D_2).
\end{equation}
As discussed in \Cref{sec:introduction}, an ideal (albeit unactionable) confidence set is 
\begin{equation}\label{eq:oracle}
    \widetilde{\mathrm{CI}} := \left\{\theta\in\Theta:\, \mathbb{M}_2(\theta) - \mathbb{M}_2(\widehat{\theta}_1) \le 0\right\}.
\end{equation}
A natural approximation is
\begin{equation}\label{eq:anti-conservative-confidence-set}
\widehat{\mathrm{CI}}^\dagger_{N} := \left\{\theta\in\Theta:\, \widehat{\mathbb{M}}_2(\theta) - \widehat{\mathbb{M}}_2(\widehat{\theta}_1) \le 0\right\}.
\end{equation}
We now state the first validity result for $\widehat{\mathrm{CI}}^\dagger_{N}$. To this end, we introduce several key objects, beginning with the $\beta$-mixing coefficient following \citet{bradley2005basic}.
\begin{definition}[$\beta$-mixing coefficient]Given the data splits $D_1$ and $D_2$ with induced marginals $P^1, P^2$ and joint law $P^{1,2} := \mathcal{L}(D_1,D_2)$, the $\beta$-mixing coefficient is
\begin{equation}\label{eq:beta-mixing}
\beta(n_1, r) = d_{\mathrm{TV}}(P^{1,2}, P^{1}\otimes P^{2})
\end{equation}
where $d_{\mathrm{TV}}(\cdot, \cdot)$ denotes total variation distance and $r$ corresponds to the gap size in \eqref{eq:k-partition}. We omit the dependence on $n_1$ and write $\beta(r) := \beta(n_1, r)$.
\end{definition}
The coefficient $\beta(r)$ quantifies the dependence between $D_1$ and $D_2$. Under independence, $\beta(r) = 0$ for all $r \ge 0$. Under $m$-dependence \citep{hoeffding1994central}, $\beta(r) = 0$ for all $r \ge m$. Many practically relevant processes satisfy $\beta(r) \to 0$ as $r \to \infty$, including certain Markov chains \citep{blum1963strong}, linear AR models with absolutely continuous innovations \citep{chanda1974strong} (Bernoulli innovations notably fail to be strong mixing \citep{andrews1984non}), near epoch sequences \citep{ibragimov1962some}, and weakly physically dependent processes \citep{wu2005nonlinear, heinrichs2026note}. See \cite{bradley2005basic, kiessler2009weak} and reference therein for further examples interacting interacting particle systems. 

Next, for $\theta\in\Theta$, we define
\begin{equation}\label{eq:notation-MSE-curvature}
    \begin{split}
        \mathbb{V}_{2}(\theta) &:= \mathbb{E}_{P^2}[|(\widehat{\mathbb{M}}_2 - \mathbb{M}_2)(\theta) - (\widehat{\mathbb{M}}_2 - \mathbb{M}_2)(\theta(P^N))|^2] \quad\textrm{and}\\\mathbb{C}_{2}(\theta) &:= \mathbb{M}_2(\theta) - \mathbb{M}_2(\theta(P^N)).
    \end{split}
\end{equation}
The quantity $\mathbb{V}_{2}(\theta)$ is the mean squared error of the estimated optimization objective at two points $\theta, \theta(P^N)$. The estimator $\widehat{\mathbb{M}}_2$ is permitted to be biased. The quantity $\mathbb{C}_{2}(\theta)$ is known as the curvature, which quantifies the ``difficulty'' of estimating $\theta(P^N)$. Both are defined for non-random $\theta \in\Theta$; when evaluated at the random point $\widehat{\theta}_1\in\Theta$, they become random variables. For brevity, write
\begin{equation}\label{eq:variance-curvature-ratio}
    \begin{split}
\widehat{\mathbb{V}}_{2}=\mathbb{V}_{2}( \widehat\theta_1), \quad \widehat{\mathbb{C}}_{2}=\mathbb{C}_{2}(\widehat\theta_1) \quad \textrm{and} \quad \widehat\Delta_2 = \widehat{\mathbb{C}}_{2}/\widehat{\mathbb{V}}_{2}^{1/2}.
    \end{split}
\end{equation}
\begin{remark}
   The evaluation of $\M(\theta, P^{2})$ at $P^{2} \not \in \mathcal{P}^N$ is an abuse of notation. This is justified under the assumption that $\theta(P^N) = \theta(P_i)$ for all $i$, so $\mathbb{C}_{2}(\theta)$ continues to reflect the curvature relative to the population target $\theta(P^N)$. 
\end{remark}
The first validity result is as follows:
\begin{theorem}\label{thm:coverage-anti-conservative-confidence-set}
The confidence set $\widehat{\mathrm{CI}}_N^{\dagger}$ in \eqref{eq:anti-conservative-confidence-set} satisfies
\[
\mathbb{P}_{P^N}\!\left(\theta(P^N) \notin \widehat{\mathrm{CI}}_{N}^{\dagger}\right)
\;\le\;
\mathbb{E}_{P^{1}}\!\left[\min\left\{\frac{1}{\widehat\Delta_2^2}, 1\right\}
\right]
+
\beta(r).
\]
Moreover, if $\ratio^2 \overset{p}{\to} \infty$ and $\beta(r) = o(1)$ uniformly over all $P^N\in\mathcal{P}^N$, then $\widehat{\mathrm{CI}}_N^{\dagger}$ is asymptotically uniformly valid at confidence level $1$.
\end{theorem}
% When $\widehat{\mathbb{M}}_2$ is unbiased , the bound in \Cref{thm:coverage-anti-conservative-confidence-set} can be sharpened.
\begin{theorem}\label{thm:coverage-anti-conservative-confidence-set-unbiased}
Suppose $ \widehat{\mathbb{M}}_2(\widehat{\theta}_1)-\widehat{\mathbb{M}}_2(\theta(P^N)) $ is an unbiased estimator of $\widehat{\mathbb{C}}_{2}$, in the sense that  $\E_{P^{2}}[\widehat{\mathbb{M}}_2(\theta)]=\mathbb{M}_2(\theta)$ for all $\theta \in \Theta$. Then, the confidence set $\widehat{\mathrm{CI}}_N^{\dagger}$ in \eqref{eq:anti-conservative-confidence-set} satisfies
\[
\mathbb{P}_{P^N}\!\left(\theta(P^N) \notin \widehat{\mathrm{CI}}_{N}^{\dagger}\right)
\;\le\;
\mathbb{E}_{P^1}\left[\frac{1}{1+\widehat\Delta^2_2}
\right]
+
\beta(r).
\]
\end{theorem}

The proofs of both results appear in \Cref{appsec:proof-of-coverage-anti-conservative}. Both results are stated under following generality: uniqueness of $\theta(P^N)$ is not assumed, the observations need not be independent, and no specific structure is imposed on the optimization objective.

The miscoverage bound depends on the ratio $\ratio^2 = \widehat{\mathbb{C}}_{2}^2/\widehat{\mathbb{V}}_{2}$, which increases with curvature $\mathbb{C}_{2}^2$ and decreases with estimation error $\widehat{\mathbb{V}}_{2}$. Because $\widehat{\mathrm{CI}}_N^{\dagger}$ is constructed without reference to a nominal level, it provides an agnostic bound. When $\ratio^2 \to \infty$ in probability, the set is valid at confidence level 1. Two illustrative examples follow.
\begin{example}[U-statistics]\label{example-u-statistics}
    Let $Z_1, \ldots, Z_{N}$ be IID observations from $\mathcal{N}(\mu, \sigma^2)$ with unknown $\sigma^2$. Consider inference for $\mu^2$, which admits the representation
    \begin{equation}
        \mu^2 = \argmin_{\theta \in \mathbb{R}}\, \E_{P^2}[(Z_1Z_2-\theta)^2].\nonumber
    \end{equation}
    Set $D_1 = \{Z_{1}, \ldots, Z_{n_1}\}$ and $D_2 =\{Z_{n_1+1}, \ldots, Z_{N}\}$ with $|D_2|=n_2$. The optimization objective can be estimated unbiasedly from $D_2$ via the U-statistic
    \begin{equation}
    \widehat{\mathbb{M}}_2(\theta) = \binom{n_2}{2}^{-1}  \sum_{n_1+1 \le i < j \le N} (Z_i Z_j - \theta)^2.\nonumber
    \end{equation}
    Direct calculation yields 
    \begin{equation*}
        \ratio^2 = (\widehat{\theta}_1 - \mu^2)^2\left(\frac{8\sigma^4}{n_2(n_2-1)} + \frac{16\mu^2\sigma^2}{n_2}\right)^{-1},
    \end{equation*}
    which reflects different behaviors depending on whether $\mu = 0$ or $\mu \neq 0$. In particular, $\ratio \overset{p}{\to}\infty$ whenever 
    \begin{equation*}
        \min\left\{\frac{n_2|\widehat{\theta}_1 - \mu^2|}{\sigma^2}, \frac{n_2^{1/2}|\widehat{\theta}_1 - \mu^2|}{|\mu| \sigma}\right\} \overset{p}{\to} \infty.
    \end{equation*}
    This is satisfied, for instance, $(\widehat{\theta}_1 - \mu^2)^2 \ge c > 0$ for some constant $c$, and $n_2 \to \infty$ with $\mu$ and $\sigma$ fixed. \Cref{supp:example1} provides the explicit distribution of $\ratio$ for the constant and the U-statistic estimators, from which the upper bound of \Cref{thm:coverage-anti-conservative-confidence-set-unbiased} can be evaluated analytically.
\end{example}

\begin{example}[Super-efficient initial estimator]\label{example:hodges-estimator}
Let $Z_1, \ldots, Z_{N}$ be IID observations from $\mathcal{N}(\mu, \sigma^2)$ with unknown $\sigma^2$. Consider inference for $\mu$, which corresponds to
    \begin{equation}
        \mu = \argmin_{\theta \in \mathbb{R}}\, \E_{P^2}[(Z_1-\theta)^2].\nonumber
    \end{equation}
    Using $D_2 = \{Z_{n_1+1}, \ldots, Z_{N}\}$, the objective is estimated unbiasedly by
    \begin{equation}
        \widehat{\mathbb{M}}_2(\theta) = \frac{1}{n_2} \sum_{n_1+1 \le i \le N} (Z_i - \theta)^2.\nonumber
    \end{equation}
    Direct calculation gives $\ratio^2 = n_2(\widehat{\theta}_1 - \mu)^2/(4\sigma^2)$. As the initial estimator $\widehat{\theta}_1$ based on $D_1$, take Hodges' estimator $\widehat \theta_1 := \bar{Z}_{n_1} \mathbf{1}\{|\bar{Z}_{n_1} |\ge {n_1}^{-1/4}\}$ where $ \bar{Z}_{n_1} = {n_1}^{-1}\sum_{i=1}^{n_1} Z_i$. 
    When $\mu=0$, Hodges' estimator satisfies ${n_1}(\widehat{\theta}_1 - \mu)^2 \overset{p}{\to} 0$, exhibiting super-efficiency. In this case, $\ratio^2 \overset{p}{\to} 0$. When $\mu = {n_1}^{-1/4}/2$, one can verify that ${n_1}(\widehat{\theta}_1 - \mu)^2 \overset{p}{\to} \infty$, and consequently  $\ratio^2 \overset{p}{\to} \infty$. Hence, $\widehat{\mathrm{CI}}_N^{\dagger}$ is asymptotically valid at confidence level 1. The same conclusion holds when $\widehat \theta_1$ is inconsistent. Then we also have $\ratio^2 \overset{p}{\to} \infty$ regardless of $\mu$. Again, \Cref{supp:example2} provides the explicit distribution of $\ratio$ for the constant estimator, sample mean and Hodges' estimator, from which the upper bound of \Cref{thm:coverage-anti-conservative-confidence-set-unbiased} can be evaluated analytically.
\end{example}
The two examples illustrate when \Cref{thm:coverage-anti-conservative-confidence-set-unbiased} yields an informative bound. The confidence set achieves zero miscoverage whenever $\ratio \to \infty$ in probability. This happens, for instance, when the initial estimator is inconsistent (the constant estimators in \Cref{example-u-statistics} and \Cref{example:hodges-estimator}), or when it converges sufficiently slowly (Hodges' estimator in a certain neighborhood of $\mu = 0$ in \Cref{example:hodges-estimator}). The bounds become uninformative when $\ratio \to 0$ in probability, as with the Constant or Hodges' estimator at $\mu =0$. Crucially, however, an uninformative upper bound does not imply that the confidence set itself is uninformative.

\Cref{fig:example1} and \Cref{fig:example3} display the empirical miscoverage of $\widehat{\mathrm{CI}}_{N}^{\dagger}$ as well as the analytical upper bound of \Cref{thm:coverage-anti-conservative-confidence-set-unbiased} for both examples. The details of the numerical experiments can be found in \Cref{supp:example1-sim} and \Cref{sup:numerical-example3}. A visible gap between the bound and the empirical performance confirms that the bound is conservative. This is because \Cref{thm:coverage-anti-conservative-confidence-set} and \Cref{thm:coverage-anti-conservative-confidence-set-unbiased} are proved for general stochastic optimization. Once additional structure is imposed on the objective, the same confidence set \eqref{eq:anti-conservative-confidence-set} remains non-trivial for any choice of $\widehat{\theta}_1$. The bound from the forthcoming \Cref{thm:coverage-anti-conservative-confidence-set-empirical-risk}, visible in \Cref{fig:example3}, tracks the empirical miscoverage far more closely. 
\begin{figure}
\centering
\begin{subfigure}{0.8\textwidth}
  \centering
  \includegraphics[width=0.7\linewidth]{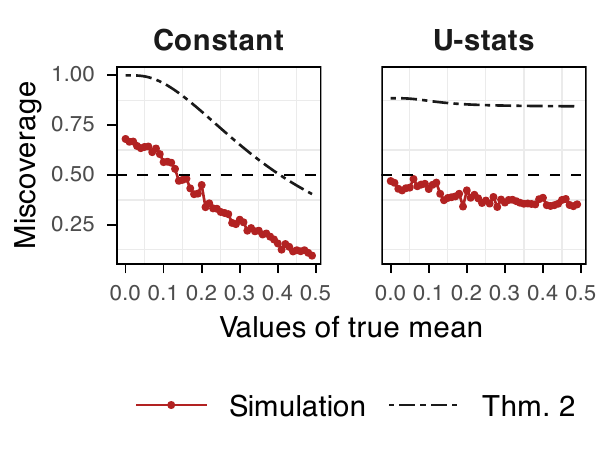}
\captionsetup{width=0.9\textwidth}
  \caption{Estimated miscoverage of $\widehat{\mathrm{CI}}_N^\dagger$ in \eqref{eq:anti-conservative-confidence-set} for the U-statistics problem of \Cref{example-u-statistics}. The $X$-axis displays the true mean $\mu$ and the $Y$-axis displays the empirical miscoverage over $1000$ replications. The performances of two estimators, the constant estimator at zero and the U-statistics estimator, are shown in red. The analytical upper bound of \Cref{thm:coverage-anti-conservative-confidence-set-unbiased} is shown as a dashed line.}  \label{fig:example1}
\end{subfigure}  
\begin{subfigure}{0.8\textwidth}
  \centering
  \includegraphics[width=\linewidth]{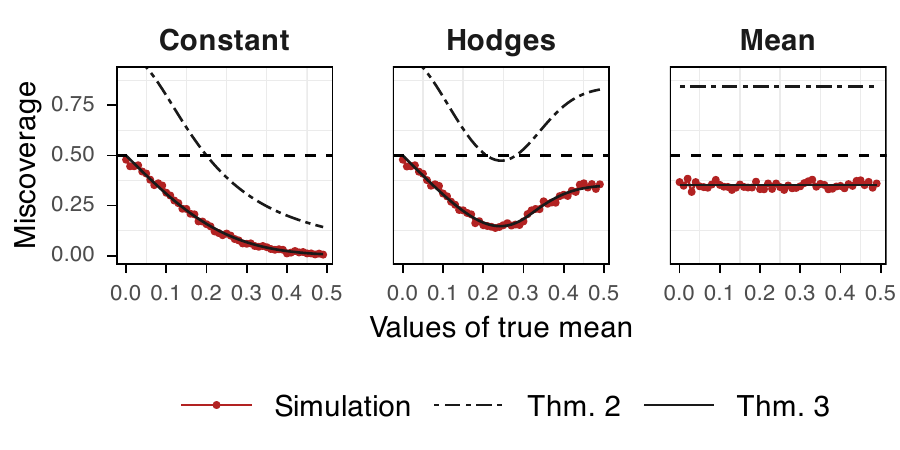}
  \captionsetup{width=0.9\textwidth}
  \caption{Estimated miscoverage probability of the confidence set $\widehat{\mathrm{CI}}_N^\dagger$ in \eqref{eq:anti-conservative-confidence-set} for the normal mean problem of \Cref{example:hodges-estimator} and \Cref{example:hodge-revisited}.  The $X$-axis displays the true mean $\mu$ and the $Y$-axis displays the empirical miscoverage over $1000$ replications. The performances of three estimators, the constant estimator at zero, Hodges' estimator and the sample mean, are shown in red. The analytical upper bound of \Cref{thm:coverage-anti-conservative-confidence-set-unbiased} is shown as a dashed line; the upper bound of \Cref{thm:coverage-anti-conservative-confidence-set-empirical-risk} is shown as a solid line.}\label{fig:example3}
\end{subfigure}
\caption{The empirical and theoretical miscoverage of the confidence set.}
\end{figure}
\section{Confidence Sets for M-estimation}\label{sec:general-M-estimation}
More refined results become available when the objective takes the form of an expected loss function, a setting commonly referred to as \emph{M-estimation} or \emph{empirical risk minimization}. Let $m_\theta : \mathcal{Z} \mapsto \mathbb{R}$ be a measurable function, indexed by $\theta \in \Theta$, and define
\begin{equation}\label{eq:m-estimation-definition}
    \M_2(\theta) := \frac{1}{n_2}\sum_{i\in I_2} \mathbb{E}_{P_i}[m_\theta(Z_i)] \quad \textrm{and} \quad \widehat{\mathbb{M}}_2(\theta) := \frac{1}{n_2}\sum_{i \in I_2} m_\theta(Z_i).
\end{equation}
For instance, taking $Z=(X^\top,Y)$ and $m_{\theta}(Y,X) := (Y-\theta^\top X)^2$ with $\Theta = \mathbb{R}^d$ corresponds to linear regression (without assuming linearity), while $m_{\theta}(Z) := -\log p(Z; \theta)$ for a (possibly misspecified) parametrized family of likelihood $p(Z; \theta)$ yields maximum likelihood estimation. This framework also includes more general nonparametric or constrained problems.

\subsection{Validity under Independence}\label{sec:validity-under-ind}
Assume that $Z_1, \ldots, Z_N$ are independent. For $i \in I_2$, define the centered differences 
\begin{equation}\label{eq:centered-xi}
    \begin{split}
 \widehat\xi_i = m_{\widehat\theta_1}(Z_i)-m_{\theta(P^N)}(Z_i) -\mathbb{E}_{P_i}[m_{\widehat\theta_1}(Z)-m_{\theta(P^N)}(Z)\mid D_1].
    \end{split}
\end{equation}
% This additional structure allows for more refined validity guarantees.
\begin{theorem}\label{thm:coverage-anti-conservative-confidence-set-empirical-risk}
The confidence set $\widehat{\mathrm{CI}}_N^{\dagger}$ in \eqref{eq:anti-conservative-confidence-set} satisfies
\begin{equation}\label{eq:standard-non-uniform-BE}
    \begin{split}
&\mathbb{P}_{P^N}\!\left(\theta(P^N) \notin \widehat{\mathrm{CI}}_{N}^{\dagger}\right)
\;\le\;
\mathbb{E}_{P^{1}}[1-\Phi(\ratio)] \\
&\quad  +\E_{P^1}\left[\min\left\{1, C\,\sum_{i\in I_2} \mathbb{E}_{P_i}\left[\frac{|\widehat\xi_i|^2}{n_2^2\widehat{\mathbb{V}}_{2}(1 + \ratio)^2}\min\left\{1,\,\frac{|\widehat \xi_i|}{n_2\widehat{\mathbb{V}}_{2}^{1/2}(1 + \ratio)}\right\}\bigg| D_1\right]\right\}\right],
    \end{split}
\end{equation}
where $C > 0$ is a universal constant. Whenever the second term vanishes uniformly over $P^N \in \mathcal{P}^N$, the set $\widehat{\mathrm{CI}}_{N}^{\dagger}$ is asymptotically uniformly valid at level $1/2$ since $\ratio \ge 0$ almost surely.
\end{theorem}
The proof appears in \Cref{appsec:proof-of-M-estimation} and uses a non-uniform Berry-Esseen bound (e.g., Theorem 2.1 of \citet{chen2001non}). The ratio $\ratio$ quantifies the degree of conservativeness. When $\ratio \to 0$  in probability, the miscoverage approaches 1/2 exactly; when $\ratio \to \infty$ in probability, the set becomes increasingly conservative and the miscoverage tends to zero. The conditions under which the remainder term vanishes also depend on $\ratio$. The least favorable case is when $\ratio \to 0$ in probability. In such case, the assumption on $\widehat\xi_i$ reduces to the classical Lindeberg-Feller conditions (conditional on $D_1$). When $\ratio \to \infty$ in probability, finite $\widehat{\mathbb{V}}_2$ suffices and the remainder decays faster than in standard Berry-Esseen bounds.

\begin{example}[Super-efficient initial estimator, Revisited]\label{example:hodge-revisited}
    Consider the same setting as \Cref{example:hodges-estimator}. \Cref{thm:coverage-anti-conservative-confidence-set-empirical-risk} implies that 
    \begin{equation}
    \begin{split}
&\mathbb{P}_{P^N}\!\left(\theta(P^N) \notin \widehat{\mathrm{CI}}_{N}^{\dagger}\right)
\;\le\;
\E_{P^1}\left[1-\Phi(\ratio) + \frac{C\sigma^{3/2}}{n_2^{1/2}(1+\ratio)^3}\right].
    \end{split}
\end{equation}
Taking $\widehat{\theta}_1$ to be Hodges' estimator and combining with the analysis of \Cref{example:hodges-estimator}, the miscoverage of $\widehat{\mathrm{CI}}_{N}^{\dagger}$ exhibits three regimes: it approaches $1/2$ when $\mu \approx 0$, where super-efficiency causes 
$\ratio \overset{p}{\to} 0$; it tends to zero when $\mu \approx n_1^{-1/4}$ where $\ratio \overset{p}{\to} \infty$; for for large $\mu$, where Hodges' estimator becomes comparable with the sample mean, the miscoverage becomes similar to that based on the sample mean as $\widehat{\theta}_1$. 
\end{example}

\Cref{fig:example3} compares the analytic upper bound of \Cref{thm:coverage-anti-conservative-confidence-set-empirical-risk} with the empirical miscoverage. The bound exactly tracks the empirical performance with visibly negligible approximation error. This confirms that $1-\Phi(\ratio)$ characterizes the miscoverage with high precision. Taken together, $\widehat{\mathrm{CI}}_{N}^{\dagger}$ is asymptotically valid at level $1/2$ for any $\widehat{\theta}_1$ and any $P^N$ under which the remainder term vanishes. The miscoverage approaches $1/2$ when the estimator converges ``too fast", and tends to zero when the estimator converges slowly or is inconsistent. For estimators satisfying $\ratio = O_P(1)$, the confidence set is valid at level $1/2$ but may in practice be conservative. The bound of \Cref{thm:coverage-anti-conservative-confidence-set-unbiased} does not capture this conservativeness precisely, whereas \Cref{thm:coverage-anti-conservative-confidence-set-empirical-risk} aligns with the observed miscoverage almost exactly.

\begin{remark}[Extensions beyond M-estimation]
The result of this section extends to optimization objectives of a more complex form than \eqref{eq:m-estimation-definition}. Notable examples include criteria defined as U-statistics or higher-order U-statistics \citep{bose2018u, diciccio2022clt} and U-quantile functionals \citep{choudhury1988generalized}. A special case was already studied in \Cref{example-u-statistics}. More refined validity results in these settings follow from the same proof strategy, with the Berry-Esseen bound replaced by the appropriate analogue for the statistic of interest; see, for instance, \citet{zhao1983non,bentkus1997edgeworth, wang2002non} and \citet{chen2007normal}. As evident in \Cref{fig:example1}, the miscoverage can exceed $1/2$ for U-statistics. This occurs because the normal approximation is valid only under non-degeneracy; the degenerate case requires separate treatment.
\end{remark}

\subsection{Validity under Dependence}
The proof of \Cref{thm:coverage-anti-conservative-confidence-set-empirical-risk} relies crucially on a non-uniform Berry--Esseen bound for independent observations. An extension to dependent data is possible through a martingale approximation argument \citep[Section 2]{Wu20014martingale}, without imposing any specific dependence structure on the observations.

Suppose the dependent observations $\widehat \xi_1, \ldots, \widehat \xi_N$ are split according to
\begin{equation}\label{eq:martingale-split}
    I_1 := \{1, \ldots, n_1\} \quad \textrm{and} \quad I_2 := \{n_1+1, \ldots, N\} \quad \textrm{with} \quad |I_2| = n_2.
\end{equation}
Let $\mathcal{H}_0$ be the $\sigma$-algebra generated by $\{\widehat\xi_i\}_{i \in I_1}$, and for each $k \in I_2$, let $\mathcal{H}_k$ be the $\sigma$-algebra generated by $\mathcal{H}_0$ and $\{\widehat \xi_i\}_{i = n_1 + 1}^k$, so that $\mathcal{H}_0 \subseteq \ldots \subseteq \mathcal{H}_N$ is a filtration supporting the full sequence. The estimator $\widehat \theta_1$ is $\mathcal{H}_0$-measurable by construction.
\begin{theorem}\label{thm:coverage-martingale}
Recall $\widehat\xi_i$ defined in \eqref{eq:centered-xi}. Define the martingale approximation 
    \begin{equation}\label{eq:martingale-approx}
        \widetilde{\xi}_i = \sum_{r\in I_2}(\E[\widehat\xi_r|\mathcal{H}_i]-\E[\widehat\xi_r|\mathcal{H}_{i-1}]) \quad \textrm{for} \quad i \in I_2.
    \end{equation}
    For $\delta \in (0, \infty)$, define 
    \begin{equation}
        \begin{split}
        L_{2\delta}&:= \sum_{i \in I_2} \E_{P_i}\left[\left|\frac{\widetilde\xi_i}{n_2\widehat{\mathbb{V}}_2^{1/2}}\right|^{2+2\delta} \,|\, D_1\right] \quad \textrm{and} \\
        M_{2\delta}&:= \E_{P^1}\left(\left|\sum_{i \in I_2} \E_{P_i}\left[\frac{\widetilde\xi_i^2}{n_2^2\widehat{\mathbb{V}}_2} \mid \mathcal{H}_{i-1}\right]-1\right|^{1+\delta} \mid D_1 \right)\quad.
    \end{split}
    \end{equation}
    The confidence set $\widehat{\mathrm{CI}}_N^{\dagger}$ in \eqref{eq:anti-conservative-confidence-set} satisfies
    \begin{equation}\label{eq:martingale-coverage-bound-general}
        \begin{split}
           &\mathbb{P}_{P^N}\!\left(\theta(P^N) \notin \widehat{\mathrm{CI}}_{N}^{\dagger}\right) \\
           &\quad \le \mathbb{E}_{P^1}[1-\Phi(\ratio)] + \mathbb{E}_{P^1}\left[\min\left\{1,  C_\delta \frac{(L_{2\delta} +  M_{2\delta})^{1/(3+2\delta)}}{1+|\ratio|^{2+2\delta}}\right\}\right],
        \end{split}
    \end{equation}
    where $C_\delta > 0$ is a constant depending only on $\delta$.
\end{theorem}
The proof of \Cref{thm:coverage-martingale} appears in \Cref{appsec:proof-of-coverage-martingale}. The original sequence $\{\widehat \xi_i\}_{i \in I_2}$ is replaced by the martingale difference sequence $\{\widetilde \xi_i\}_{i \in I_2}$ defined in \eqref{eq:martingale-approx}, to which the non-uniform Berry-Esseen bound of \citet{haeusler1988nonuniform} is applied. No specific dependence structure is imposed on the original sequence. The term $L_{2\delta}$ controls the $(2+2\delta)$-th moment of the normalized martingale increments; $M_{2\delta}$ quantifies the $L^{1+\delta}$ deviation of the quadratic variation from the marginal variance. 

Compared with \Cref{thm:coverage-anti-conservative-confidence-set-empirical-risk}, which is its counterpart under independence, \Cref{thm:coverage-martingale} also shows that the miscoverage probability remains governed by $\ratio \ge 0$ once the remainder terms are negligible under suitable moment conditions. In particular, the confidence set is asymptotically valid at level $1/2$ whenever the remainder terms vanish, extending the conclusions of \Cref{sec:validity-under-ind} to general dependent observations.
\begin{remark}[Martingale approximation] For general dependent observations,
\begin{equation}
        \widetilde{\xi}_i = \widehat\xi_i - \E[\widehat\xi_i|\mathcal{H}_{i-1}] +\sum_{r=i+1}^N(\E[\widehat\xi_r|\mathcal{H}_i]-\E[\widehat\xi_r|\mathcal{H}_{i-1}]) \quad \textrm{for} \quad i \in I_2.
    \end{equation}
    The second term vanishes, for instance, when $\{\widehat \xi_i\}_{i \in I_2}$ is itself a martingale difference sequence. This follows since for $r > i$,
    \begin{equation*}
        \E[\widehat\xi_r|\mathcal{H}_{i}] = \E[\E[\widehat\xi_r|\mathcal{H}_{r-1}] \mid \mathcal{H}_{i}] = 0 \quad \textrm{whenever} \quad \E[\widehat\xi_r|\mathcal{H}_{r-1}] = 0.
    \end{equation*}
\end{remark}
\begin{remark}[Sharpness and faster rates]
The bound \eqref{eq:martingale-coverage-bound-general} is sharp in its dependence on $L_{2\delta}$ and $M_{2\delta}$; see Section 3 of \citet{haeusler1988rate} for a matching lower bound based on an example. Recovering a $\sqrt{n_2}$-rate comparable to the independent case in  \eqref{eq:standard-non-uniform-BE} requires additional assumptions, under which sharper non-uniform Berry-Esseen bounds are available; see, for instance, \citet{fan2017non}
\end{remark}

\begin{remark}[Other dependence structures]
The approach underlying \Cref{thm:coverage-martingale} extends to other dependence structures by substituting the martingale Berry--Esseen bound with an analogous result tailored to the structure of interest; see, for instance, \citet{chen2004normal, hormann2009berry, hafouta2022non, liu2023wasserstein}.
\end{remark}
\subsection{Summary of Validity Results}
The validity results for $\widehat{\mathrm{CI}}_N^\dagger$ established thus far are summarized in \Cref{tab:base-confidence-sets}.

\begin{table}[ht]
\centering
\renewcommand{\arraystretch}{1.4}
\begin{tabular}{llcccp{4.5cm}}
\toprule
Optimization & Dependence & Theorem & Sig. Level & Assumption \\
\midrule
General
  & $\beta$-mixing
  & Thm.~\ref{thm:coverage-anti-conservative-confidence-set}
  & Conservative
  & Finite variance\\
General
  & $\beta$-mixing
  & Thm.~\ref{thm:coverage-anti-conservative-confidence-set-unbiased}
  & Conservative
  &\thead{Finite variances; \\Unbiased $\M_2(\cdot)$}\\
M-estimation
  & Indep.
  & Thm.~\ref{thm:coverage-anti-conservative-confidence-set-empirical-risk}
  & $1/2$
  & Lindeberg--Feller \\
% M-estimation
%   & General
%   & Thm.~\ref{thm:general-dependence}
%   & $1/2$
%   & \thead{Martingale approximation\\$(2+\delta)$-th moment; \\Consistent quadratic variation} \\
M-estimation
  & General
  & Thm.~\ref{thm:coverage-martingale}
  & $1/2$
  & \thead{$(2+\delta)$-th moment; \\Consistent quadratic variation} \\
\bottomrule
\end{tabular}
\caption{Validity results for the confidence set $\widehat{\mathrm{CI}}_N^\dagger$
defined in \eqref{eq:anti-conservative-confidence-set}. ``Conservative''
indicates that the miscoverage tends to zero as
$\ratio \to \infty$ in probability; ``$1/2$'' indicates
asymptotic validity at level $1/2$ whenever the remainder terms vanish.
The Lindeberg--Feller condition in the third row is sufficient but not
necessary; see the discussion following \Cref{thm:coverage-anti-conservative-confidence-set-empirical-risk}.}
\label{tab:base-confidence-sets}
\end{table}

For general stochastic optimization under $\beta$-mixing, the set $\widehat{\mathrm{CI}}_N^\dagger$ is asymptotically conservative as long as $\ratio
\to \infty$ in probability. This is achieved when $\widehat \theta_1$ is inconsistent or converges at a sub-optimal rate (\Cref{example-u-statistics} and \Cref{example:hodges-estimator}). For M-estimation problems of the form \eqref{eq:m-estimation-definition} under non-identical distributions, the same set $\widehat{\mathrm{CI}}_N^\dagger$ achieves level $1/2$ for any estimator, including super-efficient ones (\Cref{example:hodge-revisited}),  provided the remainder term vanishes. The dependent extension in \Cref{thm:coverage-martingale} accommodates arbitrary dependence structures with no further restriction. Notably, \Cref{thm:coverage-anti-conservative-confidence-set,thm:coverage-anti-conservative-confidence-set-unbiased,thm:coverage-anti-conservative-confidence-set-empirical-risk,thm:coverage-martingale} are all dimension-free and proved without any reference to the dimension or complexity of $\Theta$.
\section{Finer Control over the Coverage Level}\label{sec:coverage-slpha}
The results of \Cref{sec:general-M-estimation} establish that the set $\widehat{\mathrm{CI}}_N^\dagger$ defined in \eqref{eq:anti-conservative-confidence-set} satisfies the honest validity guarantee \eqref{eq:asymptotic-validity} under both independent and dependent observations. In particular, miscoverage is bounded by $1/2$ under M-estimation, and tends to zero when $\ratio \to \infty$ in probability. While the set is thus valid across a range of levels, it does not allow the practitioner to prescribe a desired significance level $\alpha$. This section describes how to construct confidence sets at any prespecified level.
\subsection{Data-splitting Approach}\label{sec:data-splitting}
Suppose a miscoverage bound of the following form:
\begin{equation}\label{eq:median-valid}
\mathbb{P}_{P^N} \left(\theta(P^N) \notin \widehat{\mathrm{CI}}_{N}^{\dagger}\right) \le p + \mathfrak{R}_{N, P^N},
\end{equation}
is available for some $p \in (0,1)$ and  $\mathfrak{R}_{N, P^N} \ge 0$. Under M-estimation, \Cref{thm:coverage-anti-conservative-confidence-set-empirical-risk,thm:coverage-martingale} establish \eqref{eq:median-valid} with $p=1/2$ for both independent and dependent observations.

Towards  achieving a target level $\alpha < p$, consider the following construction. For integers $B \ge 1$ and $r \ge 0$, let $S_1, G_1, S_2 \ldots, G_{B-1}, S_B$ be a partition of $\{1,2,\ldots,N\}$ with $|G_\ell| = r$ for $1\le \ell \le B-1$, where each $G_\ell$ serves as a gap of $r$ indices separating consecutive bins $S_\ell$ and $S_{\ell + 1}$. Assume that $Z_1,\ldots, Z_N$ satisfies
\begin{equation*}
    \beta(n, r) \le \beta^\dagger(r)  \quad \textrm{for all} \quad n\ge 1,
\end{equation*}
that is, the $\beta$-mixing coefficient as defined in \eqref{eq:beta-mixing} only depends on the gap size $r$. 
\begin{theorem}\label{thm:hulc-like}
Suppose each set $\widehat{\mathrm{CI}}^\dagger_{\ell}$, constructed from $\{Z_i \,:\, i \in S_\ell\}$ as in \eqref{eq:anti-conservative-confidence-set}, satisfies \eqref{eq:median-valid} with same $p \in (0,1)$ and remainder $\mathfrak{R}_{N_0, P^N}$ where $N_0 = \min_{\ell}|S_\ell|$. Then for any $\alpha \in (0, p)$, setting $B \ge \lceil \log_{p}(\alpha) \rceil$, the union $\widehat{\mathrm{CI}}^{\mathtt{DS}}_{N, \alpha} = \bigcup_{\ell=1}^B \widehat{\mathrm{CI}}^\dagger_{\ell}$ satisfies
    \begin{equation}
    \begin{split}
        &\mathbb{P}_{P^N}\left(\theta(P^N) \notin \widehat{\mathrm{CI}}^{\mathtt{DS}}_{N, \alpha}\right) \le \alpha (1+ p^{-1}\mathfrak{R}_{N_0, P^N})^{B} + (B-1)\beta^\dagger(r).
    \end{split}
    \end{equation}
    Whenever $\mathfrak{R}_{N_0, P^N} \to 0$ and $\beta^\dagger(r) \to 0$ uniformly over $P^N \in \mathcal{P}^N$, the set $\widehat{\mathrm{CI}}^{\texttt{DS}}_{N, \alpha}$ is asymptotically uniformly valid at level $\alpha$.
\end{theorem}
The proof appears in \Cref{supsec:coverage-proofs}. For $p = 1/2$ and common significance levels $\alpha \in \{0.1, 0.05, 0.01\}$, this procedure requires $B \in \{4, 5, 7\}$ bins respectively. The admissible range of $\alpha$ is implicitly restricted by the sample size, since we need $N \ge B \ge \lceil \log_{1/2}(\alpha)\rceil$. A closely related procedure was proposed by \citet{kuchibhotla2024hulc}. 

\Cref{thm:hulc-like} is not limited to M-estimation: the result holds for general stochastic optimization whenever \eqref{eq:median-valid} is available with an appropriate $p$. As \Cref{example-u-statistics} illustrates, the constant estimator in the U-statistics problem can yield miscoverage exceeding $1/2$ due to degeneracy, so a larger $p$ would be required there.

% For each $B$ bin, define $\overline D_\ell = \{Z_i \,:\, i \in S_\ell\}$, and construct the confidence set \eqref{eq:anti-conservative-confidence-set}, and denote it $\widehat{\mathrm{CI}}^\dagger_{\ell}$. The final confidence set at level $\alpha$ is the union $\widehat{\mathrm{CI}}^{\texttt{DS}}_{N, \alpha} = \bigcup_{\ell=1}^B \widehat{\mathrm{CI}}^\dagger_{\ell}$ where $B \ge \lceil \log_{1/2}(\alpha) \rceil$ and \texttt{DS} stands for data-splitting. 

\subsection{Lower Confidence Bounds Approach}\label{sec:LCB-basedCI}
Returning to the general stochastic optimization problem \eqref{eq:def-m-functional}, suppose one can construct a data-dependent function $\widehat t_{\alpha}: \Theta \times \Theta \mapsto \mathbb{R}$, measurable with respect to $D_2$, such that
\begin{equation}\label{eq:lower-bound}
	\mathbb{P}_{P^{2}} \left(\M_2(\theta(P^N)) - \M_2(\widehat{\theta}_1) \ge \widehat{\M}_2(\theta(P^N)) - \widehat{\M}_2(\widehat{\theta}_1) - \widehat{t}_\alpha(\theta(P^N), \widehat{\theta}_1) | D_1\right)\ge 1-\alpha_N,
\end{equation} 
where in many practical situations  $\alpha_N = \alpha + o(1)$. This yields the lower confidence bound (LCB)-based confidence set:
\begin{equation}\label{eq:CI-without-upper-bound}
\widehat{\mathrm{CI}}^{\mathtt{LCB}}_{N, \alpha} := \left\{\theta\in\Theta:\, \widehat{\mathbb{M}}_2(\theta) - \widehat{\mathbb{M}}_2(\widehat{\theta}_1) -\widehat t_{\alpha}(\theta, \widehat{\theta}_1)\le 0\right\}.
\end{equation}
\begin{theorem}\label{thm:general_LCB}
Suppose $\widehat t_{\alpha}$ satisfies \eqref{eq:lower-bound}. Then 
\begin{equation*}
    \mathbb{P}_{P^N}\left(\theta(P^N) \not\in \widehat{\mathrm{CI}}^{\mathtt{LCB}}_{N, \alpha} \right) \le \alpha_N + \beta(r). 
\end{equation*}
\end{theorem}
The proof appears in \Cref{supsec:coverage-proofs}. This result holds at the same level of generality as \Cref{thm:coverage-anti-conservative-confidence-set}: no structure is imposed on the optimization problem and the observations need not be independent. The only requirements are the basic inequality \eqref{eq:zeroth-order} and the lower bound condition \eqref{eq:lower-bound}.

It may initially appear puzzling that constructing a confidence set for $\theta(P^N)$ requires building a lower confidence bound for $\M_2(\theta(P^N)) - \M_2(\widehat{\theta}_1)$, which itself depends on $\theta(P^N)$. A stronger but more transparent restatement of \eqref{eq:lower-bound} clarifies what is actually needed:
\begin{equation}\label{eq:reinterpret-lower-upper}
    \begin{split}
\inf_{P\in\mathcal{P}}\,\inf_{\theta,\theta'\in\Theta}\,\mathbb{P}_{P^{2}} \left(\M_2(\theta) - \M_2(\theta') \ge \widehat{\M}_2(\theta) - \widehat{\M}_2(\theta') - \widehat{t}_\alpha(\theta, \theta') | D_1\right)\ge 1-\alpha_N.
    \end{split}
\end{equation}
That is, one needs a lower confidence bound for $\mathbb{M}_2(\theta) - \mathbb{M}_2(\theta')$ for every (non-stochastic) pair $\theta, \theta'\in\Theta$. Under M-estimation with $\mathbb{M}_2(\theta) =n_2^{-1}\sum_{i\in I_2}\mathbb{E}_{P_i}[m_{\theta}(Z_i)]$, the difference reduces to $\mathbb{M}_2(\theta) - \mathbb{M}_2(\theta') = n_2^{-1}\sum_{i\in I_2}\mathbb{E}_{P_i}[m_{\theta}(Z_i) - m_{\theta'}(Z_i)]$, for which such bounds follow from concentration inequalities or the central limit theorem.

Since the map $(\theta,\theta') \mapsto \M_2(\theta) - \M_2(\theta')$ is always real-valued, a lower confidence bound can in principle be constructed without reference to the complexity of $\Theta$. This contrasts sharply with ``classical'' approaches based on the weak convergence of $r_N(\widehat \theta_1-\theta(P^N))$ to a limit process, which often depends heavily on the complexity of $\Theta$.

\subsection{Studentized Confidence Sets under M-estimation}\label{sec:stutendized-cs}
This subsection develops a concrete choice of $\widehat t_\alpha(\cdot, \cdot)$ satisfying \eqref{eq:lower-bound} under M-estimation \eqref{eq:m-estimation-definition} with independent but not necessarily identically distributed observations. The requirement reduces to a lower confidence bound for the expected difference of loss functions based on a sample mean, for which a natural construction uses the central limit theorem for t-statistics \citep{Bentkus1996, bentkus1996berry}. For $\theta \in \Theta$, define the sample variance of $m_{\theta}(Z_i) - m_{\widehat\theta_1}(Z_i) = (m_{\theta} - m_{\widehat\theta_1})(Z_i)$ as 
\begin{equation}
\label{eq:sample-variance}
		\widehat \sigma_{\theta, \widehat{\theta}_1}^2 :=     \frac{1}{n_2-1}\sum_{i \in I_2} \left\{(m_{\theta}-m_{ \widehat{\theta}_1})(Z_i) - \frac{1}{n_2}\sum_{j\in I_2}(m_{\theta}-m_{ \widehat{\theta}_1})(Z_j)\right\}^2.
\end{equation}
The studentized confidence set is\begin{equation}\label{eq:CI-CLT}
    \widehat{\mathrm{CI}}^{\mathtt{CLT}}_{N, \alpha} := \left\{\theta\in\Theta:\, \widehat{\mathbb{M}}_2(\theta) - \widehat{\mathbb{M}}_2(\widehat{\theta}_1)\le n_2^{-1/2}z_\alpha \widehat \sigma_{\theta, \widehat\theta_1}\right\},
\end{equation}
where $z_{\alpha}$ denotes the $(1-\alpha)$-th quantile of the standard normal. This is the special case of $\widehat{\mathrm{CI}}^{\mathtt{LCB}}_{N, \alpha}$ defined in \eqref{eq:CI-without-upper-bound} with $\widehat t_{\alpha}(\theta, \widehat\theta_1)= n_2^{-1/2}z_\alpha \widehat \sigma_{\theta, \widehat\theta_1}$. Several validity results follow; all proofs appear in \Cref{supp:studentized-proofs}.

\begin{theorem}\label{thm:studentized-nonIID}
Recall $\widehat\xi_i$ defined in \eqref{eq:centered-xi}. For any $V > 0$, define the truncated random variables, 
\begin{equation}
    \overline{\xi}_i = V^{-1}\widehat{\xi}_i\mathbf{1}\{|\widehat{\xi}_i| \le V\}, \quad \textrm{and} \quad M^2 = \sum_{i \in I_2} \mathrm{Var}_{P_i}[\overline{\xi}_i| D_1].
\end{equation}
There exists an absolute constant $C > 0$ such that for any $\alpha \in (0,1)$,
\begin{equation}\label{eq:studentized-nonIID}
\begin{split}
    &\mathbb{P}_{P^N}\left(\theta(P^N) \not\in \widehat{\mathrm{CI}}^{\mathtt{CLT}}_{N, \alpha}\right) \le \alpha \\
    &\quad + \mathbb{E}_{P^1}\left[\min\left\{1,C\left(\sum_{i\in I_2}\mathbb{P}_{P_i}(\widehat{\xi}^2_i > V^2 | D_1) + \sum_{i\in I_2}\frac{|\mathbb{E}_{P_i}[\overline{\xi}_i| D_1]|}{M} + \sum_{i \in I_2} \frac{\mathbb{E}_{P_i}[|\overline{\xi}_i|^3| D_1]}{M^3}\right)\right\}\right].
\end{split}
\end{equation}
\end{theorem}
\begin{theorem}\label{thm:studentized-katz}
    Under \eqref{eq:m-estimation-definition}, one has $n_2^2\widehat{\mathbb{V}}_2 = \sum_{i\in I_2}\mathrm{Var}_{P_i}[\widehat \xi_i | D_1]$ by definition \eqref{eq:variance-curvature-ratio}. Setting $V^2 = n_2^2\widehat{\mathbb{V}}_2$ in \Cref{thm:studentized-nonIID} gives  \begin{equation}\label{eq:studentized-katz}
    \begin{split}
            &\mathbb{P}_{P^N}\left(\theta(P^N) \not\in \widehat{\mathrm{CI}}^{\mathtt{CLT}}_{N, \alpha}\right) \le \alpha \\
    &\quad + \E_{P^1}\left[\min\left\{1, C\sum_{i\in I_2} \mathbb{E}_{P_i}\left[\frac{|\widehat\xi_i|^2}{n_2^2\widehat{\mathbb{V}}_{2}}\min\left\{1,\,\frac{|\widehat \xi_i|}{n_2\widehat{\mathbb{V}}_{2}^{1/2}}\right\}\bigg|D_1\right]\right\}\right],
    \end{split}
\end{equation}
for any $\alpha \in (0,1)$ where $C > 0$ is an absolute constant.
\end{theorem}
\begin{theorem}\label{cor:studentized-IID}
    Assume $\{Z_i\, : \, i \in I_2\}$ is IID. Set $V_*$ as the largest solution that satisfies:
    \begin{equation}
        V_*^2  = \E_{P^2}[\widehat \xi_1^2\mathbf{1}\{\widehat \xi_1^2 \le V_*^2 n_2\}| D_1],
    \end{equation}
    and define 
    \begin{equation}\label{eq:DAN-remainder}
        R_* = n_2\mathbb{P}_{P^2}(\widehat{\xi}^2_1 > n_2V_*^2| D_1) + n_2|\mathbb{E}_{P^2}[\overline{\xi}_1| D_1]| + n_2\mathbb{E}_{P^2}[|\overline{\xi}_1|^3| D_1]
    \end{equation}
    where $\overline{\xi}_1 = (V_*^2 n_2)^{-1/2}\widehat \xi_1\mathbf{1}\{\widehat \xi_1^2 \le V_*^2 n_2\}$.
    Then, \eqref{eq:studentized-nonIID} becomes 
    \begin{equation}\label{eq:studentized-IID}
        \mathbb{P}_{P^N}\left(\theta(P^N) \not\in \widehat{\mathrm{CI}}^{\mathtt{CLT}}_{N, \alpha}\right) \le \alpha + \mathbb{E}_{P^1}[\min\{1, CR_*\}],
    \end{equation}
   for any $\alpha \in (0,1)$.
\end{theorem}
The remainder in \eqref{eq:studentized-nonIID} vanishes under the Lindeberg-Feller condition, though this is sufficient but not necessary. See a counterexample in Example 1.1 of \citet{bentkus1996berry}. \Cref{thm:studentized-katz} and \Cref{cor:studentized-IID} are direct consequences of \Cref{thm:studentized-nonIID}. \Cref{thm:studentized-katz} restates the result under more familiar moment conditions. The remainder in \eqref{eq:studentized-katz} should be compared with that in \Cref{thm:coverage-anti-conservative-confidence-set-empirical-risk}: beyond vanishing under the Lindeberg--Feller condition, it yields an explicit convergence rate of $n_2^{-\delta/2}$ whenever $\E_{P_i}[|\widehat \xi_i|^{2+\delta} | D_1]$ is finite for some $\delta \in (0,1]$. When $\{Z_i\, : \, i \in I_2\}$ is IID, the condition weakens significantly.  In particular, the second term of \eqref{eq:studentized-IID} tends to zero when $\widehat \xi_1$ belongs to the domain of attraction of the normal law (DAN). Also under IID observations, the folllowing multiplicative error bound becomes available.
\begin{theorem}\label{thm:studentized-IID-multi} Assume that $\{Z_i\, : \, i \in I_2\}$ is IID. Define 
\begin{equation*}
    z_{\alpha, n_2} = z_\alpha\sqrt{\frac{n_2}{n_2-1+z_\alpha^2}},\end{equation*}
and let $V_{*}$ be the largest solution satisfying
    \begin{equation}
        V_{*}^2(1+z_{\alpha, n_2}^2) = \E_{P^2}[\widehat \xi_1^2\mathbf{1}\{\widehat \xi_1^2 \le V_*^2 n_2\} |  D_1].
    \end{equation}
   Then, for any $\alpha \in (0,1)$, there exist absolute constants $C, C', C'' > 0$ such that 
    \begin{equation}\label{eq:studentized-IID-multi}
    \begin{split}
        &\mathbb{P}_{P^N}\left(\theta(P^N) \not\in \widehat{\mathrm{CI}}^{\mathtt{CLT}}_{N, \alpha}\right)\le \min\{1,C\min \left\{R_1, R_2\right\}\},
    \end{split}
    \end{equation}
    where 
    \begin{equation*}
    \begin{split}
            R_1 &= \left(1-\Phi(z_{\alpha, n_2})\right)  \mathbb{E}_{P^1}[\exp(R_*)\mathbf{1}\{R_*\le C' (1+|z_{\alpha, n_2}|)^2\}]\quad \text{and}\\
        R_2 &= \exp(-z_{\alpha,n_2}^2)\mathbb{P}_{P^1}(R_* \le C''),
    \end{split}
    \end{equation*}
     with $R_*$ defined as \eqref{eq:DAN-remainder}.
\end{theorem}
The proof follows from Theorem 2 of \citet{robinson2005self} and sub-Gaussianity result of t-statistics in \citet{gine1997student}, yielding a tighter bound whenever $R_*$ is small. This result also only requires $\widehat \xi_1$ to belong to the domain of attraction of the normal law

\Cref{thm:studentized-nonIID,thm:studentized-katz,cor:studentized-IID,thm:studentized-IID-multi} fail to capture the refined conservativeness governed by the ratio $\ratio$ as in \Cref{thm:coverage-anti-conservative-confidence-set-empirical-risk} or \Cref{thm:coverage-martingale}. Existing non-uniform Berry-Esseen bounds for t-statistics \citep{shao1999cramer, jing2003self, wang2005limit} are not directly applicable since we require results for non-central t-statistics whose limiting distributions are non-standard \citep{Bentkus2007Limiting}. Under additional assumptions, the following sharper result is available.

\begin{theorem}\label{thm:large-deviation}
    Assume that $\{Z_i\, : \, i \in I_2\}$ is IID and let $\sigma^2 = \mathrm{Var}[\widehat \xi_1 | D_1]$. Define
        \begin{gather}
              \Xi_{n_2, t} = \E_{P^2}\left[\frac{(1+t)^3 |\widehat\xi_1|^3}{\sqrt{n_2}\sigma^3} \min\left\{1,\frac{(1+t) |\widehat\xi_1|}{\sqrt{n_2}\sigma} \right\} | D_1\right], \quad \rho_N =n_2^{-1/2}\E_{P^2}[|\widehat\xi_1|^3/\sigma^3| D_1],\quad \textrm{and}\nonumber\\
    \Psi_{n_2, t} = \exp\left(\frac{(t+\ratio)^2t}{4}\left(\frac{2(t+\ratio)}{3t}-2\right)\frac{\E_{P^2}[\widehat \xi_1^3 | D_1]}{\sqrt{n_2}\sigma^3}\right).\nonumber
        \end{gather}
    Then, there exist finite absolute constants $C, C', C''$ such that 
    \begin{equation}\label{eq:moderate-deviation}
    \begin{split}
        &\mathbb{P}_{P^N}\left(\theta(P^N)\not\in \widehat{\mathrm{CI}}^{\mathtt{CLT}}_{N, \alpha} \right) \\
        &\quad = \E_{P^1}\left[\left(1-\Phi(z_\alpha+\ratio)\right)\Psi_{n_2, z_\alpha}\exp(C' \Xi_{n_2, z_\alpha})\{1 + C''(1+z_\alpha)\rho_{n_2}\}\right],
    \end{split}
    \end{equation}
    for all $\ratio\le z_\alpha/5$ and $0 \le z_\alpha \le \rho_{n_2}^{-1} C$.
\end{theorem}

The proof follows from Theorem 1 of \citet{wang2009relative}; see also Theorem 2.1 of \citet{shao2016cramer} for a relevant result. Crucially, \eqref{eq:moderate-deviation} is an equality, not an inequality, and recovers the precise dependence on $\ratio$ in the tail. When $\E_{P^2}[|\widehat{\xi}^3| | D_1] < \infty$ and $z_\alpha = o(N^{1/6})$ \eqref{eq:moderate-deviation} implies
\begin{equation*}
    \mathbb{P}_{P^N}\left(\theta(P^N)\not\in \widehat{\mathrm{CI}}^{\mathtt{CLT}}_{N, \alpha} \right)/\E_{P^1}[1-\Phi(z_\alpha+\ratio)] \to 1.
\end{equation*}
Two natural extensions remain open: relaxing the finite third moment assumption to the domain of attraction of a stable law, and extending the range of validity beyond $\ratio\le z_\alpha/5$. Both directions are non-trivial problems in their own right and merit dedicated treatment; they are left for future work.

\subsection{Summary of Prescribed-level Methods}
The constructions proposed in this section extend the based confidence set $\widehat{\mathrm{CI}}^\dagger_{N}$ as in \eqref{eq:anti-conservative-confidence-set} to a prescribed significance level $\alpha$. the associated conditions are summarized in \Cref{tab:prescribed-level-confidence-sets}.

\begin{table}[ht]
\centering

\renewcommand{\arraystretch}{1.4}
\begin{tabular}{llccc}
\toprule
 & Optimization & Dependence & Theorem & Assumption \\
\midrule
Data-split
  & M-estimation
  & $\beta$-mixing
  & Thm.~\ref{thm:hulc-like}
  & Conditions of Thm.~\ref{thm:coverage-anti-conservative-confidence-set-empirical-risk}--\ref{thm:coverage-martingale} \\
LCB \eqref{eq:CI-without-upper-bound}
  & General
  & $\beta$-mixing
  & Thm.~\ref{thm:general_LCB}
  & Lower bound \eqref{eq:lower-bound} \\
CLT \eqref{eq:CI-CLT}
  & M-estimation
  & Indep.
  & Thm.~\ref{thm:studentized-nonIID}--\ref{thm:studentized-katz}
  & Lindeberg--Feller \\
CLT \eqref{eq:CI-CLT}
  & M-estimation
  & IID\
  & Thm.~\ref{cor:studentized-IID}--\ref{thm:large-deviation}
  & DAN \\
\bottomrule
\end{tabular}
\caption{Validity results for the prescribed-level confidence sets proposed in
this section. All constructions achieve the nominal level $\alpha$ either in finite samples
(empirical Bernstein) or asymptotically (CLT-based and data-splitting); the data-splitting method restricts the admissible range of $\alpha$, depending on the sample size.
``Indep.'' allows for non-identical distributions. The Lindeberg-Feller condition is sufficient but not necessary. DAN denotes the domain of attraction of the normal law.}
\label{tab:prescribed-level-confidence-sets}
\end{table}

Most results in this section are stated for M-estimation, with the exception of \Cref{thm:general_LCB}; the concrete construction satisfying \eqref{eq:lower-bound} is nevertheless studied for M-estimation. The data-splitting approach is based on \Cref{thm:coverage-anti-conservative-confidence-set-empirical-risk}--\Cref{thm:coverage-martingale} and therefore extends to the observations with general dependence. Among the distributional assumptions presented, the weakest is the DAN condition of \Cref{cor:studentized-IID,thm:studentized-IID-multi,thm:large-deviation}; Two directions are deferred to future work: extending the CLT-based confidence set \eqref{eq:CI-CLT} to self-normalized martingales, for which \citet{fan2018berry} provides a relevant Berry--Esseen bound; and obtaining a precise characterization of miscoverage analogous to \Cref{thm:large-deviation} for non-central t-statistics, which requires moderate and large deviation results not yet available in the literature.
\begin{remark}[Comparison with existing methods]
\citet{vogel2008universal} study a confidence set essentially equivalent to \eqref{eq:anti-conservative-confidence-set} without sample-splitting. Their condition CI2, which enters directly into their validity guarantee, typically requires control of the covering numbers or metric entropy of $\Theta$, whereas the proposed methods require no such complexity conditions. 
\cite{dey2024anytime} study a confidence set for the same target but impose a ``\emph{strong central condition}'', requiring a finite moment-generating function for $\widehat{\xi}_i$, or some sub-exponential tails. By contrast, the present results require only moment conditions. In the IID case, \Cref{cor:studentized-IID} requires no finite moments as long as $\widehat{\xi}_i \in \mathrm{DAN}$. The confidence set \eqref{eq:CI-CLT} was previously studied by \citet{chakravarti2019gaussian} and \citet{park2023robust}. \citet{chakravarti2019gaussian} test against a null that is a set of multivariate Gaussians. \cite{park2023robust} introduce synthetic noise to prevent the variance from vanishing, thereby inflating the confidence set; this inflation is unnecessary under the Lindeberg--Feller condition. Finally, while this manuscript focuses on inference for $\theta(P^N)$ itself, one may instead be interested in low-dimensional summary such as a univariate projection $\ell ^\top \theta(P^N)$ for (non-random) $\ell \in \mathbb{R}^d$. Honest inference for such projections was studied for Z-estimation framework by \citet{chang2024confidence}, whose validity results impose smoothness conditions on $\M(\theta, P^N)$ and growth restrictions of the form  $\mathrm{polylog}(d) = o(n)$. 
\end{remark}
\begin{remark}[Lower Confidence Bounds based on Concentration Inequalities]
An alterative route to constructing $\widehat t_\alpha(\cdot, \cdot)$ satisfying \eqref{eq:lower-bound} is through concentration inequalities. 
There is a wide range of tools for this purpose; see \citet{Boucheron2013Concentration} for classical results and \citet{hao2019bootstrapping, ramdas2023randomized, waudby2024estimating, bates2021distribution} for more recent developments. This approach may yield finite-sample valid confidence sets under stronger tail assumptions. Results based on the one-sided empirical Bernstein inequality are developed in \Cref{suppsec:empirical-bernstein}. 
\end{remark}
\section{Convergence Rates of the Confidence Sets}\label{sec:convergence-rates}

\subsection{Diameter Bounds under General Optimization}
This section establishes non-asymptotic bounds on the diameter of the proposed confidence sets, illustrating adaptive rates of convergence that depend on the unknown curvature of the objective. Throughout, $\theta(P^N)$ is assumed to be the unique minimizer of \eqref{eq:def-m-functional}, and the following conditions are imposed.

	\begin{enumerate}[label=\textbf{(A\arabic*)},leftmargin=2cm]
    \item \label{as:margin} There exist constants $c_0, \gamma \ge 0$ such that 
    \begin{equation*}
        \M_2(\theta)-\M_2(\theta(P^N))\ge c_0\|\theta-\theta(P^N)\|^{1+\gamma}
    \end{equation*}
    for all $\theta \in \Theta$.
    \item \label{as:local-entropy}There exists a function $\phi_{n_2} : \mathbb{R}_+ \mapsto \mathbb{R}_+$ such that 
    \begin{equation}
    	\E^*_{P^2} \left[\sup_{\|\theta-\theta(P^N)\| < \delta}|(\widehat{\mathbb{M}}_2 - \mathbb{M}_2)(\theta) - (\widehat{\mathbb{M}}_2 - \mathbb{M}_2)(\theta(P^N))| \right] \le  \phi_{n_2}(\delta)\label{eq:maximal-inequality}
    \end{equation}
    for every $n_2\ge 1$ and $\delta > 0$, where $\E^*_{P^2}[\cdot]$ denotes outer expectation. Furthermore, $\phi_{n_2}(x)/x^q$ is assumed non-increasing for some $q < 1+\gamma$.
    \item \label{as:rate-initial-estimator} 
    For every $n_1, n_2 \ge 1$ and $\varepsilon_{\mathtt{init}} > 0$, the initial estimator 
$\widehat\theta_1$ based on $D_1$ satisfies
\begin{equation}
    \mathbb{P}_{P^1}\!\left(
    \mathbb{E}_{P^2}[
    |\widehat{\M}_2(\widehat\theta_1) - \widehat{\M}_2(\theta(P^N))
    | \,|\, D_1
    ] \geq C_{\mathtt{init}} \cdot s_{n_1, n_2}
    \right) \leq \varepsilon_{\mathtt{init}},
\end{equation}
where $s_{n_1, n_2}, C_{\mathtt{init}}$ are non-negative constants. 
\end{enumerate}

The parameter $\gamma$ in \ref{as:margin} links the optimization problem at hand to the curvature within the parameter space. This condition implies that $\theta(P^N)$ is a strong global minimizer of $\theta\mapsto\mathbb{M}_2(\theta)$~\citep{drusvyatskiy2013tilt}. The inequality \ref{as:local-entropy} is known as a maximal inequality \citep[Section 2.3.1]{van1996weak} with the modulus  $\phi_{n_2}(\cdot)$ reflecting the local complexity of $\Theta$. To handle possible measurability issues, the outer expectation  $\E^*_{P^{2}}[\cdot]$ is adopted. See Section 1.2 of \cite{van1996weak}. Finally, \ref{as:rate-initial-estimator} pertains to the convergence rate of the initial estimator. As shown below, diameter of the proposed confidence set depends on all three quantities. \ref{as:margin} and \ref{as:local-entropy} are standard in the analysis of M-estimators, as discussed in Theorem 3.2.5 of~\cite{van1996weak}. They also appear in \citet{kim1990cube} to explain the differences between regular and irregular M-estimators in the estimation context.
% \begin{remark}
%    Condition \ref{as:margin} is stated for all $\theta \in \Theta$. In most cases, however, this holds only locally in some neighborhood of the optimum $\theta(P^N)$---See Section~\ref{sec:quantile} for a concrete example. Relaxing \ref{as:margin} to its local analog is not trivial. For example, when \ref{as:margin} holds locally and $\Theta$ is unbounded, the conditions above may not be sufficient to claim that the proposed confidence set is bounded. To this end, we envision the use of the naive confidence set $\widehat{\mathrm{CI}}_N^{\dagger}$ in~\eqref{eq:anti-conservative-confidence-set} with a potentially inconsistent $\widehat{\theta}_1$ to first obtain a bounded confidence set and then consider the intersection of final confidence set with $\widehat{\mathrm{CI}}_N^{\dagger}$ to obtain a provably bounded confidence set. This intersection allows us to weaken assumptions~\ref{as:margin} and~\ref{as:local-entropy} by restricting to a bounded subset of $\Theta$, even if $\Theta$ is unbounded to start with. 
% \end{remark}

For a set $A$ equipped with a metric $\|\cdot\|$, define $\mathrm{Diam}_{\|\cdot\|}(A) := \sup\{\|a-b\|\, :\, a, b\in A\}$. 
\begin{theorem}\label{thm:hulc-width}Assume \ref{as:margin}--\ref{as:rate-initial-estimator}. Define $r_{n_2}$ as any value satisfying
\begin{equation}
    r_{n_2}^{-2} \phi_{n_2}(c_0^{-1/(1+\gamma)}r_{n_2}^{2/(1+\gamma)}) \le 1, \label{rq:define-rn}
\end{equation}
and set \begin{equation*}
    \mathrm{R}^{\dagger}_{N}:= c_0^{-1/(1+\gamma)}(r_{n_2}^{2/(1+\gamma)}  + s_{n_1, n_2}^{1/(1+\gamma)}).
\end{equation*}
Then, for any $n_1, n_2\ge 1$ and $\varepsilon > 0$, writing $\mathbb{P}^*_{P^N}(\cdot)$ for outer probability over $P^N$,
\begin{equation*}
\mathbb{P}^*_{P^N}\left(\mathrm{Diam}_{\|\cdot\|}\big(\widehat{\mathrm{CI}}^{\dagger}_{N}\big)\le C\varepsilon^{-1/(1+\gamma-q)} \mathrm{R}^{\dagger}_{N}\right) \ge 1-\varepsilon-\varepsilon_{\mathtt{init}}-\beta(r), 
\end{equation*}
and $C$ is a constant depending on $\gamma, q$ and $C_{\mathtt{init}}$.  
\end{theorem}
The proof appears in \Cref{sec:hulc-width}. The proof requires no independence structure on the observations and no particular form of the optimization problem, so the result extends beyond the M-estimation framework. The diameter shows dependence on the curvature parameter $\gamma$ in \ref{as:margin} despite the confidence set being constructed without prior knowledge of $\gamma$, reflecting adaptation to the local geometry of the problem. Diameter bounds for the data-splitting construction of \Cref{sec:data-splitting} follow as an immediate corollary. A companion bound for the LCB-based set \eqref{eq:CI-without-upper-bound} is developed in \Cref{suppsec:addition-diameter}.

% Under M-estimation structure, more concrete results are available with the following choices:
% \begin{equation}
%     \begin{split}
%         \widehat t_{\alpha}(\theta, \widehat\theta_1) = \begin{cases}
%             n_2^{-1/2}\sqrt{2 \log(2/\alpha)} \widehat \sigma_{\theta, \widehat\theta_1}+ (n_2-1)^{-1}\cdot7/3B\log(2/\alpha) & \mbox{for } \eqref{eq:CI-EB} \\
%             n_2^{-1/2}z_\alpha \widehat \sigma_{\theta, \widehat\theta_1}& \mbox{for } \eqref{eq:CI-CLT} 
%         \end{cases}
%     \end{split}
% \end{equation}
% In this case, rather than verifying \ref{as:general-t-entropy}, more specific conditions tailored to M-estimation can be imposed, which are easier to check in many statistical applications. 
\subsection{Diameter Bounds under M-estimation}
\Cref{sec:stutendized-cs} developed a confidence set for M-estimation at any prescribed significance level $\alpha$ using $\widehat t_{\alpha}(\theta, \widehat\theta_1)  =n_2^{-1/2}z_\alpha \widehat \sigma_{\theta, \widehat\theta_1}$. This subsection establishes a non-asymptotic diameter bound for this set. The conditions below are stated in the M-estimation notation of \eqref{eq:m-estimation-definition} under independent but not necessarily identically distributed observations.
\begin{enumerate}[label=\textbf{(A\arabic*)},leftmargin=2cm]\setcounter{enumi}{3}
\item \label{as:square-process} There exist functions $\omega_{n_2,\mathtt{emp}}, \omega_{\mathtt{pop}}: \mathbb{R}_+ \mapsto \mathbb{R}_+$ such that 
    \begin{equation}
    	\E_{P^2}^* \left[\sup_{\|\theta-\theta(P^N)\| < \delta}\left|\frac{1}{n_2}\sum_{i\in I_2}(m_\theta - m_{\theta(P^N)})^2 - \mathbb{E}_{P_i}[(m_\theta - m_{\theta(P^N)})^2]\right| \right]\le  \omega^2_{n_2,\mathtt{emp}}(\delta)\nonumber
    \end{equation}
    and 
    \begin{equation}
        \sup_{\|\theta-\theta(P^N)\| < \delta}\frac{1}{n_2}\sum_{i\in I_2}\mathbb{E}_{P_i}[(m_\theta - m_{\theta(P^N)})^2(Z_i)]\le  \omega^2_{\mathtt{pop}}(\delta)\nonumber\label{eq:square-process-pop}
    \end{equation}
    for every $n_2\ge 2$ and $\delta > 0$. The combined modulus is $\omega^2_{n_2}(\delta) = \omega^2_{n_2,\mathtt{emp}}(\delta)   +\omega^2_{\mathtt{pop}}(\delta)$, and $\omega_{n_2}(x)/x^q$ is assumed non-increasing for some $q < 1+\gamma$.
    \item \label{as:rate-initial-estimator3} For every $n_1\ge 1$ and $n_2 \ge 2$, and $\widetilde\varepsilon_{\mathtt{init}} > 0$, the initial estimator based on $D_1$ satisfies
    \begin{equation}
        \mathbb{P}_{P^1}\left(\frac{1}{n_2}\mathbb{E}_{P^2| P^1}\left[\frac{1}{n_2}\sum_{i \in I_2} \widehat \xi_i^2\right] + \widehat{\mathbb{C}}_2^2 > \widetilde C_{\mathtt{init}}\widetilde {s}^2_{n_1, n_2}\right) \le \widetilde \varepsilon_{\mathtt{init}}.\nonumber
    \end{equation}
\end{enumerate}
Condition \ref{as:square-process} resembles \ref{as:local-entropy}, but controls the growth rate of the localized \emph{squared} empirical process rather than the process itself. Condition \ref{as:rate-initial-estimator3} plays the role of \ref{as:rate-initial-estimator}; the proof shows that ${s}^2_{n_1, n_2} \le \widetilde {s}^2_{n_1, n_2}$ under independence. The following result is proved in \Cref{sup:proof-clt-width}.
\begin{theorem}\label{thm:clt-width}Assume $Z_1,\ldots, Z_N$ are independent and \ref{as:margin}, \ref{as:local-entropy}, \ref{as:square-process} and \ref{as:rate-initial-estimator3}. Define $r_{n_2}$ as in \eqref{rq:define-rn}, $u_{n_2}$ as any value satisfying
\begin{equation}
    u_{n_2}^{-2} \omega_{n_2}(c_0^{-1/(1+\gamma)}u_{n_2}^{2/(1+\gamma)}) \le n_2^{1/2}.\label{rq:define-un-clt}
\end{equation}
Set 
\begin{equation*}
    \mathrm{R}_N^{\mathtt{CLT}} = c_0^{-1/(1+\gamma)}(r_{n_2}^{2/(1+\gamma)}  + u_{n_2}^{2/(1+\gamma)}+\widetilde s_{n_1, n_2}^{1/(1+\gamma)}).
\end{equation*}
Then, for any $n_1\ge 1$, $n_2 \ge 2$ and $\varepsilon > 0$, 
\begin{equation*}
\mathbb{P}^*_{P^N}\left(\mathrm{Diam}_{\|\cdot\|}\big(\widehat{\mathrm{CI}}^{\mathtt{CLT}}_{N, \alpha}\big)\le C \left(\frac{1+|z_\alpha|}{\varepsilon}\right)^{1/(1+\gamma-q)} \mathrm{R}_N^{\mathtt{CLT}}\right) \ge 1-\varepsilon - \widetilde \varepsilon_{\mathtt{init}},
\end{equation*}
where $C$ is a constant depending only on $\gamma, q$, and $\widetilde C_{\mathtt{init}}$. 
\end{theorem}

% The results are summarized in \Cref{tab:convergence-rates-confidence-sets} below.

% \begin{table}[ht]
% \centering
% \renewcommand{\arraystretch}{1.4}
% \begin{tabular}{llccc}
% \toprule
%  & Optimization & Dependence & Theorem & Assumption \\
% \midrule
% \eqref{eq:anti-conservative-confidence-set}
%   & General
%   & $\beta$-mixing
%   & Thm.~\ref{thm:hulc-width}
%   & \ref{as:margin}--\ref{as:rate-initial-estimator} \\
% LCB \eqref{eq:CI-without-upper-bound}
%   & General
%   & $\beta$-mixing
%   & Thm.~\ref{thm:LCB-general-width}
%   & \ref{as:margin}--\ref{as:rate-initial-estimator2} \\
% EB \eqref{eq:CI-EB}
%   & M-estimation
%   & Indep.
%   & Thm.~\ref{thm:eb-width}
%   & \thead{$\|m_\theta - m_{\theta'}\|_\infty \le B$\\\ref{as:margin},\ref{as:local-entropy},\ref{as:square-process},\ref{as:rate-initial-estimator3}}\\
% CLT \eqref{eq:CI-CLT}
%   & M-estimation
%   & Indep.
%   & Thm.~\ref{thm:clt-width}
%   & \thead{\ref{as:margin},\ref{as:local-entropy},\ref{as:square-process},\ref{as:rate-initial-estimator3}}\\
% \bottomrule
% \end{tabular}
% \caption{Summary of convergence rate results. \Cref{thm:hulc-width} and \Cref{thm:LCB-general-width} are established under general stochastic optimization, while \Cref{thm:eb-width} and \Cref{thm:clt-width} specialize to M-estimation with a sample variance estimator. \Cref{thm:clt-width} applies to both bounded and unbounded loss functions and exhibits better dependence $\alpha$.}
% \label{tab:convergence-rates-confidence-sets}
% \end{table}
\begin{remark}[Controlling \ref{as:square-process}]
Light-tail assumptions on
$m_{\theta}-m_{\theta(P^N)}$, including sub-Gaussian and sub-exponential conditions, are commonly imposed to obtain $\omega_{n_2}$ for unbounded processes. This assumption can be relaxed substantially. \Cref{prop:squared-Gn} in \Cref{sec:technical-lemma} provides general construction of $\omega_{n_2}$ via a truncation argument under two settings: finite moment conditions, and sub-Weibull tails \citep{kuchibhotla2022moving}, the latter yielding sharper rates. See also Proposition~3.1 of \citet{gine2000exponential} and Proposition~B.1 of \citet{kuchibhotla2022least} for related results.
\end{remark}

\subsection{Localizing the Diameter Analysis}
In many statistical applications, \ref{as:margin}, \ref{as:local-entropy}, and \ref{as:square-process} hold only locally in a neighborhood of $\theta(P^N)$. For convergence rate analysis of M-estimators, this is  addressed by assuming consistency of the estimator first and then restricting attention to a neighborhood of $\theta(P^N)$ \citep[Theorem 3.2.5]{van1996weak}. This approach does not extend directly to confidence set analysis, where the defining inequality must be evaluated at every $\theta \in \Theta$. The following two conditions provide the additional global structure needed to localize the diameter analysis.
\begin{enumerate}[label=\textbf{(A\arabic*-\texttt{global})},itemindent=2cm]
    \item \label{as:margin-global} For $\rho > 0$, there exist a function $C_\rho : [\rho, \infty)\mapsto \mathbb{R}_+$ with generalized inverse $C^{-1}_\rho(s) = \inf\{r \ge \rho\, :\, C_\rho(r) \ge s\}$ and a function $g: [1,\infty) \mapsto [1, \infty)$, such that
    \begin{equation*}
        \M_2(\theta) - \M_2(\theta(P^N)) \geq C_\rho(\|\theta - \theta(P^N)\|) \quad \text{for all}\quad \|\theta-\theta(P^N)\| \geq \rho
    \end{equation*}
    and 
    \begin{equation*}
        C^{-1}_\rho(\lambda r) \leq g(\lambda) C^{-1}_\rho(r)\quad \text{for all} \quad r \geq \rho,\, \lambda \geq 1.
    \end{equation*}
\end{enumerate}
\begin{enumerate}[label=\textbf{(A\arabic*-\texttt{ratio})},itemindent=2cm]
\setcounter{enumi}{1}
    \item \label{as:ratio-process} For $\rho > 0$, there exists $R(n_2, \rho)$ with $\lim_{t\to \infty} R(t, \rho) = 0$, such that
    \begin{equation*}
        \mathbb{P}_{P^2}^*\left(\sup_{\|\theta-\theta(P^N)\| > \rho} \left|\frac{(\widehat\M_2-\M_2)(\theta) -(\widehat\M_2-\M_2)(\theta(P^N))}{\M_2(\theta)-\M_2(\theta(P^N))}\right| \ge C_{\mathtt{ratio}} R(n_2, \rho)\right) \le \varepsilon_{\mathtt{ratio}},
    \end{equation*}
    where $C_{\mathtt{ratio}} > 0$ is a constant depending on $\varepsilon_{\mathtt{ratio}}$.
\end{enumerate}
Condition \ref{as:margin-global} requires the population objective to grow globally away from $\theta(P^N)$, but imposes no curvature condition on the growth. It is satisfied whenever $C_\rho(r) = cr^p$ for some $c, p > 0$, corresponding to polynomial growth of any order with $g(\lambda) = \lambda^{1/p}$. Condition \ref{as:margin-global} allows for more exotic choices of growth functions. Condition \ref{as:ratio-process} is a ratio-type empirical process condition as in \cite{gine2006concentration}, which can be shown to concentrate without local curvature. 
\begin{theorem}
    \label{thm:hulc-width-local}Assume \ref{as:margin} and \ref{as:local-entropy} hold for all $\|\theta - \theta(P^N)\| \le \rho$, and that \ref{as:rate-initial-estimator}, \ref{as:margin-global}, \ref{as:ratio-process} and $\beta(r)=0$ hold. Then for $n_2$ sufficiently large,
\begin{align*}
\mathrm{Diam}_{\|\cdot\|}\big(\widehat{\mathrm{CI}}^{\dagger}_{N}) = O_{\mathbb{P}^*}\left(\max\{\mathrm{R}^{\dagger}_{N}, \mathrm{Q}^{\dagger}_{N}\mathbf{1}\{\mathrm{Q}^{\dagger}_{N} \ge \rho\}\}\right)
\end{align*}
where $\mathrm{Q}^{\dagger}_{N}:= C_\rho
^{-1}\left(s_{n_1, n_2}\right)$ and $\mathrm{R}^{\dagger}_{N}$ is as defined in \Cref{thm:hulc-width}. 
\end{theorem}
The full non-asymptotic statement without assumption $\beta(r) = 0$ is provided in \Cref{suppsec:hulc-width-local-proof}. The diameter bound comprises two terms. The first term $\mathrm{R}^\dagger_N$ is the local rate from \Cref{thm:hulc-width} depending on the curvature and the modulus. The second term $\mathrm{Q}^\dagger_N$ is a conservative global radius that eventually vanishes when $s_{n_1, n_2}$ is sufficiently small. Once $\mathrm{Q}^\dagger_N$ falls below $\rho$, the bound reduces to the local rate $\mathrm{R}^\dagger_N$ alone.

An analogous localization applies to the CLT-based set $\widehat{\mathrm{CI}}^{\mathtt{CLT}}_{N, \alpha}$ as in \eqref{eq:CI-CLT}. We place one additional condition controlling the ratio-limit process of the squared empirical process. 
\begin{enumerate}[label=\textbf{(A\arabic*-\texttt{ratio})},itemindent=2cm]
\setcounter{enumi}{3}
    \item \label{as:ratio-square-process} For $\rho > 0$, there exist functions $S_{\mathtt{emp}}(n_2, \rho, \alpha)$ and $S_{\mathtt{pop}}(n_2, \rho, \alpha)$ with \[\lim_{t \to\infty} S_{\mathtt{emp}}(t, \rho, \alpha) = 0 \quad \textrm{and}\quad \lim_{t \to\infty} S_{\mathtt{pop}}(t, \rho, \alpha) = 0,\] such that
    \begin{equation*}
        \begin{split}
        &\mathbb{P}_{P^2}^*\left(\sup_{\|\theta-\theta(P^N)\| > \rho} \frac{z_\alpha^2}{n_2^2}\left|\frac{\sum_{i\in I_2}(m_\theta - m_{\theta(P^N)})^2 - \mathbb{E}_{P_i}[(m_\theta - m_{\theta(P^N)})^2]}{\{\M_2(\theta)-\M_2(\theta(P^N))\}^2}\right| \ge \widetilde C_{\mathtt{emp}} S_{\mathtt{emp}}\right) \\
        &\qquad \le \varepsilon_{\mathtt{\mathtt{emp}}}
        \end{split}
    \end{equation*}
    and
    \begin{equation*}
        \sup_{\|\theta-\theta(P^N)\| > \rho}\frac{z_\alpha^2\sum_{i\in I_2}\mathbb{E}_{P_i}[(m_\theta - m_{\theta(P^N)})^2(Z_i)]}{n_2^2\{\M_2(\theta)-\M_2(\theta(P^N))\}^2} \le S_{\mathtt{pop}},
    \end{equation*}
    where $\widetilde C_{\mathtt{emp}} > 0$ is a constant depending on $\varepsilon_{\mathtt{\mathtt{emp}}}$, and $S_{\mathtt{emp}} = S_{\mathtt{emp}}(n_2, \rho, \alpha)$, $S_{\mathtt{pop}} = S_{\mathtt{pop}}(n_2, \rho, \alpha)$ for brevity.
\end{enumerate}
\begin{theorem}\label{thm:clt-width-local}
    Assume $Z_1,\ldots, Z_N$ are independent, \ref{as:margin}, \ref{as:local-entropy}, \ref{as:square-process} hold for all $\|\theta - \theta(P^N)\| \le \rho$ and that \ref{as:rate-initial-estimator3},  \ref{as:margin-global}, \ref{as:ratio-process}, and \ref{as:ratio-square-process} hold. Then for $n_2$ sufficiently large, 
\begin{equation*}
\mathrm{Diam}_{\|\cdot\|}\big(\widehat{\mathrm{CI}}^{\mathtt{CLT}}_{N, \alpha}\big) = O_{\mathbb{P}^*}\bigg(\max\{(1+|z_\alpha|)^{1/(1+\gamma-q)}\mathrm{R}^{\mathtt{CLT}}_{N}, \mathrm{Q}^{\mathtt{CLT}}_{N,\alpha}\mathbf{1}\{\mathrm{Q}^{\mathtt{CLT}}_{N,\alpha} \ge \rho\}\}\bigg),
\end{equation*}
where $ \mathrm{Q}^{\mathtt{CLT}}_{N,\alpha}:= C_\rho
^{-1}\left((1+|z_\alpha|)\widetilde s_{n_1, n_2}\right)$ and $\mathrm{R}^{\mathtt{CLT}}_{N}$ is as defined in \Cref{thm:clt-width}. 
\end{theorem}
The full non-asymptotic statement is provided in \Cref{suppsec:clt-width-local-proof}. 

\section{Some Improvements for Less Conservative Sets}\label{sec:improvements}
The confidence sets, $\widehat{\mathrm{CI}}^{\dagger}_{N}$ and $\widehat{\mathrm{CI}}^{\mathtt{CLT}}_{N, \alpha}$ tends to be conservative. The miscoverage probability often lies strictly below the prescribed level, depending on the behavior of $\ratio$ in \Cref{thm:coverage-anti-conservative-confidence-set-empirical-risk}. This section describes two potential remedies.

\subsection{Contraction of the Initial Estimator}
The proposed construction is agnostic to the choice of $\widehat{\theta}_1$, provided it is obtained approximately independently of $\widehat \M_2(\cdot)$. This flexibility allows to contract $\ratio$ towards zero. One possible device is the convex combination
\begin{equation}
    \widehat\theta'_1(\theta, \lambda) = (1-\lambda) \theta + \lambda \widehat\theta_1 , \quad \lambda \in [0, 1], 
\end{equation}
When $\theta = \theta(P^N)$ and $\lambda$ is close to zero, this shrinks the initial estimator toward the population minimizer.

\begin{example}
    Consider the normal mean setting of \Cref{example:hodges-estimator}. The derivation there gives
    \begin{equation*}
        \frac{\mathbb{C}_2^2(\widehat\theta'_1(\theta, \lambda))}{\mathbb{V}_2(\widehat\theta'_1(\theta, \lambda))} = \frac{n_2((1-\lambda) \mu + \lambda\widehat\theta_1 - \mu)^2}{4\sigma^2} = \lambda^2 \frac{\widehat{\mathbb{C}}_2^2}{\widehat{\mathbb{V}}_2} = \lambda^2 \ratio^2.
    \end{equation*}
   Any $\lambda \in [0,1)$ strictly reduces this ratio.
\end{example}
In view of \Cref{thm:coverage-anti-conservative-confidence-set-empirical-risk}, reducing $\ratio$ brings the miscoverage closer to $1/2$, preventing the confidence set from becoming overly conservative. Taking $\lambda = 0$ leads to a non-informative set as it flattens the curvature and corresponding confidence set \eqref{eq:anti-conservative-confidence-set} will simply be an entire parameter space $\Theta$. Thus $\lambda$ should tend to zero slowly, for instance, at a logarithmic rate. A systematic methodological and theoretical treatment is deferred to future work.

\subsection{Curvature Estimation}
The basic inequality \eqref{eq:zeroth-order} underlying the construction uses the trivial upper bound of zero for $\M_2(\theta)- \M_2(\theta(P^N))$. This may be too loose for parameters far from the minimizer. A sharper confidence set can be obtained by replacing this bound with a upper confidence bound for the curvature. Specifically, suppose one has access to $\widehat u(\theta, \widehat \theta_1)$, satisfying
\begin{equation}\label{eq:curvature-ucb}
    \mathbb{M}_2(\theta(P^N)) - \mathbb{M}_2(\widehat{\theta}_1) \le \min\{\widehat u(\theta(P^N), \widehat \theta_1), 0\}
\end{equation}
with high probability. The corresponding confidence set is 
\begin{equation}\label{eq:anti-conservative-confidence-set-UCB}
\widehat{\mathrm{CI}}^{\mathtt{UCB}}_{N} := \left\{\theta\in\Theta:\, \widehat{\mathbb{M}}_2(\theta) - \widehat{\mathbb{M}}_2(\widehat{\theta}_1) \le \min\{\widehat u(\theta, \widehat \theta_1), 0\}\right\},
\end{equation}
which satisfies $\widehat{\mathrm{CI}}^{\mathtt{UCB}}_{N} \subseteq \widehat{\mathrm{CI}}^{\dagger}_{N}$ by construction. This approach can also be combined with the lower confidence bound methods of \Cref{sec:LCB-basedCI}; for instance, \eqref{eq:CI-CLT} becomes
\begin{equation}\label{eq:LCB-set-UCB}
    \widehat{\mathrm{CI}}^{\mathtt{CLT}-\mathtt{UCB}}_{N, \alpha} := \left\{\theta\in\Theta:\, \widehat{\mathbb{M}}_2(\theta) - \widehat{\mathbb{M}}_2(\widehat{\theta}_1) -n_2^{-1/2}z_\alpha \widehat \sigma_{\theta, \widehat\theta_1}\le \min\{\widehat u(\theta, \widehat \theta_1), 0\}\right\}.
\end{equation}
This technique is employed by \citet{takatsu2025precise} to correct the bias of universal inference \citep{wasserman2020universal} under misspecification. \citet{takatsu2025precise} also establish that the resulting confidence set achieves asymptotically exact coverage under a product rate condition on the estimation errors of $\widehat{\theta}_1$ and $\widehat u(\theta, \widehat\theta_1)$, reminiscent of double robustness in the semiparametric literature \citep{bickel1982adaptive, pfanzagl1985contributions, klaassen1987consistent}.

\begin{example}\label{example:mean-UCB}
    In the normal mean setting of Example~\ref{example:hodges-estimator}, direct calculation gives 
    \begin{equation*}
       \mathbb{M}_2(\mu) - \mathbb{M}_2(\widehat{\theta}_1) = -(\mu -  \widehat\theta_1)^2.
    \end{equation*}
    Taking $\widehat u(\theta, \widehat{\theta}_1) = -(\theta-\widehat{\theta}_1)^2$ satisfies \eqref{eq:curvature-ucb} almost surely. 
\end{example}
\begin{example}[Misspecified linear regression; \citet{takatsu2025precise}]\label{example:LR-UCB}
    Let $(X_1, Y_1), \ldots, (X_N, Y_N) \in \mathbb{R}^d \times \mathbb{R}$ and consider the best linear projection:
    \begin{equation*}
        \theta(P^N) = \argmin_{\theta \in \mathbb{R}^d} \E_{P^2}[(Y-\theta^\top X)^2].
    \end{equation*}
Direct calculation gives   
\begin{equation*}
        \mathbb{M}_2(\theta(P^N)) - \mathbb{M}_2(\widehat{\theta}_1) = -(\theta(P^N) - \widehat{\theta}_1)^\top \E_{P^2}[XX^\top](\theta(P^N) - \widehat{\theta}_1).
    \end{equation*}
    \citet{takatsu2025precise} considers a simple estimator 
    \begin{equation*}
        \widehat u(\theta, \widehat\theta_1) = -(\theta-\widehat\theta_1)^\top \left(\frac{1}{n_2}\sum_{i \in I_2}X_iX_i^\top \right) (\theta-\widehat\theta_1),
    \end{equation*}
    which yields a substantially smaller confidence set in practice.
\end{example}
For many irregular problems, constructing a valid $\widehat u(\theta, \widehat{\theta}_1)$ is itself a difficult or impossible task, without introducing structural assumptions that the proposed framework is designed to avoid. Two examples above are instances of the quadratic curvature regime, i.e., \cref{as:margin} with $\gamma = 1$, which permits estimation of the curvature. See \Cref{fig:demo} for an illustration.
\section{On Computation}\label{sec:computation}
Since the proposed confidence sets are constructed via test inversion, explicit computation may appear challenging in general. Moreover, the sets involving sample variance estimation, namely \eqref{eq:CI-CLT}, can be non-convex due to the dependence of $\widehat\sigma^2_{\theta, \widehat\theta_1}$ on $\theta$ in both the left and right-hand sides of the defining inequality. This non-convexity is visible in \Cref{fig:demo}.

Despite these challenges, the geometry is tractable in some high-dimensional settings. First, the set \eqref{eq:anti-conservative-confidence-set} is convex if and only if the mapping $\theta \mapsto \widehat\M_2(\theta)$ is quasi-convex, that is, 
\begin{equation}\label{eq:quasi-convex}
    \widehat\M_2(t_1) \le 0 \quad \mathrm{and} \quad \widehat\M_2(t_2) \le 0\implies \widehat\M_2(
    \lambda t_1 +(1-\lambda)t_2) \le 0 \quad \textrm{for all}\quad  \lambda \in [0,1].
\end{equation}
The complex dependence in the variance term can be removed by evaluating it at a fixed reference point. Introduce a third data split $D_3$, independent of $D_2$ up to the $\beta$-mixing, and replace $\widehat\sigma^2_{\theta, \widehat\theta_1}$ with 
\begin{equation}\label{eq:variance-estimator-D3}
    \widehat \sigma_{\widehat{\theta}_3, \widehat{\theta}_1}^2 :=     \frac{1}{n_2-1}\sum_{i \in I_2} \left\{(m_{\widehat{\theta}_3}-m_{ \widehat{\theta}_1})(Z_i) - \frac{1}{n_2}\sum_{j\in I_2}(m_{\widehat{\theta}_3}-m_{ \widehat{\theta}_1})(Z_j)\right\}^2,
\end{equation}
where $\widehat{\theta}_3 = \widehat\theta(D_3)$. 
With this substitution, the resulting confidence sets are convex under \eqref{eq:quasi-convex}. Two concrete high-dimensional examples where the geometry can be characterized exactly are presented below.
\subsection{High-dimensional Mean Inference}
Consider normal mean setting of Example~\ref{example:hodges-estimator} with observations partitioned into three parts via $I_1 \cup I_2 \cup I_3 \subseteq \{1, \ldots, N\}$. The population parameter is 
\begin{equation*}
    \theta(P^N) = \argmin_{\theta \in \mathbb{R}^d}\frac{1}{n_2}\sum_{i\in I_2} \E_{P_i}\|Z_i - \theta\|_2^2. 
\end{equation*}
Let $\widehat\theta_1$ and $\widehat\theta_3$ be estimators, obtained from $\{Z_i \, :\, i \in I_1\}$ and $\{Z_i \, :\, i \in I_3\}$ respectively. Consider the confidence sets
\begin{align*}
    % \widehat{\mathrm{CI}}_{N}^{\mathtt{mean}, 1} &= \left\{\theta\in\mathbb{R}^d:\, \sum_{i\in I_2} \|Z_i - \theta\|_2^2 - \|Z_i - \widehat\theta_1\|_2^2\le 0\right\}, \\
    % \widehat{\mathrm{CI}}_{N}^{\mathtt{mean}, 2} &= \left\{\theta\in\mathbb{R}^d:\, \sum_{i\in I_2} \|Z_i - \theta\|_2^2 - \|Z_i - \widehat\theta_1\|_2^2\le -n_2\|\theta-\widehat\theta\|_2^2\right\}, \\
    \widehat{\mathrm{CI}}_{N, \alpha}^{\mathtt{CLT}, 1} &
    =\left\{\theta\in\mathbb{R}^d:\, \frac{1}{n_2}\sum_{i\in I_2} \|Z_i - \theta\|_2^2 - \|Z_i - \widehat\theta_1\|_2^2\le n_2^{-1/2}z_\alpha  \widehat \sigma_{\widehat{\theta}_3, \widehat\theta_1}\right\}, \quad \mathrm{and}\\\widehat{\mathrm{CI}}_{N,\alpha}^{\mathtt{CLT}, 2} &= \left\{\theta\in\mathbb{R}^d:\, \frac{1}{n_2}\sum_{i\in I_2} \|Z_i - \theta\|_2^2 - \|Z_i - \widehat\theta_1\|_2^2\le n_2^{-1/2}z_\alpha  \widehat \sigma_{\widehat{\theta}_3, \widehat\theta_1}-\|\theta-\widehat\theta\|_2^2\right\}.
\end{align*}
where $\widehat{\mathrm{CI}}_{N, \alpha}^{\mathtt{CLT}, 1}$ is the set \eqref{eq:CI-CLT} with variance estimated at $\widehat \theta_3$, and $\widehat{\mathrm{CI}}_{N,\alpha}^{\mathtt{CLT}, 2}$ further combines the upper confidence bound of \Cref{example:mean-UCB} as in \eqref{eq:LCB-set-UCB}. At $\alpha = 1/2$ where $z_\alpha=0$, both sets become equivalent to \eqref{eq:anti-conservative-confidence-set} and \eqref{eq:anti-conservative-confidence-set-UCB} respectively.
\begin{theorem}\label{thm:geometry-mean}
    Both confidence sets are Euclidean balls. Denoting 
    \begin{equation*}
        \overline Z_{2} = n_2^{-1}\sum_{i\in I_2} Z_i \quad \textrm{and} \quad \widehat{H} = \overline Z_2  - \widehat\theta_1,
    \end{equation*}
    the centers and squared radii are as follows:
    \begin{enumerate}
        % \item[\normalfont(1)] $\widehat{\mathrm{CI}}_{N}^{\mathtt{mean}, 1}$: center $\overline Z_2$, squared radius $\|\widehat H\|_2^2$, 
        % \item[\normalfont(2)] $\widehat{\mathrm{CI}}_{N}^{\mathtt{mean}, 2}$: center $(\overline Z_2+\widehat\theta_1)/2$, squared radius  $\|\widehat H\|^2_2/4$,
        \item[\normalfont(1)] $\widehat{\mathrm{CI}}_{N, \alpha}^{\mathtt{CLT}, 1}$: center $\overline Z_2$, squared radius $\|\widehat H\|^2_2 + n_2^{-1/2}z_\alpha  \widehat \sigma_{\widehat{\theta}_3, \widehat\theta_1}$, and 
        \item[\normalfont(2)] $\widehat{\mathrm{CI}}_{N, \alpha}^{\mathtt{CLT}, 2}$: center $(\overline Z_2+\widehat\theta_1)/2$, squared radius $\|\widehat H\|^2_2/4 + n_2^{-1/2}z_\alpha  \widehat \sigma_{\widehat{\theta}_3, \widehat\theta_1}/2$.
    \end{enumerate}
\end{theorem}

\subsection{High-dimensional Misspecified Linear Regression}
Consider $(X_1^\top, Y_1)^\top, \ldots, (X_{N}^\top, Y_{N})^\top \in \mathbb{R}^d \times \mathbb{R}$ with three-way partition $I_1 \cup I_2 \cup I_3 \subseteq \{1, \ldots, N\}$. The population parameter is the best linear projection:
\begin{equation*}
    \theta(P^N) = \argmin_{\theta \in \mathbb{R}^d}\frac{1}{n_2} \sum_{i\in I_2} \E_{P_i}[(Y_i - \theta^\top X_i)^2].
\end{equation*}
Let $\widehat\theta_1$ and $\widehat\theta_3$ be arbitrary estimators in from $\{Z_i \, :\, i \in I_1\}$ and $\{Z_i \, :\, i \in I_3\}$ respectively, and write $\widehat\Gamma = {n_2}^{-1}\sum_{i\in I_2} X_i X_i^\top$. Consider the confidence sets
\begin{equation}\label{eq:linear-regression-cs}
    \begin{split}
    %      \widehat{\mathrm{CI}}_{N}^{\mathtt{LR}, 1} &= \left\{\theta\in\mathbb{R}^d:\, \sum_{i\in I_2} (Y_i - \theta^\top X_i)^2 - (Y_i - \widehat\theta_1^\top X_i)^2\le 0\right\}, \\
    % \widehat{\mathrm{CI}}_{N}^{\mathtt{LR}, 2} &= \left\{\theta\in\mathbb{R}^d:\, \sum_{i\in I_2} (Y_i - \theta^\top X_i)^2 - (Y_i - \widehat\theta_1^\top X_i)^2\le -\|\theta-\widehat\theta_1\|_{{n_2}\widehat\Gamma}^2\right\}, \\
    \widehat{\mathrm{CI}}_{N, \alpha}^{\mathtt{CLT}, 1} &= \left\{\theta\in\mathbb{R}^d:\, \frac{1}{n_2}\sum_{i\in I_2} (Y_i - \theta^\top X_i)^2 - (Y_i - \widehat\theta_1^\top X_i)^2\le {n_2}^{-1/2}z_\alpha  \widehat \sigma_{\widehat{\theta}_3, \widehat\theta_1}\right\}, \quad \mathrm{and}\\\widehat{\mathrm{CI}}_{N,\alpha}^{\mathtt{CLT}, 2} &= \left\{\theta\in\mathbb{R}^d:\, \frac{1}{n_2}\sum_{i\in I_2} (Y_i - \theta^\top X_i)^2 - (Y_i - \widehat\theta_1^\top X_i)^2\le {n_2}^{-1/2}z_\alpha  \widehat \sigma_{\widehat{\theta}_3, \widehat\theta_1}-\|\theta-\widehat\theta_1\|_{\widehat\Gamma}^2\right\}.
    \end{split}
\end{equation}
% where $\widehat{\mathrm{CI}}_{N, \alpha}^{\mathtt{LR}, 1}$ is based on \eqref{eq:CI-CLT}; and $\widehat{\mathrm{CI}}_{N,\alpha}^{\mathtt{LR}, 2}$ combines the \eqref{eq:CI-CLT} and the upper confidence bound as in \eqref{eq:LCB-set-UCB}. The upper confidence bound for linear regression is derived in \Cref{example:LR-UCB}. When $\alpha = 1/2$, the sets become equivalent to \eqref{eq:anti-conservative-confidence-set} and \eqref{eq:anti-conservative-confidence-set-UCB}; 
\begin{theorem}\label{thm:geometry-LR}
    Both confidence sets are ellipsoids with respect to the $\widehat\Gamma$ matrix norm. Denoting
    \begin{equation*}
        \theta_{\mathrm{OLS}} = \left(\sum_{i\in I_2} X_iX_i^\top\right)^{-1}\sum_{i\in I_2} X_iY_i\quad \text{and} \quad \widehat H = \widehat\theta_1 - \theta_{\mathrm{OLS}},
    \end{equation*}
    both confidence sets take the form $\{\theta\in\mathbb{R}^d:\, \|\theta-\mathrm{center}\|_{\widehat \Gamma}^2 \le \mathrm{radius}^2\}$ with
    \begin{enumerate}
        % \item[\normalfont(1)] $\widehat{\mathrm{CI}}_{N}^{\mathtt{LR}, 1}$: center $\theta_{\mathrm{OLS}}$, squared radius $\|\widehat H\|_{\widehat \Gamma}^2$, 
        % \item[\normalfont(2)] $\widehat{\mathrm{CI}}_{N}^{\mathtt{LR}, 2}$: center $(\theta_{\mathrm{OLS}}+\widehat\theta_1)/2$, squared radius  $\|\widehat H\|^2_{\widehat \Gamma}/4$,
        \item[\normalfont(1)] $\widehat{\mathrm{CI}}_{N, \alpha}^{\mathtt{CLT}, 1}$: center $\theta_{\mathrm{OLS}}$, squared radius $\|\widehat H\|^2_{\widehat \Gamma} +{n_2}^{-1/2}z_\alpha  \widehat \sigma_{\widehat{\theta}_3, \widehat\theta_1}$, and 
        \item[\normalfont(2)] $\widehat{\mathrm{CI}}_{N, \alpha}^{\mathtt{CLT}, 2}$: center $(\theta_{\mathrm{OLS}}+\widehat\theta_1)/2$, squared radius $\|\widehat H\|^2_{\widehat \Gamma}/4 + {n_2}^{-1/2}z_\alpha  \widehat \sigma_{\widehat{\theta}_3, \widehat\theta_1}/2$.
    \end{enumerate}
\end{theorem}

\subsection{Further Remarks on Computation}
Other high-dimensional problems may exhibit additional parameter space structure that can be exploited computationally. For instance, Manski's maximum score estimator studied in the forthcoming \Cref{sec:application}, has $\theta \in \mathbb{S}^{d-1}$. 
This constraint allows diameter approximations, outlined in \Cref{supp:manski-algorithm} and employed in the numerical study of \Cref{sec:num}. More generally, the computational aspects are problem-specific and require case-by-case analysis. An alternative that applies broadly is to estimate critical values via nonparametric regression on repeated parameter draws, as proposed by \citet[Section 7]{park2023robust}.

When the confidence set is made convex through the variance estimator \eqref{eq:variance-estimator-D3} or when $\alpha = 1/2$, its volume can be approximated efficiently. Evaluating set membership reduces to a single function evaluation and $\widehat\theta_1$ always belongs to the confidence set. These facts together make the extensive methods for volume approximation of convex bodies directly applicable. See \citet{cousins2018gaussian} and references therein.

Finally, even when computing the full set \eqref{eq:CI-CLT} is difficult, it remains useful for testing: evaluating the defining inequality at a single null value reduces the problem to one function evaluation, and the diameter convergence rate provides an explicit uniform critical radius. More broadly, the existence of the set \eqref{eq:CI-CLT} itself has theoretical value as a concrete procedure that is dimension-agnostic, adaptive to the unknown curvature of the objective, whose validity requires no assumptions beyond the domain of attraction of the normal law.

\section{Statistical Applications}\label{sec:application}
This section establishes theoretical guarantees for the proposed confidence sets in several concrete statistical problems. For each application, the confidence sets $\widehat{\mathrm{CI}}^{\mathtt{CLT}}_{N, \alpha}$ are constructed directly from the general formulae \eqref{eq:CI-CLT}. At $\alpha = 1/2$, the set reduces to $\widehat{\mathrm{CI}}^{\dagger}_{N}$ as in \eqref{eq:anti-conservative-confidence-set} since $z_\alpha = 0$. Two sets of theoretical guarantees are provided for each application: sufficient conditions for validity and convergence rates for the diameter. To this end, we can observe
\begin{align*}
    \widehat{\mathrm{CI}}^{\mathtt{CLT}}_{N, \alpha_1} \subseteq \widehat{\mathrm{CI}}^{\dagger}_{N}\subseteq \widehat{\mathrm{CI}}^{\mathtt{CLT}}_{N, \alpha_2} \quad \text{when} \quad \alpha_2 \le 1/2 \le \alpha_1.
\end{align*}
Hence 
\begin{align*}
    \mathbb{P}_{P^N}\!\left(\theta(P^N) \notin \widehat{\mathrm{CI}}^{\mathtt{CLT}}_{N,\alpha_2}\right) \le \mathbb{P}_{P^N}\!\left(\theta(P^N) \notin \widehat{\mathrm{CI}}^{\dagger}_{N}\right) \quad \text{for any} \quad \alpha_2 \le 1/2,
\end{align*}
and 
\begin{align*}
    \mathrm{Diam}_{\|\cdot\|}(\widehat{\mathrm{CI}}^{\mathtt{CLT}}_{N,\alpha_1}) \le \mathrm{Diam}_{\|\cdot\|}(\widehat{\mathrm{CI}}^{\dagger}_{N}) \quad \text{for any} \quad 1/2 \le \alpha_1.
\end{align*}
Thus the validity result for $\widehat{\mathrm{CI}}^{\dagger}_{N}$ also implies the validity for $\widehat{\mathrm{CI}}^{\mathtt{CLT}}_{N,\alpha}$ whenever $\alpha \le 1/2$. Similarly, the diameter bound for $\widehat{\mathrm{CI}}^{\dagger}_{N}$ also implies the same diameter bound for $\widehat{\mathrm{CI}}^{\mathtt{CLT}}_{N,\alpha}$ with $1/2 \le \alpha$. All results follow from the general theory of \Cref{sec:general-M-estimation} and \Cref{sec:convergence-rates} applied to the specific scenarios. 

Throughout, $Z_1, \ldots, Z_N$ are assumed independent but not necessarily identically distributed. The observations are split with $|I_1| = n_1$, $|I_2|=n_2$ and $r=0$ so that $N=n_1 + n_2$. 

\subsection{High-dimensional Mean Inference}\label{sec:mean}
Consider independent observations $X_1, \ldots, X_N \in \mathbb{R}^d$ with a common mean. The inference of interest is the expectation of $X$ under $P^2$:
\begin{equation*}
    \theta(P^N) := \argmin_{\theta \in \mathbb{R}^d}\, \frac{1}{n_2}\sum_{i\in I_2}\E_{P_i}\|X_i-\theta\|_2^2.
\end{equation*}
The covariance matrix of $X_i$ is allowed to vary across $i$, denoted $\Sigma_i := \E_{P_i}[(X_i-\theta(P^N))(X_i-\theta(P^N))^\top]$. Although mean estimation may appear elementary, inference for the mean in growing dimensions under weak distributional assumptions remains an active area of research \citep{lugosi2019mean}. Denote the average covariance matrices within each split by
\begin{equation*}
    \bar \Sigma_{k } = n_k^{-1} \sum_{i \in I_k} \Sigma_i \quad \textrm{for} \quad k \in \{1, 2\}. 
\end{equation*} 
\paragraph{Validity.} 
\begin{theorem}\label{thm:mean-ci-validity}
Let $X_i^\circ =\bar\Sigma_2^{-1/2}(X_i - \theta(P^N))$ and define 
\begin{equation*}
        R_{n_2} = \sup_{u \in \mathbb{S}^{d-1}}\sum_{i\in I_2} \mathbb{E}_{P_i}\left[\frac{\langle X_i^\circ, u\rangle^2}{n_2}\min\left\{1,\,\frac{|\langle X_i^\circ, u\rangle|}{\sqrt{n_2}}\right\}\right].
    \end{equation*}
    There exists a universal constant $C > 0$ such that the following hold. 
    \begin{enumerate}
    \item For $\alpha \in (0,1)$ and $n_2 \ge 2$,
    \begin{equation*}
    \mathbb{P}_{P^N}\!\left(\theta(P^N) \notin \widehat{\mathrm{CI}}^{\mathtt{CLT}}_{N,\alpha}\right)
\;\le\; \alpha + \min\{1, CR_{n_2}\}.
    \end{equation*}
    \item For $\alpha \in (0, 1/2]$ and $n_2 \ge 1$, setting $\ratio = \sqrt{n_2}\|\bar\Sigma_2^{-1/2}(\widehat\theta_1 - \theta(P^N))\|/2$,
    \begin{equation*}
        \mathbb{P}_{P^N}\left(\theta(P^N) \notin \widehat{\mathrm{CI}}^{\mathtt{CLT}}_{N,\alpha}\right)
\le \min \left\{\E_{P^1}\left[1-\Phi(\ratio) + \frac{C}{(1+\ratio)^2}\right],\; \alpha + \min\{1, CR_{n_2}\}\right\}.
    \end{equation*}
\end{enumerate}
\end{theorem}
The remainder term $R_{n_2}$ satisfies
\begin{equation*}
    R_{n_2} \le C n_2^{-(1+\delta/2)}\sup_{u \in \mathbb{S}^{d-1}}\sum_{i\in I_2} \mathbb{E}_{P_i}[|\langle X_i^\circ, u\rangle|^{2+\delta}].
\end{equation*}
The condition $\sup_{u \in \mathbb{S}^{d-1}}\mathbb{E}_{P_i}[|\langle X_i^\circ, u\rangle|^{2+\delta}] \le K$ for some constant $K \ge 1$ is the so-called $L_{2+\delta}$-$L_2$ norm equivalence, under which $R_{n_2} \le CKn_2^{-\delta/2}$ without any dependence on $d$. This condition, with particular emphasis on $\delta=1$, is widely employed in high-dimensional covariance matrix estimation \citep{minsker2018sub, mendelson2020robust} and high-dimensional least squares \citep{oliveira2016lower, catoni2016pac, mourtada2022distribution}. 
This assumption is considerably less restrictive than imposing the sub-Gaussianity of $X$ since any such $X$ satisfies the $L_{2+\delta}$-$L_2$ norm equivalence with $\delta\ge2$. See Remarks 2.19, 2.20 and Figure S.7 of \cite{patil2022mitigating} for useful visual comparison. Note that \Cref{thm:mean-ci-validity} establishes the validity under weaker conditions.

When $\alpha \le 1/2$, validity is further controlled by the first argument of the outer minimum, which decays whenever $\ratio$ is large. This occurs when the initial estimator $\widehat\theta_1$ converges to $\theta(P^N)$ slower than $n_2^{-1/2}$. To illustrate this in high-dimensional problems, if $\widehat \theta_1$ is the sample mean on $D_1$, then
\begin{equation*}
    \|\widehat\theta_1 - \theta(P^N)\| \asymp \sqrt{d/n_1} \implies  \ratio \asymp \sqrt{dn_2/n_1}/\lambda^{1/2}_{\max}(\bar\Sigma_2).
\end{equation*}
Remarkably, when $n_1 \asymp n_2$ and $\lambda_{\max}(\bar\Sigma_2) = O(1)$, the validity condition reduces to $d \to \infty$ alone, with no dependence on the proportion $d/n_1$ or $d/n_2$. This is substantially weaker than any moment equivalent condition.  

\paragraph{Width Analysis.} 

\begin{theorem}\label{thm:mean-ci-width}
There exists a universal constant $C > 0$ such that the following hold.
\begin{enumerate}
    \item For $\alpha \in [1/2,1)$, $n_1, n_2 \ge 1$, and any $\varepsilon \in (0,1)$, with probability at least $1-\varepsilon$, 
    \begin{equation*}
        \mathrm{Diam}_{\|\cdot\|_2}\bigl(\widehat{\mathrm{CI}}_{N,\alpha}^{\mathtt{CLT}}\bigr)
    \leq C\varepsilon^{-1/2}
    \left\{\sqrt{\frac{\mathrm{tr}(\bar\Sigma_2)}{n_2}} 
    + \|\widehat{\theta}_1 - \theta(P^N)\|_2\right\}.
    \end{equation*}
    \item For $\alpha \in ( 0, 1)$, $n_1 \ge 1$, and any $\varepsilon \in (0,1)$, let 
$\widetilde{s}_{n_1,n_2}$ be as in \ref{as:rate-initial-estimator3}. With probability at least $1-\varepsilon$, provided $\max\{2,z_\alpha^2 C'_{\varepsilon}\}\le n_2$ 
    \begin{equation*}
        \mathrm{Diam}_{\|\cdot\|_2}\bigl(\widehat{\mathrm{CI}}_{N,\alpha}^{\mathtt{CLT}}\bigr)
    \leq C_{\varepsilon}\left(1+|z_{\alpha}|\right)
    \left\{\sqrt{\frac{\mathrm{tr}(\bar\Sigma_2)}{n_2}} 
    + \widetilde{s}_{n_1,n_2}^{1/2}\right\},
    \end{equation*}
where $C_{\varepsilon}$ and $C'_{\varepsilon}$ depend on $\varepsilon$, but not on $d$ or $\alpha$.
\end{enumerate}
\end{theorem}
\begin{corollary}\label{cor:mean-plugin}
     Suppose the initial estimator satisfies, for all $n_1 \geq N_1$, \begin{align}\label{eq:mean-init-requirement}
        \|\widehat\theta_1 - \theta(P^N)\|_2^2 = O_{P^1}\left(\frac{\mathrm{tr}(\bar \Sigma_1)}{n_1}\right).
    \end{align}
    For any $\varepsilon \in (0,1)$, and $n_1 \ge N_1$, with probability at least $1-\varepsilon$,
    \begin{align*}
            \mathrm{Diam}_{\|\cdot\|_2}\bigl(\widehat{\mathrm{CI}}_{N,\alpha}^{\mathtt{CLT}}\bigr)
    \leq C_{\varepsilon}\left(1+|z_{\alpha}|\right)
    \left\{\sqrt{\frac{\mathrm{tr}(\bar\Sigma_2)}{n_2}} 
    + \sqrt{\frac{\mathrm{tr}(\bar\Sigma_1)}{n_1}}\right\},
        \end{align*}
provided $n_2 \ge 1$ when $\alpha \ge 1/2$, and $n_2 \ge  \max\{2,z_\alpha^2 C'_{\varepsilon}\}$ when $\alpha < 1/2$, 
    where $C_{\varepsilon}$ and $C'_{\varepsilon}$ depend on $\varepsilon$, but not on $d$ or $\alpha$.
\end{corollary}

The minimum sample size requirements reflect a genuine distinction the two cases. For $\alpha \ge 1/2$, the rate holds under $n_2 \ge 1$ by directly analyzing the analytical expression for the diameter given in \Cref{thm:geometry-mean}, bypassing sample variance estimation entirely.  For $\alpha < 1/2$, the sample size must be large enough depending on $\varepsilon$, as the result relies on \Cref{thm:clt-width-local} which requires concentration of the sample variance.

Condition \eqref{eq:mean-init-requirement} is satisfied when $\widehat{\theta}_1$ is the sample mean based on $D_1$. When $\mathrm{tr}(\bar\Sigma_1)=\mathrm{tr}(\bar\Sigma_2)$, \Cref{cor:mean-plugin} implies that the balanced split $n_1 = n_2 =N/2$ minimizes the diameter, yielding the rate $\sqrt{2\mathrm{tr}(\bar\Sigma_1)/N}$, which is minimax optimal: it matches the exact risk of mean estimation under a multivariate Gaussian, as established by the formal lower bound argument of \citet[Section~5]{lee2022optimal}. Both \Cref{thm:mean-ci-validity} and \Cref{thm:mean-ci-width} impose no restriction on $d$, hence both the validity guarantee and the convergence rate are dimension-agnostic. 

The results in this subsection extend to inference for the Fr\'{e}chet mean on a general metric space $\Theta$, where \ref{as:margin} is referred to as a growth condition, or variance inequality. The quadruple condition studied in \citet{Schotz2019convergence} can be used to establish \ref{as:local-entropy}. A detailed treatment is deferred to future work.

\subsection{High-dimensional Misspecified Linear Regression}\label{sec:ols} 
Consider independent observations $(X_1^\top, Y_1)^\top, \ldots, (X_N^\top,Y_N)^\top \in \mathbb{R}^d\times\mathbb{R}$ generated from
\begin{align*}
    Y_i = \theta(P^N)^\top X_i + \varepsilon_i \quad\text{where}\quad \sum_{i \in I_2}\mathbb{E}_{P_i}[\varepsilon_i X_i] = 0 \quad\text{and}\quad\mathbb{E}_{P_i}[\varepsilon_i^2 | X_i] = \sigma_i^2.
\end{align*}
The inference of interest is the best linear projection under $P^2$: 
\begin{equation*}
    \theta(P^N) := \argmin_{\theta \in \mathbb{R}^d}\, \frac{1}{n_2}\sum_{i \in I_2}\E_{P_i}[(Y-\theta^\top X)^2].
\end{equation*}
Denote the average Gram matrices within each split by 
\begin{align*}
    \bar \Gamma_k = n_k^{-1}\sum_{i \in I_k}\E_{P_i}[X_i X_i^\top] \quad \textrm{for} \quad k \in \{1,2\}.
\end{align*}
We assume that $\bar \Gamma_2$ is invertible such that $\theta(P^N)$ exists even when the regression function $\E_{P^2}[Y_i| X_i]$ is not linear. 
% The constructed confidence set is denoted by $\widehat{\mathrm{CI}}_{N, \alpha}^{\mathtt{LR}}$ with the loss function $m_{\theta}(y,x) = (y-\theta^\top x)^2$; when $\alpha =1/2$ it coincides with \eqref{eq:anti-conservative-confidence-set}, otherwise it is constructed following \eqref{eq:CI-CLT}. 
We introduce the following assumptions:
\begin{enumerate}[label=\textbf{(B\arabic*)},leftmargin=2cm]
\item \label{as:eigen_value} 
There exists a constant $\bar{\sigma} > 0$ such that $\sigma_i \le \bar{\sigma}$ for all $i \in I_2$.
\item \label{as:linreg-moment-covariate} 
There exist constants $q_x \ge 2$ and $L \ge 1$ such that 
\begin{align*}
   \E_{P_i}[(u^\top \bar\Gamma_{2}^{-1/2}X_i)^{q_x}]\le L^{q_x},
\end{align*}
for all $u \in \mathbb{S}^{d-1}$ and $i \in I_2$.

\item \label{as:error_moment}
There exists constants $q_y \ge 2$ and $K > 0$ such that
\begin{align*}
    \E_{P_i}[|Y_i - \theta(P^N)^\top X_i|^{q_y} \mid X_i]\le K^{q_y}
\end{align*}
for all $ i \in I_2$.
\end{enumerate}
Condition \ref{as:eigen_value} bounds the conditional variance of the population residuals. Condition \ref{as:linreg-moment-covariate} imposes an $L_{q_x}$-$L_2$ norm equivalence on covariates, allowing for heavy-tailed distributions and forging sub-Gaussian or sub-exponential assumptions. Condition \ref{as:error_moment} requires finite conditional moments of the population residuals.
\paragraph{Validity.}
% We provide two results depending on the consistency of the initial estimator. 
% The following result is obtained:
\begin{theorem}\label{thm:LR-valid}
Assume \ref{as:eigen_value} and \ref{as:linreg-moment-covariate} with $q_x = 4$. Let $\Sigma_i = \mathrm{Cov}_{P_i}(X_i)$, and suppose there exits constants $\underline{\sigma},\underline{\lambda} > 0$ such that $\underline{\sigma} \le \sigma_i$, and $\underline{\lambda} \le \lambda_{\min}(\bar\Gamma_2^{-1/2}\Sigma_i \bar\Gamma_2^{-1/2})$ for all $i \in I_2$. Define the sandwich covariance matrix and normalized score
    \begin{equation*}
        \bar H_2 = \frac{1}{n_2}\sum_{i \in I_2} \mathrm{Cov}_{P_i}(X_i \epsilon_i) \quad \textrm{and} \quad W_i^\circ = \bar H_2^{-1/2}(X_i\epsilon_i - \E_{P_i}[X_i\epsilon_i]),
    \end{equation*}
    and define
    \begin{equation*}
    \begin{split}
        R_{n_2} &= \inf_{\delta > 0}\left\{\frac{2L^2\delta}{\underline{\sigma} \sqrt{\underline{\lambda}}} + \mathbb{P}_{P^1}(\|\widehat\theta_1 - \theta(P^N)\|_{\bar\Gamma_2} > \delta)\right\}\\
        &\quad +\sup_{u \in \mathbb{S}^{d-1}}\sum_{i\in I_2} \mathbb{E}_{P_i}\left[\frac{\langle W_i^\circ, u\rangle^2}{n_2}\min\left\{1,\,\frac{|\langle W_i^\circ, u\rangle|}{\sqrt{n_2}}\right\}\right].
    \end{split}
    \end{equation*}
    There exists a universal constant $C > 0$ such that the following hold.
        \begin{enumerate}
    \item For $\alpha \in (0,1)$ and $n_2 \ge 2$,
    \begin{equation*}
    \mathbb{P}_{P^N}\!\left(\theta(P^N) \notin \widehat{\mathrm{CI}}^{\mathtt{CLT}}_{N,\alpha}\right)
\;\le\; \alpha + \min\{1, CR_{n_2}\}.
    \end{equation*}
    \item For $\alpha \in (0, 1/2]$ and $n_2 \ge 1$, setting 
    \begin{equation*}
        \widetilde\Delta_2^2 = \frac{n_2\|\widehat{\theta}_1 -\theta(P^N)\|_{\bar\Gamma_2}^2}{4\bar\sigma^2 + 2L^4\|\widehat{\theta}_1 -\theta(P^N)\|_{\bar\Gamma_2}^2},
    \end{equation*}
    it holds 
    \begin{equation*}
        \mathbb{P}_{P^N}\left(\theta(P^N) \notin \widehat{\mathrm{CI}}^{\mathtt{CLT}}_{N,\alpha}\right)
\le \min \left\{\E_{P^1}\left[1-\Phi(\widetilde\Delta_2) + \frac{C}{(1+\widetilde\Delta_2)^2}\right],\; \alpha + \min\{1, CR_{n_2}\}\right\}.
    \end{equation*}
\end{enumerate}
\end{theorem}
As in \Cref{thm:mean-ci-validity}, the remainder $R_{n_2}$ becomes negligible under a finite $(2+\delta)$-th moment condition $\sup_{u \in \mathbb{S}^{d-1}}\mathbb{E}_{P_i}[|\langle W_i^\circ, u\rangle|^{2+\delta}] \le K$. The first term in $R_{n_2}$ also requires consistency of $\widehat\theta_1$ in the $\bar\Gamma_2$-norm, which implicitly places a dimension requirement, such that $d = o(n_1)$. 

When $\alpha \le 1/2$, validity holds under weaker conditions when $\widetilde\Delta_2$ is large. In particular, 
\begin{equation*}
    \widetilde\Delta_2 \to \infty \quad\text{when} \quad \min\left\{\frac{n_2 \|\widehat{\theta}_1 -\theta(P^N)\|_{\bar\Gamma_2}^2}{\bar\sigma^2}, \frac{n_2}{L^4}\right\} \to \infty.
\end{equation*}
This does not require consistency of $\widehat{\theta}_1$. If the initial estimator is inconsistent or converges slowly, for instance, a penalized estimator with a large regularization parameter, then $\widetilde\Delta_2$ diverges as $n_2 \to \infty$. In this regime, $n_1$ may be much smaller than $d$. As an illustration, consider the case when $\widehat{\theta}_1$ is the OLS estimator on $D_1$. Then 
\begin{equation*}
    n_2\|\widehat\theta_1 - \theta(P^N)\|^2_{\bar\Gamma_2} \gtrsim \frac{dn_2}{n_1}\cdot\lambda_{\min}(\bar\Gamma_1^{-1/2}\bar\Gamma_2\bar\Gamma_1^{-1/2}).
\end{equation*}
Therefore when $n_1 \asymp n_2$ and $\lambda_{\max}(\bar\Gamma_1^{-1/2}\bar\Gamma_2\bar\Gamma_1^{-1/2})$ bounded away from zero, validity at any level $\alpha \le 1/2$ holds as long as $d,n_2 \to \infty$ and does not require any proportional behavior of $n_2$ and $d$. 

\paragraph{Width Analysis.}
% Next, we provide the width of the confidence set in terms of the matrix-norm $\|\cdot\|_{\Gamma_P}$:
% \kt{
\begin{theorem}\label{thm:LR-ci-width}
    There exists a universal constant $C > 0$ such that the following hold.
\begin{enumerate}
    \item Assume \ref{as:eigen_value} and \ref{as:linreg-moment-covariate} with $q_x>2$. For $\alpha \in [1/2,1)$, $n_1 \ge 1$, and any $\varepsilon \in (0,1)$, with probability at least $1-\varepsilon - \exp(-Cn_2)$,
    \begin{equation*}
        \mathrm{Diam}_{\|\cdot\|_{ \bar\Gamma_2}}\bigl(\widehat{\mathrm{CI}}_{N,\alpha}^{\mathtt{CLT}}\bigr)
    \leq C\varepsilon^{-1/2}
    \left\{\sqrt{\frac{\bar\sigma^2 d}{n_2}} 
    + \|\widehat{\theta}_1 - \theta(P^N)\|_{\bar\Gamma_2}\right\}
    \end{equation*}
    provided that $n_2$ satisfies $n_2 \ge \mathfrak{C}d$ where $\mathfrak{C}$ depends only on $q_x$.
    \item Assume \ref{as:eigen_value}, \ref{as:linreg-moment-covariate} with $q_x\ge 4$, \ref{as:error_moment}, and let 
$\widetilde{s}_{n_1,n_2}$ be as in \ref{as:rate-initial-estimator3}. For $\alpha \in (0, 1)$, $n_1 \ge 1$, and any $\varepsilon \in (0,1)$, with probability at least $1-\varepsilon$, 
    \begin{equation*}
        \mathrm{Diam}_{\|\cdot\|_{\bar\Gamma_2}}\bigl(\widehat{\mathrm{CI}}_{N,\alpha}^{\mathtt{CLT}}\bigr)
    \leq C_{\varepsilon}(1+|z_{\alpha}|)\left\{\sqrt{\frac{\overline{\sigma}^2 d}{n_2}} + \widetilde s_{n_1, n_2}^{1/2}\right\}
    \end{equation*}
provided that $n_2$ satisfies 
\begin{equation}\label{eq:lr-sample-req-main}
    C'_{\varepsilon^\circ} \max\bigg\{((1+|z_\alpha|) d\log (2d) L^4)^{q_x/(q_x-2)}, ((1+|z_\alpha|)(1+K) dL^2)^{p/(p-1)}\bigg\} \le n_2,
\end{equation}
where $p = \min\{q_y, q_x/2\}$, and $C_{\varepsilon}, C'_{\varepsilon}$ depend on $\varepsilon$, but not on $d$ or $\alpha$.
\end{enumerate}
\end{theorem}

\begin{corollary}\label{cor:lr-plugin}
    Suppose the initial estimator satisfies, for all $n_1 \geq N_1$, \begin{align}\|\widehat \theta_1-\theta(P^N)\|^2_{\bar\Gamma_1} = O_{P^1}\left(\frac{d \bar\sigma^2}{n_1}\right).\label{eq:requirement-ols}
    \end{align}
    Assume \ref{as:eigen_value} and \ref{as:linreg-moment-covariate} with $q_x > 2$ when $\alpha \ge 1/2$ and $q_x \ge 4$ when $\alpha < 1/2$. Additionally assume \ref{as:error_moment} when $\alpha < 1/2$. For any $\varepsilon \in (0, 1)$, $n_1 \ge N_1$, 
    \begin{align*}
\mathrm{Diam}_{\|\cdot\|_{\bar\Gamma_2}}\bigl(\widehat{\mathrm{CI}}_{N,\alpha}^{\mathtt{CLT}}\bigr)
    \leq C_{\varepsilon}(1+|z_{\alpha}|)
    \left\{\sqrt{\frac{\bar\sigma^2 d}{n_2}} 
    + \sqrt{\frac{\bar\sigma^2 d}{n_1}}\lambda^{1/2}_{\max}(\bar\Gamma_1^{-1/2}\bar\Gamma_2\bar\Gamma_1^{-1/2})\right\}
        \end{align*}
        with probability at least $1-\varepsilon - \exp(-Cn_2)$, 
        provided $n_2 \ge \mathfrak{C}d$ when $\alpha \ge 1/2$, and with probability $1-\varepsilon$ provided $n_2$ satisfies \eqref{eq:lr-sample-req-main} when $\alpha < 1/2$, where $C$ is a universal constant, $p = \min\{q_y, q_x/2\}$, $\mathfrak{C}$ depends on $q_x > 2$, $C_{\varepsilon}$ and $C'_{\varepsilon}$ depend on $\varepsilon$, but not on $d$ or $\alpha$.
\end{corollary}
\Cref{thm:LR-ci-width} reveals a salient distinction between $\alpha \ge 1/2$ and $\alpha < 1/2$. For $\alpha \ge1/2$, the width guarantee requires the weak moment assumptions \ref{as:eigen_value} and \ref{as:linreg-moment-covariate} with $q_x > 2$, and a sample size $n_2 \ge \mathfrak{C}d$ for $\mathfrak{C}$ depending only on $q_x$. This is obtained by directly analyzing the analytical expression for the diameter in \Cref{thm:geometry-LR}, bypassing sample variance estimation, as in the mean inference case. For $\alpha < 1/2$, the width guarantee is obtained as an application of \Cref{thm:clt-width-local} under stronger assumptions to control concentration of the sample variance, namely \ref{as:linreg-moment-covariate} with $q_x \ge 4$ and \ref{as:error_moment}. Under the weakest moment conditions implied by these assumptions, the sample size requirement becomes $d = o(\sqrt{n_2})$, reflecting a well-known quadratic barrier. Throughout, the assumptions permit heavy-tailed covariates without any sub-Gaussian condition.

Condition \eqref{eq:requirement-ols} is satisfied when $\widehat{\theta}_1$ is the OLS estimator based on $D_1$. When $\bar\Gamma_2 = \bar\Gamma_1$, \Cref{cor:lr-plugin} implies that the balances split $n_1 = n_2 = N/2$ minimizes the diameter, yielding the rate $\sqrt{2\bar\sigma^2 d/N}$, which is minimax optimal as established by Theorem 1 of \citet{mourtada2022exact}. The combined guarantees of \Cref{thm:LR-valid} and \Cref{thm:LR-ci-width} at $\alpha = 1/2$ yield dimension-agnostic validity with diameter rate $\sqrt{d\bar\sigma^2/n_2}$ under $\mathfrak{C} d \le n$, which appears to be new. The same guarantee extends to any fixed $\alpha \le 1/2$ by employing the data-splitting method of \Cref{sec:data-splitting}.

Recently, \citet{chang2024confidence} proposed confidence sets for high-dimensional OLS based on $Z$-estimation. Theorem~8 of \citet{chang2024confidence} establishes a similar diameter bound but requires $d = o(n^{1-2/q_x})$, matching the first term of \eqref{eq:eqreuiment_for_n2_simplified} in \Cref{thm:LR-ci-width} with $\alpha < 1/2$. \citet{chang2023inference} provides a one-step bias-corrected method, improving the requirement to $d = o(n^{2/3})$. 

\begin{remark}
    Since $\mathbb{C}_{P^2}(\theta) = \|\theta - \theta(P^N)\|^2_{\bar\Gamma_{2}} \ge \lambda_{\min}(\bar\Gamma_{2})\|\theta - \theta(P^N)\|^2_2$,
    \ref{as:margin} holds with $c_0 =  \lambda_{\min}(\bar\Gamma_{2})$ and $\gamma=1$. Therefore, \Cref{thm:LR-ci-width} can be restated in terms of the $\|\cdot\|_2$-norm with an additional assumption that $\lambda_{\min}(\bar\Gamma_{2}) > \underbar{$\lambda$}$ for some constant $\underbar{$\lambda$} > 0$.
\end{remark}

\begin{remark}
    The proposed framework can be easily extended to penalized least squares, where the minimizer is defined as\begin{equation}
    \theta(P^N) := \argmin_{\theta \in \mathbb{R}^d}\, \frac{1}{n_2}\sum_{i \in I_2}\E_{P_i}[(Y-\theta^\top X)^2] + \lambda(\theta),\label{eq:pen-objective}
\end{equation}
and $\lambda : \Theta \mapsto \mathbb{R}_+\cup\{+\infty\}$ is a convex function of $\theta$, which may depend on $n$ (but not on data). We obtain the following validity result.
\begin{theorem}\label{thm:LR-valid-pen}
    Assume $\lambda(\cdot)$ is a convex function. Then, the same conclusion as \Cref{thm:LR-valid} holds for the confidence set $\widehat{\mathrm{CI}}^{\mathtt{CLT}}_{N,\alpha}$ for the penalized objective \eqref{eq:pen-objective}. 
\end{theorem}
% In this case, we can construct the confidence set as
% \begin{align*}
%     \left\{\theta \in \Theta \, : \, \mathbb{P}_n(m_{\theta}-m_{\widehat \theta_1}) + \lambda(\theta) - \lambda(\widehat \theta_1)\le n^{-1/2}z_{\alpha} \widehat\sigma_{\theta, \widehat\theta_1}\right\}
% \end{align*}
% with $m_{\theta}(y,x) = (y-\theta^\top x)^2$ and $\widehat\sigma_{\theta, \widehat\theta_1}$ is defined as in \eqref{eq:sample-variance}. The validity of this set holds under the same assumptions as Theorem~\ref{thm:LR-valid}. However, the corresponding Width Analysis requires considerably more effort, as it involves the limiting behavior of the sequence $\{\lambda(\theta(P^N)) - \lambda(\widehat \theta_1)\}$ as $\|\widehat \theta_1 - \theta(P^N)\| = o_P(1)$. 
\end{remark}
\subsection{Manski's Discrete Choice Model}\label{sec:manski}
Consider independent observations $(X_1^\top,Y_1)^\top, \ldots (X_N^\top,Y_N)^\top \in \mathbb{R}^d \times \{-1,1\}$ generated from the binary response model:
\begin{equation}
    Y_i := \mathrm{sgn}(\theta(P^N)^\top X_i + \varepsilon_i) \quad \text{where} \quad \mathrm{sgn}(t) = 2\mathbf{1}\{t \ge 0\}-1,
\end{equation}
where the error $\varepsilon_i$ has zero conditional median given $X_i$, i.e., $\mathrm{med}(\varepsilon_i | X_i)=0$, but is otherwise allowed to depend on $X_i$. The inference of interest is
\begin{align*}
    \theta(P^N) := \argmax_{\theta \in \mathbb{S}^{d-1}}\, \sum_{i \in I_2} \E_{P_i}[Y_i\, \mathrm{sgn}(\theta^\top X_i)].
\end{align*}
A natural estimator of $\theta(P^N)$ is maximum score estimator \citep{manski1975maximum}, defined as
\begin{align}\label{eq:maximum-score-estimator}
    \widehat\theta_1 := \argmax_{\theta \in \mathbb{S}^{d-1}}\, \sum_{i\in I_1} Y_i\, \mathrm{sgn}(\theta^\top X_i).
\end{align}
The asymptotic behavior of $\widehat \theta_1$ is non-standard \citep{manski1985semiparametric, kim1990cube}. The inference for this problem is challenging, and \cite{cattaneo2020bootstrap} proposes a bootstrap-based approach. The following assumptions on the joint distribution of $(X_i, \varepsilon_i)$ are standard in the literature \citep{mukherjee2019nonstandard, mukherjee2021optimal}
\begin{enumerate}[label=\textbf{(B\arabic*)},leftmargin=2cm]
\setcounter{enumi}{3}
\item \label{as:tsybakov} Let $\eta_{P_i}(x) := \mathbb{P}_{P_i}(Y=1 \mid X=x)$. There exist constants $C_0 > 0$, $0 < t^* < 1/2$, and $\gamma > 0$, such that 
\begin{align*}
    \mathbb{P}_{P_i}\left(\left|\eta_{P_i}(X)- \frac{1}{2}\right| < t\right) \le C_0t^{1/\gamma},
\end{align*}
for all $0 < t < t^*$ and $i \in I_2$.
\item \label{as:covariate-manski}
There exists a constant $c_1 > 0$, not depending on $n_2$ or $d$, such that 
\begin{align*}
    c_1\|\theta-\theta(P^N)\|_2 \le \mathbb{P}_{P_i}\left(\mathrm{sgn}(\theta^\top  X_i) \neq \mathrm{sgn}(\theta(P^N)^\top  X_i)\right)
\end{align*}
for all $\theta \in \mathbb{S}^{d-1}$ and $i \in I_2$.
\end{enumerate}
Assumption~\ref{as:tsybakov} is the \textit{low noise (margin) assumption} in the classification literature \citep{mammen1999smooth, tsybakov2004optimal}. It quantifies the deviation of the conditional class probability from $1/2$ near the decision boundary; as $\gamma \longrightarrow 0$, the decision boundary becomes well-separated, representing the most favorable situation for estimation. Assumption \ref{as:covariate-manski} relates the distribution of covariates $X_i$ to the geometry in the parameter space $\mathbb{S}^{d-1}$. See \cite{Audibert2007Fast, mukherjee2021optimal} for further discussion. 
\paragraph{Validity.}
\begin{theorem}\label{thm:manski-validity-consistent}
Assume \ref{as:covariate-manski} and define $R_{n_2} = (n_2c_1\|\widehat{\theta}_1 -\theta(P^N)\|)^{-1/2}$. There exists a universal constant $C > 0$ such that the following hold.
        \begin{enumerate}
    \item For $\alpha \in (0, 1)$ and $n_2 \ge 2$,
    \begin{equation*}
    \mathbb{P}_{P^N}\!\left(\theta(P^N) \notin \widehat{\mathrm{CI}}^{\mathtt{CLT}}_{N,\alpha}\right)
\;\le\; \alpha + \E_{P^1}[\min\{1, CR_{n_2}\}].
    \end{equation*}
    \item Additionally assume \ref{as:tsybakov}. For $\alpha \in (0, 1/2]$ and $n_2 \ge 1$, setting 
    \begin{equation*}
        \widetilde\Delta_2^2 = \mathfrak{C}n_2\|\widehat\theta_1- \theta(P^N)\|\min\{\|\widehat\theta_1- \theta(P^N)\|^{2\gamma}, (t^*)^2\},
    \end{equation*}
    with $\mathfrak{C}$ depending on $C_0$ and $c_1$,
    it holds 
    \begin{equation*}
    \begin{split}
    &\mathbb{P}_{P^N}\left(\theta(P^N) \notin \widehat{\mathrm{CI}}^{\mathtt{CLT}}_{N,1/2}\right)
\\
&\quad \le \min \left\{\E_{P^1}\left[1-\Phi(\widetilde\Delta_2) + \frac{C}{(1+\widetilde\Delta_2)^2}\right],\; \alpha + \E_{P^1}[\min\{1, CR_{n_2}\}]\right\}.
    \end{split}
    \end{equation*}
\end{enumerate}
\end{theorem}
The validity requires $n_2 \|\widehat{\theta}_1 -\theta(P^N)\| \to \infty$ in probability. Theorem 3.2 of \cite{mukherjee2019nonstandard} establishes that for Manski's estimator,
\begin{align*}
    n_2\|\widehat{\theta}_1 -\theta(P^N)\| = O_{P^1}\left(n_2 \cdot \left(\frac{d \log(n_1/d)}{n_1}\right)^{1/(1+2\gamma)}\right).
\end{align*}
With an even split $n_1 \asymp n_2$, this diverges for all $\gamma$ bounded away from zero. For small $d/n_1$ and $\gamma \to 0$ (the well-separated, low dimensional case), $\widehat{\theta}_1$ converges nearly at rate $n_1^{-1}$ and $n_2\|\widehat{\theta}_1 -\theta(P^N)\|$ remains bounded, so validity may fail. The condition is satisfied for growing $d$ or any $\gamma$. Similarly, when $n_1 \gg n_2$, the initial estimator converges too fast and validity may fail. For $\alpha = 1/2$, the bound decays at the faster rate.

Validity holds trivially for both levels when $\widehat\theta_1$ is inconsistent for $\theta(P^N)$ since $\|\widehat\theta_1 - \theta(P^N)\|$ is bounded away from zero. Natural examples include penalized logistic regression \citep{cessie1992ridge}, support vector machines, and the smoothed maximum score estimator of \citet{horowitz1992smoothed} when the bandwidth is set to a fixed constant rather than optimally tuned.

\paragraph{Width Analysis.}
In the following analysis, we assume that $X_i$ for $i \in I_2$ are identically distributed while still allowing for $\varepsilon_i$ to be heterogeneous conditioning on $X_i$. This simplifies the analysis as $\mathbb{P}_{P_i}\left(\mathrm{sgn}(\theta^\top  X_i) \neq \mathrm{sgn}(\theta(P^N)^\top  X_i)\right)$ no longer depends on $i \in I_2$. 
\begin{theorem}\label{thm:manski-width}
Assume \ref{as:tsybakov}, \ref{as:covariate-manski} and let $\widetilde s_{n_1, n_2}$ be as in \ref{as:rate-initial-estimator3}. For $n_1 \ge 1$, $\alpha \in (0,1)$ and any $\varepsilon \in (0,1)$, with probability at least $1-\varepsilon$,
\begin{equation*}
\begin{split}
    &\mathrm{Diam}_{\|\cdot\|_2}\big(\widehat{\mathrm{CI}}^{\mathtt{CLT}}_{N, \alpha}\big) \\
&\quad\le C_\varepsilon\max\bigg\{\left(1+|z_\alpha|\right)^{2/(1+2\gamma)}\left(\left(\frac{d\log(n_1/d)}{n_1}\right)^{1/(1+2\gamma)} + \widetilde s_{n_1, n_2}^{1/(1+\gamma)}\right), \\
&\quad\quad\mathrm{Q}^{\mathtt{CLT}}_{N,\alpha}\mathbf{1}\{\mathrm{Q}^{\mathtt{CLT}}_{N,\alpha} \ge (t^*)^{1/\gamma}\}\bigg\},
\end{split}
\end{equation*} 
provided $\max\{2, C_\varepsilon'(1+|z_\alpha|)^2d\} \le n_2$, where $\mathrm{Q}_{N, \alpha}^{\mathtt{CLT}}= (1+|z_\alpha|)\widetilde s_{n_1, n_2}/t^*$, $C_\varepsilon$ depends on $\varepsilon$, $C_0$, $c_1$ and $\gamma$, while $C_\varepsilon'$ depends on $\varepsilon$, $t^*$ and $\gamma$. 
\end{theorem}

For the specific case where $\widehat{\theta}_1$ is Manski's maximum score estimator, we introduce the following additional assumption.
\begin{enumerate}[label=\textbf{(B\arabic*)},leftmargin=2cm]
\setcounter{enumi}{5}
\item \label{as:covariate-manski-2}
There exists constants $C_1, C_2 > 0$, not depending on $n_2$ or $d$, such that for all $\theta \in \mathbb{S}^{d-1}$ and $i \in I_2$,
\begin{align*}
    \mathbb{P}_{P_i}\left(\mathrm{sgn}(\theta^\top  X_i) \neq \mathrm{sgn}(\theta(P^N)^\top  X_i)\right) \le C_1\|\theta-\theta(P^N)\|_2,
\end{align*}
and when $\|\theta-\theta(P^N)\|_2 \le \delta$ for some $\delta \in (0, (t^*)^{1+1/\gamma})$, 
\begin{align*}
    \mathbb{P}_{P_i}\left(\mathrm{sgn}(\theta^\top  X_i) \neq \mathrm{sgn}(\theta(P^N)^\top  X_i)\right) \le C_2\|\theta-\theta(P^N)\|_2^{1+\gamma}.
\end{align*}
\end{enumerate}
Together, Assumptions \ref{as:tsybakov} and \ref{as:covariate-manski} establish a lower bound on the curvature $\mathbb{C}_{2}(\theta) \gtrsim \|\theta - \theta(P^N)\|^{1+\gamma}$ locally, while \ref{as:covariate-manski-2} provides the matching local upper bound $\mathbb{C}_{2}(\theta) \lesssim \|\theta - \theta(P^N)\|^{1+\gamma}$. the first part of \ref{as:covariate-manski-2} has appeared in the literature, for instance, Assumption 2.8 of \cite{mukherjee2021optimal}. The second part is considerably stronger, and the margin condition \ref{as:tsybakov} must be sharp. This is not implied by any of the standard assumptions in the maximum score literature. The data-generating distribution in \Cref{sec:num} satisfies this condition. Note that \Cref{thm:manski-validity-consistent} and \Cref{thm:manski-width} require only the lower bound. 
\begin{corollary}\label{cor:manski-plugin}
    Assume \ref{as:tsybakov}, \ref{as:covariate-manski}, \ref{as:covariate-manski-2} and suppose the initial estimator satisfies for all $n_1 \geq N_1$, \begin{align}\label{eq:manski-requirement}\|\widehat\theta_1 - \theta(P^N)\|_2 = O_{P^1}\left(\frac{d\log(n_1/d)}{n_1}\right)^{1/(1+2\gamma)}.
    \end{align}
     For any $\varepsilon \in (0, 1)$ and $n_1 \ge N_1$, with probability at least $1-\varepsilon$, 
\begin{equation*}
\begin{split}
    &\mathrm{Diam}_{\|\cdot\|_2}\big(\widehat{\mathrm{CI}}^{\mathtt{CLT}}_{N, \alpha}\big) \\
&\quad\le C_\varepsilon\max\bigg\{\left(1+|z_\alpha|\right)^{2/(1+2\gamma)}\left(\left(\frac{d\log(n_1/d)}{n_1}\right)^{1/(1+2\gamma)} + \left(\frac{d\log(n_2/d)}{n_2}\right)^{1/(1+2\gamma)} \right), \\
&\quad\quad\mathrm{Q}^{\mathtt{CLT}}_{N,\alpha}\mathbf{1}\{\mathrm{Q}^{\mathtt{CLT}}_{N,\alpha} \ge (t^*)^{1/\gamma}\}\bigg\},
\end{split}
\end{equation*} 
provided $\max\{2, C_\varepsilon'(1+|z_\alpha|)^2d\} \le n_2$, where $\mathrm{Q}_{N, \alpha}^{\mathtt{CLT}}= (1+|z_\alpha|)\|\widehat{\theta}_1 - \theta(P^N)\|_2/t^*$ and $C_\varepsilon$ depends on $\varepsilon$, $C_0$, $c_1$ $\gamma$, $C_1$ and $C_2$, while $C_\varepsilon'$ depend on $\varepsilon$, $t^*$ and $\gamma$.
\end{corollary}

Under \ref{as:tsybakov} and \ref{as:covariate-manski}, Assumption \ref{as:margin} only holds locally, and the width guarantees for both $\alpha = 1/2$ and $\alpha \neq 1/2$ are obtained through \Cref{thm:clt-width-local}, with no distinction between the two cases. When the initial estimator is consistent and $\mathrm{Q}^{\mathtt{CLT}}_{N,\alpha}$ is eventually bounded by $(t^*)^{1/\gamma}$, with high probability, the diameter is determined by the first term, scaling as $\left(d\log(n_2/d)/n_2\right)^{1/(1+2\gamma)}$, and exhibiting adaptive behavior to the unknown curvature $\gamma$. 

Theorem 3.2 of \citet{mukherjee2019nonstandard} establishes that \eqref{eq:manski-requirement} is satisfied by the standard maximum score estimator, and the rate of convergence matches that of the minimax lower bound up to a logarithmic factor (See Theorem 3.4 of \citet{mukherjee2019nonstandard}). Whenever $n_1$ is large enough such that $\|\widehat{\theta}_1 - \theta(P^N)\|_2 \le (t^*)^{1+1/\gamma}/(1+|z_\alpha|)$, with high probability, the diameter is dominated by the first term of \Cref{cor:manski-plugin}. The balanced split $n_1 = n_2 = N/2$ minimizes the diameter, yielding the rate $\left(d\log(N/d)/N\right)^{1/(1+2\gamma)}$.

\subsection{Quantile without Positive Densities}\label{sec:quantile}
Consider an IID observation $X_1, \ldots, X_N \in \mathbb{R}$ where the inference of interest is the $\eta$-quantile defined as
\begin{equation*}
\theta(P^N) := \inf \left\{t \, : F_{P^N}(t) \ge \eta \right\} \quad \text{for} \quad \eta \in (0,1)
\end{equation*}
and $F_{P^N}(t) := \mathbb{P}_{P^N}(X\le t)$. It is well-known that $\theta(P^N)$ minimizes the following ``quantile" loss:
\begin{equation*}
    \theta(P^N) := \argmin_{\theta \in \mathbb{R}}\, \E_{P^N}[\eta(X-\theta)_+ + (1-\eta)(\theta-X)_+].
\end{equation*}
We work under IID observations as the setting and the forthcoming assumption become unnatural under non-identical observations. The method and the theoretical results do not require IID observations. The sample quantile centered at $\theta(P^N)$ converges to a Gaussian distribution when scaled by $N^{1/2}$ if the distribution of $X$ has a strictly positive density at $\theta(P^N)$. If the density at $\theta(P^N)$ is zero or non-existent, however, the sample quantile converges at a rate depending on the H\"{o}lder smoothness of the $F_{P^N}(t)$ in the neighborhood of $\theta(P^N)$. In this case, the limiting distribution is no longer Gaussian and also depends on the H\"{o}lder smoothness of the $F_{P^N}(t)$ in the neighborhood of $\theta(P^N)$ \citep{smirnov1952limit}. Although finite-sample valid, distribution-free confidence intervals for quantiles already exist \citep{scheffe1945non}, we present this result to illustrate the behavior of the proposed method in irregular settings. We quantity the smoothness of the $F_{P^N}(t)$ near $\theta(P^N)$ as follows:
\begin{enumerate}[label=\textbf{(B\arabic*)},leftmargin=2cm]
\setcounter{enumi}{6}
\item \label{as:cdf-Holder} 
There exist $\delta_0 >0$, $M_0,M_1 \in (0,\infty)$ and $M_0 >M_1$ such that 
\begin{equation*}
    |F_{P^N}(\theta) - F_{P^N}(\theta(P^N)) - M_0|\theta - \theta(P^N)|^{\gamma} \mathrm{sgn}(\theta-\theta(P^N))| \le M_1|\theta-\theta(P^N)|^\gamma
\end{equation*}
for all $\theta$ such that $|\theta-\theta(P^N)| \le \delta_0$.
\end{enumerate}
The H\"{o}lder smoothness as described in \ref{as:cdf-Holder} should be compared to \citet[Equation (6)]{knight1998limiting}. When $\gamma=1$, this assumption becomes equivalent to requiring that the density at the true $\eta$-quantile is bounded away from zero. 

\paragraph{Validity.}
\begin{theorem}\label{thm:quantile-validity} Assume \ref{as:cdf-Holder} and define 
\begin{equation*}
    \begin{split}
        R_{n_2} &= \inf_{\delta_0 \ge \rho > 0}\left\{ 2\sqrt{\frac{2M_0\rho^\gamma }{\eta(1-\eta)}}+ \mathbb{P}_{P^1}(|\widehat\theta_1 - \theta(P^N)| > \rho)\right\} + \frac{1}{\sqrt{n_2\eta(1-\eta)}}.
    \end{split}
\end{equation*}
    There exists a universal constant $C > 0$ such that the following hold.
        \begin{enumerate}
    \item For $\alpha \in (0, 1)$ and $n_2 \ge 2$,
    \begin{equation*}
    \mathbb{P}_{P^N}\!\left(\theta(P^N) \notin \widehat{\mathrm{CI}}^{\mathtt{CLT}}_{N,\alpha}\right)
\;\le\; \alpha + \min\{1, CR_{n_2}\}.
    \end{equation*}
    \item For $\alpha \in (0, 1/2]$ and $n_2 \ge 1$, setting $\widetilde\Delta_2^2 = \mathfrak{C} n_2 \min\{|\widehat{\theta}_1 - \theta(P^N)|^{2\gamma}, \delta_0^{2\gamma}\}$ with $\mathfrak{C}$ depending on $M_0$ and $M_1$,
    \begin{equation*}
        \mathbb{P}_{P^N}\left(\theta(P^N) \notin \widehat{\mathrm{CI}}^{\mathtt{CLT}}_{N,1/2}\right)
\le \min \left\{\E_{P^1}\left[1-\Phi(\widetilde\Delta_2) + \frac{C}{(1+\widetilde\Delta_2)^2}\right],\; \alpha + \min\{1, CR_{n_2}\}\right\}.
    \end{equation*}
\end{enumerate}
\end{theorem}
\Cref{thm:quantile-validity} reveals a phase transition in the validity conditions, depending on the data split. Suppose $\widehat\theta_1$ is the sample quantile based on $D_1$, satisfying
\begin{align*}
    \mathbb{P}_{P^1}(|\widehat\theta_1 - \theta(P^N)| > n_1^{-1/(2\gamma)}) \le \varepsilon.
\end{align*}
As an illustration, we set $\rho = n_1^{-1/(2\gamma)}$ in the infimum, assuming that $n_1$ is large enough so that $\rho \le \delta_0$. Then, the remainder $R_{n_2}$ is negligible whenever
\begin{equation*}
n_1^{1/2} \eta(1-\eta) \to \infty \quad \textrm{and} \quad n_2 \eta(1-\eta) \to \infty.
\end{equation*}
The validity requirement is similar to that of the CLT, which precludes the extreme quantile levels $(1-\eta), \eta \to 0$ at rate faster than $n_2$ and $n_1^{1/2}$. 

For $\alpha \le 1/2$, validity holds under weaker condition of $\widetilde\Delta_2 \to \infty$, which reduces to 
\begin{equation*}
    \frac{n_2}{n_1}\cdot (n_1|\widehat{\theta}_1 - \theta(P^N)|^{2\gamma}) \to \infty.
\end{equation*}
This condition is free of $\eta$. Hence, when $n_2 \gg n_1$, the asymptotic validity holds for any value of $\eta$, including extreme quantile levels. For fixed $\eta$, validity conditions are agnostic to the H\"{o}lder smoothness $\gamma$.

% The result of Theorem~\ref{thm:quantile-validity} allows for $\gamma \equiv \gamma_n$ to depend on the sample size $n$. When there exist constants $c$ and $C$ such that $0 < c \le \gamma_n \le C < 1$ for all $n$, the validity holds under the consistency of the initial estimator. When $\gamma_n \longrightarrow 0$ or $\gamma_n \longrightarrow 1$, then there is a restriction on how quickly $\gamma_n$ can tend to these extreme values depending on $\beta$ and the convergence rate of $\widehat{\theta}_1$. This requirement may be relaxed under the alternative, but possibly stronger, assumptions on $P$---see Remark~\ref{remark:consistency}.

\paragraph{Width Analysis.}
\begin{theorem}\label{thm:quantile-ci-width}
Assume \ref{as:cdf-Holder} and let $\widetilde s_{n_1, n_2}$ be as in \ref{as:rate-initial-estimator3}. For $n_1 \ge 1$, $\alpha \in (0,1)$ and any $\varepsilon \in (0,1)$, with probability at least $1-\varepsilon$,
\begin{equation*}
\mathrm{Diam}_{|\cdot|}\big(\widehat{\mathrm{CI}}^{\mathtt{CLT}}_{N, \alpha}\big) \le C_\varepsilon\max\bigg\{\left(1+|z_\alpha|\right)^{1/\gamma}(n_2^{-1/(2\gamma)} + \widetilde s_{n_1, n_2}^{1/(1+\gamma)}), \mathrm{Q}^{\mathtt{CLT}}_{N,\alpha}\mathbf{1}\{\mathrm{Q}^{\mathtt{CLT}}_{N,\alpha} \ge \delta_0\}\bigg\},
\end{equation*} 
provided $\max\{2, C_\varepsilon'(1+|z_\alpha|)^2\delta_0^{2\gamma}\} \le n_2$, where $\mathrm{Q}_{N, \alpha}^{\mathtt{CLT}}= \delta_0^{-\gamma}(1+|z_\alpha|)\widetilde s_{n_1, n_2}$ and $C_\varepsilon$ depends on $M_0, M_1$ and $\gamma$, while $C_\varepsilon'$ only depend on $\varepsilon$.

\end{theorem}
\begin{corollary}\label{cor:quantile-plugin}
Assume \ref{as:cdf-Holder} and suppose the initial estimator satisfies for all $n_1 \geq N_1$, \begin{align}\label{eq:quantile-initial}|\widehat\theta_1 - \theta(P^N)| \ = O_{P^1}(n_1^{-1/(2\gamma)}).
    \end{align}
     For any $\varepsilon \in (0, 1)$ and $n_1 \ge N_1$, with probability at least $1-\varepsilon$, 
\begin{equation*}
\mathrm{Diam}_{|\cdot|}\big(\widehat{\mathrm{CI}}^{\mathtt{CLT}}_{N, \alpha}\big) \le C_\varepsilon\max\bigg\{\left(1+|z_\alpha|\right)^{1/\gamma}(n_2^{-1/(2\gamma)} + n_1^{-1/(2\gamma)}), \mathrm{Q}^{\mathtt{CLT}}_{N,\alpha}\mathbf{1}\{\mathrm{Q}^{\mathtt{CLT}}_{N,\alpha} \ge \delta_0\}\bigg\},
\end{equation*} 
provided $\max\{2, C_\varepsilon'(1+|z_\alpha|)^2\delta_0^{2\gamma}\} \le n_2$, where $\mathrm{Q}_{N, \alpha}^{\mathtt{CLT}}= \delta_0^{-\gamma}(1+|z_\alpha|)|\widehat\theta_1-\theta(P^N)|$ and $C_\varepsilon$ depends on $M_0, M_1$ and $\gamma$, while $C_\varepsilon'$ only depends on $\varepsilon$.
\end{corollary}
Under \ref{as:cdf-Holder}, Assumption \ref{as:margin} only holds locally. The width guarantees are obtained through \Cref{thm:clt-width-local}, and there is no distinction between the two cases: $\alpha \ge 1/2$ and $\alpha < 1/2$. When the initial estimator is consistent and $\mathrm{Q}^{\mathtt{CLT}}_{N,\alpha}$ is eventually bounded by $\delta_0$, with high probability, the diameter is determined by the first term, scaling as $n_2^{-1/(2\gamma)}$, and exhibiting adaptive behavior to the unknown curvature $\gamma$. 

Condition \eqref{eq:quantile-initial} is satisfied when $\widehat{\theta}_1$ is the sample quantile based on $D_1$. Whenever $n_1$ is large enough such that $|\widehat{\theta}_1 - \theta(P^N)| \le \delta_0^{1+\gamma}/(1+|z_\alpha|)$, with high probability, the diameter is dominated by the first term of \Cref{cor:quantile-plugin}. The balanced split $n_1 = n_2 = N/2$ minimizes the diameter, yielding the rate $(N/2)^{-1/(2\gamma)}$, which matches the estimation rate of the sample quantile under \ref{as:cdf-Holder} as given by Example 1 of \cite{knight1998limiting}. In the special case $\gamma=1$, i.e., where the density is bounded away from zero at the $\eta$-quantile, the confidence shrinks at the parametric rate $N^{-1/2}$. 

% \begin{remark}[The proof of Theorem~\ref{thm:quantile-ci-width}]
% The width analyses in this section are mostly performed as the direct applications of Theorem~\ref{thm:concentration-width} or Theorem~\ref{thm:concentration-width-unbounded}. The proof of Theorem~\ref{thm:quantile-ci-width}, however, differs significantly since the curvature assumption \ref{as:margin} only holds locally for $\theta$ such that $\|\theta-\theta(P^N)\| < \delta$. As a result, Theorem~\ref{thm:concentration-width-unbounded} cannot be applied for the parameter outside of this neighborhood. While the proof of Theorem~\ref{thm:quantile-ci-width} crucially relies on the convexity and Lipschitzness of the quantile loss, a general result in the spirit of Theorem~\ref{thm:concentration-width-unbounded} may be useful under the local analog to \ref{as:margin}. Towards this task, one may need to extend the ratio-type empirical process to the unbounded function spaces \citep{gine2006concentration}. We are currently investigating this direction. 
% \end{remark}

\subsection{Discrete Argmin Inference}\label{sec:argmin-inference}
This application is motivated by \citet{zhang2024winners}, who study a prototypical problem in general model selection. The framework extends naturally to constructing confidence sets that contain the best predictor minimizing the population risk among  $\{f_1,\ldots, f_d\}$ where $f_i$ may correspond to a machine learning method trained on the same data; consequently $f_i$ and $f_j$ for $i\neq j$ can be highly correlated.

Consider independent observations $X_1, \ldots, X_N \in \mathbb{R}^d$. Observations share a common mean vector $\mu = \E_{P_i}[X_i]$ but have heterogeneous covariance matrices $\Sigma_i = \mathrm{Cov}_{P_i}(X_i)$ and we write $\bar \Sigma_2 = n_2^{-1}\sum_{i \in I_2}\Sigma_i$. In this problem, the inference of interest is the index set corresponding to the minimum marginal mean:
\begin{align*}
    \mathcal{S}^* := \theta(P^N) = \argmin_{j \in \{1, \ldots, d\}}\, \frac{1}{n_2}\sum_{i \in I_2}\E_{P_i}[e_j^\top X].
\end{align*}
Its complement $\mathcal{S}^c = \{1, \ldots, d\} \setminus \mathcal{S}^*$ denotes the non-argmin index set. Following \citet{zhang2024winners}, we do not assume that $\mathcal{S}^*$ is a singleton set, allowing the proposed confidence set to accommodate ties. 
\paragraph{Validity.}
\begin{theorem}\label{thm:arginf-valid}
Denote $\delta_{j,k} = (e_j - e_k)^\top \mu$ and $D_i^{j,k} =(e_j - e_k)^\top (X_i -\mu)$, and 
\begin{equation*}
    \sigma^2_{j,k} = (e_j-e_k)^\top \bar\Sigma_2 (e_j-e_k) = [\bar\Sigma_2]_{jj} - 2[\bar\Sigma_2]_{jk} + [\bar\Sigma_2]_{kk},
\end{equation*}
where $[A]_{jk}$ denotes the $(j,k)$th entry of the matrix $A$. Define 
\begin{equation}
    R_{n_2}  = \max_{(j, k)\in \mathcal{S}^c\times \mathcal{S}^*}\sum_{i\in I_2} \mathbb{E}_{P_i}\left[\frac{|D_i^{j, k}|^2}{n_2\sigma_{j, k}^2}\min\left\{1,\,\frac{|D_i^{j, k}|}{n_2^{1/2}\sigma_{j, k}}\right\}\right].\label{as:pairwise}
\end{equation}
    There exists a universal constant $C > 0$ such that the following hold. 
    \begin{enumerate}
    \item For $\alpha \in (0, 1)$ and $n_2 \ge 2$,
    \begin{equation*}
    \mathbb{P}_{P^N}\!\left(\theta(P^N) \notin \widehat{\mathrm{CI}}^{\mathtt{CLT}}_{N,\alpha}\right)
\;\le\; \alpha + \min\{1, CR_{n_2}\}.
    \end{equation*}
    \item For $\alpha \in (0, 1/2]$ and $n_2 \ge 1$, setting $\widetilde\Delta_2^2 = n_2\min_{(j, k)\in \mathcal{S}^c\times \mathcal{S}^*}\delta_{j,k}^2/\sigma_{j, k}^2$, it holds 
    \begin{equation*}
\mathbb{P}_{P^N}\left(\theta(P^N) \notin \widehat{\mathrm{CI}}^{\mathtt{CLT}}_{N,1/2}\right)
\le \min \left\{\E_{P^1}\left[1-\Phi(\widetilde\Delta_2) + \frac{C}{(1+\widetilde\Delta_2)^2}\right],\; \alpha + \min\{1, CR_{n_2}\}\right\}.
    \end{equation*}
\end{enumerate}
\end{theorem}
% \begin{theorem}
%     \begin{align*}
%         &\mathbb{P}_{P^N}(\theta(P^N) \not\in \widehat{\mathrm{CI}}_{N,\alpha}^{\mathtt{CLT}}) \\
%         &\qquad\le \alpha +\sup_{j\neq \theta^*}\min\left\{1, \sum_{i \in I_2}\E_{P_i}\left[\frac{W_{j\theta^*}^2}{n_2 e_j^\top \bar\Sigma_2e_{\theta^*} } \min\left\{1, \frac{|W_{j\theta^*}|}{\sqrt{n_2 e_j^\top \bar\Sigma_2e_{\theta^*} }}\right\} \right]\right\},
%     \end{align*}
%     where $W_{i}^{j,\theta^*} = e_{j}^\top (X_i - \E_{P_i}[X]) - e_{\theta^*}^\top (X_i - \E_{P_i}[X])$
% \end{theorem}
% \begin{enumerate}[label=\textbf{(B\arabic*)},leftmargin=2cm]
% \setcounter{enumi}{5}
% \item \label{as:pairwise} 
% For any $i \neq j $, it holds that 
% \begin{align*}
%     \sup_{P\in\mathcal{P}}\, \E_P\left[\frac{W_{ij}^2}{\E_P [W_{ij}^2]} \min\left\{1, \frac{|W_{ij}|}{n^{1/2}(\E_P [W_{ij}^2])^{1/2}}\right\} \right] = o(1)
% \end{align*}
% as $n \to\infty$ where $W_{ij} = e_i^\top (X-\E_P[X])- e_j^\top (X-\E_P[X])$.
% \end{enumerate}

\cite{zhang2024winners} establish validity of their method under the assumption that the smallest eigenvalue of the covariance matrix of $X$ is bounded away from zero (See Theorem 3.1 of \cite{zhang2024winners}). As pointed out in their work, this assumption may be violated in practice when the components of $X$ are highly correlated, such as in model selection for LASSO (see Section 6.2 of \citet{zhang2024winners}). 

For any $\alpha \in (0,1)$, validity holds whenever the $L_{n+\delta}$-$L_2$ moment equivalence holds; see also Theorem 2.1 of \citet{kim2025locally}, which became available after the initial version of this manuscript. For $\alpha \le 1/2$, validity holds additionally whenever $\widetilde\Delta_2 \to \infty$, requiring no moment assumptions beyond finite variance. When $\sigma_{j,k}$ is bounded away from zero, $\widetilde\Delta_2 \to \infty$ holds as long as $\delta_{j,k} \gg n_2^{-1/2}$, allowing for shrinking mean gaps. On the other hand, $\sigma_{j,k} \to 0$ is favorable for establishing $\widetilde\Delta_2 \to \infty$. This happens, for instance, when $[\bar\Sigma_2]_{jj} \approx[\bar\Sigma_2]_{kk}$ and the $j$th and $k$th components are strongly positively correlated. In this case the, validity can hold even when $\delta_{j,k}$ vanishes faster than $n_2^{-1/2}$. 

The convergence rate of the diameter does not translate directly to the discrete parameter space. We refer to \citet{kim2025locally}, which studies the proposed method in the context of testing and provides minimax power analysis.
\section{Numerical Illustration}\label{sec:num}
This section provides an empirical illustration of the proposed method for inference in high-dimensional and irregular settings. Manski's maximum score estimator is used as the primary example, as it captures both dimension-agnostic validity and curvature-adaptivity in a single framework. Additional results for high-dimensional mean estimation, misspecified linear regression, and median estimation without positive density are presented in Section~\ref{supp:num}. 

\subsection{Data-generating Distributions and Experimental Setup}
 For given sample size $N$ and dimension $d$, we generate 
\begin{equation}
    X_i \sim \mathcal{N}(0, \Sigma) \quad \textrm{where} \quad \Sigma_{i,j} = 0.1^{|i-j|}/d.
\end{equation}
Set $\beta_0 = (1/\sqrt{d}, \ldots, 1/\sqrt{d})^\top$ so $\beta_0 \in \mathbb{S}^{d-1}$. Conditional on $X_i$, the binary response is generated as
\begin{equation}
    \mathbb{P}(Y_i=1 \mid X_i) = \Phi(\mathrm{sgn}(\beta_0^\top X_i) \cdot  |\beta_0^\top X_i|^{\gamma-1}).
\end{equation}
This data-generating distribution approximately satisfies \ref{as:tsybakov} with parameter $\gamma$. Numerical studies are conducted for $\gamma \in \{1/2, 1, 2\}$, under which the maximum score estimator is expected to converge at rates, $(d/N)^{1/2}, (d/N)^{1/3}$ and $(d/N)^{1/5}$ respectively, up to logarithmic factors. The proposed CLT-based confidence set is compared against subsampling with estimated rate of convergence \citep{bertail1999subsampling} and the nonparametric bootstrap. Although the bootstrap is known to be inconsistent for this problem \citep{sen2010inconsistency}, it is included as a reference.

Two experimental settings are considered.
\begin{enumerate}
    \item[\normalfont(1)] \textbf{Experiment 1 (low dimension).} Set $d=2$ and $N \in \{100, 200, \ldots, 1000\}$. Two base estimators are compared: Manski's maximum score estimator (\texttt{Manski}) and logistic regression (\texttt{Logistic}).   The latter is inconsistent due to model misspecification and serves to illustrate the robustness of the proposed method to the choice of initial estimator. Data are split evenly between $D_1$ and $D_2$.
    \item[\normalfont(2)] \textbf{Experiment 2 (high dimension).} Set $N=200$, and $d \in \{10, \ldots, 50\}$. The maximum score estimator is computationally infeasible in this regime; since the proposed method remains valid for any initial estimator, two alternatives are considered: the smoothed maximum score estimator (\texttt{SmoothManski}) \citep{horowitz1992smoothed} and penalized logistic regression (\texttt{PenLogistic}) \citep{cessie1992ridge}. Neither estimator has known limiting distributions under misspecification, to the best of our knowledge. Data are split evenly between $D_1$ and $D_2$.
\end{enumerate}
See \Cref{supp:num} for omitted implementation details. 

\subsection{Validity for Irregular and High-dimensional Settings}
\paragraph{Experiment 1:}\Cref{fig:manski-2d-coverage} displays the estimated coverage of all confidence sets constructed at the $90\%$ nominal level, based on $500$ replications for the proposed method and $300$ replications for the resampling methods. The $X$-axis shows sample size $N$ and the $Y$-axis shows estimated coverage. From left to right, the panels correspond to data-generating distributions with $\gamma = 1/2, 1$, and $2$.

The proposed CLT-based confidence set achieves coverage near or above $90\%$ across all values of $\gamma$ and for both initial estimators. In particular, when $\gamma = 2$, coverage is closest to the nominal level, consistent with the regime where $\ratio$ is numerically close to zero. See \Cref{thm:large-deviation}. On the other hands, the performance of the resampling methods vary. Subsampling with estimated rate achieves near-nominal coverage for $\gamma=1/2$, but undercovers significantly as $\gamma$ increases. The bootstrap undercovers for all values of $\gamma$, particularly when a consistent estimator is used. Its reported coverage falls well below $60\%$ lies outside the plotted range.

\paragraph{Experiment 2:} \Cref{fig:manski-high-dim-coverage} displays the estimated coverage of all confidence sets constructed at the $90\%$ nominal level. The $X$-axis shows sample dimension $d$, with $N=200$ fixed. The proposed CLT-based confidence sets remain valid for $90\%$ nominal level, yet it becomes more conservative as dimension of the sample grows. For a fixed sample size and dimension, the coverage becomes closer as $\gamma$ increases (from left panel to the right). These are expected behaviors characterized by the behavior of $\ratio$. As for the comparative methods, subsampling with estimated rate of convergence produces extremly conservative sets with $100\%$ coverage and bootstrap fails for large dimension. We emphasize that both base estimators, \texttt{SmoothManski} or \texttt{{PenLogistic}}, are inconsistent estimators as the maximum score estimator is computational infeasible. Only the proposed methods establish validity completely agnostic to the choice of the estimator. 

\subsection{Diameter for Irregular and High-dimensional Settings}
\paragraph{Experiment 1:} \Cref{fig:manski-2d-width} displays the average diameter of each confidence set over $500$ replications ($300$ for resampling methods) with the $X$-axis displaying the sample size on a log scale and with the $Y$-axis displaying the average diameter on a log scale. The slope of each line corresponds to the exponent in the rate of convergence, with theoretical values $-1/2, -1/3$ and $-1/5$ for $\gamma = 1/2, 1, 2$ respectively. The slope for the proposed method is estimated via linear regression and overlaid on each panel. The diameter of the proposed confidence set converges at a rate closely matching the theoretical value in each case, demonstrating curvature-adaptive convergence without prior knowledge of $\gamma$.
\paragraph{Experiment 2:} \Cref{fig:manski-multi-width} displays the average diameter of each confidence set. For the proposed method, the diameter is estimated by the sampling method proposed in \Cref{supp:manski-algorithm}. The $X$-axis shows sample dimension $d$, with $N=200$ fixed. The slope of each line corresponds to the exponent in the rate of convergence, with theoretical values $1/2, 1/3$ and $1/5$ for $\gamma = 1/2, 1, 2$ respectively. The slope for the proposed method is estimated via linear regression and overlaid on each panel. We observe that the observed rate of the diameter scales similarly to the theoretical rates, but not as closely matching as \Cref{fig:manski-2d-width}. This is expected as the size of the confidence sets depend on the convergence rate of the initial estimator, in this case, \texttt{SmoothManski} or \texttt{{PenLogistic}}. Neither of them is expected to achieve the same rate as \texttt{Manski} is computationally infeasible. 
\begin{figure}
\centering
\begin{subfigure}{0.8\textwidth}
  \centering
  \includegraphics[width=\linewidth]{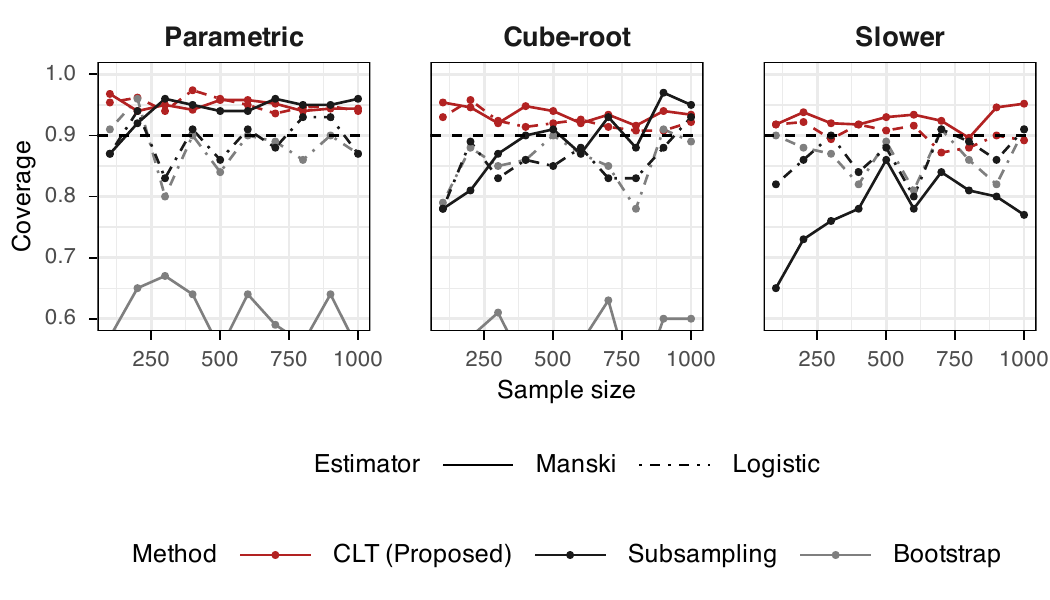}
\end{subfigure}
\caption{Estimated coverage of the proposed confidence set and two sampling methods, targeted at the $90\%$ nominal level. The $X$-axis displays the total sample size $N$ and the $Y$-axis displays the estimated coverage over $500$ replications. From left to right, the panels correspond to $\gamma=1/2,1,2$. Two base estimators are considered: Manski's maximum score estimator and logistic regression. The proposed method achieves coverage above $90\%$ across all settings with a certain conservativeness agreeing with the theoretical result. Subsampling with estimated rate achieves nominal coverage only for $\gamma=0$ and large $N$. The performance of the resampling methods vary, and it generally deteriorates as $\gamma$ increases.}
\label{fig:manski-2d-coverage}
\end{figure}
\begin{figure}
\centering
\begin{subfigure}{0.8\textwidth}
  \centering
  \includegraphics[width=\linewidth]{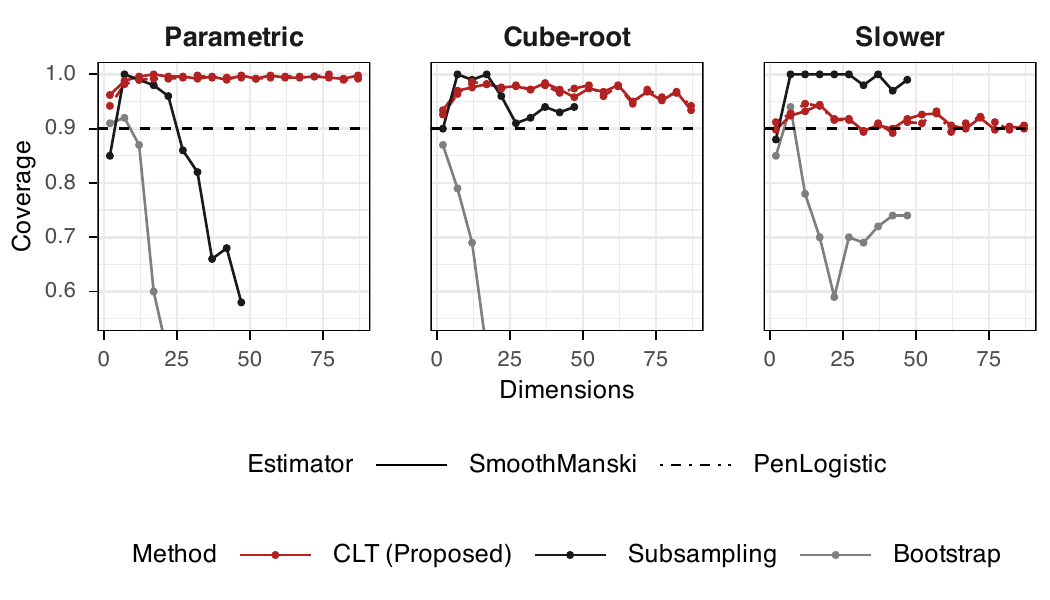}
\end{subfigure}
\caption{Estimated coverage of the proposed confidence set and two sampling methods, targeted at the $90\%$ nominal level. The $X$-axis displays the total sample size $N$ and the $Y$-axis displays the estimated coverage over $500$ replications. From left to right, the panels correspond to $\gamma=1/2,1,2$. Two base estimators are considered: Manski's maximum score estimator and logistic regression. The proposed method achieves coverage above $90\%$ across all settings with a certain conservativeness agreeing with the theoretical result. Subsampling with estimated rate achieves nominal coverage only for $\gamma=0$ and large $N$. The performance of the resampling methods vary, and it generally deteriorates as $\gamma$ increases.}
\label{fig:manski-high-dim-coverage}
\end{figure}

\begin{figure}
\centering
\begin{subfigure}{0.8\textwidth}
  \centering
  \includegraphics[width=\linewidth]{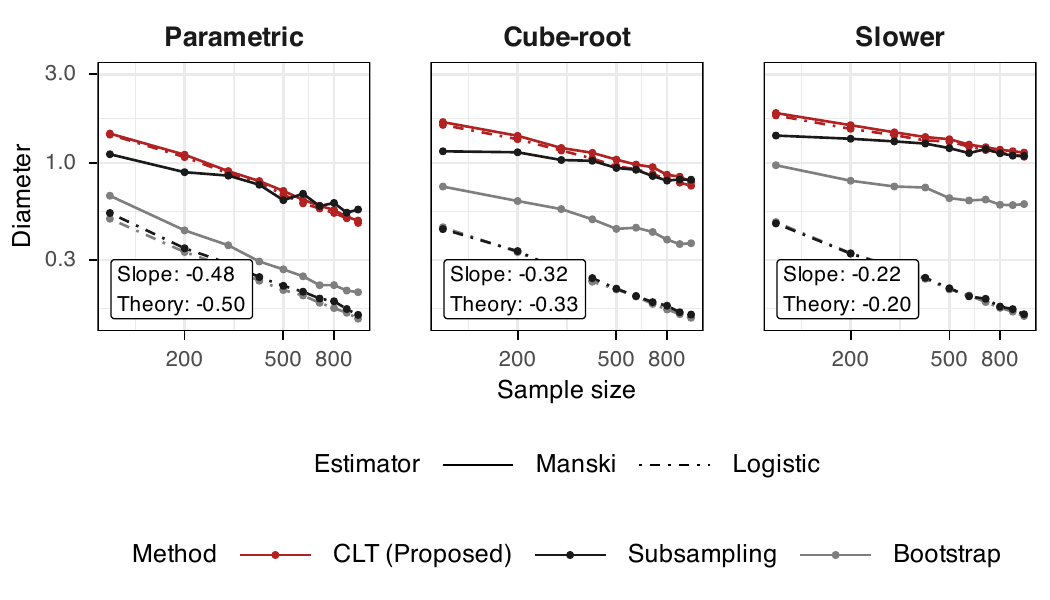}
\end{subfigure}
\caption{Average diameter of the proposed confidence set on a log--log scale. The $X$-axis displays sample size on a log scale the $Y$-axis displays the average diameter of the confidence sets on a log scale over $500$ replications. From left to right, the panels correspond to $\gamma = 1/2, 1$ and $2$, with theoretical rates correspond to $N^{-1/2}, N^{-1/3}$ and $N^{-1/5}$. The slope for the proposed method estimated by linear regression is reported in the figure. The observed slopes closely match the theoretical rates, demonstrating that the proposed confidence set adapts to the unknown smoothness parameter $\gamma$ without requiring prior knowledge of the convergence rate.}
\label{fig:manski-2d-width}
\end{figure}

\begin{figure}
\centering
\begin{subfigure}{0.8\textwidth}
  \centering
  \includegraphics[width=\linewidth]{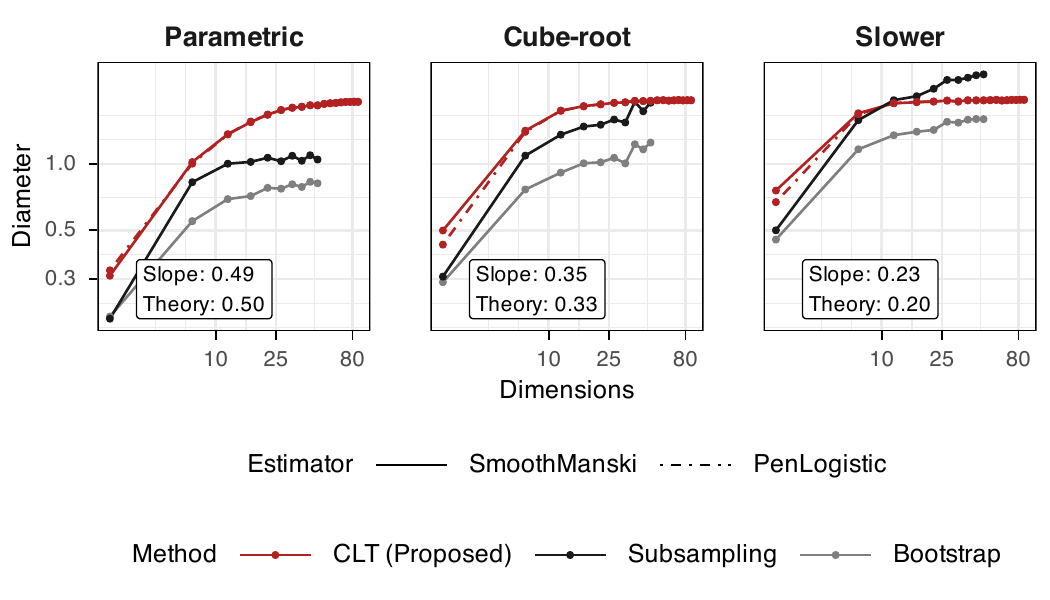}
\end{subfigure}
\caption{Average diameter of the proposed confidence set on a log--log scale. The $X$-axis displays sample size on a log scale the $Y$-axis displays the average diameter of the confidence sets on a log scale over $500$ replications. From left to right, the panels correspond to $\gamma = 1/2, 1$ and $2$, with theoretical rates correspond to $N^{-1/2}, N^{-1/3}$ and $N^{-1/5}$. The slope for the proposed method estimated by linear regression is reported in the figure. The observed slopes closely match the theoretical rates, demonstrating that the proposed confidence set adapts to the unknown smoothness parameter $\gamma$ without requiring prior knowledge of the convergence rate.}
\label{fig:manski-multi-width}
\end{figure}
\section{Concluding Remarks}\label{sec:conclusions}
This manuscript introduces a general framework for constructing confidence sets for solutions of stochastic optimization problems, rendering empirical risk minimization as special cases. The proposed method employs sample splitting, which facilitates validity across both regular and irregular settings. In particular, the method offers a dimension-agnostic solution, which becomes applicable in high-dimensional problems where standard asymptotic theory breaks down. The manuscript provides a unified treatment of validity, conservativeness, and the diameter of the resulting confidence sets.

The theoretical properties are illustrated through several challenging statistical applications. For high-dimensional and misspecified linear regression, the proposed confidence set achieves dimension-agnostic validity and recovers the $\sqrt{d/N}$ rate under weak moment conditions, with a dimensional requirement that appears to be new. For Manski's maximum score estimator, the confidence set adapts to the unknown margin condition, yielding rates that depend on the Tsybakov noise parameter; a confidence set with this adaptive behavior in high-dimension also appears to be new.

The following are problems where honest inference under weak distributional assumptions is limited or largely absent, and which are closely related to the themes of this manuscript.
\begin{itemize}
    \item \textbf{Generalized linear models.} Extending the framework to high-dimensional generalized linear models, including logistic and Poisson regression.
    \item \textbf{Constrained optimization.} Quantifying convergence rates under general constraints, which likely requires handling the geometry of the feasible set near the solution using tools from variational analysis.
    \item \textbf{Irregular parameter spaces.} Analyzing the problem where the parameter space poses structural challenges, including the space of probability distributions equipped with optimal transport metrics, discrete parameter spaces, and Hadamard spaces for Fr\'{e}chet mean inference.
    \item \textbf{Nuisance parameters.} Extending the framework to problems where the target is a functional of a higher-dimensional object estimated in a first stage, and where it is of interest to understand whether the elbow effect between parametric and functional rates of convergence can be recovered.
    \item \textbf{Dependence beyond mixing.} Studying dependence structures beyond $\beta$-mixing and martingale differences, such as graphical dependence arising in Ising models or other Markov random fields.
    \item \textbf{Removing sample splitting.} Investigating whether sample splitting can be avoided using tools from algorithmic stability or differential privacy, which provide alternative mechanisms for decoupling estimation and inference.
    \item \textbf{Probabilistic tools.} Developing new probabilistic tools including non-uniform Berry-Esseen bounds, moderate and large deviation inequalities for degenerate U-statistics, and non-central
t-statistics under weak moment conditions.
\end{itemize}
Each of these directions is independently motivated by the limitations and extensions identified in the present work. Any one of them in isolation, or any combination thereof, represents a research program of both theoretical and practical interest.
% The authors are currently exploring several directions for extending the proposed method. One key area is the constrained optimization problems, where similar irregularities emerge when the solution lies on the boundary of the constrained set. This issue includes problems involving shape constraints and sparsity. In particular, there is a lack of inferential tools for LASSO and Dantzig selectors \citep{candes2007dantzig} despite their widespread use, making the investigation in this area of significant interest. 
% Achieving efficiency in estimating $\M(\theta, \eta, P)$ will likely require the first-order bias correction from the estimation error of $\eta$, which is an interesting area to investigate. 
% Third, the manuscript did not focus on optimization problems involving U-statistics or U-quantiles. Given that the CLT for U-statistics is well-established, we anticipate the general framework to be applicable to these problems as well. 
% this manuscript considered scenarios where the sample size $n$ is fixed and not data-dependent. We envision extending this framework to data-dependent stopping rules, or \emph{anytime-valid} inference, can be achieved when the loss function is equipped with certain concentration properties, such as sub-Gaussian tails \citep{schreuder2020nonasymptotic} or under the asymptotic confidence sequence framework formalized by \cite{waudby2024time}. Pursuing these extensions will require considerable additional effort and represent substantial methodological advances.

\section*{Acknowledgements}
The first author gratefully acknowledges Woonyoung Chang for the series of helpful discussions. We also thank Christof Sch\"{o}tz for providing us comments on the proof of Theorem~\ref{thm:mean-ci-width} in the initial manuscript and informing us of the application to Fr\'{e}chet means. 
\bibliographystyle{apalike}
\bibliography{ref.bib}

\newpage
% \newpage
% \singlespacing

\setcounter{section}{0}
\setcounter{equation}{0}
\setcounter{figure}{0}
\renewcommand{\thesection}{S.\arabic{section}}
\renewcommand{\theequation}{E.\arabic{equation}}
\renewcommand{\thefigure}{A.\arabic{figure}}
\renewcommand{\theHsection}{S.\arabic{section}}
\renewcommand{\theHequation}{E.\arabic{equation}}
\renewcommand{\theHfigure}{A.\arabic{figure}}
% \begin{appendices}
\clearpage
  \begin{center}
  \Large {\bf Supplement to ``Honest Inference for Stochastic Optimization''}
  \end{center}
       \vspace{1mm}
\begin{abstract}
This supplement contains the proofs of all the main results in the paper and some supporting lemmas. 
\end{abstract}
\vspace{1mm}

% \begin{proof}[Proof of Theorem~\ref{thm:coverage-anti-conservative-confidence-set}]
\section{Review of History}\label{appsec:references-history}
This section summarizes the historical developments of the key idea behind the content of this manuscript.  
\begin{itemize}
    \item Inverting the asymptotic risk of (irregular) estimators to construct confidence sets has a long history. The idea dates back at least to \citet[Equation 8.11]{Stein1981}, who foreshadowed the possibility of the inference for a shrinkage estimator of the multivariate Gaussian mean in large dimensions. 
    \item  Later, \citet{Li1989} applied a similar risk inversion framework to nonparametric regression with Gaussian errors and introduced the concept of ``honest'' confidence sets, which subsequently sparked developments in adaptive nonparametric inference. \citet[Theorem 3.1]{beran1996} and \citet[Theorem 3.1]{Beran1998} extended these ideas, explicitly crediting \citet{Stein1981}, and proposed the inversion of the asymptotic normality based on the central limit theorem (CLT), which they call modulation of estimators. Here, sample-splitting was not considered.
    \item Parallel developments in stochastic programming analyzed the risk of constrained optimization problems. Early references include \citet[Theorem 3.2]{Shapiro1989} and \citet[Theorem 4.4]{geyer1994asymptotics}, both of whom leveraged the CLT to establish asymptotic distributions. Confidence set construction in this setting was explicitely mentioned by \citet{geyer1996asymptotics}. The asymptotic behavior of the risk under general loss functions and constraints was studied in great generality by \cite{pflug1991asymptotic, pflug1995asymptotic, pflug2003stochastic}. Here, as well, sample splitting was not considered.
    \item \citet{Robins2006} were among the first to combine sample-splitting with with risk inversion based on the CLT. Their Theorem 3.4 established that the validity of the CLT depended only on the sample size tending to infinity and a Feller condition on the univariate risk space, making their result effectively ``dimension-agnostic". 
    % \cite{nickl2013confidence} also employed sample-splitting to construct confidence sets that adapt to sparsity, though their approach relied on concentration inequalities rather than the CLT. 
    During this period, statistical literature primarily focused on squared error loss in nonparametric regression, with exceptions such as \citet{Hoffmann2011, Carpentier2013}. Meanwhile, operations research literature examined inference for constrained optimization solutions with general loss functions. Inspired by the series of works by \cite{pflug1991asymptotic, pflug1995asymptotic, pflug2003stochastic}, \citet{vogel2008universal}  investigated risk inversion for confidence sets, but without incorporating sample-splitting. Related works include \citet{vogel2008confidence, vogel2017confidence} and \cite{guigues2017non}.
    \item \citet{kim2020dimension} later introduced the term ``dimension-agnostic" to describe properties similar to those established by \citet[Theorem 3.4]{Robins2006}, though they did not cite the earlier work. For instance, Theorem 4.2 of \citet{kim2020dimension} should be compared to Theorem 3.4 of \citet{Robins2006}. They applied sample-splitting and CLT-based inversion to high-dimensional hypothesis testing problems, such as goodness-of-fit testing using Gaussian maximum mean discrepancy (MMD). \cite{chakravarti2019gaussian} uses the similar methodology based on sample-splitting and CLT for testing the relative fit of Gaussian mixtures.  \citet{park2023robust} also employ sample-splitting and CLT for the inference on population maximum likelihood estimation under model misspecification, among other techniques. Neither \citet{chakravarti2019gaussian, kim2020dimension} nor \citet{park2023robust} developed a general theory for M-estimation such as width/diameter of the resulting confidence sets. 
    \item Although the explicit application of sample-splitting and CLT inversion to general M-estimation has not been previously explored, the methodological approach is a natural consequence of prior work, including \cite{Beran1998}, \citet{Robins2006}, and \citet{vogel2008confidence}. Consequently, we do not claim innovation in methodological front, as such confidence sets would likely have emerged given the historical trajectory of the field. Instead, our contribution lies in analyzing the properties of these confidence sets, including their validity and width.
    \item M-estimators are known to exhibit locally adaptive rates of convergence, depending on problem-specific geometric factors such as curvatures \citep{kim1990cube, van1996weak}. This notion of adaptivity has not been investigated within the ``adaptive" inference literature on nonparametric submodels, such as \citet{Robins2006} and \citet{Patschkowski2019}. The concept of locally adaptive confidence sets is particularly relevant to the M-estimation framework, holding both methodological and theoretical significance. While some adaptive confidence sets have been studied in shape-restricted regression \citep{Yang2019, dumbgen2003optimal, Bellec2021}, these approaches typically assume strong distributional conditions such as (sub-)Gaussian errors. One of the key contributions of this work is the establishment of adaptive confidence sets for general M-estimation under weaker distributional assumptions.
\end{itemize}
Research on adaptive confidence sets for nonparametric models was particularly active from the 1990s to the 2010s. These studies generally relied on concentration inequalities to establish validity, requiring precise error quantification for adaptive nonparametric estimators. Additional historical developments can be found in Chapter 8.4 of \citet{gine2021mathematical}.

\section{Proofs from \Cref{sec:general}}\label{appsec:proof-of-coverage-anti-conservative}
\subsection{Proof of Theorem~\ref{thm:coverage-anti-conservative-confidence-set}}
% \begin{proof}[\bfseries{Proof of \Cref{thm:coverage-anti-conservative-confidence-set}}]
Let $\widetilde D_1$ be a random element, defined on a (possibly extended) probability space, such that
$\mathcal L(\widetilde D_1) = \mathcal L(D_1)$ and
$\widetilde D_1$ is independent of $D_2$. Define $\widehat\theta_1 := \widehat\theta(D_1)$ and
$\widetilde\theta_1 := \widehat\theta(\widetilde D_1)$,
where $\widehat\theta(\cdot)$ is a measurable function. By Berbee’s coupling lemma \citep{berbee1979random}, also in Chapter 5, Lemma 5.1 of \citet{rio2017asymptotic}, there exists a coupling satisfying
\[
\mathbb P(\widetilde D_1 \neq D_1)
\le
\beta(r).
\]
Since $\{\widehat{\theta}_1 \not=\widetilde{\theta}_1\} \subseteq \{\widetilde D_1 \neq D_1\}$, it implies $\mathbb{P}(\widehat{\theta}_1 \not=\widetilde{\theta}_1) \le \beta(r)$. Denote $\mathcal{L}(\widetilde D_1) = \widetilde P^1$, $\mathcal{L}(D_1) = P^1$, and $\mathcal{L}(D_2) = P^2$. Then the miscoverage probability can be written as
    \begin{align*}
        &\mathbb{P}_{P^N}(\theta(P^N) \not\in \widehat{\mathrm{CI}}^\dagger_{N}) \\
        &\quad = \mathbb{P}_{P^N}\left( \widehat{\mathbb{M}}_2(\theta(P^N)) - \widehat{\mathbb{M}}_2(\widehat{\theta}_1) \ge 0\right) \\
        &\quad\le \mathbb{P}_{P^N}\left( \widehat{\mathbb{M}}_2(\theta(P^N)) - \widehat{\mathbb{M}}_2(\widehat{\theta}_1) \ge 0 \, \cap\, \{\widehat{\theta}_1=\widetilde{\theta}_1\}\right) + \mathbb{P}(\widehat{\theta}_1 \not=\widetilde{\theta}_1)\\
        &\quad= \mathbb{E}_{\widetilde P^1}\left[\mathbb{P}_{P^2|\widetilde P^1}\left( \widehat{\mathbb{M}}_2(\theta(P^N)) - \widehat{\mathbb{M}}_2(\widehat{\theta}_1) \ge 0 \, \cap\, \{\widehat{\theta}_1=\widetilde{\theta}_1\}\, |\, \widetilde{\theta}_1\right)\right] + \mathbb{P}(\widehat{\theta}_1 \not=\widetilde{\theta}_1).
    \end{align*}
    Conditional on $\widetilde{\theta}_1$, $\widehat{\mathbb{M}}_2(\cdot)$ depends only on $D_2$, which is independent of $\widetilde D_1$. Applying Chebyshev’s inequality (see Equation (3) of \citet{pinelis2010between}) gives
    \begin{align*}
        &\mathbb{P}_{P^2}\left( \widehat{\mathbb{M}}_2(\theta(P^N)) - \widehat{\mathbb{M}}_2(\widehat{\theta}_1) \ge 0 \, \cap\, \{\widehat{\theta}_1=\widetilde{\theta}_1\}\mid\widetilde{\theta}_1\right)\\
        &\quad \le \min \left\{\frac{\mathbb{E}_{P^2}[|(\widehat{\mathbb{M}}_2 - \mathbb{M}_2)(\theta(P^N)) - (\widehat{\mathbb{M}}_2 - \mathbb{M}_2)(\widehat{\theta}_1)|^2\mid \widetilde{\theta}_1]}{(\mathbb{M}_2(\widehat{\theta}_1) - \mathbb{M}_2(\theta(P^N)))^2}, 1\right\} \\
        &\quad = \min \left\{\frac{\mathbb{V}_2(\widehat{\theta}_1)}
{\mathbb{C}_2^2(\widehat{\theta}_1)}, 1\right\} = \min\left\{\frac{1}{\ratio^2}, 1\right\}.
\end{align*}
We conclude the result by taking the expectation over $\widetilde D_1$, which has the same law as $D_1$.

\subsection{Proof of Theorem~\ref{thm:coverage-anti-conservative-confidence-set-unbiased}}\label{appsec:proof-of-coverage-anti-conservative-unbiased}
% \begin{proof}[\bfseries{Proof of \Cref{thm:coverage-anti-conservative-confidence-set}}]
The coupling argument is identical to that of the proof of \Cref{thm:coverage-anti-conservative-confidence-set}. Since $\widehat{\mathbb{M}}_2(\theta(P^N)) - \widehat{\mathbb{M}}_2(\widehat{\theta}_1)$ is unbiased, we can now invoke Cantelli's inequality (see Equation (2) of \citet{pinelis2010between}) instead of Chebyshev's inequality. This gives 
    \begin{align*}
        &\mathbb{P}_{P^2}\left( \widehat{\mathbb{M}}_2(\theta(P^N)) - \widehat{\mathbb{M}}_2(\widehat{\theta}_1) \ge 0 \, \cap\, \{\widehat{\theta}_1=\widetilde{\theta}_1\}\, |\, \widetilde{\theta}_1\right)\\
        &\quad = \mathbb{P}_{P^2}\left( \frac{\widehat{\mathbb{M}}_2(\theta(P^N)) - \widehat{\mathbb{M}}_2(\widehat{\theta}_1) + \widehat{\mathbb{C}}_2}{\widehat{\mathbb{V}}^{1/2}_2}\ge \frac{\widehat{\mathbb{C}}_2}{\widehat{\mathbb{V}}^{1/2}_2} \, \cap\, \{\widehat{\theta}_1=\widetilde{\theta}_1\}\, |\, \widetilde{\theta}_1\right)\le \frac{1}{1+\widehat{\mathbb{C}}_2^2/\widehat{\mathbb{V}}_2} = \frac{1}{1+\ratio^2}.
\end{align*}
We conclude the result by taking the expectation over $\widetilde D_1$, which has the same law as $D_1$.
\section{Proofs from \Cref{sec:general-M-estimation}}\label{appsec:proof-of-M-estimation}
\subsection{Proof of \Cref{thm:coverage-anti-conservative-confidence-set-empirical-risk}}
Since $P^1$ and $P^2$ are independent, we can write the miscoverage probability as  
\begin{align*}
    \mathbb{P}_{P^N}\left( \widehat{\mathbb{M}}_2(\theta(P^N)) - \widehat{\mathbb{M}}_2(\widehat{\theta}_1) \ge 0 \right) = \E_{P^1}\left[\mathbb{P}_{P^{2}|P^{1}}\left( \widehat{\mathbb{M}}_2(\theta(P^N)) - \widehat{\mathbb{M}}_2(\widehat{\theta}_1) \ge 0 \right)\right].
\end{align*}
Conditionally on $D_1$, we have
    \begin{align*}
    &\mathbb{P}_{P^{2}|P^{1}}\left( \widehat{\mathbb{M}}_2(\theta(P^N)) - \widehat{\mathbb{M}}_2(\widehat{\theta}_1) \ge 0 \right)\\
        &\quad = \mathbb{P}_{P^{2}|P^{1}}\left(\sum_{i\in I_2} m_{\theta(P^N)}(Z_i)-m_{\widehat\theta}(Z_i)\ge 0 \right)\\
        &\quad = \mathbb{P}_{P^{2}|P^{1}}\left( -\sum_{i \in I_2} \widehat \xi_i/(n_2\widehat{\mathbb{V}}_{2}^{1/2}) \ge \widehat{\mathbb{C}}_{2}/\widehat{\mathbb{V}}^{1/2}_{2} \right)\\
    &\quad \le \mathbb{P}\left(Z \ge \widehat{\mathbb{C}}_{2}/\widehat{\mathbb{V}}^{1/2}_{2} \,\bigg| \, D_1\right) \\
    &\quad\quad+ \left|\mathbb{P}_{P^{2}|P^{1}}\left( -\sum_{i \in I_2} \widehat \xi_i/(n_2\widehat{\mathbb{V}}_{2}^{1/2}) \ge \widehat{\mathbb{C}}_{2}/\widehat{\mathbb{V}}^{1/2}_{2} \right)-\mathbb{P}\left(Z \ge \widehat{\mathbb{C}}_{2}/\widehat{\mathbb{V}}^{1/2}_{2} \,\bigg| \, D_1\right)\right| 
    \end{align*}
    where $Z$ denotes a standard Normal random variable.  Conditioning on $D_1$, the standardized sum $-n_2^{-1}\widehat{\mathbb{V}}_2^{-1/2}\sum_{i \in I_2} \widehat\xi_i$ is a sum of independent mean-zero random variables with unit variance. Hence, the last remainder term can be controlled by a non-uniform Berry-Esseen bound (such as Theorem 2.1 of \citet{chen2001non}):
    \begin{align*}
    &\left|\mathbb{P}_{P^{2}|P^{1}}\left( -\sum_{i \in I_2} \widehat \xi_i/(n_2\widehat{\mathbb{V}}_{2}^{1/2}) \ge \widehat{\mathbb{C}}_{2}/\widehat{\mathbb{V}}^{1/2}_{2} \right)-\mathbb{P}\left(Z \ge \widehat{\mathbb{C}}_{2}/\widehat{\mathbb{V}}^{1/2}_{2} \,\bigg| \, D_1\right)\right|  \\
    &\quad \le C\sum_{i\in I_2} \mathbb{E}_{P_i}\left[\frac{|\widehat\xi_i|^2}{(n_2^2\widehat{\mathbb{V}}_{2})(1 + \widehat{\mathbb{C}}_{2}/\widehat{\mathbb{V}}^{1/2}_{2})^2}\min\left\{1,\,\frac{|\widehat\xi_i|}{n_2\widehat{\mathbb{V}}_{2}^{1/2}(1 + \widehat{\mathbb{C}}_{2}/\widehat{\mathbb{V}}_{2}^{1/2})}\right\} \bigg|D_1\right],
    \end{align*}
    where $C$ is a universal constant. Finally, taking expectation over $D_1$ and using linearity of expectation concludes the proof.

\subsection{Proof of \Cref{thm:coverage-martingale}}\label{appsec:proof-of-coverage-martingale}
We first establish the following lemma, of which \Cref{thm:coverage-martingale} is a direct consequence.
\begin{lemma}\label{thm:general-dependence}
    Let $\mathcal{H}_0 \subseteq \ldots \subseteq \mathcal{H}_N$ be a filtration, supporting $\widehat\xi_1, \ldots, \widehat\xi_N$ defined in \eqref{eq:centered-xi}, with  $\widehat\theta_1$ assumed $\mathcal{H}_0$-measurable. Define the martingale approximation 
    \begin{equation}
        \widetilde{\xi}_i = \sum_{r=1}^N(\E[\widehat\xi_r|\mathcal{H}_i]-\E[\widehat\xi_r|\mathcal{H}_{i-1}]) \quad \textrm{and} \quad U=\sum_{i=1}^N (\widehat\xi_i - \E[\widehat\xi_i|\mathcal{H}_N] + \E[\widehat\xi_i|\mathcal{H}_0]).
    \end{equation}
    Set $\widetilde{\mathbb{V}} = \mathrm{Var}(\sum_{i=1}^N \widetilde{\xi}_i|\mathcal{H}_0)$ and $\widehat{\mathbb{C}} = \sum_{i=1}^N \E[m_{\widehat\theta_1}(Z_i)-m_{\theta(P^N)}(Z_i)|\mathcal{H}_0]$. For $\delta \in (0, \infty)$, define 
    \begin{equation}
        \begin{split}
        L_{2\delta}:= \sum_{i=1}^N \E\left[\left|\frac{\widetilde\xi_i}{\widetilde{\mathbb{V}}^{1/2}}\right|^{2+2\delta} \,|\, \mathcal{H}_0\right] \quad \textrm{and} \quad  
        M_{2\delta}:= \E\left(\left|\sum_{i =1}^N \E\left[\frac{\widetilde\xi_i^2}{\widetilde{\mathbb{V}}} \mid \mathcal{H}_{i-1}\right]-1\right|^{1+\delta} \mid \mathcal{H}_0\right)\quad.
    \end{split}
    \end{equation}
    Then for any $\eta > 0$, it holds that 
    \begin{equation}
        \begin{split}
           \mathbb{P}_{P^N}\!\left(\theta(P^N) \notin \widehat{\mathrm{CI}}_{N}^{\dagger}\right) &\le \mathbb{E}\left[1-\Phi\left((1-\eta)\widehat{\mathbb{C}}  /\widetilde{\mathbb{V}}^{1/2} \right) \right] \\
    &\quad + \mathbb{E}\left[\min\left\{1, C_\delta \frac{(L_{2\delta} +  M_{2\delta})^{1/(3+2\delta)}}{1+|(1-\eta)\widehat{\mathbb{C}}/\widetilde{\mathbb{V}}^{1/2}|^{2+2\delta}}\right\}\right] \\
    &\quad + \mathbb{P}\left( U < -\eta \,\widehat{\mathbb{C}} \right),
        \end{split}
    \end{equation}
    where $C_\delta$ is a constant depending only on $\delta$.
\end{lemma}

\begin{proof}[\bfseries{Proof of \Cref{thm:general-dependence}}]
    By a telescoping identity, $\sum_{i=1}^N \widehat \xi_i = \sum_{i=1}^N \widetilde \xi_i + U$. Since $\widehat\theta_1$ is $\mathcal{H}_0$-measurable, conditioning on $\mathcal{H}_0$ gives 
\begin{align*}
    &\mathbb{P}\left(\theta(P^N) \notin \widehat{\mathrm{CI}}_{N}^{\dagger}\mid \mathcal{H}_0\right) \\
    &\quad = \mathbb{P}\left(\sum_{i=1}^N m_{\theta(P^N)}(Z_i)-m_{\widehat\theta_1}(Z_i) \ge 0\mid \mathcal{H}_0\right)\\
    &\quad=  \mathbb{P}\left(-\sum_{i=1}^N  \widehat \xi_i\ge \widehat{\mathbb{C}} \mid \mathcal{H}_0\right)\\
    &\quad=  \mathbb{P}\left(-\sum_{i=1}^N  \widetilde \xi_i/\widetilde{\mathbb{V}}^{1/2} \ge (\widehat{\mathbb{C}}  + U)/\widetilde{\mathbb{V}}^{1/2} \mid \mathcal{H}_0\right)\\
    &\quad\le  \mathbb{P}\left(-\sum_{i=1}^N  \widetilde \xi_i/\widetilde{\mathbb{V}}^{1/2} \ge (1-\eta)\widehat{\mathbb{C}}  /\widetilde{\mathbb{V}}^{1/2} \mid \mathcal{H}_0\right) + \mathbb{P}\left( U < -\eta \widehat{\mathbb{C}} \mid \mathcal{H}_0\right)\\
    &\quad\le  \mathbb{P}\left(Z \ge (1-\eta)\widehat{\mathbb{C}}  /\widetilde{\mathbb{V}}^{1/2} \mid \mathcal{H}_0\right) + \mathbb{P}\left( U < -\eta \widehat{\mathbb{C}} \mid \mathcal{H}_0\right)\\
    &\quad\quad + \left|\mathbb{P}\left(-\sum_{i=1}^N \widetilde \xi_i/\widetilde{\mathbb{V}}^{1/2} \ge (1-\eta)\widehat{\mathbb{C}}  /\widetilde{\mathbb{V}}^{1/2} \mid \mathcal{H}_0\right)-\mathbb{P}\left(Z \ge (1-\eta)\widehat{\mathbb{C}}  /\widetilde{\mathbb{V}}^{1/2} \mid \mathcal{H}_0\right)\right|,
    \end{align*}
    where $Z$ follows a standard normal distribution and is independent of $\mathcal{H}_N$. On the $\mathcal{H}_0$-measurable event $\{L_{2\delta} + M_{2\delta}\le 1\}$, Theorem 1 of \citet{haeusler1988nonuniform} gives, and obtain 
    \begin{align*}
        &\left|\mathbb{P}\left(-\sum_{i=1}^N  \widetilde \xi_i/\widetilde{\mathbb{V}}^{1/2} \ge (1-\eta)\widehat{\mathbb{C}}  /\widetilde{\mathbb{V}}^{1/2} \mid \mathcal{H}_0\right)-\mathbb{P}\left(Z \ge (1-\eta)\widehat{\mathbb{C}}  /\widetilde{\mathbb{V}}^{1/2} \mid \mathcal{H}_0\right)\right| \\
        &\quad \le C_\delta \frac{( L_{2\delta} +  M_{2\delta})^{1/(3+2\delta)}}{1+|(1-\eta)\widehat{\mathbb{C}}/\widetilde{\mathbb{V}}^{1/2}|^{2+2\delta}} + \mathbf{1}\{L_{2\delta} +  M_{2\delta} > 1\},
    \end{align*}
    where $C_\delta$ is a constant only depending on $\delta$. In the proof of Theorem 1 of \citet{haeusler1988nonuniform}, it is stated that the following bound
    \begin{align*}
        \frac{L_{2\delta} +  M_{2\delta}}{1+|(1-\eta)\widehat{\mathbb{C}}/\widetilde{\mathbb{V}}^{1/2}|^{2+2\delta}}
    \end{align*}
    is valid for $\mathbf{1}\{L_{2\delta} +  M_{2\delta} > 1\}$, which can further bounded by the trivial bound of 1. Taking expectations over $\mathcal{H}_0$ and combining the terms concludes the proof.
\end{proof}
\begin{proof}[\bfseries{Proof of \Cref{thm:coverage-martingale}}]
For $i \in I_2$, the definition of $\widehat\xi_i$ and $\mathcal{H}_k$-measurability of $\widehat\theta_1$ together imply that $\E[\widehat\xi_i|\mathcal{H}_{n_2}]  = \widehat \xi_i$ and $\E[\widehat\xi_i|\mathcal{H}_0]=0$. Then we have
\begin{align*}
    U = \sum_{i \in I_2} (\widehat\xi_i - \E[\widehat\xi_i|\mathcal{H}_{n_2}] + \E[\widehat\xi_i|\mathcal{H}_0]) = 0 \quad \text{and} \quad \sum_{i \in I_2} \widehat \xi_i = \sum_{i \in I_2} \widetilde \xi_i.
\end{align*}
Moreover, 
\begin{align*}
    \widehat{\mathbb{C}} = \sum_{i\in I_2} \E[m_{\widehat\theta_1}(Z_i)-m_{\theta(P^N)}(Z_i)|\mathcal{H}_0] = n_2 \widehat{\mathbb{C}}_2, 
\end{align*}
and
\begin{align*}
    \widetilde{\mathbb{V}}= \mathrm{Var}\left(\sum_{i \in I_2} \widetilde{\xi}_i|\mathcal{H}_0\right) = n_2^2\mathrm{Var}\left(\frac{1}{n_2}\sum_{i \in I_2} \widehat{\xi}_i|\mathcal{H}_0\right) = n_2^2\widehat{\mathbb{V}}_2.
\end{align*}
The factors of $n_2$ cancels, yielding $\ratio =\widehat{\mathbb{C}}_2/\widehat{\mathbb{V}}_2^{1/2} = \widehat{\mathbb{C}}/\widetilde{\mathbb{V}}^{1/2}$. Since $\widehat{\mathbb{C}}_2 \ge 0$ almost surely by definition, it follows that $\mathbb{P}( U  < -\eta \widehat{\mathbb{C}}_2 \mid \mathcal{H}_0) = 0$ for any $\eta > 0$. The result is obtained from applying \Cref{thm:general-dependence} and taking $\eta \to 0$. 
\end{proof}

\section{Proofs from \Cref{sec:coverage-slpha}}\label{supsec:coverage-proofs}
\subsection{Proofs of \Cref{thm:hulc-like} and \Cref{thm:general_LCB}}
\begin{proof}[\bfseries{Proof of \Cref{thm:hulc-like}}]
Denote $\overline D_\ell := \{Z_i \, :\, i \in S_\ell\}$ for $1 \le \ell \le B$. By iterative application of Berbee’s coupling lemma \citep{berbee1979random}, construct independent copies $\widetilde D_1, \ldots, \widetilde D_B$ such that $\mathcal{L}(\widetilde D_\ell) = \mathcal{L}(\overline D_\ell)$ with $\mathbb{P}(\overline D_\ell \neq \widetilde D_\ell) \le \beta^\dagger(r)$ at each coupling step. Define test functions 
\begin{align*}
    \widehat{\phi}_\ell := \mathbf{1}\{\theta(P^N) \notin \widehat{\mathrm{CI}}^\dagger_\ell\} \quad \text{and} \quad \widetilde{\phi}_\ell := \mathbf{1}\{\theta(P^N) \notin \widetilde{\mathrm{CI}}^\dagger_\ell\},
\end{align*}
where $\widehat{\mathrm{CI}}^\dagger_\ell$ and $\widetilde{\mathrm{CI}}^\dagger_\ell$ are the confidence sets constructed from $\overline D_\ell$ and $\widetilde D_\ell$ respectively. By the union bound at each $B-1$ coupling steps, $\mathbb{P}(\exists\, \ell \in \{1, \ldots, B\} : \widehat{\phi}_\ell \neq \widetilde{\phi}_\ell) \le (B-1)\beta^\dagger(r).$ On the complimentary event $\{\widehat{\phi}_\ell= \widetilde{\phi}_\ell \textrm{ for all }\ell\}$, the confidence sets $\widetilde{\mathrm{CI}}^\dagger_1, \ldots, \widetilde{\mathrm{CI}}^\dagger_B$ are independent. Therefore,
\begin{align*}
\mathbb{P}_{P^N}\!\left(\theta(P^N) \notin \widehat{\mathrm{CI}}^{\texttt{DS}}_{N,\alpha}\right)
&= \mathbb{P}_{P^N}\!\left(\bigcap_{\ell=1}^B \left\{\theta(P^N) \notin \widehat{\mathrm{CI}}^\dagger_\ell\right\}\right) \\
&\le \mathbb{P}_{P^N}\!\left(\bigcap_{\ell=1}^B \left\{\theta(P^N) \notin \widetilde{\mathrm{CI}}^\dagger_\ell\right\}\right) + (B-1)\beta^\dagger(r) \\
&= \prod_{\ell=1}^B \mathbb{P}_{P^N}\!\left(\theta(P^N) \notin \widetilde{\mathrm{CI}}^\dagger_\ell\right) + (B-1)\beta^\dagger(r) \\
&\le \left(p+ \mathfrak{R}_{N_0, P^N}\right)^{\!B} + (B-1)\beta^\dagger(r) \\
&\le \alpha\left(1 + p^{-1}\mathfrak{R}_{N_0, P^N}\right)^B + (B-1)\beta^\dagger(r),
\end{align*}
where $N_0 = \min_\ell |S_\ell|$ denotes the smallest bin size, the third line uses independence, the fourth line applies \eqref{eq:median-valid} and the last line uses $p^B \le \alpha$.
\end{proof}
\begin{proof}[\bfseries{Proof of \Cref{thm:general_LCB}}]
By Berbee's lemma, construct $\widetilde{D}_1$ with $\mathcal{L}(D_1) = \mathcal{L}(\widetilde{D}_1)$, and with $\widetilde{D}_1$ and $D_2$ independent. Let $\widetilde\theta_1 := \widehat{\theta}(\widetilde D_1)$ and define 
\begin{equation*}
\Omega := \left\{
        \mathbb{M}_2(\theta(P^N)) - \mathbb{M}_2(\widetilde{\theta}_1) \geq \widehat{\mathbb{M}}_2(\theta(P^N)) - \widehat{\mathbb{M}}_2(\widetilde{\theta}_1)-\widehat t_{\alpha}(\theta(P^N), \widetilde\theta_1)\right\}.
\end{equation*}
By \eqref{eq:lower-bound}, we have $\mathbb{P}(\Omega^c) \le \alpha$.
Conditioning on $\Omega \, \cap \, \{D_1 = \widetilde D_1\}$, the basic inequality \eqref{eq:zeroth-order} gives 
\begin{align*}
&\mathbb{P}_{P^N}\left(\theta(P^N) \notin \widehat{\mathrm{CI}}^{\texttt{LCB}}_{N,\alpha}\, \cap\, \Omega \, \cap \, \{D_1 = \widetilde D_1\}\right)\\
&\quad = \mathbb{P}_{P^N}\left(\widehat{\mathbb{M}}_2(\theta(P^N)) - \widehat{\mathbb{M}}_2(\widetilde{\theta}_1)-\widehat t_{\alpha}(\theta(P^N), \widetilde{\theta}_1) \ge 0\, \cap\, \Omega \, \cap \, \{D_1 = \widetilde D_1\}\right)\\
&\quad = \mathbb{P}_{P^N}\left(\mathbb{M}_2(\theta(P^N)) - \mathbb{M}_2(\widetilde{\theta}_1)  \ge 0\, \cap\, \Omega \, \cap \, \{D_1 = \widetilde D_1\}\right) = 0.
\end{align*}
Putting together, we conclude
\begin{align*}
\mathbb{P}_{P^N}\!\left(\theta(P^N) \notin \widehat{\mathrm{CI}}^{\texttt{LCB}}_{N,\alpha}\right)
&\le \mathbb{P}(\Omega^c) + \mathbb{P}(D_1 \neq \widetilde{D}_1) \le \alpha + \beta(r). 
\end{align*}
\end{proof}
\subsection{Proofs of \Cref{thm:studentized-nonIID}---\Cref{thm:large-deviation}}\label{supp:studentized-proofs}
\begin{proof}[\bfseries{Proof of \Cref{thm:studentized-nonIID}}]
    The goal is to prove \eqref{eq:reinterpret-lower-upper} with $\beta(r) = 0$ holding from independence. The result then follows by \Cref{thm:general_LCB}. Observe
    \begin{align*}
    &\mathbb{P}_{P^{2}}\left(\frac{1}{n_2}\sum_{i \in I_2}\widehat\xi_i  \ge \widehat{t}_\alpha(\theta(P^N), \widehat{\theta}_1) |D_1\right) \\
        & \quad = \mathbb{P}_{P^{2}}\left(\frac{1}{n_2}\sum_{i \in I_2}\widehat\xi_i  \ge n_2^{-1/2}z_\alpha \widehat \sigma_{\theta, \widehat\theta_1}|D_1\right)\\
        & \quad \le \left| \mathbb{P}_{P^{2}}\left(n_2^{-1/2} \widehat \sigma_{\theta, \widehat\theta_1}^{-1}\sum_{i \in I_2}\widehat\xi_i  \ge z_\alpha|D_1\right) - \mathbb{P}(Z\ge z_\alpha) \right|+\mathbb{P}(Z\ge z_\alpha) \\
        & \quad \le C \left\{\sum_{i\in I_2}\mathbb{P}_{P_i}(\widehat{\xi}^2_i > V^2 |D_1) + M^{-1}\sum_{i\in I_2}|\mathbb{E}_{P_i}[\overline{\xi}_i |D_1]| + M^{-3}\sum_{i \in I_2} \mathbb{E}_{P_i}[|\overline{\xi}_i|^3| D_1]\right\}+\alpha
    \end{align*}
    where the last inequality is by Corollary 1.1 of \cite{bentkus1996berry}. The result is obtained after taking expectation over $D_1$.
\end{proof}
\begin{proof}[\bfseries{Proof of \Cref{thm:studentized-katz}}]
    \Cref{thm:studentized-nonIID} holds after truncating at 1, such that 
    \begin{align*}
        &\mathbb{P}_{P^N}\left(\theta(P^N) \not\in \widehat{\mathrm{CI}}^{\mathtt{CLT}}_{N, \alpha}\right) \le \alpha \\
        &\quad +\E_{P^1}\left[\min\left\{1, C \left(\sum_{i\in I_2}\mathbb{P}_{P_i}(\widehat{\xi}^2_i > V^2 |D_1) + \frac{\sum_{i\in I_2}|\mathbb{E}_{P_i}[\overline{\xi}_i|D_1]|}{M} +\frac{\sum_{i \in I_2} \mathbb{E}_{P_i}[|\overline{\xi}_i|^3| D_1]}{M^3}\right)\right\}\right].
    \end{align*}
    Since this holds for any $V$, we can further bound the remainder term with the choice $V = n_2\widehat{\mathbb{V}}^{1/2}_2$. The following proof is based on \cite{katz1963note}.
     Define $\mathcal{G}$ to be the class of all non-decreasing functions $g : (0, \infty) \mapsto (0, \infty)$ and that $x/g(x)$ is non-decreasing on $(0, \infty)$. Fix $g \in \mathcal{G}$. Using the properties of $\mathcal{G}$ such that $g \in \mathcal{G}$ is non-decreasing and non-negative, we have
\begin{align}
&\{|\widehat \xi_i| \ge n_2\widehat{\mathbb{V}}^{1/2}_2\} \subseteq \{|\widehat \xi_i|g(|\widehat \xi_i|) \ge n_2\widehat{\mathbb{V}}^{1/2}_2 g(\widehat{\mathbb{V}}^{1/2}_2)\}.
\end{align}
Then we have 
\begin{align*}
    \sum_{i\in I_2}\mathbb{P}_{P_i}(\widehat{\xi}^2_i > n_2^2\widehat{\mathbb{V}}_2 | D_1) \le \sum_{i\in I_2}\mathbb{P}_{P_i}(\widehat{\xi}^2_ig(|\widehat{\xi_i}|) > n_2^2\widehat{\mathbb{V}}_2g(n_2\widehat{\mathbb{V}}_2^{1/2})| D_1) \le \frac{\sum_{i\in I_2}\E_{P_i}[\widehat{\xi}^2_ig(|\widehat{\xi_i}|)|D_1]}{n_2^2\widehat{\mathbb{V}}_2g(n_2\widehat{\mathbb{V}}_2^{1/2})},
\end{align*}
where $\mathbf{1}\{A > B\} \le A/B$ for $A, B \ge 0$. Next, observe that 
\begin{align*}
    n_2^{-1}\widehat{\mathbb{V}}_2^{-1/2}\widehat{\xi}_i &= n_2^{-1}\widehat{\mathbb{V}}_2^{-1/2}\widehat{\xi}_i\mathbf{1}\{|\widehat{\xi}_i| \le n_2\widehat{\mathbb{V}}_2^{1/2}\} + n_2^{-1}\widehat{\mathbb{V}}_2^{-1/2}\widehat{\xi}_i\mathbf{1}\{|\widehat{\xi}_i| > n_2\widehat{\mathbb{V}}_2^{1/2}\} \\
    &= \overline{\xi}_i+ n_2^{-1}\widehat{\mathbb{V}}_2^{-1/2}\widehat{\xi}_i\mathbf{1}\{|\widehat{\xi}_i| > n_2\widehat{\mathbb{V}}_2^{1/2}\}.
\end{align*}
Since $\E_{P_i}[\widehat{\xi}_i|D_1] = 0$, we have 
\begin{align}\label{eq:katz-property}
    \E_{P_i}[\overline{\xi}_i|D_1]  = -n_2^{-1}\widehat{\mathbb{V}}_2^{-1/2}\E_{P_i}[\widehat{\xi}_i\mathbf{1}\{|\widehat{\xi}_i| > n_2\widehat{\mathbb{V}}_2^{1/2}\}|D_1].
\end{align}
Next, we bound the variance. Observe 
\begin{align*}
  1-M^2 &= \frac{\sum_{i\in I_2}\E_{P_i}[\widehat \xi_i^2|D_1]}{n_2^2\widehat{\mathbb{V}}_2}-\sum_{i \in I_2} \E_{P_i}[\overline{\xi}_i^2| D_1] +\sum_{i \in I_2} (\E_{P_i}[\overline{\xi}_i| D_1])^2\\
  &\le \frac{2\sum_{i\in I_2}\E_{P_i}[\widehat \xi_i^2\mathbf{1}\{|\widehat \xi_i| \ge n_2\widehat{\mathbb{V}}_2^{1/2}\}|D_1]}{n_2^2\widehat{\mathbb{V}}_2} \le \frac{2\sum_{i\in I_2}\E_{P_i}[|\widehat{\xi}_i|^2g(|\widehat{\xi}_i|)|D_1]}{n_2^2\widehat{\mathbb{V}}_2g(n_2\widehat{\mathbb{V}}_2^{1/2})},
\end{align*}
where we used \eqref{eq:katz-property}. Next,  
\begin{align*}
    \sum_{i\in I_2}|\mathbb{E}_{P_i}[\overline{\xi}_i| D_1]| &= \sum_{i\in I_2}|\E_{P_i}[n_2^{-1}\widehat{\mathbb{V}}_2^{-1/2}\widehat{\xi}_i\mathbf{1}\{|\widehat{\xi}_i| > n_2\widehat{\mathbb{V}}_2^{1/2}\}|D_1]| \\
    &\le \sum_{i\in I_2}\E_{P_i}[n_2^{-1}\widehat{\mathbb{V}}_2^{-1/2}|\widehat{\xi}_i|\mathbf{1}\{|\widehat{\xi}_i| > n_2\widehat{\mathbb{V}}_2^{1/2}\}|D_1]\\
    &\le \sum_{i\in I_2}\E_{P_i}[n_2^{-1}\widehat{\mathbb{V}}_2^{-1/2}|\widehat{\xi}_i|\mathbf{1}\{|\widehat{\xi}_i|g(|\widehat{\xi}_i|) > n_2\widehat{\mathbb{V}}_2^{1/2}g(n_2\widehat{\mathbb{V}}_2^{1/2})\}|D_1]\\
    &\le \frac{\sum_{i\in I_2}\E_{P_i}[|\widehat{\xi}_i|^2g(|\widehat{\xi}_i|)|D_1]}{n_2^2\widehat{\mathbb{V}}_2g(n_2\widehat{\mathbb{V}}_2^{1/2})}.
\end{align*}
Finally, we have
\begin{align*}
    \sum_{i \in I_2} \mathbb{E}_{P_i}[|\overline{\xi}_i|^3| D_1] &= \frac{\sum_{i \in I_2} \mathbb{E}_{P_i}[|\widehat{\xi}_i|^3\mathbf{1}\{|\widehat{\xi}_i| \le n_2\widehat{\mathbb{V}}_2^{1/2}\}| D_1]}{n_2^3\widehat{\mathbb{V}}_2^{3/2}}\\
    &= \sum_{i \in I_2} \mathbb{E}_{P_i}\left[\frac{|\widehat{\xi}_i|^2|\widehat\xi_i|g(|\widehat \xi_i|)\mathbf{1}\{|\widehat{\xi}_i| \le n_2\widehat{\mathbb{V}}_2^{1/2}\}}{n_2^2\widehat{\mathbb{V}}_2 \cdot n_2\widehat{\mathbb{V}}_2^{1/2}g(|\widehat \xi_i|)}\bigg| D_1\right]\\
    % &= \sum_{i \in I_2} \mathbb{E}_{P_i}\left[\frac{|\widehat{\xi}_i|^2|\widehat\xi_i|g(|\widehat \xi_i|)\mathbf{1}\{|\widehat{\xi}_i| \le n_2\widehat{\mathbb{V}}_2^{1/2}\}}{\widehat{\mathbb{V}}_2 \cdot n_2\widehat{\mathbb{V}}_2^{1/2}g(|\widehat \xi_i|)}\bigg| D_1\right]\\
    &\le \frac{\sum_{i \in I_2} \mathbb{E}_{P_i}[|\widehat{\xi}_i|^2g(|\widehat \xi_i|)| D_1]}{n_2^2\widehat{\mathbb{V}}_2g(n_2\widehat{\mathbb{V}}_2^{1/2})}.
\end{align*}

We now assume that 
\begin{align*}
    \frac{\sum_{i\in I_2}\E_{P_i}[|\widehat{\xi}_i|^2g(|\widehat{\xi}_i|)|D_1]}{n_2^2\widehat{\mathbb{V}}_2g(n_2\widehat{\mathbb{V}}_2^{1/2})} \le \frac{1}{3}.
\end{align*}
Under this assumption, we have $1/3 \le M^3$, and thus $M$ is bounded from below. Therefore, 
\begin{align*}
    &\mathbb{P}_{P^N}\left(\theta(P^N) \not\in \widehat{\mathrm{CI}}^{\mathtt{CLT}}_{N, \alpha}\right) \le \alpha & \\
    &\quad  + \min\left\{1, C\, \mathbb{E}_{P^1}\left[\sum_{i\in I_2}\mathbb{P}_{P_i}(\widehat{\xi}^2_i > \widehat{\mathbb{V}}_2 | D_1) + \frac{\sum_{i\in I_2}|\mathbb{E}_{P_i}[\overline{\xi}_i|D_1]|}{M} +\frac{\sum_{i \in I_2} \mathbb{E}_{P_i}[|\overline{\xi}_i|^3| D_1]}{M^3}\right] \right\}&\\
    &\quad \le \alpha + \min\left\{1, C\frac{\sum_{i \in I_2} \mathbb{E}_{P_i}[|\widehat{\xi}_i|^2g(|\widehat \xi_i|)| D_1]}{n_2^2\widehat{\mathbb{V}}_2g(n_2\widehat{\mathbb{V}}_2^{1/2})}\right\}.
\end{align*}
Note that choosing $C \ge 3$,  the bound is still valid under the case where 
\begin{align*}
    \frac{\sum_{i\in I_2}\E_{P_i}[|\widehat{\xi}_i|^2g(|\widehat{\xi}_i|)|D_1]}{n_2^2\widehat{\mathbb{V}}_2g(n_2\widehat{\mathbb{V}}_2^{1/2})} \ge \frac{1}{3},
\end{align*}
since the result follows trivially.
Finally, choosing $g(|x|) = \min\{|x|, n_2\widehat{\mathbb{V}}_2^{1/2}\}$, we conclude 
\begin{align*}
    &\mathbb{P}_{P^N}\left(\theta(P^N) \not\in \widehat{\mathrm{CI}}^{\mathtt{CLT}}_{N, \alpha}\right) & \\&\quad \le  \alpha + \min\left\{1, C\frac{\sum_{i \in I_2} \mathbb{E}_{P_i}[|\widehat{\xi}_i|^2\min(|\widehat \xi_i|, n_2\widehat{\mathbb{V}}_2^{1/2})| D_1]}{n_2^3\widehat{\mathbb{V}}_2^{3/2}}\right\}\\
    &\quad =\alpha + \min\left\{1, C\sum_{i \in I_2} \mathbb{E}_{P_i}\left[\frac{|\widehat{\xi}_i|^2}{n_2^2\widehat{\mathbb{V}}_2}\min\left\{\frac{|\widehat \xi_i|}{n_2\widehat{\mathbb{V}}_2^{1/2}}, 1\right\}\bigg| D_1\right]\right\}.
\end{align*}
\end{proof}
\begin{proof}[\bfseries{Proof of \Cref{cor:studentized-IID}}]
    The proof is identical to that of \Cref{thm:studentized-nonIID}, except that we use Corollary 1.2 of \cite{bentkus1996berry}.
\end{proof}
\begin{proof}[\bfseries{Proof of \Cref{thm:studentized-IID-multi}}]
First, we prove the case with $\alpha \le 1/2$. We denote, conditionally on $D_1$, 
\begin{align*}
    T_{n_2} = n_2^{-1/2} \widehat \sigma_{\theta, \widehat\theta_1}^{-1}\sum_{i \in I_2}\widehat\xi_i, \quad S_{n_2} =\sum_{i \in I_2}\widehat\xi_i, \quad \text{and} \quad V_{n_2}^2=\sum_{i \in I_2}\widehat\xi_i^2.
\end{align*}
By the algebraic identity (1.2) of \citet{gine1997student}, the events $\{T_{n_2} \ge z_\alpha\}$ and $\{S_{n_2}/V_{n_2}\ge z_{\alpha, n_2}\}$ coincide, so we have 
\begin{align*}
    &\mathbb{P}_{P^N}\left(\theta(P^N) \not\in \widehat{\mathrm{CI}}^{\mathtt{CLT}}_{N, \alpha}\right) = \mathbb{E}_{P^1}[\mathbb{P}_{P^2}(T_{n_2} \ge z_\alpha|D_1)]= \mathbb{E}_{P^1}\left[\mathbb{P}_{P^2}\left(\frac{S_{n_2}}{V_{n_2}} \ge z_{\alpha, n_2}\bigg|D_1\right)\right].
\end{align*}
The bound $R_1$ follows by the same argument as in \Cref{thm:studentized-nonIID}, using Theorem 2 of \citet{robinson2005self} when $R_* \le C/(1+z_{\alpha, n_2})^2$, for some universal constant $C$. 

For $\alpha \ge 1/2$, it follows that $z_\alpha \le 0$. Consider the following objects:
\begin{align*}
    \widetilde T_{n_2} = n_2^{-1/2} \widehat \sigma_{\theta, \widehat\theta_1}^{-1}\sum_{i \in I_2}(-\widehat\xi_i), \quad \widetilde S_{n_2} =\sum_{i \in I_2}(-\widehat\xi_i), \quad \text{and} \quad \widetilde V_{n_2}^2 = V_{n_2}^2=\sum_{i \in I_2}\widehat\xi_i^2.
\end{align*}
Then 
\begin{align*}
    \{S_{n_2} \ge z_{\alpha, n_2} V_{n_2} \} &= \{S_{n_2} \ge -z_{1-\alpha, n_2} V_{n_2} \} = \{\widetilde S_{n_2} \le z_{1-\alpha, n_2} \widetilde V_{n_2} \}.
\end{align*}
Using equation (1.6) of \citet{robinson2005self}, we have
\begin{align*}
    \mathbb{P}_{P^N}\left(\theta(P^N) \not\in \widehat{\mathrm{CI}}^{\mathtt{CLT}}_{N, \alpha}\right) &= 1-\mathbb{E}_{P^1}\left[\mathbb{P}_{P^2}\left(\frac{\widetilde S_{n_2}}{\widetilde V_{n_2}} \ge z_{1-\alpha, n_2}\bigg|D_1\right)\right]\\
    &\le 1-(1-\Phi(z_{1-\alpha, n_2}))\exp(-AR_*)\\
    &=\Phi(z_{1-\alpha, n_2})\exp(-AR_*)
\end{align*}
when $R_* \le (1+z_{1-\alpha, n_2})^2/A$ and $A > 0$ is a universal constant. Finally, using the identity,
\begin{align*}
    z_{1-\alpha, n_2} = -z_{\alpha, n_2} \quad \text{and} \quad \Phi(z_{1-\alpha, n_2}) = 1- \Phi(z_{\alpha, n_2}),
\end{align*}
we conclude the result.

It remains to establish $R_2$. We use Lemma 2.2 and Theorem 2.8 of \citet{gine1997student}, asserts that stochastic boundedness of the self-normalized statistics implies sub-Gaussianity. Denote 
\begin{align*}
    c(2) = \sqrt{2}\left(\frac{4e}{3}+1\right)^2.
\end{align*}
For any $a$, Lemma 2.2  
\begin{align*}
    \mathbb{E}_{P^2}\left[\left|\frac{S_{n_2}}{V_{n_2}}\right|^2\bigg|D_1\right] &\le c(2) \left(\mathbb{E}_{P^2}\left[\left|\frac{S_{n_2}}{V_{n_2}}\right|\bigg|D_1\right]\right)^2\\
    &\le c(2) \left\{a + \sqrt{\mathbb{E}_{P^2}\left[\left(\frac{S_{n_2}}{V_{n_2}}\right)^2\bigg|D_1\right]}\sqrt{\mathbb{P}_{P^2}\left(\left|\frac{S_{n_2}}{V_{n_2}}\right|\ge a\bigg|D_1\right)}\right\}^2\\
    &\le 2c(2) \left\{a^2 + \mathbb{E}_{P^2}\left[\left(\frac{S_{n_2}}{V_{n_2}}\right)^2\bigg|D_1\right]\mathbb{P}_{P^2}\left(\left|\frac{S_{n_2}}{V_{n_2}}\right|\ge a\bigg|D_1\right)\right\}\\
    &\le 2c(2) \left\{a^2 + \mathbb{E}_{P^2}\left[\left(\frac{S_{n_2}}{V_{n_2}}\right)^2\bigg|D_1\right]\left(\mathbb{P}_{P^2}(|Z|\ge a|D_1) + C'R_*\right)\right\}
\end{align*}
where the first inequality follows from the equation 2.8 of \citet{gine1997student}, the second by H\"{o}lder's inequality, and the last follows by Corollary 1.2 of \citet{bentkus1996berry}, which we leveraged in \Cref{cor:studentized-IID}. Here, $C' > 0$ is a universal constant. Choose $a$ such that 
\begin{align*}
    \mathbb{P}_{P^2}(|Z|\ge a|D_1) = \frac{1}{4c(2)},
\end{align*}
and fix $a$ as a universal constant. Rearranging, we obtain
\begin{align*}
    &\mathbb{E}_{P^2}\left[\left|\frac{S_{n_2}}{V_{n_2}}\right|^2\bigg|D_1\right]  \le 2c(2)a^2 + \frac{1}{2}\mathbb{E}_{P^2}\left[\left|\frac{S_{n_2}}{V_{n_2}}\right|^2\bigg|D_1\right] + 2c(2)C'R_*\mathbb{E}_{P^2}\left[\left|\frac{S_{n_2}}{V_{n_2}}\right|^2\bigg|D_1\right] \\
    &\quad \Leftrightarrow \left(\frac{1}{2} - 2c(2)C'R_*\right)_+\mathbb{E}_{P^2}\left[\left|\frac{S_{n_2}}{V_{n_2}}\right|^2\bigg|D_1\right]  \le 2c(2)a^2\\
    &\quad \Leftrightarrow \mathbb{E}_{P^2}\left[\left|\frac{S_{n_2}}{V_{n_2}}\right|^2\bigg|D_1\right]  \le \frac{4c(2)a^2}{(1 - 4c(2)C'R_*)_+} =: M^2.
\end{align*}
By Cauchy-Schwarz, $\mathbb{E}_{P^2}\left[\left|S_{n_2}/{V_{n_2}}\right||D_1\right] \le M$, and thus $S_{n_2}/V_{n_2}$ is stochastically bounded. Theorem 2.5 of \citet{gine1997student} then gives, $t > 0$
\begin{align*}
    \mathbb{E}_{P^2}\!\left[\exp\!\left(t\left|\frac{S_{n_2}}{V_{n_2}}\right|\right)\bigg|D_1\right] \le 2\exp\!\left(2^{-1/2}c(2) M^2 t^2\right).
\end{align*}
When $R_* \le C'' := 1/(4c(2)C'R^*)$, the denominator is bounded away from zero, and thus $M^2$ is bounded by a universal constant. We conclude that 
\begin{align*}
    \mathbb{E}_{P^1}\left[\mathbb{P}_{P^2}\!\left(\frac{S_{n_2}}{V_{n_2}} \ge z_{\alpha,n_2}\bigg|D_1\right)\right]
    \le \min\left\{1, C'\exp(-z_{\alpha,n_2}^2)\mathbb{P}_{P^1}(R_* \le C'')\right\},
\end{align*}
which concludes the result. 
\end{proof}
\begin{proof}[\bfseries{Proof of \Cref{thm:large-deviation}}]
    The result is a direct application of Theorem 1 of \citet{wang1996asymptotics} to 
    \begin{align*}
        &\mathbb{P}_{P^N}\left(\theta(P^N)\not\in \widehat{\mathrm{CI}}^{\mathtt{CLT}}_{N, \alpha} |D_1 \right) = \mathbb{P}_{P^N}\left(\sum_{i\in I_2} \widehat\xi_i \ge n_2^{1/2}z_\alpha \widehat \sigma_{\theta(P^N), \widehat\theta_1} + n_2\widehat{\mathbb{C}}_2 |D_1\right) .
    \end{align*}
\end{proof}
\section{Proofs from \Cref{sec:convergence-rates}}\label{sec:proofs-convergence-rates}
\subsection{Proof of \Cref{thm:hulc-width}}\label{sec:hulc-width}
Any element $\theta \in \Theta$ in the confidence set $\widehat{\mathrm{CI}}^\dagger_{N}$, defined as \eqref{eq:anti-conservative-confidence-set}, satisfies the following:
	\begin{align*}
		&\widehat\M_2(\theta) - \widehat\M_2(\widehat\theta_1)\le 0\\
        &\quad \Longleftrightarrow \M_2(\theta)-\M_2(\theta(P^N))\\
        &\quad\quad \le -\left((\widehat\M_2-\M_2)(\theta) -(\widehat\M_2-\M_2)(\theta(P^N))\right)+\widehat\M_2(\widehat\theta_1) -\widehat\M_2 (\theta(P^N))
	\end{align*}
Hence, the confidence set \eqref{eq:anti-conservative-confidence-set} is contained as 
\begin{align*}
    \widehat{\mathrm{CI}}^\dagger_{N} &:= \left\{\theta \in \Theta\, :\, \widehat\M_2(\theta) - \widehat\M_2(\widehat\theta_1)\le 0\right\} \\
    &\subseteq \left\{\theta \in \Theta\, :\, \M_2(\theta) - \M_2(\widehat\theta_1) \right.\\
    &\quad\le\left.-\left((\widehat\M_2-\M_2)(\theta) -(\widehat\M_2-\M_2)(\theta(P^N))\right)+\widehat\M_2(\widehat\theta_1) -\widehat\M_2 (\theta(P^N))\right\}\\
    &\subseteq \bigg\{\theta \in \Theta\, :\, c_0\|\theta - \theta(P^N)\|^{1+\gamma}\\
    &\quad\le\left.\left|(\widehat\M_2-\M_2)(\theta) -(\widehat\M_2-\M_2)(\theta(P^N))\right|+|\widehat\M_2(\widehat\theta_1) -\widehat\M_2 (\theta(P^N))|\right\} =:\overline{\mathrm{CI}}^\dagger_{N},
\end{align*}
where we used \ref{as:margin}. Given $R \ge 0$, we consider the partition of the parameter space $\Theta$ into 
\begin{align*}
    B := \left\{\theta \in \Theta\, :\,\|\theta - \theta(P^N)\| \le R\right\} \quad \textrm{and}\quad 
    B^c := \left\{\theta \in \Theta\, :\,\|\theta - \theta(P^N)\| > R\right\}.
\end{align*}
In our case we choose 
\begin{align}\label{eq:radius-def}
    R = 2^{M} c_0^{-1/(1+\gamma)}\left(r_{n_1}^{2/(1+\gamma)} +s_{n_1,n_2}^{1/(1+\gamma)}\right),
\end{align}
where $M$ will be specified later. 
The goal is to show that the confidence set $\overline{\mathrm{CI}}^\dagger_{N}$ is contained in $B$ with high probability, which implies $\widehat{\mathrm{CI}}^\dagger_{N}$ is also contained in $B$. It is then equivalent to show that $\overline{\mathrm{CI}}^\dagger_{N}$ intersects with $B^c$ with small probability. It then follows that 
\begin{align*}
    \mathbb{P}^*_{P^N}(\overline{\mathrm{CI}}^\dagger_{N} \cap B^c) &= \mathbb{E}_{\widetilde P^1}\left[\mathbb{P}^*_{P^2}(\overline{\mathrm{CI}}^\dagger_{N} \cap B^c )\right] + \beta(r),
\end{align*}
where the notation $\widetilde D_1$ comes from coupling arguments as in \Cref{thm:coverage-anti-conservative-confidence-set}. It remains to evaluate the conditional probability. We have 
\begin{align*}
    &\mathbb{P}^*_{P^2}(\overline{\mathrm{CI}}^\dagger_{N} \cap B^c) \\
    &\quad \le \mathbb{P}^*_{P^2}\left(c_0\|\theta - \theta(P^N)\|^{1+\gamma} \le 2\left|(\widehat\M_2-\M_2)(\theta) -(\widehat\M_2-\M_2)(\theta(P^N))\right| \cap B^c|\widetilde D_1\right)\\
    &\quad\quad + \mathbb{P}^*_{P^2}\left(c_0\|\theta - \theta(P^N)\|^{1+\gamma} \le 2|\widehat\M_2(\widehat\theta_1) -\widehat\M_2 (\theta(P^N))| \cap B^c|\widetilde D_1\right)= \mathbf{I} + \mathbf{II}.
\end{align*}

The second term can be controlled by \ref{as:rate-initial-estimator} and the choice \eqref{eq:radius-def}. Conditioning on the event where 
\begin{align*}
    \Omega_{\mathtt{init}} :=\left\{ \mathbb{E}_{P^2}[|\widehat\M_2(\widehat\theta_1) -\widehat\M_2 (\theta(P^N))| \widetilde D_1] \le C_{\mathtt{init}}s_{n_1, n_2}\right\},
\end{align*}
it follows 
\begin{align}
    \mathbf{II} &= \mathbb{P}^*_{P^2}\left(c_0\|\theta - \theta(P^N)\|^{1+\gamma} \le 2|\widehat\M_2(\widehat\theta_1) -\widehat\M_2 (\theta(P^N))| \cap B^c|\widetilde D_1\right) + \mathbb{P}_{\widetilde P^1}(\Omega^c_{\mathtt{init}})\nonumber\\ &
    \le2\cdot 2^{-M(1+\gamma)}C_{\mathtt{init}} + \varepsilon_{\mathtt{init}}\nonumber
\end{align}
by Markov inequality. Moving onto the term $\mathbf{I}$, we first observe that 
\begin{align*}
    \mathbf{I} &=\mathbb{P}^*_{P^2}\left(c_0\|\theta - \theta(P^N)\|^{1+\gamma} \le 2\left|(\widehat\M_2-\M_2)(\theta) -(\widehat\M_2-\M_2)(\theta(P^N))\right|\right.\\
    &\quad\quad\quad\bigg. \mbox{ for } \|\theta-\theta(P^N)\| \ge \mathrm{R} \bigg) \\
     &\le \mathbb{P}^*_{P^2}\left(c_0\|\theta - \theta(P^N)\|^{1+\gamma} \le 2\left|(\widehat\M_2-\M_2)(\theta) -(\widehat\M_2-\M_2)(\theta(P^N))\right|\right.\\
    &\quad\quad\quad\bigg. \mbox{ for } \|\theta-\theta(P^N)\| \ge 2^{M} c_0^{-1/(1+\gamma)}r_{n_2}^{2/(1+\gamma)}  \bigg).
\end{align*}
We define the ``shell": 
\[S_j = \{\theta \in \Theta : 2^{j}c_0^{-1/(1+\gamma)}r_{n_2}^{2/(1+\gamma)} \le \|\theta-\theta(P)\| < 2^{j+1}c_0^{-1/(1+\gamma)}r_{n_2}^{2/(1+\gamma)}\}\]
for each $j  \in \{0\} \cup \mathbb{N}$. It then follows that 
\begin{align*}
    &\mathbb{P}^*_{P^2}\left(c_0\|\theta - \theta(P^N)\|^{1+\gamma} \le 2\left|(\widehat\M_2-\M_2)(\theta) -(\widehat\M_2-\M_2)(\theta(P^N))\right|\right.\\
    &\quad\quad\quad\bigg. \mbox{ for } \|\theta-\theta(P^N)\| \ge 2^{M/(1+\gamma)} c_0^{-1/(1+\gamma)}r_{n_2}^{2/(1+\gamma)}  \bigg)\\
    &\quad= \mathbb{P}^*_{P^2}\bigg(\exists\, (j \ge M, \theta \in S_j) \,:\, c_0\|\theta - \theta(P^N)\|^{1+\gamma}\\
    &\quad\quad \quad\quad\le \left. 2\left|(\widehat\M_2-\M_2)(\theta) -(\widehat\M_2-\M_2)(\theta(P^N))\right| \right) \\
    &\quad \le \sum_{j=M}^\infty \mathbb{P}^*_{P^2}\bigg(\exists\,  \theta \in S_j \,:\,c_0\|\theta - \theta(P^N)\|^{1+\gamma}\le  2\left|(\widehat\M_2-\M_2)(\theta) -(\widehat\M_2-\M_2)(\theta(P^N))\right|\bigg) \\
    & \quad\le \sum_{j=M}^\infty \mathbb{P}^*_{P^2}\bigg(2^{j(1+\gamma)}r_{n_2}^{2}\le  2\sup_{\theta \in S_j}\left|(\widehat\M_2-\M_2)(\theta) -(\widehat\M_2-\M_2)(\theta(P^N))\right|\bigg)\\
    & \quad\le 2\sum_{j=M}^\infty 2^{-j(1+\gamma)}r_{n_2}^{-2}  \mathbb{E}^*_{P^2}\bigg[\sup_{\theta \in S_j}\left|(\widehat\M_2-\M_2)(\theta) -(\widehat\M_2-\M_2)(\theta(P^N))\right|\bigg]\\
    & \quad \le 2\sum_{j=M}^\infty 2^{-j(1+\gamma)}r_{n_2}^{-2} \phi_{n_2}(2^{j+1} c_0^{-1/(1+\gamma)}r_{n_2}^{2/(1+\gamma)})\\
    & \quad \le 2\sum_{j=M}^\infty 2^{-j(1+\gamma)}2^{q(j+1)}r_{n_2}^{-2} \phi_{n_2}( c_0^{-1/(1+\gamma)}r_{n_2}^{2/(1+\gamma)})\\
    & \quad \le 2\sum_{j=M}^\infty 2^{-j(1+\gamma)}2^{q(j+1)}.
\end{align*} 
Now the last term can be written as 
\begin{align*}
    2\sum_{j=M}^\infty 2^{-j(1+\gamma)}2^{q(j+1)}  = 2C_{q,\gamma} 2^{-M(1+\gamma-q)} \quad \textrm{where} \quad C_{q,\gamma} = \frac{2^q}{1-2^{q-(1+\gamma)}}.
\end{align*}
where $C_{q,\gamma}$ is a constant only depending on $q$ and $\gamma$. Putting together we have 
\begin{align*}
&\mathbb{P}^*_{P^N}\left(\mathrm{Diam}_{\|\cdot\|}\big(\widehat{\mathrm{CI}}^{\dagger}_{N}\big)> 2^{M}c_0^{-1/(1+\gamma)}(r_{n_2}^{2/(1+\gamma)}  + s_{n_1,n_2}^{1/(1+\gamma)})\right) \\
&\quad \le 2\cdot C_{\mathtt{init}}2^{-M(1+\gamma)}+ 2\cdot C_{q,\gamma} 2^{-M(1+\gamma-q)} + \varepsilon_{\mathtt{init}}  + \beta(r)\\
&\quad \lesssim_{q,\gamma, C_{\mathtt{init}}}  2^{-M(1+\gamma-q)} + \varepsilon_{\mathtt{init}}  + \beta(r).
\end{align*}
We conclude the result by choosing $M$ to be 
\begin{align*}
    M = \frac{\log (\mathfrak{C}/\varepsilon)}{(1+\gamma-q) \cdot \log 2} \quad\textrm{and} \quad 2^{M} = \left(\frac{\mathfrak{C}}{\varepsilon}\right)^{1/(1+\gamma-q)},
\end{align*}
where $\mathfrak{C}=2(C_{\mathtt{init}}+C_{q,\gamma})$.
\subsection{Proof of \Cref{thm:clt-width}}\label{sup:proof-clt-width}
Denoting $m_{\theta}(Z_i)-m_{ \widehat{\theta}_1}(Z_i)=(m_{\theta}-m_{ \widehat{\theta}_1})(Z_i)$, we observe that
\begin{align*}
    \widehat \sigma^2_{\theta, \widehat\theta_1} &= \frac{n_2}{n_2-1}\left(\frac{1}{n_2}\sum_{i \in I_2} \{(m_{\theta}-m_{ \widehat{\theta}_1})(Z_i)\}^2 - \left(\frac{1}{n_2}\sum_{j \in I_2} \{(m_{\theta}-m_{ \widehat{\theta}_1})(Z_i)\right)^2\right) \\
    &\le \frac{2}{n_2}\sum_{i \in I_2} \{(m_{\theta}-m_{ \widehat{\theta}_1})(Z_i)\}^2.
\end{align*}
Hence we have 
\begin{align*}
    \widehat t_{\alpha}(\theta, \widehat \theta_1) &= z_\alpha n_2^{-1/2} \sqrt{\frac{2}{n_2}\sum_{i \in I_2} \{(m_{\theta}-m_{ \widehat{\theta}_1})(Z_i)\}^2} \\
    &\le z_\alpha n_2^{-1/2} \sqrt{\frac{2}{n_2}\sum_{i \in I_2} 2\{(m_{\theta}-m_{\theta(P^N)}(Z_i)\}^2 + 2\{(m_{\theta(P^N)}-m_{ \widehat{\theta}_1})(Z_i)\}^2}\\
    &\le 2z_\alpha n_2^{-1/2} \sqrt{\left|\frac{1}{n_2}\sum_{i \in I_2} \{(m_{\theta}-m_{ \theta(P^N)})(Z_i)\}^2 - \frac{1}{n_2}\sum_{i\in I_2}\E_{P_i}[\{(m_{\theta}-m_{\theta(P^N)})(Z)\}^2]\right|} \\
    &\quad + 2z_\alpha n_2^{-1/2}\sqrt{\left| \frac{1}{n_2}\sum_{i\in I_2}\E_{P_i}[\{(m_{\theta}-m_{ \theta(P^N)})(Z_i)\}^2]\right|} \\
    &\quad + 2z_\alpha n_2^{-1/2}\sqrt{\frac{1}{n_2}\sum_{i \in I_2} \{(m_{\theta(P^N)}-m_{ \widehat{\theta}_1})(Z_i)\}^2} \\
    &= \mathfrak{R}_1 + \mathfrak{R}_2 + \mathfrak{R}_3.
\end{align*}
Using this expression, we have 
\begin{align*}
    \widehat{\mathrm{CI}}^{\mathtt{CLT}}_{N, \alpha} &:= \left\{\theta \in \Theta\, :\, \widehat\M_2(\theta) - \widehat\M_2(\widehat\theta_1)\le \widehat t_{\alpha}(\theta, \widehat \theta_1)\right\} \\
    &\subseteq \bigg\{\theta \in \Theta\, :\, c_0\|\theta - \theta(P^N)\|^{1+\gamma}\\
    &\quad\le\left.\left|(\widehat\M_2-\M_2)(\theta) -(\widehat\M_2-\M_2)(\theta(P^N))\right|+|\widehat\M_2(\widehat\theta_1) -\widehat\M_2 (\theta(P^N))|\right.\\
    &\quad \quad + \mathfrak{R}_1 + \mathfrak{R}_2 + \mathfrak{R}_3\bigg\} =: \overline{\mathrm{CI}}^{\mathtt{CLT}}_{N, \alpha}.
\end{align*}
We can now use the same logic as the proof of \Cref{thm:hulc-width}. We then have 
\begin{align*}
    &\mathbb{P}^*_{P^2|\widetilde P^1}(\overline{\mathrm{CI}}^{\mathtt{CLT}}_{N, \alpha}  \cap B^c) \\
    &\quad \le \mathbb{P}^*_{P^2|\widetilde P^1}\left(c_0\|\theta - \theta(P^N)\|^{1+\gamma} \le 5\left|(\widehat\M_2-\M_2)(\theta) -(\widehat\M_2-\M_2)(\theta(P^N))\right| \cap B^c\right)\\
    &\quad\quad + \mathbb{P}^*_{P^2|\widetilde P^1}\left(c_0\|\theta - \theta(P^N)\|^{1+\gamma} \le 5|\widehat\M_2(\widehat\theta_1) -\widehat\M_2 (\theta(P^N))| \cap B^c\right)\\
    &\quad\quad + \mathbb{P}^*_{P^2|\widetilde P^1}\left(c_0\|\theta - \theta(P^N)\|^{1+\gamma} \le 5\mathfrak{R}_1 \cap B^c\right)\\
    &\quad\quad + \mathbb{P}^*_{P^2|\widetilde P^1}\left(c_0\|\theta - \theta(P^N)\|^{1+\gamma} \le 5\mathfrak{R}_2 \cap B^c\right)\\
    &\quad\quad + \mathbb{P}^*_{P^2|\widetilde P^1}\left(c_0\|\theta - \theta(P^N)\|^{1+\gamma} \le 5\mathfrak{R}_3 \cap B^c\right)= \mathbf{I} + \mathbf{II} + \mathbf{III} + \mathbf{IV} + \mathbf{V}.
\end{align*}
The terms $\mathbf{I}$ is already controlled in the proof of \Cref{thm:hulc-width}, such that $\mathbf{I} \le 5 \cdot C_{q,\gamma}2^{-M(1+\gamma-q)}$. For $\mathbf{II}$, using Markov inequality after squaring both sides, we get
\begin{align*}
    \mathbf{II} &= \mathbb{P}^*_{P^2|\widetilde P^1}\left(c_0\|\theta - \theta(P^N)\|^{1+\gamma} \le 5|\widehat\M_2(\widehat\theta_1) -\widehat\M_2 (\theta(P^N))| \cap B^c\right)\nonumber\\ 
    & =\mathbb{P}^*_{P^2|\widetilde P^1}\left(c_0^2\|\theta - \theta(P^N)\|^{2+2\gamma} \le 25|\widehat\M_2(\widehat\theta_1) -\widehat\M_2 (\theta(P^N))|^2 \cap B^c\right).
\end{align*}
We observe that 
\begin{align*}
    &\mathbb{E}_{P^2|\widetilde P^1}[|\widehat\M_2(\widehat\theta_1) -\widehat\M_2 (\theta(P^N))|^2]\\
    &\quad =\mathbb{E}_{P^2|\widetilde P^1}\left[\left|\frac{1}{n_2}\sum_{i \in I_2} (m_{\theta(P^N)}-m_{ \widehat{\theta}_1})(Z_i)\right|^2\right] \\
    &\quad= \mathbb{E}_{P^2|\widetilde P^1}\left[\left|\frac{1}{n_2}\sum_{i \in I_2} \widehat \xi_i\ \right|^2\right] + \left|\frac{1}{n_2}\sum_{i \in I_2}  \mathbb{E}_{P_i}[(m_{\theta(P^N)}-m_{ \widehat{\theta}_1})(Z_i) |D_1] \right|^2
    \\
    &\quad= \frac{1}{n_2}\mathbb{E}_{P^2|\widetilde P^1}\left[\frac{1}{n_2}\sum_{i \in I_2} \widehat \xi_i^2\right] + \left|\frac{1}{n_2}\sum_{i \in I_2}  \mathbb{E}_{P_i}[(m_{\theta(P^N)}-m_{ \widehat{\theta}_1})(Z_i) |D_1] \right|^2 \\
    &\quad = \frac{1}{n_2}\mathbb{E}_{P^2|\widetilde P^1}\left[\frac{1}{n_2}\sum_{i \in I_2} \widehat \xi_i^2\right] + \widehat{\mathbb{C}}_2^2
\end{align*}
where we used the fact that $\mathbb{E}_{P^2}[\widehat \xi_i\widehat \xi_j | D_1]=0$ due to independence and $\mathbb{E}_{P^2}[\widehat \xi_i| D_1] = 0$. Hence, conditioning on the event 
\begin{align*}
    \widetilde\Omega_{\mathtt{init}} := \left\{ \frac{1}{n_2}\mathbb{E}_{P^2|\widetilde P^1}\left[\frac{1}{n_2}\sum_{i \in I_2} \widehat \xi_i^2\right] + \widehat{\mathbb{C}}_2^2 \le \widetilde C_{\mathtt{init}}\widetilde s_{n_1, n_2}\right\},
\end{align*}
we have 
\begin{align*}
    \mathbf{II} \le &\mathbb{P}^*_{P^2|\widetilde P^1}\left(c_0^2\|\theta - \theta(P^N)\|^{2+2\gamma} \le 25|\widehat\M_2(\widehat\theta_1) -\widehat\M_2 (\theta(P^N))|^2 \cap B^c\right)\\
    & \le 25  \cdot\widetilde C_{\mathtt{init}} 2^{-M(1+\gamma)} + \mathbb{P}(\widetilde\Omega_{\mathtt{init}}^c)\le 25  \cdot\widetilde C_{\mathtt{init}} 2^{-M(1+\gamma)} + \widetilde\varepsilon_{\mathtt{init}}.
\end{align*}
For  $\mathbf{III}$, we follow the same chain of logic as $\mathbf{I}$ in the proof of \Cref{thm:hulc-width}, and 
\begin{align*}
    \mathbf{III} &\le \mathbb{P}^*_{P^2|\widetilde P^1}\left(c_0\|\theta - \theta(P^N)\|^{1+\gamma} \le 5\mathfrak{R}_1\right.\\
    &\quad\quad\quad\quad\bigg. \mbox{ for } \|\theta-\theta(P^N)\| \ge 2^{M} c_0^{-1/(1+\gamma)}u_{n_2}^{2/(1+\gamma)}  \bigg)\\
    &=\mathbb{P}^*_{P^2|\widetilde P^1}\left(c_0^2\|\theta - \theta(P^N)\|^{2+2\gamma} \le 25\mathfrak{R}_1^2\right.\\
    &\quad\quad\quad\quad\bigg. \mbox{ for } \|\theta-\theta(P^N)\| \ge 2^{M} c_0^{-1/(1+\gamma)}u_{n_2}^{2/(1+\gamma)}  \bigg)\\
    &\le 25 \cdot 4z_\alpha^2 n_2^{-1} \sum_{j=M}^\infty 2^{-2j(1+\gamma)} u_{n_2}^{-4}\omega^2_{n_2}(2^{j+1}c_0^{-1/(1+\gamma)}u_{n_2}^{2/(1+\gamma)})\\
    &\le 25 \cdot 4z_\alpha^2 n_2^{-1} \sum_{j=M}^\infty 2^{-2j(1+\gamma)} 2^{2q(j+1)}u_{n_2}^{-4}\omega^2_{n_2}(c_0^{-1/(1+\gamma)}u_{n_2}^{2/(1+\gamma)})\\
    &\le 25 \cdot 4z_\alpha^2 \sum_{j=M}^\infty 2^{-2j(1+\gamma)} 2^{2q(j+1)} \le 100 \cdot z_\alpha^2  \cdot C_{q,\gamma}^2 2^{-2M(1+\gamma-q)}. 
\end{align*}
The bound for $\mathbf{IV}$ is analogous. Finally for $\mathbf{V}$, observe that
\begin{align*}
    &n_2^{-1} \mathbb{E}_{P^2|\widetilde P^1}\left[\frac{1}{n_2}\sum_{i \in I_2} \{(m_{\theta(P^N)}-m_{ \widehat{\theta}_1})(Z_i)\}^2\right] \\
    &\quad = \frac{1}{n_2} \mathbb{E}_{P^2|\widetilde P^1}\left[\frac{1}{n_2}\sum_{i \in I_2} \widehat \xi_i^2\right] + \frac{1}{n_2}\mathbb{E}_{P^2|\widetilde P^1}\left[\frac{1}{n_2}\sum_{i \in I_2} \{\E_{P_i}[(m_{\theta(P^N)}-m_{ \widehat{\theta}_1})(Z_i)|D_1]\}^2\right]\\
    &\quad = \frac{1}{n_2} \mathbb{E}_{P^2|\widetilde P^1}\left[\frac{1}{n_2}\sum_{i \in I_2} \widehat \xi_i^2\right] + \frac{1}{n_2}\mathbb{E}_{P^2|\widetilde P^1}\left[\frac{1}{n_2}\sum_{i \in I_2} \mathbb{C}^2_i(\widehat{\theta}_1)\right]\le \widetilde s_{n_1, n_2}^{2}.
\end{align*}
Hence, on the event $\widetilde\Omega_{\mathtt{init}}$ this term is controlled as $\mathbf{II}$. Putting together, we have 
\begin{align*}
&\mathbb{P}^*_{P^N}\left(\mathrm{Diam}_{\|\cdot\|}\big(\widehat{\mathrm{CI}}^{\mathtt{CLT}}_{N, \alpha}\big)> 2^{M}c_0^{-1/(1+\gamma)}(r_{n_2}^{2/(1+\gamma)}  +u_{n_2}^{2/(1+\gamma)} +{s'}_{n_1,n_2}^{1/(1+\gamma)})\right)  \\
&\quad \lesssim C_{q,\gamma} 2^{-M(1+\gamma-q)} + \widetilde C_{\mathtt{init}}2^{-M(1+\gamma)} +  z_\alpha^2 C^2_{q,\gamma} 2^{-2M(1+\gamma-q)} + \widetilde C_{\mathtt{init}}  z_\alpha^2 2^{-2M(1+\gamma)} + \widetilde \varepsilon_{\mathtt{init}}\\
&\quad \le \mathfrak{C} (1+2^{-M(1+\gamma-q)}z_\alpha^2) 2^{-M(1+\gamma-q)} + \widetilde \varepsilon_{\mathtt{init}},
\end{align*}
where $\mathfrak{C}$ is a constant depending on $q,\gamma,\widetilde C_{\mathtt{init}}$. We conclude the claim by choosing $M$ to be 
\begin{align*}
    M = \frac{\log ((1+|z_\alpha|)\mathfrak{C}/\varepsilon)}{(1+\gamma-q) \cdot \log 2} \quad\textrm{and} \quad 2^{M} = \left(\frac{\mathfrak{C}(1+|z_\alpha|)}{\varepsilon}\right)^{1/(1+\gamma-q)}.
\end{align*}
% \begin{align*}
%     &\mathbb{P}\left(c_0\|\theta-\theta(P^N)\|^{1+\gamma} \le z_{\alpha}n_2^{-1/2}\sqrt{\frac{2}{n_2}\sum_{i \in I_2} \{(m_{\theta}-m_{ \widehat{\theta}_1})(Z_i)\}^2}\right) \\
%     &= \mathbb{P}\left(c_0^2\|\theta-\theta(P^N)\|^{2+2\gamma} \le z^2_{\alpha}n_2^{-1}\left(\frac{2}{n_2}\sum_{i \in I_2} \{(m_{\theta}-m_{ \widehat{\theta}_1})(Z_i)\}^2\right)\right)\\
%     % &\le \mathbb{P}\left(c_0^2\|\theta-\theta(P^N)\|^{2+2\gamma} \le z^2_{\alpha}n_2^{-1}\left(\frac{4}{n_2}\sum_{i \in I_2} \{(m_{\theta}-m_{\theta(P^N)})(Z_i)\}^2 + \frac{4}{n_2}\sum_{i \in I_2} \{(m_{\theta(P^N)}-m_{ \widehat{\theta}_1})(Z_i)\}^2\right)\right)\\
%     &\le \mathbb{P}\left(c_0^2\|\theta-\theta(P^N)\|^{2+2\gamma} \le z^2_{\alpha}n_2^{-1}\left(\frac{8}{n_2}\sum_{i \in I_2} \{(m_{\theta}-m_{\theta(P^N)})(Z_i)\}^2\right)\right)\\
%     &\quad + \mathbb{P}\left(c_0^2\|\theta-\theta(P^N)\|^{2+2\gamma} \le z^2_{\alpha}n_2^{-1}\left(\frac{8}{n_2}\sum_{i \in I_2} \{(m_{\theta(P^N)}-m_{ \widehat{\theta}_1})(Z_i)\}^2\right)\right)
% \end{align*}

\subsection{Proof of \Cref{thm:hulc-width-local}}\label{suppsec:hulc-width-local-proof}
\Cref{thm:hulc-width-local} is a simpler version of the following non-asymptotic bound.
\begin{theorem}\label{thm:hulc-width-local-full}
Assume \ref{as:margin} and \ref{as:local-entropy} hold for all $\|\theta - \theta(P^N)\| \le \rho$, and that \ref{as:rate-initial-estimator}, \ref{as:margin-global} and \ref{as:ratio-process} hold. Then for $n_2 \ge N_2$ where $N_2$ depends on $\varepsilon_{\mathtt{ratio}}$ and $\rho$, and for any $\varepsilon > 0$
\begin{align*}
\mathbb{P}^*_{P^N}\left(\mathrm{Diam}_{\|\cdot\|}\big(\widehat{\mathrm{CI}}^{\dagger}_{N}\big)\le C\max\{\mathrm{R}^{\dagger}_{N}, \mathrm{Q}^{\dagger}_{N}\mathbf{1}\{\mathrm{Q}^{\dagger}_{N} \ge \rho\}\}\right) \ge 1-\varepsilon^\circ-\beta(r), 
\end{align*}
where $\varepsilon^\circ = \varepsilon + \varepsilon_{\mathtt{init}} + \varepsilon_{\mathtt{ratio}}$,
\begin{align*}
    \mathrm{R}^{\dagger}_{N}:= c_0^{-1/(1+\gamma)}(r_{n_2}^{2/(1+\gamma)}  + s_{n_1, n_2}^{1/(1+\gamma)}) ,\quad\mathrm{Q}^{\dagger}_{N}:= C_\rho
^{-1}\left(s_{n_1, n_2}\right),
\end{align*}
and $C$ is a constant depending on $\gamma, q, C_{\mathtt{init}}, C_{\mathtt{ratio}}, g(\cdot)$ and $\varepsilon^\circ$. 
\end{theorem}
\begin{proof}[\bfseries{Proof of \Cref{thm:hulc-width-local-full}}]
    For $\rho > 0$, we define the partition of the parameter space as
\begin{align*}
    \Theta_{\rho} := \{\theta \in \Theta\, :\, \|\theta-\theta(P^N)\|\le \rho \}\quad \textrm{and} \quad \Theta_{\rho}^c :=  \Theta \setminus \Theta_{\rho}.
\end{align*}
The proof will proceed by analyzing the two disjoint partition of the confidence set:
\begin{align*}
    \widehat{\mathrm{CI}}^\dagger_{N} = (\widehat{\mathrm{CI}}^\dagger_{N} \,\cap \, \Theta_{\rho}) \,\cup\, (\widehat{\mathrm{CI}}^\dagger_{N} \,\cap\,\Theta_{\rho}^c). 
\end{align*}
The first set can be analyzed as the proof of \Cref{thm:hulc-width} assuming \ref{as:margin} and \ref{as:local-entropy}, but their respective requirements holding only on $\Theta_{\rho}$. This result concludes that $\widehat{\mathrm{CI}}^\dagger_{N} \,\cap \, \Theta_{\rho}$ is contained in the ball with radius $C_\varepsilon\mathrm{R}_N^\dagger$ with probability greater than $1-\varepsilon$.

We now turn our attention to $\widehat{\mathrm{CI}}^\dagger_{N} \,\cap\,\Theta_{\rho}^c$. Any element $\theta \in \Theta_\rho^c$ in the confidence set $\widehat{\mathrm{CI}}_{N}^\dagger$ satisfies the following:
    \begin{align*}
    &\widehat\M_2(\theta) - \widehat\M_2(\widehat\theta_1)\le 0\\
    &\quad \Longleftrightarrow \M_2(\theta)-\M_2(\theta(P^N))+\left((\widehat\M_2-\M_2)(\theta) -(\widehat\M_2-\M_2)(\theta(P^N))\right)\\
    &\quad\quad \le \widehat\M_2(\widehat\theta_1) -\widehat\M_2 (\theta(P^N))\\
    &\quad \Longleftrightarrow \M_2(\theta)-\M_2(\theta(P^N))\left(1+\frac{(\widehat\M_2-\M_2)(\theta) -(\widehat\M_2-\M_2)(\theta(P^N))}{\M_2(\theta)-\M_2(\theta(P^N))}\right)\\
    &\quad\quad \le \widehat\M_2(\widehat\theta_1) -\widehat\M_2 (\theta(P^N))\\
    &\quad \Longrightarrow C_\rho(\|\theta-\theta(P^N)\|)\left(1-\sup_{\|\theta-\theta(P^N)\| > \rho}\left|\frac{(\widehat\M_2-\M_2)(\theta) -(\widehat\M_2-\M_2)(\theta(P^N))}{\M_2(\theta)-\M_2(\theta(P^N))}\right|\right)_+\\
    &\quad\quad \le |\widehat\M_2(\widehat\theta_1) -\widehat\M_2 (\theta(P^N))|,
\end{align*}
where we used \ref{as:margin-global} in the last step. We define the event
\begin{align*}
    \Omega_{\rho} := \left\{\sup_{\|\theta-\theta(P^N)\| > \rho}\left|\frac{(\widehat\M_2-\M_2)(\theta) -(\widehat\M_2-\M_2)(\theta(P^N))}{\M_2(\theta)-\M_2(\theta(P^N))}\right| \le 1/2\right\}.
\end{align*}
Under \ref{as:ratio-process} for fixed $\varepsilon_{\mathtt{ratio}} > 0$, we choose $N_2$ large enough such that for all $n_2 \ge N_2$, $C_{\mathtt{ratio}} R(n_2, \rho) \le 1/2$. Such $N_2$ exists in view of the limiting nature of $R(n_2, \rho)$. With such a choice of $N_2$, we have $\mathbb{P}_{P^2}( \Omega_{\rho}^c) \le \varepsilon_{\mathtt{ratio}}$. On the event $\Omega_{\rho}$, the confidence set satisfies the inclusion:
\begin{align*}
    \widehat{\mathrm{CI}}_{N}^\dagger \cap \Theta_{\rho}^c\subseteq \left\{\theta \in \Theta_{\rho}^c\, :\,C_\rho(\|\theta-\theta(P^N)\|)\le 2|\widehat\M_2(\widehat\theta_1) -\widehat\M_2 (\theta(P^N))| \right\}
\end{align*}
with probability greater than $1-\varepsilon_{\mathtt{ratio}}$. Furthermore, by Markov inequality, we have
\begin{align*}
    \mathbb{P}_{P^2}\left(|\widehat\M_2(\widehat\theta_1) -\widehat\M_2 (\theta(P^N))| \ge \frac{\E_{P^2}[|\widehat\M_2(\widehat\theta_1) -\widehat\M_2 (\theta(P^N))||D_1]}{\varepsilon} \bigg|D_1\right) \le \varepsilon.
\end{align*}
Hence, $|\widehat\M_2(\widehat\theta_1) -\widehat\M_2 (\theta(P^N))| \le \varepsilon^{-1}\E_{P^2}[|\widehat\M_2(\widehat\theta_1) -\widehat\M_2 (\theta(P^N))|]$ in probability greater than $1-\varepsilon$. Note that this is further bounded by $C_{\mathtt{init}}s_{n_1, n_2}$ in view of \ref{as:rate-initial-estimator}. By the fact that $\theta(P^N)$ is a unique solution, $C_\rho$ is an increasing function. Hence, 
\begin{align}\label{eq:ci-rho-complement-ball} 
    \widehat{\mathrm{CI}}_{N}^\dagger \cap \Theta_{\rho}^c\subseteq \left\{\theta \in \Theta_{\rho}^c\, :\,\|\theta-\theta(P^N)\|\le C_\rho^{-1}\big((2/\varepsilon)C_{\mathtt{init}}s_{n_1, n_2}\big) \right\},
\end{align}
with probability greater than $1-\varepsilon-\varepsilon_{\mathtt{init}}$. When the upper bound in \eqref{eq:ci-rho-complement-ball} becomes smaller than $\rho$, the right-hand side of the set inclusion becomes an empty set. Hence, we can safely replace the upper bound with 
\begin{align*}
    \mathrm{Q}_N^\dagger \mathbf{1}\{\mathrm{Q}_N^\dagger \ge \rho\} \quad \text{where} \quad \mathrm{Q}_N^\dagger = C_\rho^{-1}((2/\varepsilon)C_{\mathtt{init}}s_{n_1, n_2}).
\end{align*}
Putting together, we have 
\begin{align*}
     \widehat{\mathrm{CI}}_{N}^\dagger \cap \Theta_{\rho}&\subseteq \left\{\theta \in \Theta\, :\,\|\theta-\theta(P^N)\| \le \min(C_{\varepsilon_1}\mathrm{R}_{N}^\dagger, \rho) \right\} \quad \text{with prob. grt. than} \quad 1-\varepsilon_1\\
     \widehat{\mathrm{CI}}_{N}^\dagger \cap \Theta_{\rho}^c &\subseteq \left\{\theta \in \Theta\, :\,\|\theta-\theta(P^N)\| \le \mathrm{Q}_N^\dagger \mathbf{1}\{\mathrm{Q}_N^\dagger \ge \rho\} \right\} \quad \text{with prob. grt. than} \quad 1-\varepsilon_2.
\end{align*}
Hence we conclude that 
\begin{align*}
     \widehat{\mathrm{CI}}_{N}^\dagger &\subseteq \left\{\theta \in \Theta\, :\,\|\theta-\theta(P^N)\| \le C_{\varepsilon_1, \varepsilon_2}\max(\mathrm{R}_{N}^\dagger, \mathrm{Q}_N^\dagger \mathbf{1}\{\mathrm{Q}_N^\dagger \ge \rho\})\right\}
\end{align*}
with probability greater than $1-\varepsilon_1 - \varepsilon_2$ where the constant $C_{\varepsilon_1, \varepsilon_2}$ depends on $\varepsilon_1, \varepsilon_2$.
\end{proof}

\subsection{Proof of \Cref{thm:clt-width-local}}\label{suppsec:clt-width-local-proof}
\Cref{thm:clt-width-local} is a simpler version of the following non-asymptotic bound.
\begin{theorem}\label{thm:clt-width-local-full}
    Assume $Z_1,\ldots, Z_N$ are independent, \ref{as:margin}, \ref{as:local-entropy}, \ref{as:square-process} hold for all $\|\theta - \theta(P^N)\| \le \rho$ and that \ref{as:rate-initial-estimator3},  \ref{as:margin-global}, \ref{as:ratio-process}, and \ref{as:ratio-square-process} hold.  Then for $n_2 \ge N_2$ where $N_2$ depends on $\varepsilon_{\mathtt{ratio}},  \varepsilon_{\mathtt{emp}}$ and $\rho$, and for any $\varepsilon > 0$
\begin{align*}
\mathbb{P}^*_{P^N}\left(\mathrm{Diam}_{\|\cdot\|}\big(\widehat{\mathrm{CI}}^{\mathtt{CLT}}_{N, \alpha}\big)\le C\max\{(1+|z_\alpha|)^{1/(1+\gamma-q)}\mathrm{R}^{\mathtt{CLT}}_{N}, \mathrm{Q}^{\mathtt{CLT}}_{N,\alpha}\mathbf{1}\{\mathrm{Q}^{\mathtt{CLT}}_{N,\alpha} \ge \rho\}\}\right) \ge 1-\varepsilon^\circ, 
\end{align*}
where $\varepsilon^\circ = \varepsilon + \widetilde\varepsilon_{\mathtt{init}} + \varepsilon_{\mathtt{ratio}} + \varepsilon_{\mathtt{emp}}$,
\begin{align*}
    \mathrm{R}_N^{\mathtt{CLT}} = c_0^{-1/(1+\gamma)}(r_{n_2}^{2/(1+\gamma)}  + u_{n_2}^{2/(1+\gamma)}+\widetilde s_{n_1, n_2}^{1/(1+\gamma)}),\quad\mathrm{Q}^{\mathtt{CLT}}_{N,\alpha}:= C_\rho
^{-1}\left((1+|z_\alpha|)\widetilde s_{n_1, n_2}\right),
\end{align*}
and $C$ is a constant depending on $\gamma, q, \widetilde C_{\mathtt{init}}, C_{\mathtt{ratio}}, \widetilde C_{\mathtt{emp}}, g(\cdot)$ and $\varepsilon^\circ$. 
\end{theorem}
\begin{proof}[\bfseries{Proof of \Cref{thm:clt-width-local-full}}]
  The general proof is analogous to that of \Cref{thm:hulc-width-local-full}. In the proof of \Cref{thm:clt-width}, we have established that 
    \begin{align*}
    \widehat t_{\alpha}(\theta, \widehat \theta_1) &\le 2z_\alpha n_2^{-1/2} \sqrt{\left|\frac{1}{n_2}\sum_{i \in I_2} \{(m_{\theta}-m_{ \theta(P^N)})(Z_i)\}^2 - \frac{1}{n_2}\sum_{i\in I_2}\E_{P_i}[\{(m_{\theta}-m_{\theta(P^N)})(Z)\}^2]\right|} \\
    &\quad + 2z_\alpha n_2^{-1/2}\sqrt{\left| \frac{1}{n_2}\sum_{i\in I_2}\E_{P_i}[\{(m_{\theta}-m_{ \theta(P^N)})(Z_i)\}^2]\right|} \\
    &\quad + 2z_\alpha n_2^{-1/2}\sqrt{\frac{1}{n_2}\sum_{i \in I_2} \{(m_{\theta(P^N)}-m_{ \widehat{\theta}_1})(Z_i)\}^2}.
\end{align*}
Any element $\theta \in \Theta_\rho^c$ in the confidence set $\widehat{\mathrm{CI}}_{N,\alpha}^{\mathtt{CLT}}$ satisfies the following:
    \begin{align*}
    &\widehat\M_2(\theta) - \widehat\M_2(\widehat\theta_1)\le \widehat t_{\alpha}(\theta, \widehat \theta_1)\\
    &\quad \Longrightarrow \left\{\M_2(\theta)-\M_2(\theta(P^N))\right\}\left(1-\mathfrak{R}_{\rho, 1} -\mathfrak{R}_{\rho, 2}-\mathfrak{R}_{\rho, 3}  \right)_+\\
    &\quad\quad \le |\widehat\M_2(\widehat\theta_1) -\widehat\M_2 (\theta(P^N))| + 2z_\alpha n_2^{-1/2}\sqrt{\frac{1}{n_2}\sum_{i \in I_2} \{(m_{\theta(P^N)}-m_{ \widehat{\theta}_1})(Z_i)\}^2},
\end{align*}
where 
\begin{align*}
    \mathfrak{R}_{\rho, 1} &= \sup_{\|\theta-\theta(P^N)\| > \rho}\left|\frac{(\widehat\M_2-\M_2)(\theta) -(\widehat\M_2-\M_2)(\theta(P^N))}{\M_2(\theta)-\M_2(\theta(P^N))}\right| \\
    \mathfrak{R}_{\rho, 2} &= 2\sup_{\|\theta-\theta(P^N)\| > \rho}\sqrt{\left|\frac{z_\alpha ^2\sum_{i\in I_2}(m_\theta - m_{\theta(P^N)})^2 - \mathbb{E}_{P_i}[(m_\theta - m_{\theta(P^N)})^2]}{n_2^2\{\M_2(\theta)-\M_2(\theta(P^N))\}^2}\right| },\quad \textrm{and}\\
    \mathfrak{R}_{\rho, 3} &=2\sup_{\|\theta-\theta(P^N)\| > \rho}\sqrt{\frac{z_\alpha^2\sum_{i\in I_2}\mathbb{E}_{P_i}[(m_\theta - m_{\theta(P^N)})^2(Z_i)]}{n_2^2\{\M_2(\theta)-\M_2(\theta(P^N))\}^2}}.
\end{align*}
We define the events
\begin{align*}
    \Omega_{\rho,1} &:= \left\{\sup_{\|\theta-\theta(P^N)\| > \rho}\left|\frac{(\widehat\M_2-\M_2)(\theta) -(\widehat\M_2-\M_2)(\theta(P^N))}{\M_2(\theta)-\M_2(\theta(P^N))}\right| \le 1/6\right\}, \quad \textrm{and}\\
    \Omega_{\rho,2} &:= \left\{2\sup_{\|\theta-\theta(P^N)\| > \rho}\sqrt{\left|\frac{z_\alpha^2\sum_{i\in I_2}(m_\theta - m_{\theta(P^N)})^2 - \mathbb{E}_{P_i}[(m_\theta - m_{\theta(P^N)})^2]}{n_2^2\{\M_2(\theta)-\M_2(\theta(P^N))\}^2}\right| } \le 1/6\right\}.
\end{align*}
Under \ref{as:ratio-process} and \ref{as:ratio-square-process}, for fixed $\varepsilon_{\mathtt{ratio}}+\varepsilon_{\mathtt{emp}} > 0$, we choose $N_2$ large enough such that for all $n_2 \ge N_2$ 
\begin{align*}
    C_{\mathtt{ratio}}R(n_2, \rho) \le \frac{1}{6}, \quad C_{\mathtt{emp}}S_{\mathtt{emp}}(n_2, \rho, \alpha) \le \frac{1}{(2\cdot 6)^2} \quad \text{and} \quad S_{\mathtt{pop}}(n_2, \rho, \alpha) \le \frac{1}{(2\cdot 6)^2}.
\end{align*}
Such $N_2$ exists by the fact that $R(n_2, \rho)$, $S_{\mathtt{emp}}(n_2, \rho, \alpha)$ and $S_{\mathtt{pop}}(n_2, \rho, \alpha)$ all tend to zero. With such choice of $n_2$, we have 
\begin{align*}
    \left(1-\mathfrak{R}_{\rho, 1} -\mathfrak{R}_{\rho, 2}-\mathfrak{R}_{\rho, 3}  \right)_+ \ge 1/2
\end{align*}
with probability greater than $1-\varepsilon_{\mathtt{ratio}}-\varepsilon_{\mathtt{emp}}$.  Furthermore, we have established in the proof of \Cref{thm:clt-width} that under \ref{as:rate-initial-estimator3}, 
\begin{align*}
    |\widehat\M_2(\widehat\theta_1) -\widehat\M_2 (\theta(P^N))|^2 + 2z_\alpha^2 \left(\frac{1}{n_2^2}\sum_{i \in I_2} \{(m_{\theta(P^N)}-m_{ \widehat{\theta}_1})(Z_i)\}^2\right)^2 \lesssim \widetilde C_{\mathtt{init}}(1+z_\alpha^2)\widetilde s_{n_1, n_2}^2,
\end{align*}
with probability greater than $1-\widetilde\varepsilon_{\mathtt{init}}$. The results thus far imply that with probability greater than $1-\varepsilon_{\mathtt{ratio}}-\varepsilon_{\mathtt{emp}}-\widetilde\varepsilon_{\mathtt{init}}$,
\begin{align*}
    &\left\{2^{-1}\M_2(\theta)-\M_2(\theta(P^N))\right\}^2 \lesssim C_{\mathtt{init}}(1+z_\alpha^2)\widetilde s_{n_1, n_2}^2\\
    &\quad \Rightarrow  \left\{2^{-1}\M_2(\theta)-\M_2(\theta(P^N))\right\} \lesssim \sqrt{C_{\mathtt{init}}}(1+|z_\alpha|)\widetilde s_{n_1, n_2}\\
    &\quad \Rightarrow  C_\rho(\|\theta - \theta(P^N)\|) \lesssim \sqrt{C_{\mathtt{init}}}(1+|z_\alpha|)\widetilde s_{n_1, n_2}\\
    &\quad \Rightarrow  \|\theta - \theta(P^N)\| \lesssim C^{-1}_\rho(\sqrt{C_{\mathtt{init}}}(1+|z_\alpha|)\widetilde s_{n_1, n_2})\\
    &\quad \Rightarrow  \|\theta - \theta(P^N)\| \lesssim g(\sqrt{C_{\mathtt{init}}})C^{-1}_\rho((1+|z_\alpha|)\widetilde s_{n_1, n_2}).
\end{align*}
This concludes the claim.
\end{proof}

\subsection{Additional Results on Convergence Rates}\label{suppsec:addition-diameter}
A diameter bound is derived for the confidence set based on the lower confidence bound construction \eqref{eq:CI-without-upper-bound}. Validity of this confidence set in its general form was established under no structural assumptions on the optimization problem in \Cref{thm:general_LCB}. The following two additional conditions are needed.
\begin{enumerate}[label=\textbf{(C\arabic*)},leftmargin=2cm]
\item \label{as:general-t-entropy}There exists a function $\psi_{n_2} : \mathbb{R}_+ \mapsto \mathbb{R}_+$ such that 
    \begin{align}
    	\sup_{\eta \in \Theta}\, \E^*_{P^2} \left[\sup_{\|\theta-\theta(P^N)\| < \delta}|\widehat t_{\alpha}(\theta, \eta)-\widehat t_{\alpha}(\theta(P^N), \eta)\right] \le  \psi_{n_2}(\delta)\label{eq:general-t-modulus}
    \end{align}
    for every $n_2\ge 1$ and $\delta > 0$, and $\psi_{n_2}(x)/x^q$ is assumed non-increasing for some $q < 1+\gamma$.
    \item \label{as:rate-initial-estimator2} 
    For every $n_1, n_2 \ge 1$ and $\varepsilon'_{\mathtt{init}} > 0$, the initial estimator based on $D_1$ satisfies
    \begin{align}
        \mathbb{P}_{P^1}\!\left(
    \mathbb{E}_{P^2}[
    |\widehat t_{\alpha}(\theta(P^N), \widehat \theta_1)| \,|\, D_1
    ] \geq C'_{\mathtt{init}} \cdot s'_{n_1, n_2}
    \right) \leq \varepsilon'_{\mathtt{init}},
    \end{align}
    where $s'_{n_1, n_2}, C'_{\mathtt{init}}$ are non-negative constants. 
\end{enumerate}
Condition \ref{as:general-t-entropy} is the counterpart of \ref{as:local-entropy}, stated for the modulus for $\widehat t_{\alpha}(\theta, \eta)-\widehat t_{\alpha}(\theta(P^N), \eta)$ uniformly over $\eta \in \Theta$. Condition \ref{as:rate-initial-estimator2} essentially quantifies the rate of convergence of $\widehat\theta_1$.
\begin{theorem}\label{thm:LCB-general-width}Assume $\theta(P^N)$ is the unique solution of \eqref{eq:def-m-functional} that satisfies \ref{as:margin}--\ref{as:rate-initial-estimator2}. Define $u_{n_2}$ as any value that satisfies
\begin{align}
    u_{n_2}^{-2} \psi_{n_2}(c_0^{-1/(1+\gamma)}u_{n_2}^{2/(1+\gamma)}) \le 1,
\end{align}
and define $r_{n_2}$ as in \eqref{rq:define-rn}. 
Then, for any $n_1, n_2\ge 1$ and $\varepsilon > 0$, 
\begin{align*}
&\mathbb{P}^*_{P^N}\left(\mathrm{Diam}_{\|\cdot\|}\big(\widehat{\mathrm{CI}}^{\mathtt{LCB}}_{N, \alpha}\big)\le C\varepsilon^{-1/(1+\gamma-q)}\mathrm{R}_N^{\mathtt{LCB}})\right)\ge 1-\varepsilon-\varepsilon_{\mathtt{init}}-\varepsilon'_{\mathtt{init}}-\beta(r),
\end{align*}
where 
\begin{align}
    \mathrm{R}_N^{\mathtt{LCB}} = c_0^{-1/(1+\gamma)}(r_{n_2}^{2/(1+\gamma)} +u_{n_2}^{2/(1+\gamma)} + s_{n_1, n_2}^{1/(1+\gamma)}+{s'}_{n_1, n_2}^{1/(1+\gamma)}),\nonumber
\end{align}
and $C$ is a constant depending on $\gamma, q, C_{\mathtt{init}}$, and $C'_{\mathtt{init}}$.  
\end{theorem}
The proof is structurally identical to that of \Cref{thm:hulc-width} and we provide a general argument while highlighting the difference from \Cref{thm:hulc-width}. As with \Cref{thm:hulc-width} this result is stated for general stochastic optimization problems and  imposes no particular form on $\widehat t_{\alpha}(\cdot, \cdot)$, provided \ref{as:general-t-entropy} and \ref{as:rate-initial-estimator2} can be verified.
\begin{proof}[\bfseries{Proof of \Cref{thm:LCB-general-width}}]
    The proof is structurally identical to that of \Cref{thm:hulc-width}. As such, we only highlight the differences from \Cref{thm:hulc-width}. First, we define the superset of $\widehat{\mathrm{CI}}^{\mathtt{LCB}}_{N, \alpha}$ as follows:
\begin{align*}
    \widehat{\mathrm{CI}}^{\mathtt{LCB}}_{N, \alpha} &:= \left\{\theta \in \Theta\, :\, \widehat\M_2(\theta) - \widehat\M_2(\widehat\theta_1)\le \widehat t_{\alpha}(\theta, \widehat \theta_1)\right\} \\
    &\subseteq \bigg\{\theta \in \Theta\, :\, c_0\|\theta - \theta(P^N)\|^{1+\gamma}\\
    &\quad\le\left.\left|(\widehat\M_2-\M_2)(\theta) -(\widehat\M_2-\M_2)(\theta(P^N))\right|+|\widehat\M_2(\widehat\theta_1) -\widehat\M_2 (\theta(P^N))|\right.\\
    &\quad \quad + |\widehat t_{\alpha}(\theta, \widehat \theta_1)-\widehat t_{\alpha}(\theta(P^N), \widehat \theta_1)|+|\widehat t_{\alpha}(\theta(P^N), \widehat \theta_1)|\bigg\} =: \overline{\mathrm{CI}}^{\mathtt{LCB}}_{N, \alpha} 
\end{align*}
We then define $B$ and $B^c$ but now with 
\begin{align}
    R = 2^{M} c_0^{-1/(1+\gamma)}\left(r_{n_1}^{2/(1+\gamma)} +r_{n_1}^{2/(1+\gamma)} +s_{n_1,n_2}^{1/(1+\gamma)} + {s'}_{n_1,n_2}^{1/(1+\gamma)}\right)\nonumber.
\end{align}
We then have 
\begin{align*}
    &\mathbb{P}^*_{P^2|\widetilde P^1}(\overline{\mathrm{CI}}^{\mathtt{LCB}}_{N, \alpha}  \cap B^c) \\
    &\quad \le \mathbb{P}^*_{P^2|\widetilde P^1}\left(c_0\|\theta - \theta(P^N)\|^{1+\gamma} \le 4\left|(\widehat\M_2-\M_2)(\theta) -(\widehat\M_2-\M_2)(\theta(P^N))\right| \cap B^c\right)\\
    &\quad\quad + \mathbb{P}^*_{P^2|\widetilde P^1}\left(c_0\|\theta - \theta(P^N)\|^{1+\gamma} \le 4|\widehat\M_2(\widehat\theta_1) -\widehat\M_2 (\theta(P^N))| \cap B^c\right)\\
    &\quad\quad + \mathbb{P}^*_{P^2|\widetilde P^1}\left(c_0\|\theta - \theta(P^N)\|^{1+\gamma} \le 4|\widehat t_{\alpha}(\theta, \widehat \theta_1)-\widehat t_{\alpha}(\theta(P^N), \widehat \theta_1)| \cap B^c\right)\\
    &\quad\quad + \mathbb{P}^*_{P^2|\widetilde P^1}\left(c_0\|\theta - \theta(P^N)\|^{1+\gamma} \le 4|\widehat t_{\alpha}(\theta(P^N), \widehat \theta_1)| \cap B^c\right)= \mathbf{I} + \mathbf{II} + \mathbf{III} + \mathbf{IV}.
\end{align*}
The first two terms are already controlled in the proof of \Cref{thm:hulc-width}, and we have 
\begin{align*}
    \mathbf{I} \le \frac{4C_{q,\gamma}}{2^{M(1+\gamma-q)}}\quad\textrm{and} \quad \mathbf{II} \le \frac{4C_{\mathtt{init}}}{2^{M(1+\gamma)}} + \varepsilon_{\mathtt{init}}.
\end{align*}
The other two follow analogously. Following the peeling and Markov inequality as in 
\begin{align*}
    \mathbf{III}& \le \mathbb{P}^*_{P^2|\widetilde P^1}\left(c_0\|\theta - \theta(P^N)\|^{1+\gamma} \le 4\left|\widehat t_{\alpha}(\theta, \widehat \theta_1)-\widehat t_{\alpha}(\theta(P^N), \widehat \theta_1)\right|\right.\\
    &\quad\quad\quad\quad\bigg. \mbox{ for } \|\theta-\theta(P^N)\| \ge 2^{M} c_0^{-1/(1+\gamma)}u_{n_2}^{2/(1+\gamma)}  \bigg)\\
    &\le 4\sum_{j=M}^\infty 2^{-j(1+\gamma)}2^{q(j+1)} = \frac{4C_{q,\gamma}}{2^{M(1+\gamma-q)}}.
\end{align*} 
Similarly, condition on the event where 
\begin{align*}
    \Omega'_{\mathtt{init}} := \left\{ \mathbb{E}_{P^2}[|\widehat t_{\alpha}(\theta(P^N), \widehat \theta_1)| \widetilde D_1] \le C'_{\mathtt{init}}s'_{n_1, n_2}\right\},
\end{align*}
we obtain
\begin{align}
    \mathbf{IV} &= \mathbb{P}_{P^2|\widetilde P^1}\left(c_0\|\theta - \theta(P^N)\|^{1+\gamma} \le 2|\widehat t_{\alpha}(\theta(P^N), \widehat \theta_1)| \cap B^c\right) \nonumber\\ &
    \le4\cdot 2^{-M(1+\gamma)}{s'}^{-1}_{n_1,n_2}\mathbb{E}_{P^2|\widetilde P^1}[|\widehat t_{\alpha}(\theta(P^N), \widehat \theta_1)|]+ \mathbb{P}_{\widetilde P^1}({\Omega'}_{\mathtt{init}}^c)\le \frac{4C'_{\mathtt{init}}}{2^{M(1+\gamma)}} + \varepsilon'_{\mathtt{init}}\nonumber
\end{align}
by Markov inequality. Hence, we conclude 
\begin{align*}
&\mathbb{P}^*_{P^N}\left(\mathrm{Diam}_{\|\cdot\|}\big(\widehat{\mathrm{CI}}^{\mathtt{LCB}}_{N, \alpha}\big)> 2^{M}c_0^{-1/(1+\gamma)}(r_{n_2}^{2/(1+\gamma)}  +u_{n_2}^{2/(1+\gamma)} +s_{n_1,n_2}^{1/(1+\gamma)}+{s'}_{n_1,n_2}^{1/(1+\gamma)})\right) \\
&\quad \le 4\cdot (C_\mathtt{init}+C'_\mathtt{init})2^{-M(1+\gamma)}+ 8\cdot C_{q,\gamma} 2^{-M(1+\gamma-q)} +\varepsilon_{\mathtt{init}}+\varepsilon'_{\mathtt{init}}+ \beta(r)\\
&\quad \lesssim_{q,\gamma, C_\mathtt{init},C'_\mathtt{init}}  2^{-M(1+\gamma-q)} +\varepsilon_{\mathtt{init}}+\varepsilon'_{\mathtt{init}}+ \beta(r).
\end{align*}
We can choose the same $M$ as the proof of \Cref{thm:hulc-width} except now we have $\mathfrak{C} = 4(C_\mathtt{init}+C'_\mathtt{init}+2C_{q,\gamma})$.
\end{proof}

\section{Proofs from \Cref{sec:computation}}
\subsection{Proof of \Cref{thm:geometry-mean}}
The key identity, obtained by expanding both squared norms, is
\begin{equation}\label{eq:mean-expansion}
\frac{1}{n_2}\sum_{i \in I_2} \|Z_i - \theta\|_2^2 - \|Z_i - \widehat{\theta}_1\|_2^2 = \|\widehat{\theta}_1 - \theta\|_2^2 + 2\widehat{H}^\top(\widehat{\theta}_1 - \theta).
\end{equation}
Two cases can be proved analogously, and we show the case (2). Write out the set as
    \begin{align*}
    \widehat{\mathrm{CI}}_{N, \alpha}^{\mathtt{CLT}, 2}  &= \left\{\theta\in\mathbb{R}^d:\, \frac{1}{n_2}\sum_{i \in I_2} \|Z_i - \theta\|_2^2 - \|Z_i - \widehat\theta_1\|_2^2\le -\|\widehat\theta_1 -\theta\|_2^2 + n_2^{-1/2}z_\alpha  \widehat \sigma_{\widehat{\theta}_3, \widehat\theta_1}\right\} \\
    &= \left\{\theta\in\mathbb{R}^d:\, \|\widehat{\theta}_1 - \theta\|_2^2 + 2\widehat{H}^\top(\widehat{\theta}_1 - \theta)\le -\|\widehat\theta_1 -\theta\|_2^2 + n_2^{-1/2}z_\alpha  \widehat \sigma_{\widehat{\theta}_3, \widehat\theta_1}\right\} \\
         &= \left\{\theta\in\mathbb{R}^d:\, \widehat H (\widehat\theta_1 - \theta) + \|\widehat\theta_1 -\theta\|_2^2\le n_2^{-1/2}z_\alpha  \widehat \sigma_{\widehat{\theta}_3, \widehat\theta_1}/2\right\}\\
         &= \left\{\theta\in\mathbb{R}^d:\, (\widehat\theta_1 - \theta + \widehat H/2)^\top(\widehat\theta_1 - \theta + \widehat H/2)\le \widehat H ^\top \widehat H/4 + n_2^{-1/2}z_\alpha  \widehat \sigma_{\widehat{\theta}_3, \widehat\theta_1}/2\right\},
    \end{align*}
    which is an $\mathbb{R}^d$-ball with center given by $\widehat\theta_1 + \widehat H/2 = (\overline Z_2+\widehat\theta_1)/2$, and the radius given by $(\|\widehat H\|^2_2/4 + n_2^{-1/2}z_\alpha  \widehat \sigma_{\widehat{\theta}_3, \widehat\theta_1}/2)^{1/2}$.

\subsection{Proof of \Cref{thm:geometry-LR}}
We begin by writing out
\begin{equation}\label{eq:LR-expansion}
\frac{1}{n_2}\sum_{i \in I_2} (Y_i - \theta^\top X_i)^2 - (Y_i - \widehat\theta_1^\top X_i)^2 = (\widehat{\theta}_1 - \theta)^\top \widehat\Gamma (\widehat{\theta}_1 - \theta) + 2\widehat{\Lambda}^\top(\widehat{\theta}_1 - \theta),
\end{equation}
where we denote 
\begin{align*}
    \widehat \Lambda = {n_2}^{-1}\sum_{i\in I_2} \widehat\varepsilon_iX_i, \quad \textrm{where}\quad \widehat\varepsilon_i = Y_i - \widehat\theta_1^\top X_i\quad \textrm{and} \quad 
    \widehat\Gamma = {n_2}^{-1}\sum_{i\in I_2} X_i X_i^\top.
\end{align*}
Two cases can be proved analogously, and we show the case (2). Write out the set as
    \begin{align*}
    \widehat{\mathrm{CI}}_{N,\alpha}^{\mathtt{CLT}, 2}  &= \left\{\theta\in\mathbb{R}^d:\, (\widehat{\theta}_1 - \theta)^\top \widehat\Gamma (\widehat{\theta}_1 - \theta) + 2\widehat{\Lambda}^\top(\widehat{\theta}_1 - \theta)\le n_2^{-1/2}z_\alpha  \widehat \sigma_{\widehat{\theta}_3, \widehat\theta_1}-(\widehat{\theta}_1 - \theta)^\top \widehat\Gamma (\widehat{\theta}_1 - \theta)\right\} \\
    &= \left\{\theta\in\mathbb{R}^d:\, (\widehat{\theta}_1 - \theta)^\top \widehat\Gamma (\widehat{\theta}_1 - \theta) + \widehat{\Lambda}^\top(\widehat{\theta}_1 - \theta)\le n_2^{-1/2}z_\alpha  \widehat \sigma_{\widehat{\theta}_3, \widehat\theta_1}/2\right\} \\
    &= \left\{\theta\in\mathbb{R}^d:\, (\widehat{\theta}_1 - \theta + \widehat\Gamma^{-1}\widehat\Lambda/2)^\top \widehat\Gamma (\widehat{\theta}_1 - \theta + \widehat\Gamma^{-1}\widehat\Lambda/2)\le \widehat\Lambda{\widehat\Gamma^{-1}}\widehat\Lambda/4 + n_2^{-1/2}z_\alpha  \widehat \sigma_{\widehat{\theta}_3, \widehat\theta_1}/2\right\}
    \end{align*}
    Here, observe that 
    \begin{align*}
        \widehat\Gamma^{-1}\widehat\Lambda &= \widehat\Gamma^{-1}\left(n_2^{-1}\sum_{i \in I_2}Y_iX_i-\widehat\Gamma \widehat\theta_1\right) = \theta_{\mathrm{OLS}} -  \widehat \theta_1 \quad \textrm{and}\\
     \widehat\Lambda{\widehat\Gamma^{-1}}\widehat\Lambda&=\widehat\Lambda{\widehat\Gamma^{-1}}{\widehat\Gamma}{\widehat\Gamma^{-1}}\widehat\Lambda = \|\theta_{\mathrm{OLS}} -  \widehat \theta_1\|_{\widehat\Gamma}.
    \end{align*}
    Hence we obtain that 
    \begin{align*}
        \widehat{\mathrm{CI}}_{N,\alpha}^{\mathtt{CLT}, 2} = \left\{\theta\in\mathbb{R}^d:\, \|(\widehat{\theta}_1 + \theta_{\mathrm{OLS}})/2- \theta\|_{\widehat\Gamma}^2\le \|\theta_{\mathrm{OLS}} -  \widehat \theta_1\|_{\widehat\Gamma}^2/4 + n_2^{-1/2}z_\alpha  \widehat \sigma_{\widehat{\theta}_3, \widehat\theta_1}/2\right\}.
    \end{align*}

\section{Proofs from Statistical Applications}\label{sec:proof-applications}
This section contains all proofs associated with the statistical applications. 

We frequently use following two lemmas, which becomes useful for the analysis of validity and diameter calculation. Their proofs are postponed to \Cref{sec:technical-lemma}. 
\begin{lemma}\label{prop:linearization-prop}
    Suppose that there exists a constant $\delta_0 > 0$ and a $P_i$-dependent mean-zero random vector $G_i$, such that 
    \begin{align*}
        \frac{\E_{P_i}[(\widehat \xi_i- \langle \widehat{\theta}_1 - \theta(P^N), G_i\rangle)^2  |D_1]}{\E_{P_i}[\langle \widehat{\theta}_1 - \theta(P^N), G_i\rangle^2 |D_1]} \le \varphi(\|\widehat{\theta}_1 - \theta(P^N)\|),
    \end{align*}
    for all $i \in I_2$ and $\|\widehat{\theta}_1 - \theta(P^N)\| < \delta_0$
    where $\varphi : \mathbb{R}_+ \mapsto \mathbb{R}_+$ is continuous and $\varphi(0) = 0$. Then, 
    \begin{align*}
        &\mathbb{E}_{P^1}\left[\min\left\{1, C\sum_{i\in I_2} \mathbb{E}_{P_i}\left[\frac{|\widehat\xi_i|^2}{n_2^2\widehat{\mathbb{V}}_{2}}\min\left\{1,\,\frac{|\widehat \xi_i|}{n_2\widehat{\mathbb{V}}_{2}^{1/2}}\right\} \bigg| D_1\right]\right\}\right] \le \min\left\{1,  C'R_{n_2}\right\},
    \end{align*}
    where $C, C'$ are universal constants, 
    \begin{align*}
        R_{n_2} = \inf_{\delta < \delta_0}\left\{2\sqrt{\varphi(\delta)} + \mathbb{P}_{P^1}(\|\widehat\theta_1 - \theta(P^N)\| > \delta)\right\} + \sum_{i\in I_2} \mathbb{E}_{P_i}\left[\frac{|\langle u, G_i\rangle|^2}{\mathbb{V}_G}\min\left\{1,\,\frac{|\langle u, G_i\rangle|}{\mathbb{V}_G^{1/2}}\right\}\right],
    \end{align*}
    and $\mathbb{V}_G = \sum_{i \in I_2}\E_{P_i}[\langle u, G_i\rangle^2]$.
\end{lemma}
The following lemma provides an intermediate bound on $\omega_{n_2, \mathtt{emp}}$ 
in \ref{as:square-process} under moment conditions on the envelope. 
The standard empirical process notation is adopted throughout: for a 
measurable function $f$,
\begin{align*}
    (\mathbb{P}_{n_2} - P^2)f := \frac{1}{n_2}\sum_{i\in I_2} f(Z_i) 
    - \mathbb{E}_{P^2}[f(Z)].
\end{align*}
The localized function class and its envelope are defined respectively as
\begin{align*}
    \mathcal{M}_{\delta} := \{m_\theta - m_{\theta_0} : \|\theta - \theta_0\| 
    \leq \delta\} \quad \textrm{and} \quad M_\delta(z) := \sup_{m \in \mathcal{M}_{\delta}}|m(z)|,
\end{align*}
where $M_\delta$ is assumed measurable. 
\begin{lemma}\label{prop:squared-Gn}
Let $Z_1, \ldots, Z_{n_2}$ be independent random variables with law $P^2$, let $M_\delta$ 
be the envelope defined above, and let $\phi_{n_2}(\delta)$ be as in \ref{as:local-entropy}.
\begin{enumerate}
\item \textbf{$L^q$ envelope}: If $\E_{P^2}[|M_\delta|^q] \le C_q$ for some $q \ge 2$, then
    \begin{align*}
        \mathbb{E}^*_{P^2}\!\left[\sup_{m\in\mathcal{M}_{\delta}} |(\mathbb{P}_{n_2}-P^2) m^2|\right] \leq 16\,n_2^{2/q - 1} C_q^{2/q} + 8\cdot 8^{1/q}\, n_2^{1/q} C_q^{1/q}\,\phi_{n_2}(\delta).
    \end{align*}
\item \textbf{Sub-Weibull envelope}: If $\mathbb{E}_{P^2}[|M_\delta|^q] \leq K^q q^{q/\gamma}$ for all $q \geq 1$, then for $n_2 > e^{2/\gamma}$
    \begin{align*}
        \mathbb{E}^*_{P^2}\!\left[\sup_{m\in\mathcal{M}_{\delta}} |(\mathbb{P}_{n_2}-P^2) m^2|\right] &\leq 16n_2^{-1}e^{2/\gamma} K^2 (\gamma \log n_2) ^{2/\gamma}\\
        &\quad +16 \cdot e^{1/\gamma}K (\gamma \log n_2) ^{1/\gamma} \phi_{n_2}(\delta).
    \end{align*}
\end{enumerate}
\end{lemma}

% Similarly, the bracketing entropy integral is defined as 
% \begin{align*}
%     J_{[\, ]}(\delta, \Theta, \|\cdot\|) := \int_{0}^\delta \sqrt{1 + \log N_{[\,]}(\varepsilon, \Theta, \|\cdot\|)}\, d\varepsilon,
% \end{align*}
% crucially without taking the supremum over $Q$ and $\varepsilon$ is not normalized by the norm of the envelop function.

% An analogous result under the bracketing entropy integral is also available:
% \begin{theorem}[Theorem 2.14.17' of \citet{van1996weak}]\label{thm:local-maximal-ineq-bracket}
%     Let $\mathcal{F}$ be a collection of $P$-square integrable functions equipped with an envelop function $F \le M$. If $\mathbb{E}_P f^2 \le t^2 \mathbb{E}_P F^2$, for every $f$ and some $t \in (0,1)$, then 
%     \begin{align*}
%         \E_P\, \left[\sup_{f \in \mathcal{F}}\, |\mathbb{G}_n f|\right] \lesssim J_{[\, ]}(t, \mathcal{F}, L_2(P))\left(1+ \frac{J_{[\, ]}(t, \mathcal{F}, L_2(P))}{t^2 \sqrt{n}}M\right) 
%     \end{align*}
%     where the expectation should be regarded as an outer expectation (Chapter 1.2 of \citet{van1996weak}) when the content inside is not measurable. 
% \end{theorem}

\subsection{High-dimensional Mean Inference}
% In this example, we take $m_{\theta}(z) := \|z-\theta\|^2$.
\begin{proof}[\bfseries{Proof of Theorem~\ref{thm:mean-ci-validity}}]
The proof is an application of \Cref{thm:coverage-anti-conservative-confidence-set-empirical-risk} for $\alpha = 1/2$ and \Cref{thm:studentized-katz} for $\alpha \neq 1/2$. We collect relevant values. First, 
    \begin{align*}
        \widehat\xi_i &= \|X_i - \widehat \theta_1\|^2 - \|X_i - \theta(P^N)\|^2 - \mathbb{E}_{P_i}[\|X_i - \widehat \theta_1\|^2 - \|X_i - \theta(P^N)\|^2|D_1].\\
        & = 2\langle X_i - \theta(P^N), \theta(P^N)-\widehat\theta_1\rangle.
    \end{align*}
    Now, we denote $u \in \mathbb{S}^{d-1}$ such that $\theta(P^N)-\widehat\theta_1 = \delta_1 u$. Then 
    \begin{align*}
        |\widehat\xi_i|^2 = 4\delta_1^2 \langle X_i - \theta(P^N), u\rangle^2 \quad \text{and} \quad \widehat{\mathbb{V}}_2 = \frac{4\delta_1^2}{n_2^2}\sum_{i\in I_2} u^\top  \Sigma_i u = \frac{4\delta_1^2 u^\top \bar\Sigma_2 u}{n_2}.
    \end{align*}
    We then have 
    \begin{align*}
        &\sum_{i\in I_2} \mathbb{E}_{P_i}\left[\frac{|\widehat\xi_i|^2}{n_2^2\widehat{\mathbb{V}}_{2}}\min\left\{1,\,\frac{|\widehat \xi_i|}{n_2\widehat{\mathbb{V}}_{2}^{1/2}}\right\} \bigg|D_1\right] \\
        &\quad \le \sup_{u \in \mathbb{S}^{d-1}}\sum_{i\in I_2} \mathbb{E}_{P_i}\left[\frac{\langle X_i - \theta(P^N), u\rangle^2}{n_2 u^\top  \bar\Sigma_2 u}\min\left\{1,\,\frac{|\langle X_i - \theta(P^N), u\rangle|}{\sqrt{n_2 u^\top  \bar\Sigma_2 u}}\right\}\right]\\
        &\quad = \sup_{u \in \mathbb{S}^{d-1}}\sum_{i\in I_2} \mathbb{E}_{P_i}\left[\frac{\langle \bar\Sigma_2^{-1/2}(X_i - \theta(P^N)), u\rangle^2}{n_2}\min\left\{1,\,\frac{|\langle \bar\Sigma_2^{-1/2}(X_i - \theta(P^N)), u\rangle|}{\sqrt{n_2}}\right\}\right] \\
        &\quad = \sup_{u \in \mathbb{S}^{d-1}}\sum_{i\in I_2} \mathbb{E}_{P_i}\left[\frac{\langle X_i^\circ, u\rangle^2}{n_2}\min\left\{1,\,\frac{|\langle X_i^\circ, u\rangle|}{\sqrt{n_2}}\right\}\right]
    \end{align*}
   The result for $\alpha \neq 1/2$ is claimed by applying \Cref{thm:studentized-katz}. 

    Next consider the case $\alpha = 1/2$. To apply \Cref{thm:coverage-anti-conservative-confidence-set-empirical-risk}, we observe that  
    \begin{align*}
        \ratio^2 = \widehat{\mathbb{C}}^2_2/\widehat{\mathbb{V}}_{2} =\|\widehat\theta_1 - \theta(P^N)\|^4/\left(\frac{4\delta_1^2 u^\top \bar\Sigma_2 u}{n_2}\right) = \frac{n_2\|\bar\Sigma_2^{-1/2}(\widehat\theta_1 - \theta(P^N))\|^2}{4}.
    \end{align*}
    We then have
\begin{align*}
    &\sum_{i\in I_2}\mathbb{E}_{P_i}\left[\frac{|\widehat\xi_i|^2}{n_2^2\widehat{\mathbb{V}}_{2}(1 + \ratio)^2}\min\left\{1,\,\frac{|\widehat \xi_i|}{n_2\widehat{\mathbb{V}}_{2}^{1/2}(1 + \ratio)}\right\}\bigg|D_1\right]\\
    &\quad = \sum_{i\in I_2} \E_{P_i}\left[\frac{\langle X_i - \theta(P^N), u\rangle^2}{n_2 u^\top  \bar\Sigma_2 u(1 + \ratio)^2}\min\left\{1,\,\frac{|\langle X_i - \theta(P^N), u\rangle|}{\sqrt{n_2 u^\top  \bar\Sigma_2 u}(1 + \ratio)}\right\}\bigg|D_1\right] \le \frac{1}{(1+\ratio)^2}.
    % &\quad \le 4\E_{P^1}\left[\sup_{u \in \mathbb{S}^{d-1}}\sum_{i\in I_2} \E_{P_i}\left[\frac{\langle X_i - \theta(P^N), u\rangle^2}{n_2^2 \delta_1^2}\bigg|D_1\right]\right]\\
    % &\quad \le \E_{P^1}\left[\frac{4\lambda_{\max}(\bar\Sigma_2)}{n_2\delta_1^2}\right].
\end{align*}
By \Cref{thm:coverage-anti-conservative-confidence-set-empirical-risk} and \Cref{thm:studentized-katz}, we conclude concludes the result for $\alpha = 1/2$.
\end{proof}
\begin{proof}[\bfseries{Proof of \Cref{thm:mean-ci-width}}]
The result is a direct consequence of \Cref{lemma:mean-width-1} and \Cref{lemma:mean-width-2}
\end{proof}
\begin{lemma}\label{lemma:mean-width-1}
    When $\alpha = 1/2$, for any $\varepsilon > 0$, it holds that     \begin{align*}
    \mathbb{P}_{P^N}\!\left(
    \mathrm{Diam}_{\|\cdot\|_2}\bigl(\widehat{\mathrm{CI}}_{N,1/2}^{\mathtt{CLT}}\bigr)
    \leq 2\sqrt{2}\varepsilon^{-1/2}
    \left\{\sqrt{\frac{\mathrm{tr}(\bar\Sigma_2)}{n_2}} 
    + \|\widehat{\theta}_1 - \theta(P^N)\|_2\right\}
    \right) \geq 1 - \varepsilon,
\end{align*}
for $n_1, n_2 \ge 1$.
\end{lemma}
\begin{proof}[\bfseries{Proof of \Cref{lemma:mean-width-1}}]
    When $\alpha = 1/2$, that is $z_\alpha = 0$, \Cref{thm:geometry-mean} establishes that the diameter of the confidence set can be computed exactly as 
    \begin{align*}
        2\|\widehat H\|_2 \quad \textrm{where} \quad \widehat H = n_2^{-1}\sum_{i \in I_2} Z_i - \widehat{\theta}_1.
    \end{align*}
    By triangle inequality, we have 
    \begin{align*}
        \|\widehat H\|_2 \le \|\widehat{\theta}_1-\theta(P^N)\|_2 + \left\|n_2^{-1}\sum_{i \in I_2} Z_i-\theta(P^N)\right\|_2.
    \end{align*}
    Observe that $\theta(P^N) = \E_{P_i}[X_i]$ for all $1 \le i \le N$, under independence, we have 
    \begin{align*}
        \E_{P^N}\left[\left\|n_2^{-1}\sum_{i \in I_2} Z_i-\theta(P^N)\right\|_2^2\right] &= \E_{P^N}\left[\left\|n_2^{-1}\sum_{i \in I_2} (Z_i-\E_{P_i}[Z_i])\right\|_2^2\right] = \frac{\mathrm{tr}(\bar\Sigma_2)}{n_2}.
    \end{align*}
    Denoting
    \begin{align*}
        \mathrm{R}_N = \|\widehat{\theta}_1-\theta(P^N)\|_2 + \sqrt{\frac{\mathrm{tr}(\bar\Sigma_2)}{n_2}},
    \end{align*}
    it follows as 
    \begin{align*}
        &\mathbb{P}_{P^N}(2\|\widehat H\|_2 \ge C\varepsilon^{-1/2} \mathrm{R}_N) \le \frac{4\varepsilon \mathbb{E}_{P^N}[\|\widehat H\|_2^2/\mathrm{R}_N^2]}{C^2} \le \frac{8\varepsilon}{C^2}\E_{P^1}\left[\frac{\|\widehat{\theta}_1-\theta(P^N)\|_2^2 + \mathrm{tr}(\bar\Sigma_2)/n_2}{\mathrm{R}_N^2}\right] \le \varepsilon,
    \end{align*}
    with $C = 2\sqrt{2}$. This concludes the result. 
\end{proof}

\begin{lemma}\label{lemma:mean-width-2}
    Let 
$\widetilde{s}_{n_1,n_2}$ be as in \ref{as:rate-initial-estimator3}. 
For any $\varepsilon > 0$, setting $\varepsilon^\circ = \varepsilon + \widetilde \varepsilon_{\mathtt{init}}$, it holds that
\begin{align*}
    \mathbb{P}_{P^N}\!\left(
    \mathrm{Diam}_{\|\cdot\|_2}\bigl(\widehat{\mathrm{CI}}_{N,\alpha}^{\mathtt{CLT}}\bigr)
    \leq C_{\varepsilon^\circ}\left(1+|z_{\alpha}|\right)
    \left\{\sqrt{\frac{\mathrm{tr}(\bar\Sigma_2)}{n_2}} 
    + \widetilde{s}_{n_1,n_2}^{1/2}\right\}
    \right) \geq 1 - \varepsilon^\circ,
\end{align*}
provided $\max\{2,z_\alpha^2 C'_{\varepsilon^\circ}\}\le n_2$ where $C_{\varepsilon^\circ}$ and $C'_{\varepsilon^\circ}$ depend on $\varepsilon^\circ$, but not on $d$ or $\alpha$.
\end{lemma}
\begin{proof}[\bfseries{Proof of \Cref{lemma:mean-width-2}}]
    Throughout, we treat $\|\cdot\|\equiv\|\cdot\|_2$. The proof is a direct application of \Cref{thm:clt-width-local-full}, and thus proceeds by verifying \ref{as:margin}, \ref{as:local-entropy}, and \ref{as:square-process} to hold locally and \ref{as:margin-global}, \ref{as:ratio-process}, and \ref{as:ratio-square-process} to hold globally.

    \paragraph{Verifying \ref{as:margin}} For any $\theta \in \Theta$, 
    \begin{align*}
    m_{\theta}-m_{\theta(P^N)} := \|X-\theta\|^2 - \|X-\theta(P^N)\|^2 = 2(X-\theta(P^N))^\top(\theta-\theta(P^N)) + \|\theta-\theta(P^N)\|^2,
\end{align*}
and $\M_2(\theta)-\M_2(\theta(P^N)) =  \|\theta-\theta(P^N)\|^2$. Thus \ref{as:margin} holds with $\gamma=1$ and $c_0=1$. 

\paragraph{Verifying \ref{as:local-entropy}}
For any $\theta$ such that $\|\theta-\theta(P^N)\| \le \delta$, we have 
    \begin{align*}
        &\sup_{\|\theta-\theta(P^N)\| \le \delta}\, |m_{\theta}-m_{\theta(P^N)} - \E_{P_i}(m_{\theta}-m_{\theta(P^N)})| \\
        &\quad = \sup_{\|\theta-\theta(P^N)\| \le \delta}\, 2|\big(X-\theta(P^N)\big)^\top (\theta(P^N)-\theta)| \\
        &\quad= 2\delta \sup_{u \in \mathbb{S}^{d-1}}\,|u^\top (X-\theta(P^N))|  \\
        &\quad= 2\delta \|X-\theta(P^N)\|.
    \end{align*}
   Below, denote by $\{\epsilon_i\}_{i \in I_2}$ independent Rademacher random variables. We then obtain by symmetrization, such that, 
   \begin{align*}
       &\E_{P^2} \left[\sup_{\|\theta-\theta(P^N)\| < \delta}|(\widehat{\mathbb{M}}_2 - \mathbb{M}_2)(\theta) - (\widehat{\mathbb{M}}_2 - \mathbb{M}_2)(\theta(P^N))| \right] \\
       &\quad \le 2\E_{P^2 \times \epsilon} \left[\left|\frac{1}{n_2}\sum_{i \in I_2}2\delta \epsilon_i\|X_i-\theta(P^N)\| \right|\right] \\
       &\quad \le \frac{4}{n_2}\left(\E_{P^2 \times \epsilon}\left[\left|\sum_{i \in I_2} \epsilon_i\delta \|X_i-\theta(P^N)\|\right|^2 \right]\right)^{1/2}\\
       &\quad = \frac{4}{n_2}\left(\E_{P^2}\left[\sum_{i \in I_2}\delta^2 \|X_i-\theta(P^N)\|^2\right]\right)^{1/2}\\
       &\quad 
       \le \frac{4\delta}{n_2}\left(\sum_{i \in I_2}\E_{P_i}\|X_i-\theta(P^N)\|^2\right)^{1/2} = \phi_{n_2}(\delta).
   \end{align*}
    This satisfies the requirement with $q=1$.
   \paragraph{Verifying \ref{as:square-process}} First, we derive $\omega_{\mathtt{pop}}$. We have
    \begin{align*}
    &\E_{P_i}[(m_{\theta}-m_{\theta(P^N)})^2] \\
    &\quad= \E_{P_i}[(2(X-\theta)^\top(\theta-\theta(P^N)) + \|\theta-\theta(P^N)\|^2)^2]\\
    &\quad\le 8\E_{P_i}|(X-\theta)^\top(\theta-\theta(P^N))|^2 + 2\|\theta-\theta(P^N)\|^4\\
    &\quad\le 8\|\theta-\theta(P^N)\|^2\E_{P_i}\|X-\theta(P^N)\|^2  + 2\|\theta-\theta(P^N)\|^4.
\end{align*}
Hence we can set 
\begin{align*}
    \omega^2_{\mathtt{pop}}(\delta) =\frac{8\delta^2}{n_2}\sum_{i\in  I_2}\E_{P_i}\|X-\theta(P^N)\|^2  + 2\delta^4.
\end{align*}
In order to derive $\omega_{n_2, \mathtt{emp}}(\delta)$, we use \Cref{prop:squared-Gn}. From the earlier derivation, the local envelope function can be defined as 
\begin{align*}
     M_\delta(X_i)&= \sup_{\|\theta-\theta(P^N)\| \le \delta}\, |(m_{\theta}-m_{\theta(P^ )})(X_i)|\\
     &=\sup_{\|\theta-\theta(P^N)\| \le \delta}\, |(2(X_i-\theta(P^N))^\top(\theta-\theta(P^N)) + \|\theta-\theta(P^N)\|^2|\\
     &\le  2\delta \|X_i-\theta(P^N)\|  + \delta^2.
\end{align*}
By invoking \Cref{prop:squared-Gn} with $q=2$, we have for any $\eta > 0$,
\begin{align*}
    &\E_{P^2}\left[\sup_{\|\theta-\theta(P^N)\| < \delta}\left|\frac{1}{n_2}\sum_{i\in I_2}(m_\theta - m_{\theta(P^N)})^2 - \mathbb{E}_{P_i}[(m_\theta - m_{\theta(P^N)})^2]\right| \right]\\
    &\quad \le 16\,\E_{P^2}[|M_\delta(X)|^2] + 8\cdot 8^{1/2}\, n_2^{1/2} (\E_{P^2}[|M_\delta(X)|^2])^{1/2}\,\phi_{n_2}(\delta)\\
    &\quad \le 32(2\delta)^2 \E_{P^2}[\|X-\theta(P^N)\|^2]  + 32\delta^4 \\
    &\quad\quad + 8\cdot 8^{1/2}\, n_2^{1/2} 2\delta(\E_{P^2}[\|X-\theta(P^N)\|^2])^{1/2}\,\phi_{n_2}(\delta)\\
    &\quad\quad + 8\cdot 8^{1/2}\, n_2^{1/2} \sqrt{2}\delta^2\,\phi_{n_2}(\delta)\\
    &\quad \le C \left(\delta^2 \mathrm{tr}(\bar\Sigma_2) + \delta^3 \sqrt{\mathrm{tr}(\bar\Sigma_2)}+\delta^4\right) := \omega^2_{n_2, \mathtt{emp}}(\delta),
\end{align*}
for some universal constant $C > 0$.

Notice that this derivation does not satisfy $q < 1+\gamma$ globally. Suppose that $\delta \le \rho$, then 
\begin{align*}
    \omega^2_{\mathtt{pop}}(\delta) &=C\left(\mathrm{tr}(\bar\Sigma_2)  + \rho^2\right)\delta^2, \quad \text{and} \\
    \omega^2_{n_2, \mathtt{emp}}(\delta)&=C \left(\mathrm{tr}(\bar\Sigma_2) + \rho \sqrt{\mathrm{tr}(\bar\Sigma_2)} + \rho^2\right) \delta^2,
\end{align*}
hence we can take $q = 1$ and $q < 1+\gamma$ is now satisfied locally. 

\paragraph{Verifying \ref{as:margin-global}}
\ref{as:margin} holds globally and thus one can choose 
\begin{align*}
    C_{\rho}(\|\theta-\theta(P^N)\|) = \|\theta-\theta(P^N)\|^2.
\end{align*}

\paragraph{Verifying \ref{as:ratio-process}} 
For any $\delta > \rho$,
\begin{align*}
    &\E_{P^2} \left[\sup_{\|\theta-\theta(P^N)\| > \rho}\left|\frac{(\widehat{\mathbb{M}}_2 - \mathbb{M}_2)(\theta) - (\widehat{\mathbb{M}}_2 - \mathbb{M}_2)(\theta(P^N))}{\mathbb{M}_2(\theta) -\mathbb{M}_2(\theta(P^N))}\right| \right] \\
    &\quad \le \frac{4}{n_2 \rho}\left(\sum_{i \in I_2}\E_{P_i}\|X_i-\theta(P^N)\|^2\right)^{1/2} = \frac{4\sqrt{\mathrm{tr}(\bar\Sigma_2)}}{\sqrt{n_2} \rho} = R(n_2, \rho).
\end{align*}
Hence, \ref{as:ratio-process} holds with $C_{\mathtt{ratio}}=1/\varepsilon_{\mathtt{ratio}}$ by Markov's inequality.

\paragraph{Verifying \ref{as:ratio-square-process}} Similarly for any $\delta > \rho$,
\begin{align*}
    &\E_{P^2} \left[\sup_{\|\theta-\theta(P^N)\| > \rho}\frac{z_\alpha^2}{n_2^2}\left|\frac{\sum_{i\in I_2}(m_\theta - m_{\theta(P^N)})^2 - \mathbb{E}_{P_i}[(m_\theta - m_{\theta(P^N)})^2]}{\{\mathbb{M}_2(\theta) -\mathbb{M}_2(\theta(P^N))\}^2}\right| \right]\\
    &\quad \le \frac{C  \mathrm{tr}(\bar\Sigma_2)z_\alpha^2}{\rho^2 n_2} + \frac{C \sqrt{\mathrm{tr}(\bar\Sigma_2)}z_\alpha^2}{\rho n_2}+ \frac{Cz_\alpha^2}{n_2} = S_{\mathtt{emp}}(n_2, \rho, \alpha),
\end{align*}
and 
\begin{align*}
    &\sup_{\|\theta-\theta(P^N)\| > \rho}\frac{z_\alpha^2\sum_{i\in I_2}\mathbb{E}_{P_i}[(m_\theta - m_{\theta(P^N)})^2(Z_i)]}{n_2^2\{\M_2(\theta)-\M_2(\theta(P^N))\}^2}\le \frac{Cz_\alpha^2}{n_2}\left(\frac{\mathrm{tr}(\bar\Sigma_2)}{\rho^2} + 1\right) = S_{\mathtt{pop}}(n_2, \rho, \alpha).
\end{align*}
Hence, \ref{as:ratio-square-process} holds with $\widetilde C_{\mathtt{emp}}=1/\varepsilon_{\mathtt{emp}}$ by Markov's inequality.
\paragraph{Evaluating the rate of convergence} We now evaluate the rate of convergence by applying \Cref{thm:clt-width} and \Cref{thm:clt-width-local-full}. Denote by $C>0$ a universal constant that changes from line to line. Choose $\rho = \sqrt{\mathrm{tr}(\bar\Sigma_2)}$. Then, using the fact that $c_0 = 1$ and $\gamma = 1$, 
\begin{align*}
    r_{n_2}^{-2}\phi_{n_2}(r_{n_2}) \le 1 \Leftrightarrow Cr^{-2}_{n_2}r_{n_2}\left(\frac{\mathrm{tr}(\bar\Sigma_2)}{n_2}\right)^{1/2}\le 1 \Leftrightarrow C\left(\frac{\mathrm{tr}(\bar\Sigma_2)}{n_2}\right)^{1/2}\le r_{n_2}.
\end{align*}
Next, we evaluate the value related to $\omega_{\mathtt{pop}}$. This follows as 
\begin{align*}
    u_{n_2}^{-4}\omega^2_{\mathtt{pop}}(u_{n_2}) \le n_2 \Leftrightarrow Cu_{n_2}^{-2}\mathrm{tr}(\bar\Sigma_2)\le n_2 \Leftrightarrow C\sqrt{\frac{\mathrm{tr}(\bar\Sigma_2)}{n_2}}\le u_{n_2}.
\end{align*}
Finally, the value related to $\omega_{n_2, \mathtt{emp}}$ yields that 
\begin{align*}
   u_{n_2}^{-4}\omega^2_{n_2, \mathtt{emp}}(u_{n_2}) \le n_2\Leftrightarrow  Cu_{n_2}^{-4}\mathrm{tr}(\bar\Sigma_2)u_{n_2}^2 \le n_2 \Leftrightarrow C\sqrt{\frac{\mathrm{tr}(\bar\Sigma_2)}{n_2}}\le u_{n_2}.
\end{align*}
Then by \Cref{thm:clt-width}, we obtain 
\begin{align*}
   (1+|z_\alpha|)^{1+\gamma-q} \mathrm{R}^{\mathtt{CLT}}_{N} &=   (1+|z_\alpha|)\sqrt{\frac{\mathrm{tr}(\bar\Sigma_2)}{n_2}}  +   (1+|z_\alpha|)\widetilde s^{1/2}_{n_1, n_2}.
\end{align*}
Meanwhile, by \Cref{thm:clt-width-local-full}, we have 
\begin{align*}
\mathrm{Q}^{\mathtt{CLT}}_{N,\alpha} = (1+|z_\alpha|)^{1/2}\widetilde s^{1/2}_{n_1, n_2}
\end{align*}
in view of $C_{\rho}(\|\theta-\theta(P^N)\|) = \|\theta-\theta(P^N)\|^2$. Hence, we conclude that 
\begin{align*}
    &\max\{(1+|z_\alpha|)^{1/(1+\gamma-q)}\mathrm{R}^{\mathtt{CLT}}_{N}, \mathrm{Q}^{\mathtt{CLT}}_{N,\alpha}\mathbf{1}\{\mathrm{Q}^{\mathtt{CLT}}_{N,\alpha} \ge \rho\}\} \\
    &\quad \le (1+|z_\alpha|)\sqrt{\frac{\mathrm{tr}(\bar\Sigma_2)}{n_2}}  +   (1+|z_\alpha|)\widetilde s^{1/2}_{n_1, n_2},
\end{align*}
with probability greater than $1-\varepsilon$ as long as $n_2$ is large enough such that 
\begin{align*}
    C_\varepsilon \max\{R(n_2, \rho),S_{\mathtt{emp}}(n_2, \rho, \alpha), S_{\mathtt{pop}}(n_2, \rho, \alpha) \} \le 1/3,
\end{align*}
where $C_\varepsilon$ is a constant depending on $\varepsilon$.
% as shown in the proof of \Cref{thm:clt-width}, we have $ s_{n_1, n_2} \le \widetilde s_{n_1, n_2}$ for M-estimation with independent observations. For, $\alpha \neq 1/2$, more terms are involved. First, the value related to $\omega_{\mathtt{pop}}$ yields that 
% \begin{align*}
%     u_{n_2}^{-4}\omega^2_{\mathtt{pop}}(u_{n_2}) \le n_2 \Leftrightarrow 4(1+1/\eta)u_{n_2}^{-2}\mathrm{tr}(\bar\Sigma_2)  + (1+\eta) \le n_2.
% \end{align*}
% Choose $\eta = 1/2$. Then 
% \begin{align*}
%     \frac{12\mathrm{tr}(\bar\Sigma_2) }{n_2-1.5} \le \frac{48\mathrm{tr}(\bar\Sigma_2) }{n_2}  \le u_{n_2}^2 \quad \text{when} \quad n_2 \ge2.
% \end{align*}
% Now, we recall that 
% \begin{align*}
%     \omega^2_{n_2, \mathtt{emp}}(\delta)& =C \left(\delta^2 \mathrm{tr}(\bar\Sigma_2) + \delta^3 \sqrt{\mathrm{tr}(\bar\Sigma_2)}\right) + 16(1+1/20)\delta^4.
% \end{align*}
% % \begin{align*}
% %     C \left(u_{n_2}^{-2} \mathrm{tr}(\bar\Sigma_2) + u_{n_2}^{-1} \sqrt{\mathrm{tr}(\bar\Sigma_2)}\right)  \le n_2 - 16(1+1/20) 
% % \end{align*}
% We can then take for all $n_2 \ge 17$,
% \begin{align*}
% C\left(\frac{\sqrt{\mathrm{tr}(\bar\Sigma_2)}}{n_2} + \sqrt{\frac{\mathrm{tr}(\bar\Sigma_2)}{n_2}} \right) \le C\sqrt{\frac{\mathrm{tr}(\bar\Sigma_2)}{n_2}}  \le u_{n_2}.
% \end{align*}
% Hence \Cref{thm:clt-width} holds with 
% \begin{align*}
%     \mathrm{R}^{\mathtt{CLT}}_{N,\alpha} &= \sqrt{\frac{\mathrm{tr}(\bar\Sigma_2)}{n_2}}  + \widetilde s_{n_1, n_2}.
% \end{align*}
This concludes the result. 
\end{proof}

For \Cref{cor:mean-plugin}, we instead prove the slightly rephrased version of the corollary. 
\begin{corollary}\label{cor:mean-plugin-full}
    Suppose the initial estimator satisfies for all $n_1 \geq N_1$, \begin{align}\nonumber\mathbb{P}_{P^1}\left(\|\widehat\theta_1 - \theta(P^N)\|_2^2 \le \frac{\widetilde C_{\mathtt{init}}\mathrm{tr}(\bar \Sigma_1)}{n_1} \right) \ge 1-\widetilde\varepsilon_{\mathtt{init}},
    \end{align}
     For any $\varepsilon \in (0, 1-\widetilde \varepsilon_{\mathtt{init}})$, setting $\varepsilon^\circ = \varepsilon + \widetilde \varepsilon_{\mathtt{init}}$, $n_1 \ge N_1$, 
    \begin{align*}
            \mathbb{P}_{P^N}\!\left(
    \mathrm{Diam}_{\|\cdot\|_2}\bigl(\widehat{\mathrm{CI}}_{N,\alpha}^{\mathtt{CLT}}\bigr)
    \leq C''_{\varepsilon^\circ}\left(1+|z_{\alpha}|\right)
    \left\{\sqrt{\frac{\mathrm{tr}(\bar\Sigma_2)}{n_2}} 
    + \sqrt{\frac{\mathrm{tr}(\bar\Sigma_1)}{n_1}}\right\}
    \right) \geq 1 - \varepsilon^\circ,
        \end{align*}
        provided $n_2 \ge 1$ when $\alpha = 1/2$, and $n_2 \ge  \max\{2,z_\alpha^2 C'_{\varepsilon}\}$ when $\alpha \neq 1/2$, 
    where $ C'_{\varepsilon^\circ}$ and $C''_{\varepsilon^\circ}$ depend on $\varepsilon^\circ$, but not on $d$ or $\alpha$.
\end{corollary}
\begin{proof}[\bfseries{Proof of \Cref{cor:mean-plugin-full}}]
Verifying \ref{as:rate-initial-estimator3}, we have that 
\begin{align*}
   &\frac{1}{n_2}\mathbb{E}_{P^2| P^1}\left[\frac{1}{n_2}\sum_{i \in I_2} \widehat \xi_i^2\right] \le \frac{4\|\widehat \theta_1-\theta(P^N)\|^2}{n_2}\sum_{i \in I_2}\frac{\E_{P_i}\|X_i-\theta(P^N)\|^2}{n_2} = \frac{4\|\widehat \theta_1-\theta(P^N)\|^2\mathrm{tr}(\bar\Sigma_2)}{n_2}
\end{align*}
and $\widehat{\mathbb{C}}_2^2 = \|\widehat \theta_1-\theta(P^N)\|^4$. Consider the event 
\begin{align*}
    \Omega_{\mathtt{init}} := \left\{\|\widehat \theta_1-\theta(P^N)\|^2 \le \widetilde C_{\mathtt{init}}\frac{\mathrm{tr}(\bar\Sigma_1)}{n_1}\right\}.
\end{align*}
Then on this event, we can take 
\begin{align*}
    \widetilde s_{n_1, n_2}^2 = \frac{\mathrm{tr}(\bar\Sigma_1)\mathrm{tr}(\bar\Sigma_2)}{n_1n_2} + \frac{\mathrm{tr}^2(\bar\Sigma_1)}{n_1^2},
\end{align*}
since 
\begin{align*}
    &\frac{1}{n_2}\mathbb{E}_{P^2| P^1}\left[\frac{1}{n_2}\sum_{i \in I_2} \widehat \xi_i^2\right] + \widehat{\mathbb{C}}_2^2 \\
    &\quad\le \frac{4\mathrm{tr}(\bar\Sigma_2)}{n_2}\|\widehat \theta_1-\theta(P^N)\|^2 + \|\widehat \theta_1-\theta(P^N)\|^4 \le \max\{ 4\widetilde C_{\mathtt{init}}, \widetilde C_{\mathtt{init}}^2\}\widetilde s_{n_1, n_2}^2
\end{align*}
with probability greater than $1-\widetilde\varepsilon_{\mathtt{init}}$. Then the result follows with 
\begin{align*}
  \sqrt{\frac{\mathrm{tr}(\bar\Sigma_2)}{n_2}} +   \widetilde s_{n_1, n_2}^{1/2} &\le  \sqrt{\frac{\mathrm{tr}(\bar\Sigma_2)}{n_2}} +\left(\frac{\mathrm{tr}(\bar\Sigma_1)}{n_1}\right)^{1/4}\left(\frac{\mathrm{tr}(\bar\Sigma_2)}{n_2}\right)^{1/4} + \sqrt{\frac{\mathrm{tr}(\bar\Sigma_1)}{n_1}}\\
  &\lesssim \sqrt{\frac{\mathrm{tr}(\bar\Sigma_2)}{n_2}} + \sqrt{\frac{\mathrm{tr}(\bar\Sigma_1)}{n_1}},
\end{align*}
by AM-GM inequality. 
\end{proof}

\subsection{High-dimensional Misspecified Linear Regression}
% This section includes all proofs for the result presented in Section~\ref{sec:ols}.
\begin{proof}[\bfseries{Proof of \Cref{thm:LR-valid}}]
The proof is an application of \Cref{thm:coverage-anti-conservative-confidence-set-empirical-risk} for $\alpha = 1/2$ and \Cref{thm:studentized-katz} for $\alpha \neq 1/2$. We collect relevant values. First, 
\begin{align*}
        \widehat\xi_i &= (Y_i-\widehat \theta_1^\top X_i)^2 - (Y_i-\theta(P^N)^\top X_i)^2 - \mathbb{E}_{P_i}[(Y_i-\widehat \theta_1^\top X_i)^2 - (Y_i-\theta(P^N)^\top X_i)^2|D_1]\\
        & = 2\varepsilon_i (\theta(P^N)-\widehat\theta_1)^\top X_i - 2\E_{P_i}[\varepsilon_i (\theta(P^N)-\widehat\theta_1)^\top X_i |D_1]\\
        & \quad + \{(\theta(P^N)-\widehat \theta_1)^\top X_i\}^2-\E_{P_i}[ \{X_i^\top(\theta(P^N) - \widehat \theta_1)\}^2 |D_1].
    \end{align*}
    Under non-identically distributed settings, $\E_{P_i}[\varepsilon_i (\theta(P^N)-\widehat\theta_1)^\top X_i |D_1]$ may not be zero. 
Taking $G_i := 2\{X_i\varepsilon_i - \E_{P_i}[X_i\varepsilon_i]\}$, 
\begin{align*}
     \frac{\E_{P_i}[(\widehat \xi_i- \langle \widehat{\theta}_1 - \theta(P^N), G_i\rangle)^2  |D_1]}{\E_{P_i}[\langle \widehat{\theta}_1 - \theta(P^N), G_i\rangle^2 |D_1]}
     & \le \frac{\mathrm{Var}_{P_i}[\{(\theta(P^N)-\widehat \theta_1)^\top X_i\}^2 |D_1]}{4\E_{P_i}[\langle \widehat{\theta}_1 - \theta(P^N), \{X_i\varepsilon_i - \E_{P_i}[X_i\varepsilon_i]\}\rangle^2 |D_1]}\\
     & \le \frac{\E_{P_i}[\{(\theta(P^N)-\widehat \theta_1)^\top X_i\}^4 |D_1]}{4\sigma^2_i\E_{P_i}[\langle \widehat{\theta}_1 - \theta(P^N), X_i - \E_{P_i}[X_i]\rangle^2 |D_1]}\\
     & \le \frac{L^4\|\theta(P^N)-\widehat \theta_1\|_{\bar\Gamma_2}^2}{4\sigma^2_i\lambda_{\min}(\bar\Gamma_2^{-1/2}\Sigma_i \bar\Gamma_2^{-1/2})}.
\end{align*}
Then the requirement of \Cref{prop:linearization-prop} holds with 
\begin{align*}
    \varphi(\|\theta(P^N)-\widehat \theta_1\|_{\bar\Gamma_2}) = \frac{L^4\|\theta(P^N)-\widehat \theta_1\|_{\bar\Gamma_2}^2}{4\underline{\sigma}^2 \underline{\lambda}}.
\end{align*}
Denoting $\mathbb{V}_G = \sum_{i \in I_2}\E_{P_i}[\langle u, X_i\varepsilon_i - \E_{P_i}[X_i\varepsilon_i]\rangle^2]$, we have
\begin{align*}
    &\sum_{i\in I_2} \mathbb{E}_{P_i}\left[\frac{|\langle u, X_i\varepsilon_i - \E_{P_i}[X_i\varepsilon_i]\rangle|^2}{\mathbb{V}_G}\min\left\{1,\,\frac{|\langle u, X_i\varepsilon_i - \E_{P_i}[X_i\varepsilon_i]\rangle|}{\mathbb{V}_G^{1/2}}\right\}\right]\\
    &\quad = \sup_{u \in \mathbb{S}^{d-1}}\sum_{i\in I_2} \mathbb{E}_{P_i}\left[\frac{\langle W_i^\circ, u\rangle^2}{n_2}\min\left\{1,\,\frac{|\langle W_i^\circ, u\rangle|}{\sqrt{n_2}}\right\}\right],
\end{align*}
where
\begin{equation*}
        \bar H_2 = \frac{1}{n_2}\sum_{i \in I_2} \mathrm{Cov}_{P_i}(X_i \epsilon_i) \quad \textrm{and} \quad W_i^\circ = \bar H_2^{-1/2}(X_i\epsilon_i - \E_{P_i}[X_i\epsilon_i]).
    \end{equation*}
The result for $\alpha \neq 1/2$ is concluded by applying \Cref{prop:linearization-prop}.

For $\alpha = 1/2$, we have an additional result from \Cref{thm:coverage-anti-conservative-confidence-set-empirical-risk}. 
    % Under non-identically distributed settings, $\E_{P_i}[\varepsilon_i (\theta(P^N)-\widehat\theta_1)^\top X_i |D_1]$ may not be zero. Denoting $\widehat{\theta}_1 - \theta(P^N) = \delta_1 u$ with $u \in \mathbb{S}^{d-1}$, we obtain 
    % \begin{align*}
    %     \E_{P_i}[\widehat{\xi}_i |D_1]= \delta_1^2\E_{P_i}[\{X_i^\top u\}^2|D_1]
    % \end{align*}
    We obtain 
    \begin{align*}
        \mathrm{Var}_{P_i}[\widehat{\xi}_i |D_1] \le 4\E_{P_i}[\{\varepsilon_i (\theta(P^N)-\widehat\theta_1)^\top X_i\}^2|D_1] + 2\E_{P_i}[\{(\theta(P^N)-\widehat \theta_1)^\top X_i\}^4|D_1]
    \end{align*}
    and hence 
    \begin{align*}
        \widehat{\mathbb{V}}_2 &=\frac{1}{n_2^2}\sum_{i \in I_2} \mathrm{Var}_{P_i}[\widehat{\xi}_i |D_1] \\
        & \le \frac{4\|\theta(P^N)-\widehat\theta_1\|_{\bar\Gamma_2}^2\bar\sigma^2}{n_2^2}\sum_{i \in I_2}\E_{P_i}[\{u^\top\bar\Gamma_2^{-1/2} X_i\}^2|D_1] \\
        &\quad+ \frac{2\|\theta(P^N)-\widehat\theta_1\|_{\bar\Gamma_2}^4}{n_2^2}\sum_{i \in I_2}\E_{P_i}[\{u^\top \bar\Gamma_2^{-1/2}X_i\}^4|D_1] \\
        & \le \frac{4\|\theta(P^N)-\widehat\theta_1\|_{\bar\Gamma_2}^2\bar\sigma^2}{n_2} + \frac{2\|\theta(P^N)-\widehat\theta_1\|_{\bar\Gamma_2}^4 L^4}{n_2}.
    \end{align*}
    % For $\tau > 0$, consider an event where 
    % \begin{align*}
    %     \Omega := \left\{\|\widehat{\theta}_1 - \theta(P^N)\| \ge \tau^2 \frac{\sum_{i \in I_2}\E_{P_i}[\{\varepsilon_iu^\top X_i\}^2|D_1]}{\sum_{i \in I_2}\E_{P_i}[\{X_i^\top u\}^4|D_1]}\right\},
    % \end{align*}
    % then on such an event we have 
    % \begin{align*}
    %     \delta_1^2 \ge \tau^2 \frac{\sum_{i \in I_2}\E_{P_i}[\{\varepsilon_iu^\top X_i\}^2|D_1]}{\sum_{i \in I_2}\E_{P_i}[\{X_i^\top u\}^4|D_1]} \Rightarrow \frac{1}{\delta_1^2} \le \frac{\sum_{i \in I_2}\E_{P_i}[\{X_i^\top u\}^4|D_1]}{\tau^2 \sum_{i \in I_2}\E_{P_i}[\{\varepsilon_iu^\top X_i\}^2|D_1]}.
    % \end{align*}
    % \begin{align*}
    %     \frac{\sum_{i \in I_2}\E_{P_i}[\{\varepsilon_iu^\top X_i\}^2|D_1]}{\sum_{i \in I_2}\E_{P_i}[\{X_i^\top u\}^4|D_1]} &= \frac{\|\bar\Gamma_2^{-1/2}\|^2\sum_{i \in I_2}\E_{P_i}\sigma_i^2[\{(\bar\Gamma_2^{-1/2}t)^\top X_i\}^2|D_1]}{\sum_{i \in I_2}\E_{P_i}[\{X_i^\top (\bar\Gamma_2^{-1/2}t)\}^4|D_1]}\\
    %     &\ge \frac{\|\bar\Gamma_2^{-1/2}\|^2\underline{\sigma}^2}{L^4}
    % \end{align*}
    % This in term provides the upper bound on the variance:
    % \begin{align*}
    %     \sum_{i \in I_2} \mathrm{Var}_{P_i}[\widehat{\xi}_i |D_1]  \le (8/\tau^2 +2)\delta_1^4\sum_{i \in I_2}\E_{P_i}[\{X_i^\top u\}^4|D_1].
    % \end{align*}
    Meanwhile, the curvature is given by 
    \begin{align*}
        \widehat{\mathbb{C}}_2 
        &= \frac{2}{n_2}\sum_{i \in I_2}\E_{P_i}[\varepsilon_i (\theta(P^N)-\widehat\theta_1)^\top X_i |D_1] + \frac{1}{n_2}\sum_{i \in I_2}\E_{P_i}[\{X_i^\top(\theta(P^N) - \widehat \theta_1)\}^2 |D_1]\\
        &= \frac{1}{n_2}\sum_{i \in I_2}\E_{P_i}[\{X_i^\top(\theta(P^N) - \widehat \theta_1)\}^2|D_1] =\|\theta(P^N)-\widehat\theta_1\|_{\bar\Gamma_2}^2,
    \end{align*}
    where the second equality follows even under non-identically distributed settings. Finally, the following ratio is bounded from below as
    \begin{align*}
        \ratio^2 &\ge \frac{n_2\|\theta(P^N)-\widehat\theta_1\|_{\bar\Gamma_2}^4}{4\|\theta(P^N)-\widehat\theta_1\|_{\bar\Gamma_2}^2\bar\sigma^2 + 2\|\theta(P^N)-\widehat\theta_1\|_{\bar\Gamma_2}^4 L^4}= \frac{n_2\|\theta(P^N)-\widehat\theta_1\|_{\bar\Gamma_2}^2}{4\bar\sigma^2 + 2L^4\|\theta(P^N)-\widehat\theta_1\|_{\bar\Gamma_2}^2} = \widetilde\Delta_2^2.
    \end{align*}
    % Furthermore, as this is an increasing function in $\|\theta(P^N)-\widehat\theta_1\|_{\bar\Gamma_2}^2$, if we have a lower bound $\|\theta(P^N)-\widehat\theta_1\|_{\bar\Gamma_2}^2 \ge \tau^2$, we have
    % \begin{align*}
    %     \ratio^2 \ge \frac{n_2\tau ^2}{4\bar\sigma^2 + 2L^4\tau^2 }.
    % \end{align*}
The remainder term in \Cref{thm:coverage-anti-conservative-confidence-set-empirical-risk} becomes 
    \begin{align*}
    &\E_{P^1} \left[\min\left\{1, C\sum_{i\in I_2}\mathbb{E}_{P_i}\left[\frac{|\widehat\xi_i|^2}{n_2^2\widehat{\mathbb{V}}_{2}(1 + \ratio)^2}\min\left\{1,\,\frac{|\widehat \xi_i|}{n_2\widehat{\mathbb{V}}_{2}^{1/2}(1 + \ratio)}\right\} \bigg|D_1\right]\right\}\right]\\
    &\quad \le\E_{P^1} \left[\min\left\{1,  \frac{C}{(1 + \widetilde\Delta_2)^2} \right\}\right].
\end{align*}
This concludes the claim.
\end{proof}

\begin{proof}[\bfseries{Proof of \Cref{thm:LR-ci-width}}]
    The result is a direct consequence of \Cref{lemma:lr-direct-calculation} and \Cref{lemma:LR-ci-width-full}
\end{proof}
\begin{lemma}\label{lemma:lr-direct-calculation}
    Assume \ref{as:eigen_value} and \ref{as:linreg-moment-covariate} with $q_x>2$. For $\alpha = 1/2$, $n_1 \ge 1$, and any $\varepsilon \in (0,1)$,
    \begin{equation*}
        \mathbb{P}_{P^N}\left(\mathrm{Diam}_{\|\cdot\|_{ \Gamma_2}}\bigl(\widehat{\mathrm{CI}}_{N,1/2}^{\mathtt{CLT}}\bigr)
    \leq 4\varepsilon^{-1/2}
    \left\{\sqrt{\frac{\bar\sigma^2 d}{n_2}} 
    + \|\widehat{\theta}_1 - \theta(P^N)\|_{\Gamma_2}\right\}\right)\ge 1-\varepsilon - \exp(-Cn_2),
    \end{equation*}
    provided that $n_2$ satisfies $n_2 \ge \mathfrak{C}d$ where $\mathfrak{C}$ depends only on $q_x$, and $C$ is a universal constant. 
\end{lemma}

\begin{proof}[\bfseries{Proof of \Cref{lemma:lr-direct-calculation}}]
    When $\alpha = 1/2$, that is $z_\alpha = 0$, \Cref{thm:geometry-LR} establishes that the diameter of the confidence set can be computed exactly, in terms of $\|\cdot\|_{\widehat{\Gamma}_2}$, 
    \begin{align*}
        2\|\widehat{H}\|_{\widehat\Gamma_2} \quad \textrm{where} \quad \widehat{H} = \widehat{\theta}_1 - \theta_{\mathtt{OLS}}.
    \end{align*}
    Here, $\theta_{\mathtt{OLS}}$ is the OLS based on $D_2$. Throughout, we assume $\widehat \Gamma_2$ is invertible, that is $n_2 \ge d$. By the triangle inequality, we have 
    \begin{align*}
        \|\widehat{H}\|_{\widehat\Gamma_2}  \le \|\widehat{\theta}_1 - \theta(P^N)\|_{\widehat\Gamma_2}  + \|\theta_{\mathtt{OLS}}- \theta(P^N)\|_{\widehat\Gamma_2}.
    \end{align*}
    We observe 
    \begin{align*}
         \theta_{\mathtt{OLS}} - \theta(P^N) &= \widehat\Gamma_2^{-1}\left(\frac{1}{n_2}\sum_{i\in I_2}X_i  Y_i\right) -\theta(P^N) \\
         &= \widehat\Gamma_2^{-1}\left(\frac{1}{n_2}\sum_{i\in I_2}X_i  (X^\top \theta(P^N) + \varepsilon_i)\right) -\theta(P^N)\\
         &= \widehat\Gamma_2^{-1}\left(\frac{1}{n_2}\sum_{i\in I_2}X_i  \varepsilon_i\right).
    \end{align*}
    Hence, using \ref{as:eigen_value}, 
    \begin{align*}
        \E_{P^N}[\| \theta_{\mathtt{OLS}} - \theta(P^N)\|^2_{\widehat\Gamma_2}] &= \E_{P^N}\left[\left\|\widehat\Gamma_2^{-1}\left(\frac{1}{n_2}\sum_{i\in I_2}X_i  \varepsilon_i\right)\right\|_{\widehat\Gamma_2}\right]\\
        &\le \bar\sigma^2\E_{P^N}\left[\left\|\widehat\Gamma_2^{-1}\left(\frac{1}{n_2}\sum_{i\in I_2}X_i \right)\widehat\Gamma_2^{1/2}\right\|_2^2\right]\\
        &= \frac{\bar\sigma^2}{n_2}\E_{P^N}\left[\mathrm{tr}\left(\widehat\Gamma_2^{-1}\left(\frac{1}{n_2}\sum_{i\in I_2}X_iX_i^\top \right)\right)\right] \le \frac{d\bar\sigma^2}{n_2}.
    \end{align*}
    Also by linearity of expectation,
    \[\mathbb{E}_{P^2}[\|\widehat{\theta}_1 -\theta(P^N)\|_{\widehat\Gamma_2}| D_1] = \|\widehat{\theta}_1 -\theta(P^N)\|_{\Gamma_2}.\]
    Denoting
    \begin{align*}
        \mathrm{R}_N = \sqrt{\frac{d\bar\sigma^2}{n_2}} + \|\widehat{\theta}_1 -\theta(P^N)\|_{\Gamma_2},
    \end{align*}
    it follows from Markov's inequality, 
    \begin{align*}
        &\mathbb{P}_{P^N}\left(2\|\widehat H\|_{\widehat{\Gamma}} \ge C\varepsilon^{-1/2}\mathrm{R}_N\right) \\
        &\quad \le \frac{4\varepsilon}{C^2}\mathbb{E}_{P^N}\left[\frac{\|\widehat H\|^2_{\widehat{\Gamma}}}{\mathrm{R}_N^2}\right] \le \frac{8\varepsilon}{C^2}\mathbb{E}_{P^N}\left[\frac{(d\bar\sigma^2/n_2) + \|\widehat{\theta}_1 -\theta(P^N)\|_{\Gamma_2}^2}{\mathrm{R}_N^2}\right] \le \varepsilon
    \end{align*}
    where $C = 2\sqrt{2}$. Finally, since
    \begin{align*}
        \|\widehat H\|^2_{\widehat{\Gamma}}  = \widehat H^{\top}\widehat{\Gamma}\widehat H \ge  \lambda_{\min}(\Gamma^{-1/2}\widehat{\Gamma}\Gamma^{-1/2}) \cdot \|\widehat H\|^2_{\Gamma},
    \end{align*}
    it remains to bound $\lambda_{\min}(\Gamma^{-1/2}\widehat{\Gamma}\Gamma^{-1/2})$ from below. Note that $\Gamma^{-1/2}\widehat{\Gamma}\Gamma^{-1/2} = n_2^{-1}\sum_{i \in I_2}\widetilde X_i\widetilde X_i^\top$ where $\widetilde X_i = \bar\Gamma^{-1/2}X_i$ satisfies $n_2^{-1}\sum_{i \in I_2}\E_{P_i}[\widetilde X_i\widetilde X_i^\top] = I_d$ by construction. Theorem 1.3 of \citet{Koltchinskii2015Bounding} implies that \ref{as:linreg-moment-covariate} with $q_x > 2$, whenever $d/n_2 \le \mathfrak{C}$ for a universal constant $\mathfrak{C}$ with probability at least $1-\exp(-C'n_2)$,
    \begin{align*}
        \lambda_{\min}(\Gamma^{-1/2}\widehat{\Gamma}\Gamma^{-1/2}) \ge 1/2 \implies \|\widehat H\|_{\widehat{\Gamma}} \ge  \|\widehat H\|_{\Gamma}/\sqrt{2}.
    \end{align*}
    Therefore on this event, we have 
    \begin{align*}
        \mathbb{P}_{P^N}\left(2\|\widehat H\|_{\Gamma} \ge \sqrt{2}C\varepsilon^{-1/2}\mathrm{R}_N\right)&\le \mathbb{P}_{P^N}\left(2\|\widehat H\|_{\widehat{\Gamma}} \ge C\varepsilon^{-1/2}\mathrm{R}_N\right) +\exp(-C'n_2)\\
        &\le \varepsilon + \exp(-C'n_2).
    \end{align*}
   This concludes the claim.
\end{proof}
\begin{lemma}\label{lemma:LR-ci-width-full}
    Assume \ref{as:eigen_value}, \ref{as:linreg-moment-covariate} with $q_x\ge 4$, and \ref{as:error_moment} and let 
$\widetilde{s}_{n_1,n_2}$ be as in \ref{as:rate-initial-estimator3}. For $\alpha \neq 1/2$, $n_1 \ge 1$, and any $\varepsilon > 0$, setting $\varepsilon^\circ = \varepsilon +\widetilde\varepsilon_{\mathtt{init}}$, it holds that 
    \begin{equation*}
        \mathbb{P}_{P^N}\left(\mathrm{Diam}_{\|\cdot\|_2}\bigl(\widehat{\mathrm{CI}}_{N,\alpha}^{\mathtt{CLT}}\bigr)
    \leq C_{\varepsilon^\circ}(1+|z_{\alpha}|)\left\{\sqrt{\frac{\overline{\sigma}^2 d}{n_2}} + \widetilde s_{n_1, n_2}^{1/2}\right\}\right) \ge 1-\varepsilon^\circ,
    \end{equation*}
provided that $n_2$ satisfies 
\begin{equation*}
    C'_{\varepsilon^\circ} \max\bigg\{((1+|z_\alpha|) d\log (2d) L^4)^{q_x/(q_x-2)}, ((1+|z_\alpha|)(1+K) dL^2)^{p/(p-1)}\bigg\} \le n_2.
\end{equation*}
where $p = \min\{q_y, q_x/2\}$, and $C_{\varepsilon^\circ}, C'_{\varepsilon^\circ}$ depend on $\varepsilon^\circ$, but not on $d$ or $\alpha$.
\end{lemma}
\begin{proof}[\bfseries{Proof of \Cref{lemma:LR-ci-width-full}}]The proof is a direct application of \Cref{thm:clt-width-local-full} and thus proceeds by verifying
\ref{as:margin}, \ref{as:local-entropy}, and \ref{as:square-process} to hold locally and \ref{as:margin-global}, \ref{as:ratio-process}, and \ref{as:ratio-square-process} to hold globally. 

\paragraph{Verifying \ref{as:margin}}For any $\theta \in \Theta$, 
\begin{align*}
    m_\theta - m_{\theta(P^N)} &= (Y_i-\theta^\top X_i)^2 - (Y_i-\theta(P^N)^\top X_i)^2 \\
        & = 2\varepsilon_i (\theta(P^N)-\theta)^\top X_i + \{(\theta(P^N)-\theta)^\top X_i\}^2,
\end{align*}
and 
\begin{align*}
    &\frac{1}{n_2}\sum_{i \in I_2}\E_{P_i}\left[2\varepsilon_i (\theta(P^N)-\theta)^\top X_i \right]+ \frac{1}{n_2}\sum_{i \in I_2}\E_{P_i}\left[\{(\theta(P^N)-\theta)^\top X_i\}^2\right] \\
    &\quad= \frac{1}{n_2}\sum_{i \in I_2}\E_{P_i}\left[\{(\theta(P^N)-\theta)^\top X_i\}^2\right]\\
    &\quad= (\theta(P^N)-\theta)^\top\left(\frac{1}{n_2}\sum_{i \in I_2}\E_{P_i}\left[X_iX_i^\top \right]\right)(\theta(P^N)-\theta) = \|\theta(P^N)-\theta\|^2_{\bar\Gamma_2}.
\end{align*}
Thus \ref{as:margin} holds with $\gamma = 1$ and $c_0 = 1$.
\paragraph{Verifying \ref{as:local-entropy}} For any $\theta$ such that $\|\theta-\theta(P^N)\|_{\bar\Gamma_2} \le \delta$, we have
\begin{align*}
    &\sup_{\|\theta-\theta(P^N)\|_{\bar\Gamma_2}\le \delta} |(Y-\theta^\top X)^2 - (Y-\theta(P^N)^\top X)^2| \\
    &\quad \le \sup_{\|\theta-\theta(P^N)\|_{\bar\Gamma_2}\le \delta}\left\{2|\varepsilon(\theta(P^N)^\top X-\theta^\top X) | + |(\theta(P^N)^\top X-\theta^\top X)^2|\right\}\\
    &\quad \le \sup_{\|\theta-\theta(P^N)\|_{\bar\Gamma_2}\le \delta}\left\{2|(\theta(P^N)-\theta)^\top \bar\Gamma_2^{1/2}\bar\Gamma_2^{-1/2}\varepsilon X |\right\} \\
    &\quad\quad+ \sup_{\|\theta-\theta(P^N)\|_{\bar\Gamma_2}\le \delta}\left\{|(\theta(P^N)-\theta)^\top \bar\Gamma_2^{1/2}\bar\Gamma_2^{-1/2}X|^2\right\}.
\end{align*}
By Cauchy-Schwarz, for any $\|\theta-\theta(P^N)\|_{\bar\Gamma_2} \le \delta$, 
\begin{align*}
    |(\theta(P^N)-\theta)^\top \bar\Gamma_2^{1/2}\bar\Gamma_2^{-1/2}\varepsilon X | &\le |\varepsilon| \cdot \|(\theta(P^N)-\theta)^\top \bar\Gamma_2^{1/2}\|_2\cdot\|\bar\Gamma_2^{-1/2} X\|_2 \\
    &\le |\varepsilon|\cdot\|\theta(P^N)-\theta\|_{\bar\Gamma_2}\cdot\|\bar\Gamma_2^{-1/2} X\|_2\\
    &\le \delta|\varepsilon|\|\bar\Gamma_2^{-1/2} X\|_2, 
\end{align*}
and 
\begin{align*}
    &|(\theta(P^N)-\theta)^\top \bar\Gamma_2^{1/2}\bar\Gamma_2^{-1/2}X|^2 \le  \|(\theta(P^N)-\theta)^\top \bar\Gamma_2^{1/2}\|_2^2\cdot \|\bar\Gamma_2^{-1/2}X\|_2^2 \le  \delta^2\|\bar\Gamma_2^{-1/2}X\|_2^2.
\end{align*}
For a measurable function $f$, we denote 
\begin{align*}
    (\mathbb{P}_{n_2}-P^2) f = \frac{1}{n_2}\sum_{i \in I_2} f(X_i, Y_i) - \E_{P_i}[f(X_i, Y_i)].
\end{align*}
Then we have 
\begin{align*}
    &\E_{P^2} \left[\sup_{\|\theta-\theta(P^N)\|_{\bar\Gamma_2} < \delta}|(\widehat{\mathbb{M}}_2 - \mathbb{M}_2)(\theta) - (\widehat{\mathbb{M}}_2 - \mathbb{M}_2)(\theta(P^N))| \right]\\
    &\quad = \E_{P^2} \left[\sup_{\|\theta-\theta(P^N)\|_{\bar\Gamma_2} < \delta}|(\mathbb{P}_{n_2}-P^2)(m_{\theta} - m_{\theta(P^N)})| \right]\\
    &\quad \le 2\delta \E_{P^2} \left[\|\bar\Gamma_2^{-1/2}(\mathbb{P}_{n_2}-P^2) [\varepsilon X]\| \right] + \delta^2 \E_{P^2} \left[\|\bar\Gamma_2^{-1/2} (\mathbb{P}_{n_2}-P^2)[XX^\top ]\bar\Gamma_2^{-1/2}\|_{\mathrm{op}} \right] \\
    &\quad = \mathbf{I} + \mathbf{II},
\end{align*}
where $\|\cdot\|_{\mathrm{op}}$ denotes the operator norm.
\Cref{lemma:linear-empirical-process} and \Cref{lemma:operator-empirical-process} at the end of this section prove that 
\begin{align*}
    \mathbf{I} &\le 2\delta \sqrt{\frac{d\bar\sigma^2 }{n_2}} \quad \text{and}\quad 
    \mathbf{II}\le \left\{4\sqrt{\frac{d\log(2d)L^4}{n_2}} + 32\sqrt{2}\log(2d)\left(\frac{dL^2}{n_2^{1-2/q_x}}\right)\right\}\delta^2,
\end{align*}
under \ref{as:linreg-moment-covariate} and \ref{as:eigen_value}. 

Notice that this derivation does not satisfy $q < 1+\gamma$ globally. Assume that $\delta \le \rho$, then
\begin{align*}
    \phi_{n_2}(\delta) = \left( 2 \sqrt{\frac{d\bar\sigma^2 }{n_2}} + \left\{4\sqrt{\frac{d\log(2d)L^4}{n_2}} + 32\sqrt{2}\log(2d)\left(\frac{dL^2}{n_2^{1-2/q_x}}\right)\right\}\rho \right)\delta,
\end{align*}
hence we can take $q=1$ and $q < 1+\gamma$ is now satisfied locally.
\paragraph{Verifying \ref{as:square-process}}
We begin by analyzing $\omega_{\mathtt{pop}}(\cdot)$. It follows that 
\begin{align*}
    &(m_\theta(X_i, Y_i) - m_{\theta(P)}(X_i, Y_i))^2 \\
    &\quad \le 4|(\theta(P^N)-\theta)^\top \bar\Gamma_2^{1/2}\bar\Gamma_2^{-1/2}\varepsilon X |^2 + 2|(\theta(P^N)-\theta)^\top \bar\Gamma_2^{1/2}\bar\Gamma_2^{-1/2}X|^4,
\end{align*}
and hence 
\begin{align*}
    &\sup_{\|\theta-\theta(P^N)\|_{\bar\Gamma_2}\le \delta}\, \frac{1}{n_2}\sum_{i\in I_2}\E_{P_i}(m_\theta - m_{\theta(P)})^2(X_i, Y_i) \\
    & \quad \le  \frac{4\delta^2}{n_2}\sum_{i\in I_2}\E_{P_i}\left[|\varepsilon_i|^2\cdot|v\bar\Gamma_2^{-1/2} X |^2 \right] + \frac{2\delta^4}{n_2}\sum_{i\in I_2}\E_{P_i}\left[|v\bar\Gamma_2^{-1/2}X|^4 \right]\\
    & \quad \le  \frac{2\delta^2\bar\sigma^2}{n_2}\sum_{i\in I_2}v^\top \E_{P_i}\left[\bar\Gamma_2^{-1/2} X_iX_i^\top \bar\Gamma_2^{-1/2} \right]v + \frac{2\delta^4}{n_2}\sum_{i\in I_2}L^4\\
    &\quad = 2\delta^2 \bar\sigma^2 + 2\delta^4 L^4.
\end{align*}
When $\delta \le\rho$, we have 
\begin{align*}
    2\delta^2 \bar\sigma^2 + 2\delta^4 L^4 \le \delta^2 \left(2\bar\sigma^2 + 2\rho^2 L^4\right) =\omega_{\mathtt{pop}}^2(\delta),
\end{align*}
which satisfies $q=1$ and $q < 1+\gamma$ locally.

Next, we derive $\omega_{n_2, \mathtt{emp}}(\cdot)$. By \Cref{prop:squared-Gn} with  $q=p =\min\{q_y, q_x/2\}$, we have 
\begin{align*}
    \E_{P^2}[|M_{\delta}|^p] \le 2 \cdot 2^{p-1}\delta^p\E_{P^2}[|\varepsilon|^p\|\bar\Gamma_2^{-1/2} X\|_2^p] + 2^{p-1}\delta^{2p}\E_{P^2}[\|\bar\Gamma_2^{-1/2}X\|_2^{2p}].
\end{align*}
We observe that 
\begin{align*}
    \E_{P^2}[|\varepsilon|^p\|\bar\Gamma_2^{-1/2} X\|_2^p] &= \E_{P^2}[\E_{P^2}[|\varepsilon|^p|X]\|\bar\Gamma_2^{-1/2} X\|_2^p]\\
    % &\le \E_{P^2}[|\varepsilon|^p\|\bar\Gamma_2^{-1/2} X\|_2^p] \\
    &\le \E_{P^2}[\left(\E_{P^2}[|\varepsilon|^{q_y}|X]\right)^{p/q_y}\|\bar\Gamma_2^{-1/2} X\|_2^p]\\
    &\le K^p\E_{P^2}[\|\bar\Gamma_2^{-1/2} X\|_2^p]\\
    &\le(\E_{P^2}[\|\bar\Gamma_2^{-1/2} X\|_2^{q_x}])^{p/q_x} \\
    &\le K^pd^{p/2}L^{p}
\end{align*}
assuming $p \le q_y$. Furthermore, 
\begin{align*}
    \E_{P^2}[\|\bar\Gamma_2^{-1/2}X\|_2^{2p}] &\le \left(\E_{P^2}[\|\bar\Gamma_2^{-1/2}X\|_2^{q_x}] \right)^{(2p)/q_x} \le d^p L^{2p}
\end{align*}
assuming $2p \le q_x$. Hence by \Cref{prop:squared-Gn}, 
\begin{align*}
    &\E_{P^2}^* \left[\sup_{\|\theta-\theta(P^N)\| < \delta}\left|\frac{1}{n_2}\sum_{i\in I_2}(m_\theta - m_{\theta(P^N)})^2 - \mathbb{E}_{P_i}[(m_\theta - m_{\theta(P^N)})^2]\right| \right] \\
    &\quad \le 16\, n_2^{2/p-1}(2^{2/p}\cdot 2^{2-2/p}\delta^2 K^2d L^2 + 2^{2-2/p}\delta^4 d^2L^4) \\
    &\quad\quad+ 8\cdot 8^{1/p}\, n_2^{1/p} (2^{1/p}\cdot 2^{1-1/p}\delta K d^{1/2} L + 2^{1-1/p}\delta^2 d L^2) (\mathbf{I} + \mathbf{II}) \\
    &\quad= A_2 \delta^2 + A_3 \delta^3 +A_4 \delta^4
\end{align*}
for $p = \min\{q_y, q_x/2\}$ where 
\begin{align*}
    A_2 &= C_1 n_2^{2/p-1}K^2 dL^2 + n_2^{1/p}\sqrt{\frac{\bar\sigma^2}{n_2}}KdL \\
    A_3 &=C_2 \cdot n_2^{1/p} K d^{1/2} L \left\{4\sqrt{\frac{d\log(2d)L^4}{n_2}} + 32\sqrt{2}\log(2d)\left(\frac{dL^2}{n_2^{1-2/q_x}}\right)\right\}, \quad\text{and}\\
    A_4 &= C_3 n^{2/p-1}d^2 L^4 + 8\cdot 4\, n_2^{1/p} dL^2\left\{4\sqrt{\frac{d\log(2d)L^4}{n_2}} + 32\sqrt{2}\log(2d)\left(\frac{dL^2}{n_2^{1-2/q_x}}\right)\right\}.
\end{align*}
When $\delta \le \rho$, we can set $(A_2 + A_3 \rho  +A_4 \rho^2) \delta^2$, which satisfies $q=1$ and $q < 1+\gamma$ locally.

% Hence, by \Cref{prop:squared-Gn}, 
% \begin{align*}
%     &\E_{P^2}^* \left[\sup_{\|\theta-\theta(P^N)\| < \delta}\left|\frac{1}{n_2}\sum_{i\in I_2}(m_\theta - m_{\theta(P^N)})^2 - \mathbb{E}_{P_i}[(m_\theta - m_{\theta(P^N)})^2]\right| \right] \\
%     &\quad \le 16\, (2\delta^2 \bar\sigma^2d + 2\delta^4 d^2L^4) \\
%     &\quad\quad+ 8\cdot 8^{1/2}\, n_2^{1/2} (2\delta^2 \bar\sigma^2d + 2\delta^4 d^2L^4)^{1/2} (\mathbf{I} + \mathbf{II}).
% \end{align*}

\paragraph{Verifying \ref{as:margin-global}}
\ref{as:margin} holds globally and thus one can choose 
\begin{align*}
    C_{\rho}(\|\theta-\theta(P^N)\|) = \|\theta-\theta(P^N)\|_{\bar\Gamma_2}^2.
\end{align*}
\paragraph{Verifying \ref{as:ratio-process}} 
Next, we have
\begin{align*}
    &\E^*_{P^2} \left[\sup_{\|\theta-\theta(P^N)\| > \rho}\left|\frac{(\widehat{\mathbb{M}}_2 - \mathbb{M}_2)(\theta) - (\widehat{\mathbb{M}}_2 - \mathbb{M}_2)(\theta(P^N))}{\mathbb{M}_2(\theta) -\mathbb{M}_2(\theta(P^N))}\right| \right] \\
    &\quad \le 2\rho^{-1} \sqrt{\frac{d\bar\sigma^2 }{n_2}} + \left\{4\sqrt{\frac{d\log(2d)L^4}{n_2}} + 32\sqrt{2}\log(2d)\left(\frac{dL^2}{n_2^{1-2/q_x}}\right)\right\} = R(n_2, \rho).
\end{align*}
Hence, \ref{as:ratio-process} holds with $C_{\mathtt{ratio}}=1/\varepsilon_{\mathtt{ratio}}$ by Markov's inequality.

\paragraph{Verifying \ref{as:ratio-square-process}} Similarly for any $\delta > \rho$,
\begin{align*}
    &\E_{P^2}^* \left[\sup_{\|\theta-\theta(P^N)\| > \rho}\frac{z_\alpha^2}{n_2^2}\left|\frac{\sum_{i\in I_2}(m_\theta - m_{\theta(P^N)})^2 - \mathbb{E}_{P_i}[(m_\theta - m_{\theta(P^N)})^2]}{\{\mathbb{M}_2(\theta) -\mathbb{M}_2(\theta(P^N))\}^2}\right| \right]\\
    &\quad \le \frac{A_2(1+z_\alpha^2)}{\rho^2 n_2} + \frac{A_3(1+z_\alpha^2)}{\rho n_2}+ \frac{A_4(1+z_\alpha^2)}{n_2} = S_{\mathtt{emp}}(n_2, \rho, \alpha),
\end{align*}
and 
\begin{align*}
    &\sup_{\|\theta-\theta(P^N)\| > \rho}\frac{z_\alpha^2\sum_{i\in I_2}\mathbb{E}_{P_i}[(m_\theta - m_{\theta(P^N)})^2(Z_i)]}{n_2^2\{\M_2(\theta)-\M_2(\theta(P^N))\}^2} \le \frac{(1+z_\alpha^2)}{n_2}\left( 2\rho^{-2} \bar\sigma^2 d+ 2 L^4\right) = S_{\mathtt{pop}}(n_2, \rho, \alpha).
\end{align*}
Hence, \ref{as:ratio-square-process} holds with $\widetilde C_{\mathtt{emp}}=1/\varepsilon_{\mathtt{emp}}$ by Markov's inequality.

We assume that $n_2$ is large enough satisfies the following:
\begin{align}\label{eq:eqreuiment_for_n2}
    C_{\varepsilon^\circ}\max\{R(n_2, \rho), S_{\mathtt{emp}}(n_2, \rho, \alpha), S_{\mathtt{pop}}(n_2, \rho, \alpha)\} \le 1/3,
\end{align}
where $C_{\varepsilon^\circ}$ is a constant depending on $\varepsilon^\circ$. We choose $\rho = C'_{\varepsilon^\circ}\sqrt{\bar\sigma^2 d/n_2}$ with $C'_{\varepsilon^\circ}$ being a sufficiently large constant depending on $\varepsilon^\circ$. Then \eqref{eq:eqreuiment_for_n2} is satisfied if
\begin{equation}\label{eq:eqreuiment_for_n2_simplified}
    \begin{split}
        &(1+|z_\alpha|)\max\left\{\frac{d\log(2d) L^4}{n_2^{1-2/q_x}}, \frac{(1+K) dL^2}{n_2^{1-1/p}} \right\} \le C_{\varepsilon^\circ} \\
& \quad \Leftrightarrow C''_{\varepsilon^\circ} \max\left\{((1+|z_\alpha|) d\log (2d) L^4)^{q_x/(q_x-2)}, ((1+|z_\alpha|)(1+K) dL^2)^{p/(p-1)}\right\} \le n_2.
    \end{split}
\end{equation}
\paragraph{Evaluating the rate of convergence}
We now evaluate the rate of convergence by applying \Cref{thm:clt-width} and \Cref{thm:clt-width-local-full}. Denote $\widetilde C_{\varepsilon^\circ}$ a constant depending on $\varepsilon^\circ$ that changes from line to line. Then with the choice of $\rho$ and $n_2$ satisfying \eqref{eq:eqreuiment_for_n2_simplified}, 
\begin{align*}
    &r_{n_2}^{-2}\phi_{n_2}(r_n) \le 1  \Leftrightarrow r_{n_2}^{-1}\widetilde C_{\varepsilon^\circ}\sqrt{\frac{d\bar\sigma^2 }{n_2}}  \le 1\Leftrightarrow \widetilde C_{\varepsilon^\circ}\sqrt{\frac{d\bar\sigma^2 }{n_2}} \le r_{n_2}.
\end{align*}
Next, we evaluate the value related to $\omega_{\mathtt{pop}}$. Under the choice of $\rho$ and $n_2$ satisfying \eqref{eq:eqreuiment_for_n2_simplified}, 
\begin{align*}
   u_{n_2}^{-4}\omega^2_{\mathtt{pop}}(u_{n_2}) \le n_2 
   & \Leftrightarrow u_{n_2}^{-2}  \left(2\bar\sigma^2\rho^{-2} + 2L^4\right) \rho^2\le n_2  \Leftrightarrow \frac{\widetilde C_{\varepsilon^\circ}\bar\sigma^2 d}{n_2}  \le u_{n_2}^2,
\end{align*}
where we used $\bar\sigma \le \bar\sigma d$ and $L^4/n_2 \le C_{\varepsilon^\circ}/(d\log (2d) n^{2/q_x}) \le C_{\varepsilon^\circ}$ by \eqref{eq:eqreuiment_for_n2_simplified}.

Finally, we evaluate the value related to $\omega_{n_2, \mathtt{emp}}$. From earlier derivation, we have 
\begin{align*}
    (A_2 + A_3 \rho  +A_4 \rho^2) \delta^2 = n_2 \rho^2\left(\frac{A_2}{n_2 \rho^2} + \frac{A_3}{n_2 \rho}  +\frac{A_4}{n_2}\right) \delta^2 \le  \widetilde C_{\varepsilon^\circ}\bar\sigma^2 d\delta^2 = \omega^2_{n_2, \mathtt{emp}}(\delta),
\end{align*}
under \eqref{eq:eqreuiment_for_n2_simplified}. Hence we conclude
\begin{align*}
   u_{n_2}^{-4}\omega^2_{\mathtt{pop}}(u_{n_2}) \le n_2 
   & \Leftrightarrow  \frac{\widetilde C_{\varepsilon^\circ}\bar\sigma^2 d}{n_2}  \le u_{n_2}^2.
\end{align*}
By \Cref{thm:clt-width}, we conclude 
\begin{align*}
    (1+|z_\alpha|)^{1+\gamma-q} \mathrm{R}_N^{\mathtt{CLT}} = C_{\varepsilon^\circ}(1+|z_\alpha|) \sqrt{\frac{\bar\sigma^2 d}{n_2}}+ (1+|z_\alpha|)\widetilde s_{n_1, n_2}^{1/2}. 
\end{align*}
Meanwhile, by \Cref{thm:clt-width-local-full}, we have 
\begin{align*}
\mathrm{Q}^{\mathtt{CLT}}_{N,\alpha} = (1+|z_\alpha|)^{1/2}\widetilde s^{1/2}_{n_1, n_2}
\end{align*}
in view of $C_{\rho}(\|\theta-\theta(P^N)\|) = \|\theta-\theta(P^N)\|^2$. Hence, we conclude that 
\begin{align*}
    &\max\{(1+|z_\alpha|)^{1/(1+\gamma-q)}\mathrm{R}^{\mathtt{CLT}}_{N}, \mathrm{Q}^{\mathtt{CLT}}_{N,\alpha}\mathbf{1}\{\mathrm{Q}^{\mathtt{CLT}}_{N,\alpha} \ge \rho\}\} \\
    &\quad \le C_{\varepsilon^\circ}(1+|z_\alpha|) \sqrt{\frac{\bar\sigma^2 d}{n_2}}  +   (1+|z_\alpha|)\widetilde s^{1/2}_{n_1, n_2},
\end{align*}
which concludes the result.
\end{proof}

For \Cref{cor:lr-plugin}, we instead prove the slightly rephrased version of the corollary. 
\begin{corollary}\label{cor:lr-plugin-full}
    Suppose the initial estimator satisfies, for all $n_1 \geq N_1$, \begin{align}\nonumber\mathbb{P}_{P^1}\left(\|\widehat \theta_1-\theta(P^N)\|^2_{\bar\Gamma_1} \le \widetilde C_{\mathtt{init}}\frac{d \bar\sigma^2}{n_1} \right) \ge 1-\widetilde\varepsilon_{\mathtt{init}}.
    \end{align}
    Assume \ref{as:eigen_value} and \ref{as:linreg-moment-covariate} with $q_x > 2$ when $\alpha = 1/2$ and $q_x \ge 4$ when $\alpha \neq 1/2$. Additionally assume \ref{as:error_moment} when $\alpha \neq 1/2$.
     For any $\varepsilon \in (0, 1-\widetilde \varepsilon_{\mathtt{init}})$, setting $\varepsilon^\circ = \varepsilon + \widetilde \varepsilon_{\mathtt{init}}$, $n_1 \ge N_1$, 
    \begin{align*}
\mathrm{Diam}_{\|\cdot\|_{\bar\Gamma_2}}\bigl(\widehat{\mathrm{CI}}_{N,\alpha}^{\mathtt{CLT}}\bigr)
    \leq C_{\varepsilon^\circ}(1+|z_{\alpha}|)
    \left\{\sqrt{\frac{\bar\sigma^2 d}{n_2}} 
    + \sqrt{\frac{\bar\sigma^2 d}{n_1}}\lambda^{1/2}_{\max}(\bar\Gamma_1^{-1/2}\bar\Gamma_2\bar\Gamma_1^{-1/2})\right\}
        \end{align*}
        with probability at least $1-\varepsilon^\circ - \exp(-Cn_2)$, 
        provided $n_2 \ge \mathfrak{C}d$ when $\alpha = 1/2$, and with probability $1-\varepsilon^\circ$ provided that $n_2$ satisfies 
\begin{equation*}
    C'_{\varepsilon^\circ} \max\bigg\{((1+|z_\alpha|) d\log (2d) L^4)^{q_x/(q_x-2)}, ((1+|z_\alpha|)(1+K) dL^2)^{p/(p-1)}\bigg\} \le n_2,
\end{equation*}
when $\alpha \neq 1/2$, where $C$ is a universal constant, $p = \min\{q_y, q_x/2\}$, $\mathfrak{C}$ depends on $q_x > 2$, $C_{\varepsilon^\circ}$ and $C'_{\varepsilon^\circ}$ depend on $\varepsilon^\circ$, but not on $d$ or $\alpha$.
\end{corollary}
\begin{proof}[\bfseries{Proof of \Cref{cor:lr-plugin-full}}]
Verifying \ref{as:rate-initial-estimator3}, we have that 
\begin{align*}
    &\frac{1}{n_2}\mathbb{E}_{P^2| P^1}\left[\frac{1}{n_2}\sum_{i \in I_2} \widehat \xi_i^2\right] + \widehat{\mathbb{C}}_2^2 \\
    &\quad \le \frac{4\|\widehat{\theta}_1 - \theta(P^N)\|_{\bar\Gamma_2}^2 \bar\sigma^2}{n_2} + \frac{2\|\widehat{\theta}_1 - \theta(P^N)\|_{\bar\Gamma_2}^4L^4}{n_2} + \|\widehat{\theta}_1 - \theta(P^N)\|_{\bar\Gamma_2}^4\\
    &\quad = \frac{\|\widehat{\theta}_1 - \theta(P^N)\|_{\bar\Gamma_2}^2 \bar\sigma^2}{n_2} + C \|\widehat{\theta}_1 - \theta(P^N)\|_{\bar\Gamma_2}^4\\
    &\quad \le \frac{\|\widehat{\theta}_1 - \theta(P^N)\|_{\bar\Gamma_1}^2 \bar\sigma^2}{n_2}\lambda_{\max}(\bar\Gamma_1^{-1/2}\bar\Gamma_2\bar\Gamma_1^{-1/2}) + C \|\widehat{\theta}_1 - \theta(P^N)\|_{\bar\Gamma_1}^4\lambda^2_{\max}(\bar\Gamma_1^{-1/2}\bar\Gamma_2\bar\Gamma_1^{-1/2})
\end{align*}
since we already assumed that $CL^4 \le n_2$ in \eqref{eq:eqreuiment_for_n2_simplified}. Consider the event 
\begin{align*}
    \Omega_{\mathtt{init}} := \left\{\|\widehat \theta_1-\theta(P^N)\|^2_{\bar\Gamma_1} \le \widetilde C_{\mathtt{init}}\frac{d \bar\sigma^2}{n_1}\right\}.
\end{align*}
Then on this event, we can take 
\begin{align*}
    \widetilde s_{n_1, n_2}^2 = \frac{(d\bar\sigma^2)(d\bar\sigma^2)}{n_1n_2} \lambda_{\max}(\bar\Gamma_1^{-1/2}\bar\Gamma_2\bar\Gamma_1^{-1/2})+ \frac{(d\bar\sigma^2)^2}{n_1^2}\lambda^2_{\max}(\bar\Gamma_1^{-1/2}\bar\Gamma_2\bar\Gamma_1^{-1/2}),
\end{align*}
since 
\begin{align*}
    &\frac{1}{n_2}\mathbb{E}_{P^2| P^1}\left[\frac{1}{n_2}\sum_{i \in I_2} \widehat \xi_i^2\right] + \widehat{\mathbb{C}}_2^2 \\
    &\quad\le \frac{\|\widehat{\theta}_1 - \theta(P^N)\|_{\bar\Gamma_1}^2 \bar\sigma^2}{n_2}\lambda_{\max}(\bar\Gamma_1^{-1/2}\bar\Gamma_2\bar\Gamma_1^{-1/2}) + C \|\widehat{\theta}_1 - \theta(P^N)\|_{\bar\Gamma_1}^4\lambda^2_{\max}(\bar\Gamma_1^{-1/2}\bar\Gamma_2\bar\Gamma_1^{-1/2})\\
    &\quad\lesssim \max\{\widetilde C_{\mathtt{init}}, \widetilde C_{\mathtt{init}}^2\}\widetilde s_{n_1, n_2}^2
\end{align*}
with probability greater than $1-\widetilde\varepsilon_{\mathtt{init}}$. Then the result follows with 
\begin{align*}
  &\sqrt{\frac{d\bar\sigma^2}{n_2}} +   \widetilde s_{n_1, n_2}^{1/2} \\
  &\quad \le    \sqrt{\frac{d\bar\sigma^2}{n_2}} +\left(\frac{d\bar\sigma^2}{n_2}\right)^{1/4}\left(\frac{d\bar\sigma^2}{n_2}\right)^{1/4}\lambda^{1/4}_{\max}(\bar\Gamma_1^{-1/2}\bar\Gamma_2\bar\Gamma_1^{-1/2})+ \sqrt{\frac{d\bar\sigma^2}{n_1}} \lambda^{1/2}_{\max}(\bar\Gamma_1^{-1/2}\bar\Gamma_2\bar\Gamma_1^{-1/2})\\
  &\quad \lesssim \sqrt{\frac{d\bar\sigma^2}{n_2}} +\sqrt{\frac{d\bar\sigma^2}{n_1}} \lambda^{1/2}_{\max}(\bar\Gamma_1^{-1/2}\bar\Gamma_2\bar\Gamma_1^{-1/2})
\end{align*}
by AM-GM inequality. 
\end{proof}

\begin{lemma}\label{lemma:linear-empirical-process}
    Let $X_i$ and $\varepsilon_i = Y_i - \theta(P^N)^\top X_i$ for $i \in I_2$. Then under \ref{as:eigen_value}, 
    \begin{align*}
        \E_{P^2} \left[\|\bar\Gamma_2^{-1/2}(\mathbb{P}_{n_2}-P^2) [\varepsilon X]\| \right] \le\sqrt{\frac{d\bar\sigma^2 }{n_2}}.
    \end{align*}
\end{lemma}
\begin{proof}[\bfseries{Proof of \Cref{lemma:linear-empirical-process}}]
We have
\begin{align*}
    \E_{P^2} \left[\|\bar\Gamma_2^{-1/2}(\mathbb{P}_{n_2}-P^2) [\varepsilon X]\|_2 \right] \le \sqrt{\E_{P^2} \left[\|(\mathbb{P}_{n_2}-P^2) \bar\Gamma_2^{-1/2}[\varepsilon X]\|_2^2 \right]}.
\end{align*}
We focus on the expectation inside, and obtain 
\begin{align*}
    &\E_{P^2} \left[\|(\mathbb{P}_{n_2}-P^2) \bar\Gamma_2^{-1/2}[\varepsilon X]\|_2^2\right] \\
    &\quad = \E_{P^2} \left[\left\|\frac{1}{n_2}\sum_{i \in I_2} \bar\Gamma_2^{-1/2} \varepsilon_i X_i - \E_{P_i}[\bar\Gamma_2^{-1/2} \varepsilon_i X_i]\right\|^2_2\right]= \E_{P^2} \left[\left\|\frac{1}{n_2}\sum_{i \in I_2} X_i^\circ \right\|^2\right]
\end{align*}
where 
\begin{align*}
    X_i^\circ = \bar\Gamma_2^{-1/2} \varepsilon_i X_i-\E_{P_i}[\bar\Gamma_2^{-1/2} \varepsilon_i X_i].
\end{align*}
Since each $X_i^\circ$ is mean zero, and as $X_i^\circ$ and $X_j^\circ$ are independent for $i \neq j$, we have 
\begin{align*}
    \E_{P^2} \left[\left\|\frac{1}{n_2}\sum_{i \in I_2} X_i^\circ \right\|^2\right] &= \frac{1}{n_2^2}\sum_{i \in I_2} \E_{P_i}[\|X_i^\circ\|^2] \\
    &\le  \frac{1}{n_2^2}\sum_{i \in I_2} \mathrm{tr}(\E_{P_i}[\bar\Gamma_2^{-1/2} \varepsilon_i^2 X_iX_i^\top \bar\Gamma_2^{-1/2}])\\
    &\le  \frac{1}{n_2^2}\sum_{i \in I_2} \bar\sigma^2 \mathrm{tr}(\E_{P_i}[\bar\Gamma_2^{-1/2} X_iX_i^\top \bar\Gamma_2^{-1/2}])\\
    &=  \frac{\bar\sigma^2 }{n_2} \cdot \mathrm{tr}(\bar\Gamma_2^{-1/2} \bar\Gamma_2 \bar\Gamma_2^{-1/2}) =\frac{d\bar\sigma^2 }{n_2},
\end{align*}
where we used \ref{as:eigen_value} with the tower property of the expectation. This concludes that 
\begin{align*}
    \E_{P^2} \left[\|\bar\Gamma_2^{-1/2}(\mathbb{P}_{n_2}-P^2) [\varepsilon X]\| \right]  \le \sqrt{\frac{d\bar\sigma^2 }{n_2}}.
\end{align*}
\end{proof}
\begin{lemma}\label{lemma:operator-empirical-process}
    Let $X_i$ for $i \in I_2$. Then under \ref{as:linreg-moment-covariate} for $q_x \ge 4$,
    \begin{align*}
      &\E_{P^2} \left[\|\bar\Gamma_2^{-1/2} (\mathbb{P}_{n_2}-P^2)[XX^\top ]\bar\Gamma_2^{-1/2}\|_{\mathrm{op}} \right] \le 4\sqrt{\frac{d\log(2d)L^4}{n_2}} + 32\sqrt{2}\log(2d)\left(\frac{dL^2}{n_2^{1-2/q_x}}\right).
    \end{align*}
\end{lemma}
\begin{proof}[\bfseries{Proof of \Cref{lemma:operator-empirical-process}}]
We use Theorem I of~\cite{tropp2016expected} to bound the expectation of the operator norm. Define 
\begin{align*}
    \mathbf{v}(X) = \left\|\sum_{i\in I_2} \mathbb{E}_{P_i}\left[\left(\frac{\bar\Gamma_2^{-1/2}X_iX_i^{\top}\bar\Gamma_2^{-1/2}}{n_2} - \frac{\bar\Gamma_2^{-1/2}\Gamma_i \bar\Gamma_2^{-1/2}}{n_2}\right)^2\right]\right\|_{\mathrm{op}},
\end{align*}
and 
\begin{align*}
    \mathbf{L}^2 = \mathbb{E}_{P^2}\left[\max_{i \in I_2} \left\|\left(\frac{\bar\Gamma_2^{-1/2}X_iX_i^{\top}\bar\Gamma_2^{-1/2}}{n_2} - \frac{\bar\Gamma_2^{-1/2}\Gamma_i \bar\Gamma_2^{-1/2}}{n_2}\right)\right\|^2_{\mathrm{op}}\right].
\end{align*}
Then Theorem I of \cite{tropp2016expected} states that 
\begin{align*}
    \E_{P^2} \left[\|\bar\Gamma_2^{-1/2} (\mathbb{P}_{n_2}-P^2)[XX^\top ]\bar\Gamma_2^{-1/2}\|_{\mathrm{op}} \right]  \le 2\sqrt{(1+2\log(2d))\mathbf{v}(X)} + 4(1+2\log(2d))\mathbf{L}.
\end{align*}

Now we derive bounds for $\mathbf{v}(X)$ and $\mathbf{L}$. We begin with $\mathbf{v}(X)$. First denote 
\begin{align*}
    \widetilde X_i = \bar\Gamma_2^{-1/2}X_i\quad \textrm{and} \quad  \widetilde \Gamma_i  =\bar\Gamma_2^{-1/2}\Gamma_i \bar\Gamma_2^{-1/2}.
\end{align*}
Then, we have 
\begin{align*}
    &\E_{P_i}\left[\left(\frac{\bar\Gamma_2^{-1/2}X_iX_i^{\top}\bar\Gamma_2^{-1/2}}{n_2} - \frac{\bar\Gamma_2^{-1/2}\Gamma_i \bar\Gamma_2^{-1/2}}{n_2}\right)^2 \right]=\E_{P_i}\left[\left(\frac{\widetilde X_i\widetilde X_i^{\top}}{n_2} - \frac{\widetilde \Gamma_i}{n_2}\right)^2 \right]\\
    &\quad =\E_{P_i}\left[\frac{\widetilde X_i\widetilde X_i^{\top}\widetilde X_i\widetilde X_i^{\top}}{n_2^2} - \frac{\widetilde X_i\widetilde X_i^{\top}\widetilde \Gamma_i}{n_2^2} -\frac{\widetilde \Gamma_i\widetilde X_i\widetilde X_i^{\top}}{n_2^2} +\frac{\widetilde \Gamma_i^2}{n_2^2} \right]=\E_{P_i}\left[\frac{\|\widetilde X_i\|^2\widetilde X_i\widetilde X_i^{\top}}{n_2^2}  \right] - \frac{\widetilde \Gamma_i^2}{n_2^2}.
\end{align*}
Furthermore, we have
\begin{align*}
    \mathbf{v}(X) = \left\|\sum_{i\in I_2}\E_{P_i}\left[\frac{\|\widetilde X_i\|^2\widetilde X_i\widetilde X_i^{\top}}{n_2^2}  \right] - \frac{\widetilde \Gamma_i^2}{n_2^2}\right\|_{\mathrm{op}} \le \left\|\sum_{i\in I_2}\E_{P_i}\left[\frac{\|\widetilde X_i\|^2\widetilde X_i\widetilde X_i^{\top}}{n_2^2}  \right] \right\|_{\mathrm{op}}.
\end{align*}
Now, we focus on the expectation inside. For any $\|u\|_2 = 1$,
\begin{align*}
    u^\top \E_{P_i}[\|\widetilde X_i\|^2\widetilde X_i\widetilde X_i^{\top} ] u&= \E_{P_i}[\|\bar\Gamma_2^{-1/2}X_i\|^2|u^\top \bar\Gamma_2^{-1/2}X_i|^2 ] \\
    &\le \sqrt{\E_{P_i}[\|\bar\Gamma_2^{-1/2}X_i\|^4}\sqrt{\E_{P_i}|u^\top  \bar\Gamma_2^{-1/2}X_i|^4}.
\end{align*}
We observe that 
\begin{align*}
    \E_{P_i}[\|\bar\Gamma_2^{-1/2}X_i\|^4] &= \E_{P_i}\left[\left\{\sum_{j=1}^d(e_j^\top \bar\Gamma_2^{-1/2}X_i)^2\right\}^2\right] \le d^2\E_{P_i}\left[\frac{1}{d}\sum_{j=1}^d(e_j^\top \bar\Gamma_2^{-1/2}X_i)^4\right] \le d^2L^4,
\end{align*}
where the last step follows by \ref{as:linreg-moment-covariate} since $q_x \ge 4$. Similarly, we have $\E_{P_i}|u^\top  \bar\Gamma_2^{-1/2}X_i|^4 \le L^4$. Putting together 
\begin{align*}
    \mathbf{v}(X) \le \left\|\sum_{i\in I_2}\E_{P_i}\left[\frac{\|\widetilde X_i\|^2\widetilde X_i\widetilde X_i^{\top}}{n_2^2}  \right] \right\|_{\mathrm{op}} \le \frac{dL^4}{n_2}.
\end{align*}

Next for $\mathbf{L}$, we have, 
\begin{align*}
    \mathbf{L}^2 \le \frac{2}{n_2^2}\mathbb{E}_{P^2}\left[\max_{i \in I_2}\|\bar\Gamma_2^{-1/2}X_i\|^4 \right] + \frac{2}{n_2^2}\mathbb{E}_{P^2}\left[\max_{i \in I_2} \|\bar\Gamma_2^{-1/2}\Gamma_i \bar\Gamma_2^{-1/2}\|^2_{\mathrm{op}}\right].
\end{align*}
For the first term, we have 
\begin{align*}\mathbb{E}_{P^2}\left[\max_{i \in I_2}\|\bar\Gamma_2^{-1/2}X_i\|^4\right] &\le \left(\mathbb{E}_{P^2}\left[\max_{i \in I_2}\|\bar\Gamma_2^{-1/2}X_i\|^{q_x}\right]\right)^{4/q_x} \le \left(\sum_{i \in I_2}\mathbb{E}_{P_i}[\|\bar\Gamma_2^{-1/2}X_i\|^{q_x}]\right)^{4/q_x}.
\end{align*}
As before, we have
\begin{align*}
    \E_{P_i}\left[\|\bar\Gamma_2^{-1/2}X_i\|^{q_x}\right] &= \E_{P_i}\left[\left\{\sum_{j=1}^d (e_j^\top \bar\Gamma_2^{-1/2}X_i)^{2}\right\}^{q_x/2}\right]\\
    &\le d^{q_x/2}\E_{P_i}\left[\frac{1}{d}\sum_{j=1}^d |e_j^\top \bar\Gamma_2^{-1/2}X_i|^{q_x}\right] \le d^{q_x/2} L^{q_x}
\end{align*}
where the last step follows by \ref{as:linreg-moment-covariate}. Therefore, we conclude
\begin{align*}
    \mathbb{E}_{P^2}\left[\max_{i \in I_2}\|\bar\Gamma_2^{-1/2}X_i\|^4\right] \le n_2^{4/q_x}d^2 L^4.
\end{align*}
For $\max_{i \in I_2} \|\bar\Gamma_2^{-1/2}\Gamma_i \bar\Gamma_2^{-1/2}\|^2_{\mathrm{op}}$, we have 
\begin{align*}
    \sup_{\|u\|=1}u^\top \bar\Gamma_2^{-1/2}\Gamma_i \bar\Gamma_2^{-1/2} u = \sup_{\|u\|=1}\E_{P_i}[|u^\top \bar\Gamma_2^{-1/2}X_i|^2] \le \sup_{\|u\|=1}\big(\E_{P_i}[|u^\top \bar\Gamma_2^{-1/2}X_i|^4]\big)^{1/2} \le L^2.
\end{align*}
Putting together 
\begin{align*}
    \mathbf{L} \le \sqrt{2}n_2^{-1+2/q_x}d L^2 + \frac{\sqrt{2}L^2}{n_2} \le 2\sqrt{2}n_2^{-1+2/q_x}d L^2.
\end{align*}
The result is concluded by choosing $\eta = 1$ and plugging $\mathbf{v}(X)$ and $\mathbf{L}$ into the expression of Theorem I of \citet{tropp2016expected}.
\end{proof}
\begin{proof}[\bfseries{Proof of \Cref{thm:LR-valid-pen}}]
    Since the additional term $\lambda(\theta)$ is not random, we have 
    \begin{align*}
        \widehat\xi_i &= (Y_i-\widehat\theta^\top X)^2 + \lambda(\widehat \theta) -(Y_i-\theta(P^N)^\top X)^2 - \lambda(\theta(P^N)) \\
        &\quad - \E_{P_i}[(Y_i-\widehat\theta^\top X)^2 + \lambda(\widehat \theta) -(Y_i-\theta(P^N)^\top X)^2 - \lambda(\theta(P^N)) |D_1]\\
        &= (Y_i-\widehat\theta^\top X)^2  -(Y_i-\theta(P^N)^\top X)^2 - \E_{P_i}[(Y_i-\widehat\theta^\top X)^2 -(Y_i-\theta(P^N)^\top X)^2 |D_1].
    \end{align*}
    Thus the expression for $\widehat \xi_i$ is identical as in \Cref{thm:LR-valid}. Hence the validity result by \Cref{thm:LR-valid} remains to hold whenever $\widehat\theta_1$ is consistent. For $\alpha = 1/2$, the upper bound for variance also does not change since $\lambda(\theta)$ is constant. The expression for $\widehat{\mathbb{C}}_2$ may depend on $\lambda(\cdot)$. In particular, for any convex $\lambda$ and $g_0 \in \partial \lambda(\theta_0)$, 
    \begin{align*}
        \widehat{\mathbb{C}}_2 &= \frac{1}{n_2}\sum_{i \in I_2}\E_{P_i}[(Y_i-\widehat\theta^\top X_i)^2  -(Y_i-\theta(P^N)^\top X_i)^2+ \lambda(\widehat \theta)- \lambda(\theta(P^N)) | D_1] \\
         &\ge \frac{1}{n_2}\sum_{i \in I_2}\E_{P_i}[2(Y_i-\theta(P^N)^\top X_i)(\theta(P^N)^\top X_i-\widehat \theta^\top X_i)+g_0^\top(\widehat \theta-\theta(P^N)) | D_1] \\
         &\quad +\frac{1}{n_2}\sum_{i \in I_2}\E_{P_i}[(\widehat \theta^\top X_i-\theta(P^N)^\top X_i)^2| D_1].
    \end{align*}
    On the other hand, since $\theta(P^N)$ is a population minimizer, it follows that
    \begin{align*}
        &0 \in \frac{1}{n_2}\sum_{i \in I_2}\E_{P_i}[2(Y_i-\theta(P^N)^\top X_i)X_i] + \partial \lambda(\theta_0) \\
        &\qquad \implies \frac{1}{n_2}\sum_{i \in I_2}\E_{P_i}[2(Y_i-\theta(P^N)^\top X_i)(\theta(P^N)^\top X_i-\widehat \theta^\top X_i)+g_0^\top(\widehat \theta-\theta(P^N)) | D_1]  = 0.
    \end{align*}
    Hence the lower bound 
    \begin{align*}
        \widehat{\mathbb{C}}_2 \ge \frac{1}{n_2}\sum_{i \in I_2}\E_{P_i}[(\widehat \theta^\top X_i-\theta(P^N)^\top X_i)^2| D_1]
    \end{align*}
    remains to hold when $\lambda(\cdot)$ is convex. The rest of the proof is identical to that of \Cref{thm:LR-valid}.
\end{proof}

\subsection{Manski's Discrete Choice Model}
In this example, we take 
\[m_{\theta} := (y,x) \mapsto -\frac{1}{2}\, y \cdot \mathrm{sgn}(\theta^\top x).\]
whose envelop function for $m_{\theta}-m_{\theta(P^N)}$ is trivially given by a constant function at $1$. The leading constant $1/2$ is introduced without loss of generality.
\begin{proof}[\bfseries{Proof of \Cref{thm:manski-validity-consistent}}]The proof is an application of \Cref{thm:coverage-anti-conservative-confidence-set-empirical-risk} for $\alpha = 1/2$ and \Cref{thm:studentized-katz} for $\alpha \neq 1/2$. First, observe that 
    \begin{align*}
        &\left|\E_{P_i}\left[-\frac{1}{2}\, Y_i \cdot \mathrm{sgn}(\widehat{\theta}_1^\top X_i) + \frac{1}{2}\, Y_i \cdot \mathrm{sgn}(\theta(P^N)^\top X_i)\right]\right| \\
        &\quad = \left|\E_{P_i}\left[(2\eta_{P_i}(X_i)-1)\cdot \mathrm{sgn}(\theta(P^N)^\top X_i)\left( \mathrm{sgn}(\theta(P^N)^\top X_i) \neq  \mathrm{sgn}(\widehat{\theta}_1^\top X_i)\right)\right]\right|\\
        &\quad \le \mathbb{P}_{P_i}(\mathrm{sgn}(\theta(P^N)^\top X_i) \neq  \mathrm{sgn}(\widehat{\theta}_1^\top X_i)) = d_{\Delta, i}(\widehat{\theta}_1, \theta(P^N)).
    \end{align*} 
    This implies that 
    \begin{align*}
        &\mathrm{Var}_{P_i}[\widehat \xi_i] = \mathrm{Var}_{P_i}\left[\frac{1}{2}Y_i\left( \mathrm{sgn}(\theta(P^N)^\top X_i) -  \mathrm{sgn}(\widehat{\theta}_1^\top X_i)\right)\right] \\
    &\quad= \mathbb{P}_{P_i}\left( \mathrm{sgn}(\theta(P^N)^\top X_i) \neq  \mathrm{sgn}(\widehat{\theta}_1^\top X_i)\right) - \left[\frac{1}{2}\E_{P_i}\left[Y_i\left( \mathrm{sgn}(\theta(P^N)^\top X_i) -  \mathrm{sgn}(\widehat{\theta}_1^\top X_i)\right)\right]\right]^2\\
    &\quad\ge \mathbb{P}_{P_i}\left( \mathrm{sgn}(\theta(P^N)^\top X_i) \neq  \mathrm{sgn}(\widehat{\theta}_1^\top X_i)\right)-\left\{\mathbb{P}_{P_i}\left( \mathrm{sgn}(\theta(P^N)^\top X_i) \neq  \mathrm{sgn}(\widehat{\theta}_1^\top X_i)\right)\right\}^2\\
    &\quad= \mathbb{P}_{P_i}\left( \mathrm{sgn}(\theta(P^N)^\top X_i) \neq  \mathrm{sgn}(\widehat{\theta}_1^\top X_i)\right)\left\{1-\mathbb{P}_{P_i}\left( \mathrm{sgn}(\theta(P^N)^\top X_i) \neq  \mathrm{sgn}(\widehat{\theta}_1^\top X_i)\right)\right\} \\
    &\quad = d_{\Delta, i}(\widehat{\theta}_1, \theta(P^N)) (1-d_{\Delta, i}(\widehat{\theta}_1, \theta(P^N))).
    \end{align*}
    Finally, using the fact that $|\widehat{\xi}_i| \le 1$, we have
    \begin{align*}
        \E_{P_i}[|\widehat\xi_i|^3] &\le \E_{P_i}[|\widehat\xi_i|^2]\\
    &\le \E_{P_i}\left[\left|\frac{1}{2}\, Y_i \cdot \left\{\mathrm{sgn}(\widehat{\theta}_1^\top X_i)-\mathrm{sgn}(\theta(P^N)^\top X_i)\right\}\right|^2\right]\\
    &= \mathbb{P}_{P_i}\left(\mathrm{sgn}(\widehat{\theta}_1^\top X_i)\neq \mathrm{sgn}(\theta(P^N)^\top X_i)\right) = d_{\Delta, i}(\widehat{\theta}_1, \theta(P^N)).
    \end{align*}

Putting together, we obtain
\begin{align*}
    \mathbb{E}_{P^N}\left[\sum_{i\in I_2} \frac{|\widehat\xi_i|^2}{\widehat{\mathbb{V}}_{2}}\min\left\{1,\,\frac{|\widehat \xi_i|}{\widehat{\mathbb{V}}_{2}^{1/2}}\right\}\right] &\le \mathbb{E}_{P^N}\left[\sum_{i\in I_2} \frac{|\widehat\xi_i|^3}{\widehat{\mathbb{V}}_{2}^{3/2}}\right] \\
    &\le \mathbb{E}_{P^1}\left[\frac{\sum_{i\in I_2} |d_{\Delta, i}(\widehat{\theta}_1, \theta(P^N))|}{\left(\sum_{i \in I_2}d_{\Delta, i}(\widehat{\theta}_1, \theta(P^N))(1-d_{\Delta, i}(\widehat{\theta}_1, \theta(P^N)))\right)^{3/2}}\right].
\end{align*}
First, assume that $d_{\Delta, i}(\widehat{\theta}_1, \theta(P^N)) \le 1/2$ for all $i \in I_2$. Then 
\begin{align*}
    &\min\left\{1, C\mathbb{E}_{P^1}\left[\frac{\sum_{i\in I_2} |d_{\Delta, i}(\widehat{\theta}_1, \theta(P^N))|}{\left(\sum_{i \in I_2}d_{\Delta, i}(\widehat{\theta}_1, \theta(P^N))(1-d_{\Delta, i}(\widehat{\theta}_1, \theta(P^N)))\right)^{3/2}}\right]\right\} \\
    &\quad \le \min\left\{1, C\mathbb{E}_{P^1}\left[\frac{2^{3/2}}{\left(\sum_{i \in I_2}d_{\Delta, i}(\widehat{\theta}_1, \theta(P^N))\right)^{1/2}}\right]\right\}.
\end{align*}
Meanwhile if $d_{\Delta, i}(\widehat{\theta}_1, \theta(P^N)) \ge 1/2$ for some $i \in I_2$, then $\sum_{i \in I_2}d_{\Delta, i}(\widehat{\theta}_1, \theta(P^N)) \ge 1/2$, and the result holds trivially. Finally, in view of \ref{as:covariate-manski} and \Cref{thm:studentized-katz}, we conclude 
    \begin{align*}
&\mathbb{P}_{P^N}(\theta(P^N) \not\in \widehat{\mathrm{CI}}^{\mathtt{CLT}}_{N,\alpha}) \le \alpha +  \E_{P^1}\left[\min\left\{1, \frac{C}{(n_2\|\widehat\theta_1 - \theta(P^N)\|)^{1/2}}\right\}\right],
\end{align*}
where $C > 0$ is a constant depending on $c_1$.

Next, consider the case $\alpha = 1/2$. We derive the lower bound on the curvature. Following Proposition 1 of \citet{tsybakov2004optimal} (and similarly for Proposition 2.4 of \citet{mukherjee2021optimal}), we define the set $\mathcal{A}(\theta) := \{x : \mathrm{sgn}(\theta(P^N)^\top X) \neq \mathrm{sgn}(\theta^\top X)\}$ for each $\theta$. It then follows that 
\begin{align*}
   \E_{P_i}(m_{\theta}-m_{\theta(P^N)}) &=\frac{1}{2}\E_{P_i}\left[ Y \left(\mathrm{sgn}(\theta(P^N)^\top X)-\mathrm{sgn}(\theta^\top X)\right)\right] \\
   &=\int_{\mathcal{A}(\theta)}\, |\E_{P_i}[Y |X=x] |P_X(x)\, dx\\
   &= 2 \int_{\mathcal{A}(\theta)}\, |\eta_{P_i}(x)-1/2|P_X(x)\, dx \\
   &\ge 2\sup_{0 \le t \le t^*}t\mathbb{P}_{P_i}\left(|\eta_{P_i}(X)-1/2| \ge t \cap X \in \mathcal{A}(\theta)\right)\\
   &\ge 2\sup_{0 \le t \le t^*}t\left(d_{\Delta, i}(\theta, \theta(P^N))-\mathbb{P}_{P_i}\left(|\eta_P(X)-0.5| \le t \right)\right)\\
   &\ge 2\sup_{0 \le t \le t^*}t\left(d_{\Delta, i}(\theta, \theta(P^N))- C_0t^{1/\gamma} \right).
\end{align*}
The last inequality uses \ref{as:tsybakov}. The optimal choice of $t$ is given by 
\begin{align*}
    t = \begin{cases}
        (1+1/\gamma)^{-\gamma}C_0^{-\gamma}d^\gamma_{\Delta, i}(\theta, \theta(P^N)) & \text{when}\quad d_{\Delta, i}(\theta, \theta(P^N))\le (1+1/\gamma)C_0(t^*)^{1/\gamma}\\
        t^* & \text{otherwise.}
    \end{cases}
\end{align*}
Putting together, it follows for all $\theta \in \mathbb{S}^{d-1}$,
    \begin{align*}
        \E_{P_i}(m_{\theta}-m_{\theta(P^N)}) &\ge \left(\frac{2}{1+\gamma}\right)\frac{d^{1+\gamma}_{\Delta, i}(\theta, \theta(P^N))}{ (1+1/\gamma)^{\gamma}C_0^\gamma}\mathbf{1}\left\{d_{\Delta, i}(\theta, \theta(P^N))\le (1+1/\gamma)C_0(t^*)^{1/\gamma}\right\}\\
        &\quad +  \left(\frac{2}{1+\gamma}\right)t^*d_{\Delta, i}(\theta, \theta(P^N))\mathbf{1}\left\{d_{\Delta, i}(\theta, \theta(P^N))>(1+1/\gamma)C_0(t^*)^{1/\gamma}\right\}\\
        &\ge \mathfrak{C}d_{\Delta, i}(\theta, \theta(P^N))\min\{d^\gamma_{\Delta, i}(\theta, \theta(P^N)), t^*\},
    \end{align*}
    where $\mathfrak{C}$ depends on $C_0$. Based on the earlier calculation, we have 
\begin{align*}
    \mathrm{Var}_{P_i}[\widehat{\xi}_i] \le \E_{P_i}[\widehat{\xi}_i^2] \le d_{\Delta, i}(\theta, \theta(P^N)).
\end{align*}
Hence, we have 
\begin{align*}
    \ratio^2 &\ge  \frac{(\sum_{i \in I_2}\mathfrak{C}d_{\Delta, i}(\theta, \theta(P^N))\min\{d^\gamma_{\Delta, i}(\theta, \theta(P^N)), t^*\})^2}{\sum_{i \in I_2}d_{\Delta, i}(\theta, \theta(P^N))} \\
    &\ge \frac{(\sum_{i \in I_2}\mathfrak{C}d_{\Delta, i}(\theta, \theta(P^N))\min\{c_1\|\theta- \theta(P^N)\|^\gamma, t^*\})^2}{\sum_{i \in I_2}d_{\Delta, i}(\theta, \theta(P^N))}\\
    &\ge \mathfrak{C}^2\sum_{i \in I_2}d_{\Delta, i}(\theta, \theta(P^N))\min\{c_1^2\|\theta- \theta(P^N)\|^{2\gamma}, (t^*)^2\}\\
    &\ge \mathfrak{C}n_2\|\theta- \theta(P^N)\|\min\{\|\theta- \theta(P^N)\|^{2\gamma}, (t^*)^2\} = \widetilde \Delta_2^2,
\end{align*}
where we used \ref{as:covariate-manski} and $\mathfrak{C}$ depends on $C_0$ and $c_1$.

% Under \ref{as:covariate-manski} and $(1+1/\gamma)C_0(t^*)^{1/\gamma} > 1 \Leftrightarrow (1+1/\gamma)^\gamma C_0^\gamma(t^*) > 1 \Leftarrow C_0^\gamma(t^*) > 1$,
% \begin{align*}
%     \E_{P_i}(m_{\theta}-m_{\theta(P^N)}) \ge \mathfrak{C}\|\widehat{\theta}_1 - \theta(P^N)\|\min\{\|\widehat{\theta}_1 - \theta(P^N)\|^\gamma, t^*\},
% \end{align*}
% where $\mathfrak{C}$ depends on $c_1, C_0$ and $\gamma$.
% Based on the calculation from the proof of \Cref{thm:manski-validity-consistent}, we have 
% \begin{align*}
%     \mathrm{Var}_{P_i}[\widehat{\xi}_i] \le \E_{P_i}[\widehat{\xi}_i^2] \le d_{\Delta, i}(\theta, \theta(P^N)).
% \end{align*}
% Hence assuming \ref{as:covariate-manski2}, we have 
% \begin{align*}
%     \ratio^2 \ge  \frac{n_2^2\mathfrak{C}^2\|\widehat{\theta}_1 - \theta(P^N)\|^{2+2\gamma}}{n_2C_1\|\widehat{\theta}_1 - \theta(P^N)\|} \ge n_2 \mathfrak{C}\|\widehat{\theta}_1 - \theta(P^N)\|^{1+2\gamma}.
% \end{align*}
Finally by \Cref{thm:coverage-anti-conservative-confidence-set-empirical-risk}, we have
\begin{align*}
    &\min\left\{1, C\mathbb{E}_{P^N}\left[\sum_{i\in I_2}\frac{|\widehat\xi_i|^2}{n_2^2\widehat{\mathbb{V}}_{2}(1 + \ratio)^2}\min\left\{1,\,\frac{|\widehat \xi_i|}{n_2\widehat{\mathbb{V}}_{2}^{1/2}(1 + \ratio)}\right\}\right]\right\}\\
    &\quad \le \min\left\{1, C\E_{P^1}\left[\frac{1}{(1 + \widetilde\Delta_2)^2} \right]\right\}.
\end{align*}
This concludes the claim.
\end{proof}
\begin{proof}[\bfseries{Proof of \Cref{thm:manski-width}}] The proof is a direct application of \Cref{thm:clt-width-local-full} and thus proceeds by verifying \ref{as:margin}, \ref{as:local-entropy}, \ref{as:square-process}, to hold locally and \ref{as:margin-global}, \ref{as:ratio-process}, and \ref{as:ratio-square-process} to hold globally. We first establish the convergence rate of the diameter in terms of the pseudo-metric $d_\Delta(\theta_1, \theta_2)$ and then translate the result to $\|\cdot\|_2$. For identical distributions $X_i$ with $i \in I_2$, we define the metric over $P^2$ as 
\begin{align*}
    d_{\Delta}(\theta_1, \theta_2) = \frac{1}{n_2}\sum_{i \in I_2}d_{\Delta, i}(\theta_1, \theta_2).
\end{align*}
We define 
\begin{align*}
    \M_2(\theta) = \frac{1}{n_2}\sum_{i \in I_2}\E_{P_i}\left(-\frac{1}{2}\, Y_i \cdot \mathrm{sgn}(\theta^\top X_i)\right) \quad \text{and} \quad \widehat{\M}_2(\theta) = \frac{1}{n_2}\sum_{i \in I_2}\left(-\frac{1}{2}\, Y_i \cdot \mathrm{sgn}(\theta^\top X_i)\right).
\end{align*}
We consider the following collection of ``localized'' functions:
\begin{align*}
    \mathcal{M}_{\delta}^\Delta := \left\{  \frac{1}{2}\, y \left(\mathrm{sgn}(\theta(P^N)^\top x)-\mathrm{sgn}(\theta^\top x)\right)  \text{ for all }\theta \text{ s.t., } d_\Delta(\theta, \theta(P^N)) \le \delta \text{ and } \theta \in \mathbb{S}^{d-1}\right\}.
\end{align*}
\paragraph{Verifying \ref{as:margin}}Throughout, we assume $X_i$ is identically distributed. In the proof of \Cref{thm:manski-validity-consistent}, we have already established that 
\begin{align*}
    \frac{1}{n_2}\sum_{i\in I_2}\E_{P_i}(m_{\theta}-m_{\theta(P^N)}) \ge \mathfrak{C} d^{1+\gamma}_{\Delta}(\theta, \theta(P^N))
\end{align*}
whenever $d_\Delta(\theta, \theta(P^N)) \le (t^*)^{1/\gamma}$ where $\mathfrak{C}$ depends on $C_0, \gamma$. Hence \ref{as:margin} holds locally with $\gamma$ for $d_\Delta(\theta, \theta(P^N)) \le (t^*)^{1/\gamma}$.
\paragraph{Verifying \ref{as:local-entropy}}
For a measurable function $f$, we denote 
\begin{align}\label{eq:definition-Gn}
    \mathbb{G}_{n_2} f = n_2^{1/2}\left(\frac{1}{n_2}\sum_{i \in I_2} f(X_i, Y_i) - \E_{P_i}[f(X_i, Y_i)]\right).
\end{align}
Then 
\begin{align*}
    &\E^*_{P^2} \left[\sup_{d_\Delta(\theta, \theta(P^N)) < \delta}|(\widehat{\mathbb{M}}_2 - \mathbb{M}_2)(\theta) - (\widehat{\mathbb{M}}_2 - \mathbb{M}_2)(\theta(P^N))| \right]  \\
    &\quad = n_2^{-1/2}\E^*_{P^2} \left[\sup_{d_\Delta(\theta, \theta(P^N)) < \delta}| \mathbb{G}_{n_2}(m_{\theta} - m_{\theta(P^N)})| \right] =n_2^{-1/2}\E^*_{P^2} \left[\sup_{m \in \mathcal{M}^\Delta_\delta}\left|\mathbb{G}_{n_2} m\right| \right].
\end{align*}
  In order to control this term, we introduce the following objects. For any set $\Theta$ equipped with a metric $\|\cdot\|$, and any $\varepsilon > 0$, an $\varepsilon$-covering number $\mathcal{N}(\varepsilon,\Theta, \|\cdot\|)$ of $\Theta$ relative to the metric $\|\cdot\|$ is defined as the minimal number of $\|\cdot\|$-balls of radius less than or equal to $\varepsilon$ required for covering $\Theta$. In particular, we consider when $\Theta$ contains measurable functions of observations $Z_i \in \mathcal{Z}$ and let $Q$ be any  discrete probability measure on $Z_i$ for $i \in I_2$. We define an envelop function $F$ of the class $\Theta$ as $F := z\mapsto \sup_{f\in\Theta}|f(z)|$. The uniform entropy numbers of $\Theta$ relative to $L_r$ is defined as 
\begin{align*}
    J(\delta, \Theta, L_r) := \sup_{Q}\int_{0}^\delta \sqrt{1 + \log \mathcal{N}(\varepsilon\|F\|_{Q, r}, \Theta, L_r(Q))}\, d\varepsilon
\end{align*}
where $\|f\|_{Q,r} := \left(\sum_{i=1}^n f^r(z_i)Q(z_i)\right)^{1/r}$. 
We use the following result from \citet{van2011local}:
\begin{theorem}[Theorem 2.1 of \citet{van2011local}]\label{thm:local-maximal-ineq}
    Let $\mathcal{F}$ be a collection of $P$-square integrable functions equipped with an envelop function $F \le 1$. If $\mathbb{E}_P f^2 \le t^2 \mathbb{E}_P F^2$, for every $f$ and some $t \in (0,1)$, then 
    \begin{align*}
        \E_P^*\, \left[\sup_{f \in \mathcal{F}}\, |\mathbb{G}_n f|\right] \lesssim J(t, \mathcal{F}, L_2)\left(1+ \frac{J(t, \mathcal{F}, L_2)}{t^2 \sqrt{n}\|F\|_{P,2}}\right) \|F\|_{P, 2}.
    \end{align*}
\end{theorem}
    
First, we relate the covering number of $\mathcal{M}^\Delta_\delta$ to the VC dimension of the subgraphs of the functions in $\mathcal{M}^\Delta_\delta$. First, the function $f$ in $\mathcal{M}^\Delta_\delta$ takes values in $\{-1, 0, 1\}$. Hence, for the subgraph $\{(x,y,t) : t < f(x,y)\}$, we only need to consider cases with $-1 < t \le 0$ and $0 < t \le 1$. They are identical and only provide the case with $0 < t \le 1$. We only need to consider the set where the function evaluates to non-zero values, that is, 
\begin{align*}
    \left\{(1,x,t) : \theta(P^N)^\top x > 0, \theta^\top x \le 0, t>0\right\} \cup \left\{(-1,x,t) : \theta(P^N)^\top x \le 0, \theta^\top x > 0, t>0\right\}. 
\end{align*}
Then, 
\begin{align*}
    &\left\{(y,x,t) \, :\, 2^{-1}y\{\mathrm{sgn}(\theta(P^N)^\top x)-\mathrm{sgn}(\theta^\top x)\} \ge t, t>0\right\} \\
        &\quad \subseteq \left\{(1,x,t) : \theta(P^N)^\top x > 0, \theta^\top x \le 0, t>0\right\} \\
        &\quad\quad\cup \left\{(-1,x,t) : \theta(P^N)^\top x \le 0, \theta^\top x > 0, t>0\right\}.
\end{align*}
Each component is the intersection of two half-spaces in $\mathbb{R}^d$ whose VC-dimension is $d+1$ and the intersection of two sets in VC-classes is also VC. The corresponding VC-dimension of the resulting set space is $2d+2$. See Lemma 2.6.17 of \citet{van1996weak}. Since the subgraph of a function is VC, the covering number of the function space is 
    \begin{align*}
        \mathcal{N}(\varepsilon \|M\|_{L_2(Q)}, \mathcal{M}_\delta^\Delta, L_2(Q)) \le C d (16e)^d \left(\frac{1}{\varepsilon}\right)^{4d}
    \end{align*}
    by Theorem 2.6.7 of \citet{van1996weak} with their $r=2$ for any probability measure $Q$ and $\varepsilon \in (0,1)$ and $C$ is a universal constant. We thus obtain  
    \begin{align*}
        t\mapsto J(t, \mathcal{M}^\Delta_{\delta}, L_2) &= \sup_{Q} \int_{0}^t \sqrt{1+ \log\left(C d (16e)^d \left(\frac{1}{\varepsilon}\right)^{4d}\right)}\, d\varepsilon\\
        &\le  \sup_{Q} \int_{0}^t \sqrt{\mathfrak{C} d + 4d\log\left(\frac{1}{\varepsilon}\right)}\, d\varepsilon \le  \mathfrak{C}t \sqrt{d\log(1/t)}
    \end{align*}
    where $\mathfrak{C}$ is a universal constants that may change line by line.
    
    Next, for the variance bound, we have 
    \begin{align*}
        \frac{1}{n_2}\sum_{i \in I_2}\E_{P_i}f^2  = \frac{1}{n_2}\sum_{i \in I_2}\mathbb{P}_{P_i}\left(\mathrm{sgn}(\theta(P^N)^\top x) \neq \mathrm{sgn}(\theta^\top x)\right) = \delta.
    \end{align*}
    Thus the condition of Theorem~\ref{thm:local-maximal-ineq} holds with $t^2 = \delta$ and $F=1$. By Theorem~\ref{thm:local-maximal-ineq}, we obtain
    \begin{align}
        \E^*_{P^2}\, \left[\sup_{m \in \mathcal{M}_\delta^\Delta}\, |\mathbb{G}_{n_2} m|\right]&\lesssim J(\sqrt{\delta}, \mathcal{F}, L_2)\left(1+ \frac{J(\sqrt{\delta}, \mathcal{F}, L_2)}{\delta\sqrt{n_2}}\right)\nonumber \\
        & \lesssim\sqrt{\delta d\log(1/\delta)}\left(1+ \frac{ \sqrt{\delta d\log(1/\delta)}}{\delta \sqrt{n_2}}\right) \nonumber\\
        &= \sqrt{\delta d\log(1/\delta)}+ \frac{d\log(1/\delta)}{\sqrt{n_2}}.\nonumber
    \end{align}
    We can thus take $\phi_{n_2}$ in \ref{as:local-entropy} as
    \begin{align*}
        \delta \mapsto \phi_{n_2}(\delta) = C\left(\sqrt{\frac{\delta d\log(1/\delta)}{n_2}}+ \frac{d\log(1/\delta)}{n_2}\right),
    \end{align*}
    where $C$ is a universal constant. The requirement $q < 1+\gamma$ is satisfies with $q = 1/2$.  

\paragraph{Verifying \ref{as:square-process}}
Recalling the discussion after \Cref{thm:clt-width}, we have $\phi_{n_2} = \omega_{n_2, \mathtt{emp}}$ for bounded processes. This follows from the contraction principle \citep[Theorem 4.12]{ledoux2013probability}. Thus we only need to evaluate $\omega^2_{\mathtt{pop}}$. This follows
\begin{align*}
    &\frac{1}{n_2}\sum_{i \in I_2}\E_{P_i}[(m_{\theta}-m_{\theta(P^N)})^2(Z_i)] \\
    &\quad = \frac{1}{n_2}\sum_{i \in I_2}\mathbb{P}_{P_i}\left(\mathrm{sgn}(\theta^\top X_i)\neq \mathrm{sgn}(\theta(P^N)^\top X_i)\right) = \frac{1}{n_2}\sum_{i \in I_2}d_{\Delta, i}(\theta, \theta(P^N)).
\end{align*}
Hence $\omega^2_{\mathtt{pop}}(\delta) = \delta$, which satisfies the requirement $q < 1+\gamma$ with $q=1/2$. 
\paragraph{Verifying \ref{as:margin-global}} Setting $\delta > (t^*)^{1/\gamma}$, we have shown that 
\begin{align*}
    C_{\delta_0}(d_{\Delta}(\theta, \theta(P^N))) \ge \mathfrak{C}t^* d_{\Delta}(\theta, \theta(P^N))
\end{align*}
from the proof of \Cref{thm:manski-validity-consistent}.

\paragraph{Verifying \ref{as:ratio-process}}
For any $\delta > (t^*)^{1/\gamma}$,
\begin{align*}
    &\E^*_{P^2} \left[\sup_{d_{\Delta}(\theta, \theta(P^N)) > (t^*)^{1/\gamma}}\left|\frac{(\widehat{\mathbb{M}}_2 - \mathbb{M}_2)(\theta) - (\widehat{\mathbb{M}}_2 - \mathbb{M}_2)(\theta(P^N))}{\mathbb{M}_2(\theta) -\mathbb{M}_2(\theta(P^N))}\right| \right] \\
    &\quad \le C\left(\sqrt{\frac{ d\log(1/t^*)}{\gamma(t^*)^{1/\gamma}n_2}}+ \frac{d\log(1/t^*)}{\gamma(t^*)^{1/\gamma}n_2}\right) = R(n_2, (t^*)^{1/\gamma})
\end{align*}
where $C$ is a universal constant. Hence, \ref{as:ratio-process} holds with $C_{\mathtt{ratio}}=1/\varepsilon_{\mathtt{ratio}}$ by Markov's inequality.

\paragraph{Verifying \ref{as:ratio-square-process}} Similarly for any $\delta > (t^*)^{1/\gamma}$,
\begin{align*}
    &\E_{P^2}^* \left[\sup_{d_{\Delta}(\theta, \theta(P^N)) > (t^*)^{1/\gamma}}\frac{z_\alpha^2}{n_2^2}\left|\frac{\sum_{i\in I_2}(m_\theta - m_{\theta(P^N)})^2 - \mathbb{E}_{P_i}[(m_\theta - m_{\theta(P^N)})^2]}{\{\mathbb{M}_2(\theta) -\mathbb{M}_2(\theta(P^N))\}^2}\right| \right]\\
    &\quad \le \frac{(1+z_\alpha^2)C}{(t^*)^{2/\gamma}n_2}\left(\frac{(t^*)^{1/\gamma} d\log(1/t^*)}{\gamma n_2}+ \left(\frac{d\log(1/t^*)}{\gamma n_2}\right)^2\right) =  S_{\mathtt{emp}}(n_2, (t^*)^{1/\gamma}, \alpha),
\end{align*}
and 
\begin{align*}
    &\sup_{d_{\Delta}(\theta, \theta(P^N)) > (t^*)^{1/\gamma}}\frac{z_\alpha^2\sum_{i\in I_2}\mathbb{E}_{P_i}[(m_\theta - m_{\theta(P^N)})^2(Z_i)]}{n_2^2\{\M_2(\theta)-\M_2(\theta(P^N))\}^2} \le \frac{(1+z_\alpha^2)}{(t^*)^{1/\gamma}n_2} =S_{\mathtt{pop}}(n_2, (t^*)^{1/\gamma}, \alpha),
\end{align*}
and \ref{as:ratio-square-process} holds with $\widetilde C_{\mathtt{emp}}=1/\varepsilon_{\mathtt{emp}}$ by Markov's inequality.

We assume that $n_2$ is large enough satisfies the following:
\begin{align}\label{eq:eqreuiment_for_n2-manski}
    C_{\varepsilon^\circ}\max\{R(n_2, \rho) , S_{\mathtt{emp}}(n_2, \rho, \alpha), S_{\mathtt{pop}}(n_2, \rho, \alpha)\} \le 1/3,
\end{align}
where $C_{\varepsilon^\circ}$ is a constant depending on $\varepsilon^\circ$. Then \eqref{eq:eqreuiment_for_n2-manski} is satisfied for $n_2$ large such that 
\begin{equation*}
    C'_{\varepsilon^\circ}(1+|z_\alpha|)^2\frac{ d\log(1/t^*)}{\gamma(t^*)^{1/\gamma}} \le n_2,
\end{equation*}
where $C'_{\varepsilon^\circ}$ is a constant depending on $\varepsilon^\circ$.

\paragraph{Evaluating the rate of convergence.}Throughout, let $\mathfrak{C}$ denote a constant, depending on $C_0, \gamma$, that may change from line to line. First, we obtain the value related to $\phi_{n_2}$, that is, 
\begin{align*}
    &r_{n_2}^{-2}\phi_{n_2}(\mathfrak{C}r_{n_2}^{2/(1+\gamma)}) \le 1 \\
    &\quad \Leftrightarrow r_{n_2}^{-2}\sqrt{\frac{\mathfrak{C}r_{n_2}^{2/(1+\gamma)} d\log(1/\mathfrak{C}r_{n_2}^{2/(1+\gamma)})}{n_2}}+ r_{n_2}^{-2}\frac{d\log(1/\mathfrak{C}r_{n_2}^{2/(1+\gamma)})}{n_2} \le 1\\
    &\quad \Leftarrow r_{n_2}^{-2}\sqrt{\frac{\mathfrak{C}r_{n_2}^{2/(1+\gamma)}d\log(1/\mathfrak{C}r_{n_2}^{2/(1+\gamma)})}{n_2}} \le 1/2 \quad \textrm{and} \quad r_{n_2}^{-2}\frac{d\log(1/\mathfrak{C}r_{n_2}^{2/(1+\gamma)})}{n_2} \le 1/2.
\end{align*}
The first and the second inequalities give  
\begin{align*}
    \frac{\mathfrak{C}d}{n_2} \le r_{n_2}^{(2+4\gamma)/(1+\gamma)} \log(1/r_{n_2}) \quad \text{and} \quad \frac{\mathfrak{C}d}{n_2} \le r_{n_2}^2 \log(1/r_{n_2}).
\end{align*}
Solving these inequalities assuming $d \le n_2$, we arrive 
\begin{align*}
    &r_{n_2} \ge \mathfrak{C}\left(\left(\frac{d\log(n_2/d)}{n_2}\right)^{(1+\gamma)/(2+4\gamma)} + \left(\frac{d\log(n_2/d)}{n_2}\right)^{1/2} \right)\\
    &\quad \Leftrightarrow r^{2/(1+\gamma)}_{n_2} \ge \mathfrak{C}\left(\left(\frac{d\log(n_2/d)}{n_2}\right)^{1/(1+2\gamma)} + \left(\frac{d\log(n_2/d)}{n_2}\right)^{1/(1+\gamma)}\right).
\end{align*}
The first term is always larger then the second for $n_2 \le d$. 
Next, we evaluate the value related to $\omega^2_{\mathtt{pop}}$, which yields,
\begin{align*}
    u_{n_2}^{-4}\omega^2_{\mathtt{pop}}(\mathfrak{C}^{-1/(1+\gamma)}u_{n_2}^{2/(1+\gamma)}) \le n_2 &\Leftrightarrow  u_{n_2}^{-4}\mathfrak{C}^{-1/(1+\gamma)}u_{n_2}^{2/(1+\gamma)} \le n_2 \\
    &\Leftrightarrow  \mathfrak{C}^{-1/(1+2\gamma)}n_2^{-(1+\gamma)/2(1+2\gamma)} \le u_{n_2}.
\end{align*}
Thus we can take $u_{n_2}^{2/(1+\gamma)} =\mathfrak{C}n_2^{-1/(1+2\gamma)}$, but this is bounded up to a constant by the term coming from $r_{n_2}$. Hence \Cref{thm:clt-width} holds with 
\begin{align*}
    \mathrm{R}_{N, \alpha}^{\mathtt{CLT}} = \left(\frac{d\log(n_2/d)}{n_2}\right)^{1/(1+2\gamma)} + \widetilde{s}_{n_1, n_2}^{1/(1+\gamma)}.
\end{align*}

Next, from \ref{as:margin-global}, we have 
\begin{align*}
     \mathrm{Q}_{N, \alpha}^{\mathtt{CLT}} = C_\rho
^{-1}\left((1+|z_\alpha|)\widetilde s_{n_1, n_2}\right) = \frac{\mathfrak{C}(1+|z_\alpha|)\widetilde s_{n_1, n_2}}{t^*}.
\end{align*}

Finally, we relate this result to $\|\theta-\theta(P^N)\|$ using \ref{as:covariate-manski}. The result established thus far and \ref{as:covariate-manski} together imply that for any $\theta \in \widehat{\mathrm{CI}}^{\mathtt{Manski}}_{N,\alpha}$, 
\begin{align*}
    \|\theta - \theta(P^N)\|_2 \le  c_1^{-1}\mathfrak{C} d_{\Delta}(\theta_1, \theta(P^N)) \le 
\end{align*}
with high probability in view of in view of \Cref{thm:clt-width-local-full}. This proves the result.

\end{proof}
For \Cref{cor:manski-plugin}, we instead prove the slightly rephrased version of the corollary.
\begin{corollary}\label{cor:manski-plugin-full}
    
\end{corollary}
\begin{proof}[\bfseries{Proof of \Cref{cor:manski-plugin-full}}]
Next, we verify \ref{as:rate-initial-estimator3}. First, we have  
\begin{align*}
    \frac{1}{n_2^2}\sum_{i \in I_2} \E_{P_i}[\widehat \xi_i^2|D_1] \le \frac{1}{n_2^2}\sum_{i \in I_2}d_{\Delta, i}(\widehat{\theta}_1, \theta(P^N)) = \frac{1}{n_2}d_{\Delta}(\widehat{\theta}_1, \theta(P^N))\le C_1^2\|\widehat{\theta}_1-\theta(P^N)\|^2.
\end{align*}
A straightforward calculation gives 
\begin{align*}
    \mathbb{C}^2_2(\widehat\theta_1) \le d_{\Delta}(\widehat{\theta}_1, \theta(P^N)) \le C_1^2\|\widehat{\theta}_1-\theta(P^N)\|^2
\end{align*}
Unfortunately, this is not sufficient and we need a stronger regularity condition. Observe that 
\begin{align*}
    \mathrm{sgn}(\widehat{\theta}_1^\top X_i) \neq \mathrm{sgn}(\theta(P^N)^\top X_i) &\Rightarrow  \widehat{\theta}_1^\top X_i\cdot \theta(P^N)^\top X_i \le 0 \\
    &\Rightarrow  |\theta(P^N)^\top X_i| \le |\widehat{\theta}_1^\top X_i- \theta(P^N)^\top X_i|\\
    &\Rightarrow  |\theta(P^N)^\top X_i| \le \|\widehat{\theta}_1-\theta(P^N)\|\|X_i\|.
\end{align*}
Hence 
\begin{align*}
   \E_{P_i}(m_{\theta}-m_{\theta(P^N)}) &=\frac{1}{2}\E_{P_i}\left[ Y \left(\mathrm{sgn}(\theta(P^N)^\top X)-\mathrm{sgn}(\theta^\top X)\right)\right] \\
   &=\int_{\mathcal{A}(\theta)}\, |\E_{P_i}[Y |X=x] |P_X(x)\, dx\\
   &= 2 \int_{\mathcal{A}(\theta)}\, |\eta_{P_i}(x)-1/2|P_X(x)\, dx \\
   &\le 2 \int_{\mathcal{A}'(\theta,x)}\, |\eta_{P_i}(x)-1/2|P_X(x)\, dx \\
   &= 2 \|\widehat{\theta}_1 - \theta(P^N)\|^{\gamma}P_X(\mathcal{A}'(\theta,x)).
\end{align*}
\end{proof}

\subsection{Quantile without Positive Densities}
% For given $\eta \in (0,1)$, we consider the quantile loss defined as 
% \begin{align*}
%     m_{\theta}(x) := \eta(x-\theta)_+ + (1-\eta)(\theta-x)_+.
% \end{align*}
% Below, we denote by $\theta(P^N) \equiv \theta(P)$.
\begin{proof}[\bfseries{Proof of Theorem~\ref{thm:quantile-validity}}]
First, we establish the result for $\alpha \neq 1/2$ as an application of \Cref{thm:studentized-katz}. Consider the case when $\theta > \theta(P^N)$, then 
    \begin{align*}
    m_{\theta} - m_{\theta(P^N)} 
    & = \eta(X-\theta)_+ + (1-\eta)(\theta-X)_+ - \eta(X-\theta(P^N))_+ - (1-\eta)(\theta(P^N)-X)_+\\
    & = (1-\eta)(\theta-\theta(P^N)) \mathbf{1}\{X \le \theta(P^N)\} 
 -\eta(\theta-\theta(P^N))\mathbf{1}\{\theta < X\} \\
 &\quad+ \left\{(1-\eta)(\theta-X)-\eta(X-\theta(P^N))\right\} \mathbf{1}\{\theta(P^N) < X \le \theta\}\\
 &=(\theta-\theta(P^N)) \mathbf{1}\{X \le \theta(P^N)\} 
 -\eta(\theta-\theta(P^N))(1-\mathbf{1}\{X > \theta(P^N)\})\\
 &\quad+\eta(\theta-\theta(P^N))\mathbf{1}\{\theta(P^N) < X \le \theta\} \\
 &\quad+ \left\{\theta(P^N)-X +(1-\eta) (\theta - \theta(P^N))\right\} \mathbf{1}\{\theta(P^N) < X \le \theta\}\\
    & = (\theta-\theta(P^N)) \mathbf{1}\{X \le \theta(P^N)\} 
 -\eta(\theta-\theta(P^N))\\
 &\quad+\eta(\theta-\theta(P^N))\mathbf{1}\{\theta(P^N) < X \le \theta\} \\
 &\quad+ \left\{\theta(P^N)-X +(1-\eta) (\theta - \theta(P^N))\right\} \mathbf{1}\{\theta(P^N) < X \le \theta\}\\
    & = (\theta-\theta(P^N)) \mathbf{1}\{X \le \theta(P^N)\} 
 -\eta(\theta-\theta(P^N))+ (\theta-X) \mathbf{1}\{\theta(P^N) < X \le \theta\}.
    \end{align*}
Analogously, we have 
\[m_{\theta} - m_{\theta(P^N)} = \eta(\theta(P^N)-\theta) - (\theta(P^N)-\theta) \mathbf{1}\{X \le \theta(P^N)\} + (X-\theta) \mathbf{1}\{\theta < X \le \theta(P^N)\}\]
when $\theta(P^N) > \theta$. Taking expectations, we obtain 
\begin{align*}
    &\E_{P_i}[m_{\widehat\theta_1}(X_i) - m_{\theta(P^N)}(X_i)|D_1] \\
    &\quad= \E_{P_i}[ (\widehat\theta_1-X) \mathbf{1}\{\theta(P^N) < X \le \widehat\theta_1\}|D_1] + \E_{P_i}[ (X-\widehat\theta_1) \mathbf{1}\{\widehat\theta_1 < X \le \theta(P^N)\}|D_1].
\end{align*}
We now define 
\begin{align*}
    \widehat\xi_i = m_{\widehat\theta_1}(X_i)-m_{\theta(P^N)}(X_i) - \E_{P_i}[m_{\widehat\theta_1}(X_i) - m_{\theta(P^N)}(X_i)|D_1],
\end{align*}
and take $G_i =  \mathbf{1}\{X_i \le \theta(P^N)\}  - \eta$.  For $\widehat\theta_1 > \theta(P^N)$ and $|\widehat\theta_1 - \theta(P^N)| \le \delta_0$, under which \ref{as:cdf-Holder} holds, it follows that
    \begin{align*}
        &\frac{\mathbb{E}_{P_i}[|\widehat\xi_i  - (\widehat\theta_1-\theta(P^N))\cdot(\mathbf{1}\{X_i \le \theta(P^N)\}  - \eta)|^2|D_1]}{|\widehat\theta_1-\theta(P^N)|^2\E_{P_i}[(\mathbf{1}\{X \le \theta(P^N)\}  - \eta)^2|D_1]} \\
        &\quad=\frac{\mathbb{E}_{P_i}[|(\widehat\theta_1-X) \mathbf{1}\{\theta(P^N) < X \le \widehat\theta_1\}|D_1] - \mathbb{E}_{P_i}[(\widehat\theta_1-X) \mathbf{1}\{\theta(P^N) < X \le \widehat\theta_1\}|^2|D_1]}{|\widehat\theta_1-\theta(P^N)|^2\E_{P_i}[(\mathbf{1}\{X \le \theta(P^N)\}  - \eta)^2|D_1]} \\
        &\quad \le \frac{(\widehat\theta_1-\theta(P^N))^2\mathbb{P}_{P_i}(\theta(P^N) < X \le\widehat\theta_1) }{|\widehat\theta_1-\theta(P^N)|^2\eta(1-\eta)}\\
        &\quad\le \frac{M_1|\widehat{\theta}_1-\theta(P^N)|^\gamma +M_0|\widehat{\theta}_1 - \theta(P^N)|^{\gamma}}{\eta(1-\eta)}\le \frac{2M_0|\widehat{\theta}_1-\theta(P^N)|^\gamma }{\eta(1-\eta)}
    \end{align*}
    where we used \ref{as:cdf-Holder} and the fact that $M_0 > M_1$.  We can repeat the identical argument for $\widehat{\theta}_1 < \theta(P^N)$. Thus the requirement of \Cref{prop:linearization-prop} holds with 
    \begin{align*}
        \varphi(|\widehat{\theta}_1-\theta(P^N)|) = \frac{2M_0|\widehat{\theta}_1-\theta(P^N)|^\gamma }{\eta(1-\eta)}.
    \end{align*}
By \Cref{prop:linearization-prop}, we conclude
\begin{align*}
   R_{n_2}
        &= \inf_{\delta_0 \ge \delta > 0}\left\{2 \sqrt{\frac{2M_0\delta^\gamma }{\eta(1-\eta)}}+ \mathbb{P}_{P^1}(|\widehat\theta_1 - \theta(P^N)| > \delta)\right\} \\
        &\quad\quad +  \sum_{i\in I_2} \mathbb{E}_{P_i}\left[\frac{|\mathbf{1}\{X_i \le \theta(P^N)\}  - \eta|^2}{n_2\eta(1-\eta)}\min\left\{1,\,\frac{|\mathbf{1}\{X_i \le \theta(P^N)\}  - \eta|}{\sqrt{n_2\eta(1-\eta)}}\right\}\right].
    \end{align*}
    Furthermore, observe that 
    \begin{align*}
        \E_{P_i}|\mathbf{1}\{X_i \le \theta(P^N)\}  - \eta|^3 = (1-\eta)^3 \eta + \eta^3 (1-\eta),
    \end{align*}
    and hence we have 
    \begin{align*}
        &\sum_{i\in I_2} \mathbb{E}_{P_i}\left[\frac{|\mathbf{1}\{X_i \le \theta(P^N)\}  - \eta|^2}{n_2\eta(1-\eta)}\min\left\{1,\,\frac{|\mathbf{1}\{X_i \le \theta(P^N)\}  - \eta|}{\sqrt{n_2\eta(1-\eta)}}\right\}\right] \\
        &\quad \le \frac{(1-\eta)^3 \eta + \eta^3 (1-\eta)}{\eta(1-\eta)\sqrt{n_2\eta(1-\eta)}} =  \frac{(1-\eta)^2 + \eta^2}{\sqrt{n_2\eta(1-\eta)}} \le \frac{1}{\sqrt{n_2\eta(1-\eta)}}.
    \end{align*}
    This concludes the result by invoking \Cref{prop:linearization-prop}.

Next, we establish the result for $\alpha = 1/2$ as an application of \Cref{thm:coverage-anti-conservative-confidence-set-empirical-risk}. The proof is split into two cases: (1) $|\widehat\theta_1-\theta(P^N)| \le \delta_0$ and (2) $|\widehat\theta_1-\theta(P^N)| > \delta_0$ where $\delta_0 > 0$ is defined in \ref{as:cdf-Holder}. First we consider the case (1). We have shown that 
\begin{align*}
    \E_{P_i}[m_{\widehat\theta_1} - m_{\theta(P^N)} |D_1] &= \E_{P_i}[ (\widehat\theta_1-X) \mathbf{1}\{\theta(P^N) < X \le \widehat\theta_1\}|D_1] \\
    &\quad+ \E_{P_i}[ (X-\widehat\theta_1) \mathbf{1}\{\widehat\theta_1 < X \le \theta(P^N)\}|D_1].
\end{align*}
We observe that $(\widehat\theta_1 -X) \mathbf{1}\{\theta(P^N) < X \le \widehat\theta_1 \}$ is a non-negative random variable, taking values from $0$ to $\widehat\theta_1 -\theta(P^N)$. Then 
\begin{align*}
    &(\widehat\theta_1 -X) \mathbf{1}\{\theta(P^N) < X \le \widehat\theta_1 \} \\
    &\quad \ge \frac{1}{2}(\widehat\theta_1-\theta(P^N))\mathbf{1}\left\{\theta(P^N) < X \le \theta(P^N) + \frac{1}{2}(\widehat\theta_1-\theta(P^N))\right\} 
\end{align*}
and thus taking expectations both sides,
\begin{align*}
    &\E_{P_i}[ (\widehat\theta_1-X) \mathbf{1}\{\theta(P^N) < X \le \widehat\theta_1\} |D_1]\\
    &\quad \ge \frac{1}{2}(\widehat\theta_1-\theta(P^N))\left\{F\left(\theta(P^N) + \frac{1}{2}(\widehat\theta_1-\theta(P^N))\right) - F(\theta(P^N))\right\}\\
    &\quad \ge \frac{1}{2^{1+\gamma}}M_0 |\theta-\theta(P^N)|^{1+\gamma} - M_1\frac{1}{2^{1+\gamma}}|\theta-\theta(P^N)|^{1+\gamma}\\
    &\quad= \frac{M_0-M_1}{2^{1+\gamma}}|\theta-\theta(P^N)|^{1+\gamma}
\end{align*}
% When $\theta-\theta(P^N) \le \delta$,
% by \ref{as:cdf-Holder}, we can further lower bound the expectation as 
% \begin{align*}
%     P(m_{\theta} - m_{\theta(P^N)})  &\ge RM_0 (1-R)^{\beta}|\theta-\theta(P^N)|^{1+\beta} - M_1R(1-R)^\beta|\theta-\theta(P^N)|^{1+\beta}\\
%     &= (M_0-M_1)R (1-R)^{\gamma}|\theta-\theta(P^N)|^{1+\gamma}.
% \end{align*}
By the assumption that $M_0 > M_1$, we can conclude that 
\begin{align*}
    \E_{P_i}[m_{\widehat\theta_1} - m_{\theta(P^N)}|D_1]  \ge \mathfrak{C} (\widehat\theta_1-\theta(P^N))^{1+\gamma}
\end{align*}
where $\mathfrak{C}$ depends on $M_0$ and $M_1$. 
The case with $\theta(P^N) > \widehat\theta_1$ is analogous and omitted. 
    
Next, we consider the case where $|\widehat{\theta}_1 - \theta(P^N)| > \delta_0$. For any $\widehat{\theta}_1 = \theta(P^N) + \ell u$ such that $u \in \{-1, 1\}$ and $\ell > \delta_0$, we define $\bar\theta = \theta(P^N) + \ell u$. Since $\bar\theta = (1-\delta/\ell)\theta(P^N) + \delta/\ell \widehat{\theta}_1$, it follows by the convexity of $\theta \mapsto \E_{P_i}[m_{\theta}]$, 
\begin{align*}
     &(1-\delta/\ell) \E_{P_i}[m_{\theta(P^N)}] - \delta/\ell \E_{P_i}[m_{\theta}] \ge \E_{P_i}[m_{\bar \theta}]\\
     &\quad \Longleftrightarrow  \E_{P_i}[m_{\theta}]-\E_{P_i}[m_{\theta(P^N)}] \ge (\ell/\delta)(\E_{P_i}[m_{\bar \theta}] -\E_{P_i}[m_{\theta(P^N)}]) \ge\mathfrak{C}  |\theta-\theta(P^N)|\delta^{\gamma}.
\end{align*}
Putting together, we have established that 
\begin{align*}
    \E_{P_i}[m_{\widehat\theta_1} - m_{\theta(P^N)}|D_1]  \ge \mathfrak{C} |\widehat\theta_1-\theta(P^N)|\min\{|\widehat\theta_1-\theta(P^N)|^{\gamma}, \delta^\gamma\}.
\end{align*}
We have also shown that 
\begin{align*}
    m_{\theta} - m_{\theta(P^N)} &= (\theta(P^N)-\theta)(\eta - \mathbf{1}\{X \le \theta(P^N)\}) + (X-\theta) \mathbf{1}\{\theta < X \le \theta(P^N)\} \\
    &\quad + (\theta-\theta(P^N)) (\mathbf{1}\{X \le \theta(P^N)\} 
 -\eta)+ (\theta-X) \mathbf{1}\{\theta(P^N) < X \le \theta\},
\end{align*}
and this function is Lipschitz in the sense that 
\begin{align*}
    | m_{\widehat\theta_1} - m_{\theta(P^N)}|\le \{\max\{\eta, 1-\eta\}+1\}|\widehat\theta_1-\theta(P^N)|.
\end{align*}
Hence, we have 
\begin{align*}
    \mathrm{Var}_{P_i}[\widehat{\xi}_i |D_1] \le \E_{P_i}[  | m_{\widehat\theta_1} - m_{\theta(P^N)}|^2|D_1] \le  \{\max\{\eta, 1-\eta\}+1\}^2|\widehat\theta_1-\theta(P^N)|^2.
\end{align*}
Putting together, we obtain 
\begin{align*}
    \ratio^2 = \frac{(\sum_{i \in I_2}\E_{P_i}[m_{\widehat\theta_1} - m_{\theta(P^N)}|D_1] )^2}{\sum_{i\in I_2}\mathrm{Var}_{P_i}[\widehat{\xi}_i |D_1]} \ge \mathfrak{C} n_2 \min\{|\widehat\theta_1-\theta(P^N)|^{2\gamma}, \delta_0^{2\gamma}\} = \widetilde \Delta_2^2.
\end{align*}
Finally, the remainder term in \Cref{thm:coverage-anti-conservative-confidence-set-empirical-risk} becomes 
    \begin{align*}
    &\E_{P^1} \left[\min\left\{1, C\sum_{i\in I_2}\mathbb{E}_{P_i}\left[\frac{|\widehat\xi_i|^2}{n_2^2\widehat{\mathbb{V}}_{2}(1 + \ratio)^2}\min\left\{1,\,\frac{|\widehat \xi_i|}{n_2\widehat{\mathbb{V}}_{2}^{1/2}(1 + \ratio)}\right\} \bigg|D_1\right]\right\}\right]\\
    &\quad \le\E_{P^1} \left[\min\left\{1, \frac{C}{(1 + \widetilde\Delta_2)^2}\right\}\right].
\end{align*}
This concludes the claim.
\end{proof}

\begin{proof}[\bfseries{Proof of Theorem~\ref{thm:quantile-ci-width}}] The proof is a direct application of \Cref{thm:clt-width-local-full} and thus proceeds by verifying \ref{as:margin}, \ref{as:local-entropy}, and \ref{as:square-process} to hold locally $\rho = \delta_0$ where $\delta_0$ is defined in \ref{as:cdf-Holder}, and \ref{as:margin-global}, \ref{as:ratio-process} and \ref{as:ratio-square-process} to globally. 
\paragraph{Verifying \ref{as:margin}}
From the proof of \Cref{thm:quantile-validity}, we have established that 
\begin{align*}
    \frac{1}{n_2}\sum_{i \in I_2}\E_{P_i}[m_{\widehat\theta_1} - m_{\theta(P^N)}|D_1]  \ge \mathfrak{C} |\widehat\theta_1-\theta(P^N)|\min\{|\widehat\theta_1-\theta(P^N)|^{\gamma}, \delta_0^\gamma\},
\end{align*}
where $\mathfrak{C}$ is a constant depending on $M_0$ and $M_1$. Thus \ref{as:margin} holds with $\gamma$ and $c_0 = \mathfrak{C}$ when $|\widehat\theta_1-\theta(P^N)| \le \delta_0$.

\paragraph{Verifying \ref{as:local-entropy}}
To control this term, we observe that 
\begin{align*}
    | m_{\theta} - m_{\theta(P^N)}|\le \{\max\{\eta, 1-\eta\}+1\}|\theta-\theta(P^N)|\le 2|\theta-\theta(P^N)|,
\end{align*}
that is, the function is Lipschitz in parameter. We define the following collection of ``localized'' functions:
\begin{align*}
    \mathcal{M}_{\delta} := \left\{ m_{\theta} - m_{\theta(P^N)}  \text{ for all }\theta \text{ s.t., } |\theta - \theta(P^N)| \le \delta\right\}.
\end{align*}
Then using the notation $\mathbb{G}_{n_2}$, defined in \eqref{eq:definition-Gn}, we have 
\begin{align*}
    &\E^*_{P^2} \left[\sup_{|\theta - \theta(P^N)| <\delta}|(\widehat{\mathbb{M}}_2 - \mathbb{M}_2)(\theta) - (\widehat{\mathbb{M}}_2 - \mathbb{M}_2)(\theta(P^N))| \right]  \\
    &\quad = n_2^{-1/2}\E^*_{P^2} \left[\sup_{|\theta - \theta(P^N)| <\delta}| \mathbb{G}_{n_2}(m_{\theta} - m_{\theta(P^N)})| \right] =n_2^{-1/2}\E^*_{P^2} \left[\sup_{m \in \mathcal{M}_\delta}\left|\mathbb{G}_{n_2} m\right| \right].
\end{align*}
The last object can be related to the $\varepsilon$-bracketing numbers. Two functions $\ell$ and $u$ are defined to be an $\varepsilon$-bracket of functions $m \in \mathcal{M}_\delta$ if $l(x) \le m(x) \le u(x)$ for all $x$ and $\|\ell - u\| \le \varepsilon$. The bracketing number $\mathcal{N}_{[\,]}(\varepsilon, \mathcal{M}_\delta, \|\cdot\|)$ is the minimum number of $\varepsilon$-brackets required for covering $\mathcal{M}_\delta$. Then by Theorem 2.14.2 of \cite{van1996weak}, we have 
\begin{align*}
    \E^*_{P^2} \left[\sup_{m \in \mathcal{M}_\delta}\left|\mathbb{G}_{n_2} m\right| \right] \lesssim(\E_{P^2}M^2)^{1/2}  \int_0^1 \sqrt{1 + \log \mathcal{N}_{[\,]}(\varepsilon\|M\|_{P^2,2}, \mathcal{M}_\delta, L_2(P^2))} \, d\varepsilon
\end{align*}
where the envelop function can be taken as $M = 2\delta$, and the Lipschtiz constant is $2$. Furthermore, Theorem 2.7.11 of \cite{van1996weak} shows that 
\begin{align*}
    \mathcal{N}_{[\,]}(\varepsilon\|M\|_{P,2}, \mathcal{M}_\delta, L_2(P)) \le \mathcal{N}(\delta\varepsilon/4, \Theta_\delta, |\cdot|) \lesssim  1/\varepsilon
\end{align*}
where $\Theta_\delta = [-\delta, \delta]$. Evaluating the integral, we have 
\begin{align*}
    \E_{P^2} \left[\sup_{m \in \mathcal{M}_\delta}\left|\mathbb{G}_{n_2} m\right| \right] &\le C\delta \int_0^1\sqrt{1 + \log (1/\varepsilon)}\, d\varepsilon \le C\delta  + C\delta \int_0^1\sqrt{\log (1/\varepsilon)}\, d\varepsilon \le C\delta,
\end{align*}
where $C$ is a universal constant that changes line by line. We thus conclude that $\phi_{n_2}(\delta) = C\delta/\sqrt{n_2}$. Hence the requirement $q < 1+\gamma$ is satisfied with $q=1$ globally. 

\paragraph{Verifying \ref{as:square-process}}
Using the Lipschitz continuity (in parameter) of $m_{\theta} - m_{\theta(P^N)}$, it follows, 
\begin{align*}
    \omega_{\mathtt{pop}}^2(\delta) = \sup_{|\theta - \theta(P^N)| \le \delta}\frac{1}{n_2}\sum_{i \in I_2} \E_{P^2}[| m_{\theta} - m_{\theta(P^N)}|^2] \le 4\delta^2.
\end{align*}
To control $\omega_{n_2, \mathtt{emp}}$, we employ \Cref{prop:squared-Gn}. 
By \Cref{prop:squared-Gn} with $q=2$, we have 
\begin{align*}
    &\E_{P^2}^* \left[\sup_{\|\theta-\theta(P^N)\| < \delta}\left|\frac{1}{n_2}\sum_{i\in I_2}(m_\theta - m_{\theta(P^N)})^2 - \mathbb{E}_{P_i}[(m_\theta - m_{\theta(P^N)})^2]\right| \right]\\
    &\quad \le 64 \delta^2 + 16\cdot 8^{1/2}\, n_2^{1/2} \delta\,\phi_{n_2}(\delta) \le C\delta^2,
\end{align*}
for some universal constant $C$. Hence the requirement $q < 1+\gamma$ is satisfied with $q=1$ globally. 
\paragraph{Verifying \ref{as:margin-global}}
Setting $\delta > \delta_0$, we have shown that 
\begin{align*}
    C_{\delta_0}(|\theta - \theta(P^N)|) = \delta_0^\gamma |\theta - \theta(P^N)|.
\end{align*}

\paragraph{Verifying \ref{as:ratio-process}}
For any $\delta > \delta_0$,
\begin{align*}
    &\E^*_{P^2} \left[\sup_{\|\theta-\theta(P^N)\| > \delta}\left|\frac{(\widehat{\mathbb{M}}_2 - \mathbb{M}_2)(\theta) - (\widehat{\mathbb{M}}_2 - \mathbb{M}_2)(\theta(P^N))}{\mathbb{M}_2(\theta) -\mathbb{M}_2(\theta(P^N))}\right| \right] \le \frac{C}{\delta_0^\gamma \sqrt{n_2}}= R(n_2, \delta_0),
\end{align*}
where $C$ is a universal constant. Hence, \ref{as:ratio-process} holds with $C_{\mathtt{ratio}}=1/\varepsilon_{\mathtt{ratio}}$ by Markov's inequality.

\paragraph{Verifying \ref{as:ratio-square-process}} Similarly for any $\delta > \delta_0$,
\begin{align*}
    &\E_{P^2}^* \left[\sup_{\|\theta-\theta(P^N)\| > \delta_0}\frac{z_\alpha^2}{n_2^2}\left|\frac{\sum_{i\in I_2}(m_\theta - m_{\theta(P^N)})^2 - \mathbb{E}_{P_i}[(m_\theta - m_{\theta(P^N)})^2]}{\{\mathbb{M}_2(\theta) -\mathbb{M}_2(\theta(P^N))\}^2}\right| \right]\le \frac{(1+z_\alpha^2)C}{\delta_0^{2\gamma}n_2},
\end{align*}
and 
\begin{align*}
    &\sup_{\|\theta-\theta(P^N)\| > \delta_0}\frac{z_\alpha^2\sum_{i\in I_2}\mathbb{E}_{P_i}[(m_\theta - m_{\theta(P^N)})^2(Z_i)]}{n_2^2\{\M_2(\theta)-\M_2(\theta(P^N))\}^2} \le \frac{(1+z_\alpha^2)C}{\delta_0^{2\gamma}n_2}.
\end{align*}
Hence we can set
\begin{align*}
     S_{\mathtt{emp}}(n_2, \delta_0, \alpha)= S_{\mathtt{pop}}(n_2, \delta_0, \alpha) = \frac{(1+z_\alpha^2)C}{\delta_0^{2\gamma}n_2},
\end{align*}
and \ref{as:ratio-square-process} holds with $\widetilde C_{\mathtt{emp}}=1/\varepsilon_{\mathtt{emp}}$ by Markov's inequality.

We assume that $n_2$ is large enough satisfies the following:
\begin{align}\label{eq:eqreuiment_for_n2-quantile}
    C_{\varepsilon^\circ}\max\{R(n_2, \rho) , S_{\mathtt{emp}}(n_2, \rho, \alpha), S_{\mathtt{pop}}(n_2, \rho, \alpha)\} \le 1/3,
\end{align}
where $C_{\varepsilon^\circ}$ is a constant depending on $\varepsilon^\circ$. Then \eqref{eq:eqreuiment_for_n2-quantile} is satisfied for $n_2$ large such that $C'_{\varepsilon^\circ}(1+|z_\alpha|)^2\delta_0^{2\gamma} \le n_2$.

% \paragraph{Verifying \ref{as:ratio-square-process}}

\paragraph{Evaluating the rate of convergence}
We now evaluate the rate of convergence by applying \Cref{thm:clt-width} and \Cref{thm:clt-width-local-full}. Denote by $\mathfrak{C}$ a depending on $M_0, M_1$ and $\gamma$ that changes from line to line. For $\delta \le \delta_0$,
\begin{align*}
    r_{n_2}^{-2}\phi_{n_2}(\mathfrak{C}r_{n_2}^{2/(1+\gamma)}) \le 1 \Leftrightarrow \frac{\mathfrak{C}}{\sqrt{n_2}} \le r_{n_2}^{2\gamma/(1+\gamma)}\Leftrightarrow \frac{\mathfrak{C}}{n_2^{1/(2\gamma)}} \le r_{n_2}^{2/(1+\gamma)},
\end{align*}
\begin{align*}
    u_{n_2}^{-4}\omega^2_{\mathtt{pop}}(\mathfrak{C}u_{n_2}^{2/(1+\gamma)}) \le n_2 \Leftrightarrow \mathfrak{C} n_2^{-1}\le u_{n_2}^{\mathfrak{C}\gamma/(1+\gamma)}\Leftrightarrow \frac{4}{n_2^{1/(2\gamma)}} \le u_{n_2}^{2/(1+\gamma)},\quad \text{and}
\end{align*}
\begin{align*}
    u_{n_2}^{-4}\omega^2_{n_2, \mathtt{emp}}(\mathfrak{C}u_{n_2}^{2/(1+\gamma)}) \le n_2 \Leftrightarrow \mathfrak{C} n_2^{-1}\le u_{n_2}^{4\gamma/(1+\gamma)}\Leftrightarrow \frac{\mathfrak{C}}{n_2^{1/(2\gamma)}} \le u_{n_2}^{2/(1+\gamma)}.
\end{align*}
Hence we have $\mathrm{R}_{N}^{\mathtt{CLT}}=\mathfrak{C}(n_2^{-1/(2\gamma)} + \widetilde s_{n_1, n_2}^{1/(1+\gamma)})$ where $\mathfrak{C}$ depends on $M_0, M_1$ and $\gamma$.

Next, since $C_{\delta_0}(|\theta - \theta(P^N)|) = \delta_0^\gamma|\theta - \theta(P^N)|$, we have 
\begin{align*}
    \mathrm{Q}_{N, \alpha}^{\mathtt{CLT}}=C^{-1}_{\delta_0}((1+|z_\alpha|)\widetilde s_{n_1, n_2}) = \delta_0^{-\gamma}(1+|z_\alpha|)\widetilde s_{n_1, n_2}.
\end{align*}
This concludes the claim in view of \Cref{thm:clt-width-local-full}.

\end{proof}

For \Cref{cor:quantile-plugin}, we instead prove the slightly rephrased version of the corollary. 
\begin{corollary}\label{cor:quantile-plugin-full}
    Suppose the initial estimator satisfies for all $n_1 \geq N_1$, \begin{align}\nonumber\mathbb{P}_{P^1}\left(|\widehat\theta_1 - \theta(P^N)| \le \widetilde C_{\mathtt{init}} n_1^{-1/(2\gamma)} \right) \ge 1-\widetilde\varepsilon_{\mathtt{init}}.
    \end{align}
     Assume \ref{as:cdf-Holder}. For any $\varepsilon \in (0, 1-\widetilde \varepsilon_{\mathtt{init}})$, setting $\varepsilon^\circ = \varepsilon + \widetilde \varepsilon_{\mathtt{init}}$, $n_1 \ge N_1$, with probability at least $1-\varepsilon^\circ$, 
\begin{equation*}
\mathrm{Diam}_{|\cdot|}\big(\widehat{\mathrm{CI}}^{\mathtt{CLT}}_{N, \alpha}\big) \le C_{\varepsilon^\circ}\max\bigg\{\left(1+|z_\alpha|\right)^{1/\gamma}(n_2^{-1/(2\gamma)} + n_1^{-1/(2\gamma)}), \mathrm{Q}^{\mathtt{CLT}}_{N,\alpha}\mathbf{1}\{\mathrm{Q}^{\mathtt{CLT}}_{N,\alpha} \ge \delta_0\}\bigg\},
\end{equation*} 
provided $\max\{2, C_{\varepsilon^\circ}'(1+|z_\alpha|)^2\delta_0^{2\gamma}\} \le n_2$, where $\mathrm{Q}_{N, \alpha}^{\mathtt{CLT}}= \delta_0^{-\gamma}(1+|z_\alpha|)|\widehat\theta_1-\theta(P^N)|$ and $C_{\varepsilon^\circ}$ depends on $M_0, M_1$ and $\gamma$, while $C_{\varepsilon^\circ}'$ only depends on ${\varepsilon^\circ}$.
\end{corollary}
\begin{proof}[\bfseries{Proof of \Cref{cor:quantile-plugin-full}}]
    When $|\widehat\theta_1-\theta(P^N)| \le \delta_0$, we have 
\begin{align*}
    \frac{1}{n_2}\mathbb{E}_{P^2| P^1}\left[\frac{1}{n_2}\sum_{i \in I_2} \widehat \xi_i^2\right] &\le \frac{1}{n_2}\mathbb{E}_{P^2| P^1}\left[|m_{\widehat\theta_1}(X_i)-m_{\theta(P^N)}(X_i)|^2|D_1\right]  \\
    &\le \frac{1}{n_2}(\widehat\theta_1-\theta(P^N))^2\mathbb{P}_{P^1}(\theta(P^N) < X \le \widehat\theta_1)\\
    &\le \frac{1}{n_2}(M_1+M_0)(\widehat\theta_1-\theta(P^N))^{2+\gamma},
\end{align*}
and
\begin{align*}
    \widehat{\mathbb{C}}_2^2 \le (M_1+M_0)^2(\widehat\theta_1-\theta(P^N))^{2+2\gamma}.
\end{align*}
Consider the event 
\begin{align*}
    \Omega_{\mathtt{init}} := \left\{|\widehat\theta_1 - \theta(P^N)|\le \widetilde C_{\mathtt{init}} n_1^{-1/(2\gamma)}\right\}.
\end{align*}
Then on this event and $n_1$ sufficiently large such that $\widetilde C_{\mathtt{init}} n_1^{-1/\gamma} \le \delta_0$, \ref{as:rate-initial-estimator3} holds with
\begin{align*}
    \widetilde s_{n_1, n_2}^2 = n_2^{-(1+\gamma)/(2\gamma)}n_1^{-(1+\gamma)/(2\gamma)}+n_1^{-(2+2\gamma)/(2\gamma)}
\end{align*}
since
\begin{align*}
     &\frac{1}{n_2}\mathbb{E}_{P^2| P^1}\left[\frac{1}{n_2}\sum_{i \in I_2} \widehat \xi_i^2\right] + \widehat{\mathbb{C}}_2^2\\
     &\quad \le \frac{|\widehat\theta_1-\theta(P^N)|}{n_2}(M_1+M_0)|\widehat\theta_1-\theta(P^N)|^{1+\gamma} + (M_1+M_0)^2|\widehat\theta_1-\theta(P^N)|^{2+2\gamma}\\
     &\quad \le  \frac{\widetilde C_{\mathtt{init}}^{2+\gamma}}{n_2^{1+1/(2\gamma)}}(M_1+M_0)n_1^{-(1+\gamma)/(2\gamma)} + \widetilde C_{\mathtt{init}}^{2+2\gamma}(M_1+M_0)^2 n_1^{-(2+2\gamma)/(2\gamma)}\\
     &\quad \le \max\{\widetilde C_{\mathtt{init}}^{2+\gamma}(M_1+M_0), \widetilde C_{\mathtt{init}}^{2+2\gamma}(M_1+M_0)^2\}  \widetilde s_{n_1, n_2}^2.
\end{align*}
with probability greater than $1 - \widetilde \varepsilon_{\mathtt{init}}$. In particular, we used the fact that 
\begin{align*}
    n_2^{-1-1/(2\gamma)} = n_2^{-(2\gamma+1)/(2\gamma)} \le n_2^{-(\gamma+1)/(2\gamma)}.
\end{align*}
Then the result follows with 
\begin{align*}
    n_2^{-1/(2\gamma)} + \widetilde s_{n_1, n_2}^{1/(1+\gamma)} &\le n_2^{-1/(2\gamma)}  + n_2^{-1/2(2\gamma)}n_1^{-1/2(2\gamma)}+n_1^{-1/(2\gamma)}
 \\
 &\lesssim n_2^{-1/(2\gamma)}  +n_1^{-2/(2\gamma)}
 \end{align*}
 by AM-GM inequality. When $|\widehat\theta_1 - \theta(P^N)| \ge \delta_0$, we can instead use
 \begin{align*}
    \frac{1}{n_2}\mathbb{E}_{P^2| P^1}\left[\frac{1}{n_2}\sum_{i \in I_2} \widehat \xi_i^2\right] &\le \frac{1}{n_2}\mathbb{E}_{P^2| P^1}\left[|m_{\widehat\theta_1}(X_i)-m_{\theta(P^N)}(X_i)|^2|D_1\right]  \\
    &\le \frac{1}{n_2}(\widehat\theta_1-\theta(P^N))^2\mathbb{P}_{P^1}(\theta(P^N) < X \le \widehat\theta_1)\\
    &\le \frac{4}{n_2}(\widehat\theta_1-\theta(P^N))^{2},
\end{align*}
and
\begin{align*}
    \widehat{\mathbb{C}}_2^2 \le 4(\widehat\theta_1-\theta(P^N))^{2},
\end{align*}
where we simply used $\mathbb{P}_{P^1}(\theta(P^N) < X \le \widehat\theta_1) \le 1$. Hence we have 
\begin{align*}
    C_{\delta_0}^{-1}((1+|z_\alpha|)\widetilde s_{n_1, n_2}) = \frac{4(1+|z_\alpha|)}{\delta_0^\gamma}|\widehat\theta_1-\theta(P^N)| = \mathrm{Q}_{N, \alpha}^{\mathtt{CLT}}.
\end{align*}
Plugging them into the final expression of \Cref{thm:quantile-ci-width} concludes the claim.
\end{proof}

\subsection{Discrete Argmin Inference}
\begin{proof}[\bfseries{Proof of \Cref{thm:arginf-valid}}]
The proof is an application of \Cref{thm:coverage-anti-conservative-confidence-set-empirical-risk} for $\alpha = 1/2$ and \Cref{thm:studentized-katz} for $\alpha \neq 1/2$. 

For both cases, we observe that 
\begin{align*}
        &\mathbb{P}_{P^N}(\theta(P^N) \notin \widehat{\mathrm{CI}}_{N,\alpha}^{\mathtt{CLT}}) \\
        &\quad= \mathbb{P}_{P^N}(\theta(P^N) \notin \widehat{\mathrm{CI}}_{N,\alpha}^{\mathtt{CLT}} \,\cap \,\{\widehat\theta_1 \in \mathcal{S}^*\}) + \mathbb{P}_{P^N}(\theta(P^N) \notin \widehat{\mathrm{CI}}_{N,\alpha}^{\mathtt{CLT}} \,\cap \,\{\widehat\theta_1 \in \mathcal{S}^c\})\\
        &\quad=  \mathbb{P}_{P^N}(\theta(P^N) \notin \widehat{\mathrm{CI}}_{N,\alpha}^{\mathtt{CLT}} \,\cap \,\{\widehat\theta_1 \in \mathcal{S}^c\}),
    \end{align*}
since $\widehat\theta_1 \in \widehat{\mathrm{CI}}_{N,\alpha}^{\mathtt{CLT}}$ almost surely and by the convention \eqref{eq:miscoverage-over-set}. Hence throughout, we assume that $\widehat\theta_1 \in \mathcal{S}^c$. Observe that for any $k \in \mathcal{S}^*$,
    \begin{align*}
        \widehat \xi_i= (e_{\widehat{\theta}_1}-e_{k})^\top(X_i-\mu) = D_i^{\widehat{\theta}_1,k}
    \end{align*}
    and 
    \begin{align*}
        n_2^2\widehat{\mathbb{V}}_{2} &= \sum_{i\in I_2} \mathrm{Var}_{P_i}[\widehat \xi_i] = \sum_{i \in I_2} (e_{\widehat{\theta}_1}-e_{k})^\top \Sigma_i (e_{\widehat{\theta}_1}-e_{k}) = n_2\sigma^2_{\widehat\theta_1, k}.
    \end{align*}
    Then, the upper bound in \Cref{thm:studentized-katz} becomes 
    \begin{align*}
        &\E_{P^1}\left[\min\left\{1, C\,\sum_{i\in I_2} \mathbb{E}_{P_i}\left[\frac{|D_i^{\widehat{\theta}_1, k}|^2}{n_2\sigma_{\widehat\theta_1, k}^2}\min\left\{1,\,\frac{|D_i^{\widehat{\theta}_1, k}|}{n_2^{1/2}\sigma_{\widehat\theta_1, k}}\right\}\mathbf{1}\{\widehat\theta_1 \notin \mathcal{S}^*\}\bigg|D_1\right]\right\}\right] \\
        &\quad \le \min\left\{1, \sup_{(j, k)\in \mathcal{S}^c\times \mathcal{S}^*}C\sum_{i\in I_2} \mathbb{E}_{P_i}\left[\frac{|D_i^{j, k}|^2}{n_2\sigma_{j, k}^2}\min\left\{1,\,\frac{|D_i^{j, k}|}{n_2^{1/2}\sigma_{j, k}}\right\}\right]\right\}.
    \end{align*}

Next, we consider the case for $\alpha = 1/2$. Again, assuming that $\{\widehat\theta_1 \in \mathcal{S}^c\}$, and for any $j, k \in \mathcal{S}^c \times \mathcal{S}^*$
\begin{align*}
    \frac{\mathbb{C}^2_2(j)}{\mathbb{V}_2(j)} = \frac{n_2\delta_{j,k}^2}{\sigma_{j, k}^2} \ge \min_{(j, k)\in \mathcal{S}^c\times \mathcal{S}^*}\frac{n_2\delta_{j,k}^2}{\sigma_{j, k}^2} =\widetilde\Delta_2^2 \quad \text{where}  \quad \delta_{j,k}  = e_j^\top \mu - e_k^\top \mu.
\end{align*}
Hence \Cref{thm:coverage-anti-conservative-confidence-set-empirical-risk} implies,
\begin{align*}
    \mathbb{P}_{P^N}(\theta(P^N) \notin \widehat{\mathrm{CI}}_{N,\alpha}^{\mathtt{CLT}}) \le 1-\Phi(\widetilde\Delta_2) +\min\left\{1, \frac{C}{(1+\widetilde\Delta_2)^2}\right\},
\end{align*}
under no moment assumptions beyond finite variance. 
\end{proof}
\subsection{Auxiliary Results for Statistical Applications}\label{sec:technical-lemma}
\subsubsection{Results for Validity}
\begin{proof}[\bfseries{Proof of \Cref{prop:linearization-prop}}]
    Define random variable 
    \begin{align*}
        \Delta_{1} = \min\left\{1, C\sum_{i\in I_2} \mathbb{E}_{P_i}\left[\frac{|\widehat\xi_i|^2}{n_2^2\widehat{\mathbb{V}}_{2}}\min\left\{1,\,\frac{|\widehat \xi_i|}{n_2\widehat{\mathbb{V}}_{2}^{1/2}}\right\} \bigg| D_1\right]\right\}.
    \end{align*}
    Since $\Delta_{1} \le 1$, we can safely introduce the following indicator function, for any $0 \le \delta \le \delta_0$, 
    \begin{align*}
        &\Delta_{1} \le \Delta_{1} \mathbf{1}\{{\|\widehat{\theta}_1-\theta(P^N)\| < \delta}\} + \mathbf{1}\{{\|\widehat{\theta}_1-\theta(P^N)\| \ge \delta}\}.
    \end{align*}
    We also define 
    \[\widehat{\mathbb{V}}_H = \sum_{i \in I_2}\E_{P_i}[\langle u, G_i\rangle^2].\]
    Then, first term can be controlled as 
    \begin{align*}
        &\sum_{i\in I_2}\mathbb{E}_{P_i}\left[\frac{|\widehat\xi_i|^2}{n_2^2\widehat{\mathbb{V}}_{2}}\min\left\{1,\,\frac{|\widehat \xi_i|}{n_2\widehat{\mathbb{V}}_{2}^{1/2}}\right\} \bigg| D_1\right] \\
        &\quad = \sum_{i\in I_2} \mathbb{E}_{P_i}\left[\frac{|\widehat\xi_i|^2}{n_2^2\widehat{\mathbb{V}}_{2}}\min\left\{1,\,\frac{|\widehat \xi_i|}{n_2\widehat{\mathbb{V}}_{2}^{1/2}}\right\} \bigg| D_1\right]\\
        &\quad \quad+ \sum_{i\in I_2} \mathbb{E}_{P_i}\left[\frac{|\langle\widehat{\theta}_1 - \theta(P^N),  G_i\rangle|^2}{\widehat{\mathbb{V}}_H}\min\left\{1,\,\frac{|\langle \widehat{\theta}_1 - \theta(P^N),  G_i\rangle|}{\widehat{\mathbb{V}}_H^{1/2}}\right\}\bigg| D_1\right]\\
        &\quad\quad - \sum_{i\in I_2} \mathbb{E}_{P_i}\left[\frac{|\langle \widehat{\theta}_1 - \theta(P^N),  G_i\rangle|^2}{\widehat{\mathbb{V}}_H}\min\left\{1,\,\frac{|\langle \widehat{\theta}_1 - \theta(P^N),  G_i\rangle|}{\widehat{\mathbb{V}}_H^{1/2}}\right\}\bigg| D_1\right] \\
        &\quad =\sum_{i\in I_2} \mathbb{E}_{P_i}\left[\frac{|\langle \widehat{\theta}_1 - \theta(P^N),  G_i\rangle|^2}{\widehat{\mathbb{V}}_H}\min\left\{1,\,\frac{|\langle \widehat{\theta}_1 - \theta(P^N),  G_i\rangle|}{\widehat{\mathbb{V}}_H^{1/2}}\right\}\bigg| D_1\right]  \\
        &\quad\quad + 5 \varphi(\|\widehat{\theta}_1 - \theta(P^N)\|) + 10\sqrt{\varphi(\|\widehat{\theta}_1 - \theta(P^N)\|)},
    \end{align*}
    where the last step follows from \Cref{lemma:katz-local}. Since it is assumed that $\|\widehat\theta_1 - \theta(P^N)\|<\delta$,
    \begin{align*}
        &\sum_{i\in I_2} \mathbb{E}_{P_i}\left[\frac{|\langle u,  G_i\rangle|^2}{{\mathbb{V}}_H}\min\left\{1,\,\frac{|\langle u,  G_i\rangle|}{{\mathbb{V}}_H^{1/2}}\right\}\bigg| D_1\right]  + 5 \varphi(\delta) + 10\sqrt{\varphi(\delta)}\\
        &\quad \le \sup_{u \in \mathbb{S}^{d-1}}\, \sum_{i\in I_2} \mathbb{E}_{P_i}\left[\frac{|\langle u,  G_i\rangle|^2}{{\mathbb{V}}_H}\min\left\{1,\,\frac{|\langle u,  G_i\rangle|}{{\mathbb{V}}_H^{1/2}}\right\}\right]  + 5 \varphi(\delta) + 10\sqrt{\varphi(\delta)},
    \end{align*}
    where $\mathbb{V}_G = \sum_{i \in I_2}\E_{P_i}[\langle u, G_i\rangle^2]$. The result follows by taking expectation over $P^1$ and taking infimum over $0 \le \delta \le \delta_0$. In particular, when taken as the minimum between one, we have $\varphi(\delta) \le \sqrt{\varphi(\delta)}$, thus only $\sqrt{\varphi(\delta)}$ appears in the bound.
\end{proof}

\begin{lemma}\label{lemma:katz-local}
    Suppose $\{X_i\}_{i=1}^N$ and $\{Y\}_{i=1}^N$ are independent but not identically distributed random variables. Denote $V_X = \sum_{i=1}^N \E_{P_i}[X_i^2]$ and 
    $V_Y = \sum_{i=1}^N \E_{P_i}[Y_i^2]$
    such that $\mathbb{E}_{P_i}[|X_i-Y_i|^2] \le C \mathbb{E}_{P_i}[X_i^2]$ for all $1 \le i \le N$. Then  
    \begin{align*}
        \sum_{i=1}^N\mathbb{E}_{P_i}\left[\frac{X_i^2}{V_X}\min\left\{\frac{|X_i|}{\sqrt{V_X}},1\right\}\right] -  \sum_{i=1}^N\mathbb{E}_{P_i}\left[  \frac{Y_i^2}{V_Y}\min\left\{\frac{|Y_i|}{\sqrt{V_Y}},1\right\}\right]\le 5C + 10C^{1/2}.
    \end{align*}
\end{lemma}
\begin{proof}[\bfseries{Proof of \Cref{lemma:katz-local}}]
    Suppose $\{X_i\}_{i=1}^N$ and $\{Y_i\}_{i=1}^N$ satisfy $\mathbb{E}_{P_i}[|X_i-Y_i|^2] \le C \mathbb{E}_{P_i}[X_i^2]$ for all $ 1\le i \le N$. We observe 
    \begin{align*}
        \mathbb{E}_{P_i}\left[\left|\frac{X_i^2}{V_X}-\frac{Y_i^2}{V_Y}\right|\right] &=   \mathbb{E}_{P_i}\left[\left|\frac{X_i^2V_Y - Y_i^2V_X}{V_X V_Y}\right|\right] \\
        &=   \mathbb{E}_{P_i}\left[\left|\frac{X_i^2V_Y - Y_i^2V_X - Y_i^2V_Y +  Y_i^2V_Y}{V_X V_Y}\right|\right]\\
        &\le   \mathbb{E}_{P_i}\left[\left|\frac{(X_i^2 - Y_i^2)V_Y}{V_X V_Y}\right|\right] + \mathbb{E}_{P_i}\left[\left|\frac{Y_i^2(V_Y - V_X)}{V_X V_Y}\right|\right]\\
        &=   \mathbb{E}_{P_i}\left[\left|\frac{(X_i^2 - Y_i^2)}{V_X }\right|\right] + \left|\frac{\sum_{i=1}\E_{P_i}(Y_i^2 - X_i^2)}{V_X}\right| \frac{\E_{P_i}[Y_i^2]}{V_Y}.
    \end{align*}
    Summing over $i$, we obtain 
    \begin{align*}
        \sum_{i=1}^N \mathbb{E}_{P_i}\left[\left|\frac{X_i^2}{V_X}-\frac{Y_i^2}{V_Y}\right|\right] &\le  \sum_{i=1}^N\mathbb{E}_{P_i}\left[\left|\frac{(X_i^2 - Y_i^2)}{V_X }\right|\right] + \left|\frac{\sum_{i=1}\E_{P_i}(Y_i^2 - X_i^2)}{V_X}\right| \\
        &\le 2\sum_{i=1}^N\mathbb{E}_{P_i}\left[\frac{|X_i^2 - Y_i^2|}{V_X }\right]
    \end{align*}
    by Jensen's inequality. Since $Y^2=(Y-X)^2+2 X(Y-X)+X^2$, we obtain
\begin{align*}
	\mathbb{E}_{P_i}[| Y_i^2-X_i^2|] &= \mathbb{E}_{P_i}[|Y_i-X_i|^2] +2\mathbb{E}_{P_i}[ X_i(Y_i-X_i)]\\
	&\leq  \mathbb{E}_{P_i}[|Y_i-X_i|^2]+2\left(\mathbb{E}_{P_i}[X_i^2]\right)^{1/2}\left(\mathbb{E}_{P_i}[|Y_i-X_i|^2]\right)^{1 / 2}\\
	&\le  C \mathbb{E}_{P_i}[X_i^2]+2C^{1/2}\mathbb{E}_{P_i}[X_i^2]
	\end{align*}
by Cauchy-Schwarz inequality. Hence we conclude
\begin{align*}
    \sum_{i=1}^N \mathbb{E}_{P_i}\left[\left|\frac{X_i^2}{V_X}-\frac{Y_i^2}{V_Y}\right|\right] \le 2\sum_{i=1}^N\left[\frac{(C+2C^{1/2})\mathbb{E}_{P_i}[X_i^2]}{V_X }\right] \le 2(C+2C^{1/2}).
\end{align*}
    Define a function $g:\mathbb{R}_+ \mapsto \mathbb{R}_+$ such that 
    \begin{align*}
        g(x) = x\min\{\sqrt{x}, 1\}.
    \end{align*}
    We claim that this function is Lipschitz. When $x, y \ge 1$,
    \begin{align*}
        |g(x)-g(y)| = |x-y|.
    \end{align*}
    Next, when $x, y < 1$,
    \begin{align*}
        |g(x) - g(y)| = |x^{3/2}-y^{3/2}| = \frac{3}{2} \xi^{1/2}|x-y|,
    \end{align*}
    by the mean value theorem and $\min(x,y) \le \xi \le \max(x,y)$. Since $x,y \le 1$, we conclude $|g(x) - g(y)| = \frac{3}{2} |x-y|$. Finally, for $x \ge 1, y < 1$,
    \begin{align*}
        |g(x)-g(y)| = |x - y^{3/2}| \le  |x -1| +  |1 - y^{3/2}| \le |x-y| + \frac{3}{2}|x-y|= \frac{5}{2}|x-y|.
    \end{align*}
    The case with $y \ge 1, x<1$ is analogous. Hence $g$ is $5/2$-Lipschitz. This implies that 
    \begin{align*}
    &\frac{X_i^2}{V_X}\min\left\{\frac{|X_i|}{\sqrt{V_X}},1\right\} -    \frac{Y_i^2}{V_Y}\min\left\{\frac{|Y_i|}{\sqrt{V_Y}},1\right\} \\
    &\quad \le
        \left|g\left(\frac{X_i^2}{V_X}\right) - g\left(\frac{Y_i^2}{V_Y}\right)\right|\le \frac{5}{2}\left|\frac{X_i^2}{V_X}-\frac{Y_i^2}{V_Y}\right|.
    \end{align*}
    We conclude the result by taking the expectation both sides, summing over $1\le i \le N$ and applying the first result. 
\end{proof}
\subsubsection{Results for width analysis}

\begin{proof}[\bfseries{Proof of Lemma \ref{prop:squared-Gn}}]
Define a sequence of independent Rademacher random variables. Then, we have 
    \begin{align*}
    	\E^*_{P^2} \left[\sup_{m \in \mathcal{M}_\delta}\, |(\mathbb{P}_n-P^2) m^2| \right]
     & \le 2n_2^{-1}\E^*_{P^2\times \varepsilon} \left[\sup_{m \in \mathcal{M}_\delta}\, \left|\sum_{i \in I_2} \epsilon_i m^2(Z_i)\right|\right]\\
     & \le 2n_2^{-1}\E^*_{P^2\times \varepsilon} \left[\sup_{m \in \mathcal{M}_\delta}\, \left|\sum_{i \in I_2} \epsilon_i m^2(Z_i)\mathbf{1}\{M_\delta 
 > B\}\right|\right] \\
 &\quad+ 2n_2^{-1}\E^*_{P^2\times \varepsilon} \left[\sup_{m \in \mathcal{M}_\delta}\, \left|\sum_{i \in I_2} \epsilon_i m^2(Z_i)\mathbf{1}\{M_\delta 
 \le B\}\right|\right].
    \end{align*}
where the second inequality follows from symmetrization (see for instance, Lemma 2.3.1 of \cite{van1996weak}). We now handle two terms separately. For the unbounded part, we have 
\begin{align*}
    &\E^*_{P^2\times \varepsilon} \left[\sup_{m \in \mathcal{M}_\delta}\, \left|\sum_{i\in I_2} \epsilon_i m^2(Z_i)\mathbf{1}\{M_\delta 
 > B\}\right|\right] \le \E_{P^2} \left[\left|\sum_{i\in I_2}  M_\delta^2(Z_i)  \mathbf{1}\{M_\delta 
 > B\}\right|\right].
\end{align*}
We apply the Hoffmann-J$\o$rgensen inequality (See Proposition 6.8 of \citet{ledoux2013probability} with $p=1$), which states 
\begin{align*}
    \E_{P^2} \left[\left|\sum_{i\in I_2} M_\delta^2(Z_i)\mathbf{1}\{M_\delta 
 > B\}\right|\right] \le 8 \left(\mathbb{E}_{P^2}\left[\max_{i \in I_2} M_\delta^2(Z_i)\right] + t_0^2\right)
\end{align*}
for any $t_0$ such that 
\begin{align}\nonumber
\mathbb{P}_{P^2}\left(\sum_{i \in I_2} M_\delta^2(Z_i) \mathbf{1}\{M_\delta 
 > B\} > t_0\right) \le 1/8.
\end{align}
One can take $t_0=0$ as long as the following satisfies
\begin{align}\label{eq:condition-for-HJ}
&\mathbb{P}_{P^2}\left(\sum_{i \in I_2} M_\delta^2(Z_i) \mathbf{1}\{M_\delta 
 > B\} > 0\right)  \le \mathbb{P}_{P^2}\left(\max_{i \in I_2}M_\delta(Z_i) 
 > B\right) \le 1/8.
\end{align}
Under \eqref{eq:condition-for-HJ}, which to be verified late, the Hoffmann-J$\o$rgensen inequality implies
\begin{align*}
    2n_2^{-1}\E^*_{P^2\times \varepsilon} \left[\sup_{m \in \mathcal{M}_\delta}\, \left|\sum_{i\in I_2} \epsilon_i m^2(Z_i)\mathbf{1}\{M_\delta 
 > B\}\right|\right] &\le 16n_2^{-1} \mathbb{E}_{P^2}\left[\max_{i \in I_2} |M_\delta(Z_i)|^2\right].
\end{align*}
For the second term, we observe that the entire process is uniformly bounded by $B$. We can thus apply the contraction inequality, such as, Theorem 4.12 of \cite{ledoux2013probability} or Corollary 3.2.2 of \cite{gine2021mathematical}. This in tern implies that 
\begin{align*}
    &2n_2^{-1}\E^*_{P^2} \left[\sup_{m \in \mathcal{M}_\delta}\,\left|\sum_{i \in I_2} \epsilon_i m^2(Z_i)\mathbf{1}\{M_\delta 
 \le B\}\right|\right] \\
 &\quad \le 4B n_2^{-1}\E^*_{P^2} \left[\sup_{m \in \mathcal{M}_\delta}\,\left|\sum_{i \in I_2} \epsilon_i m(Z_i)\right|\right] \\
 &\quad \le 8B n_2^{-1}\E^*_{P^2} \left[\sup_{m \in \mathcal{M}_\delta}\,\left|\sum_{i \in I_2}m(Z_i)-\E_{P^2}[m(Z_i)]\right|\right]\\
 &\quad\le 8B \phi_{n_2}(\delta),
\end{align*}
where the first inequality is by the symmetrization and the second by desymmetrization. Now it remains to verify \eqref{eq:condition-for-HJ}.
\paragraph{Finite $q$th moment} When $\E_{P^2}[|M_\delta|^q] \le C_q$, it follows that 
\begin{align*}
&\mathbb{P}_{P^2}\left(\max_{i \in I_2}M_\delta(Z_i) 
 > B\right) \\
 &\quad =\mathbb{P}_{P^2}\left(\max_{i \in I_2}M^q_\delta(Z_i) 
 > B^q\right)\le \mathbb{P}_{P^2}\left(\sum_{i \in I_2}M^q_\delta(Z_i) 
 > B^q\right) \le \frac{n_2 \E_{P^2}[M_\delta^q]}{B^q}.
\end{align*}
In this case, we set $B^q = 8 n_2\E_{P^2}[M_\delta^q]$. The final bound is 
\begin{align*}
    &16n_2^{-1} \mathbb{E}_{P^2}\left[\max_{i \in I_2} |M_\delta(Z_i)|^2\right] + 8B \phi_{n_2}(\delta) \\
    &\quad \le 16n_2^{-1} \left(\mathbb{E}_{P^2}\left[\sum_{i \in I_2} |M_\delta(Z_i)|^{2 \cdot q/2}\right]\right)^{2/q} +8\cdot 8^{1/q} \cdot n_2^{1/q}C_q^{1/q} \phi_{n_2}(\delta) \\
    &\quad \le 16n_2^{-1+2/q} C_q^{2/q} +8\cdot 8^{1/q} \cdot n_2^{1/q}C_q^{1/q} \phi_{n_2}(\delta).
\end{align*}
\paragraph{Finite $\gamma$ Sub-Weibull}
In this case, we have $\E_{P^2}[|M_\delta|^q] \le K^q q^{q/\gamma}$ for all $q \ge 1$. Then the final bound becomes 
\begin{align*}
    16n_2^{-1+2/q} K^2 q^{2/\gamma}+8\cdot 8^{1/q} \cdot n_2^{1/q}K q^{1/\gamma} \phi_{n_2}(\delta).
\end{align*}
The second contraction term is dominated when $n_2^{1/q}q^{1/\gamma}$ is minimized. We choose $q_* = (\gamma \log n_2) > 2$, which requires $n_2 > e^{2/\gamma}$. This choice yields 
\begin{align*}
    &16n_2^{-1+2/q_*} K^2 q_*^{2/\gamma}+8\cdot 8^{1/q_*} \cdot n_2^{1/q_*}K q_*^{1/\gamma} \phi_{n_2}(\delta) \\
    &\quad =16n_2^{-1+2/(\gamma \log n_2)} K^2 (\gamma \log n_2) ^{2/\gamma}\\
    &\quad\quad +8\cdot 8^{1/(\gamma \log n_2)} \cdot n_2^{1/(\gamma \log n_2) }K (\gamma \log n_2) ^{1/\gamma} \phi_{n_2}(\delta)\\
    &\quad \le 16n_2^{-1}e^{2/\gamma} K^2 (\gamma \log n_2) ^{2/\gamma}\\
    &\quad\quad +8\cdot 2 \cdot e^{1/\gamma}K (\gamma \log n_2) ^{1/\gamma} \phi_{n_2}(\delta)
\end{align*}
where $8^{1/x} \le 2$ for $x > 2$.
\end{proof}

\section{Miscellaneous Derivation from Examples}
\subsection{Derivation for \Cref{example-u-statistics}}\label{supp:example1}
    Let $Z_1, \ldots, Z_{N}$ be IID observations following $\mathcal{N}(\mu, \sigma^2)$. We split data into two parts so that $|I_1| + |I_2| = n_1 + n_2 = N$. Recall that 
    \begin{equation}
        \theta(P^N) = \argmin_{\theta \in \mathbb{R}}\, \E_{P^2}[(Z_1Z_2-\theta)^2],\nonumber
    \end{equation}
    and
    \begin{equation}
        \widehat{\mathbb{M}}_2(\theta) = {n_2 \choose 2}^{-1} \sum_{n_1+1 \le i < j \le N} (Z_i Z_j - \theta)^2.\nonumber
    \end{equation}
    First, the curvature follows as
    \begin{align*}
        \mathbb{C}_{P^{2}}(\theta) &= \E_{P^{2}}[(Z_1Z_2-\theta)^2] - \E_{P^{2}}[(Z_1Z_2-\mu^2)^2] \\
        &= \E_{P^{2}}[(Z_1Z_2-\mu^2+\mu^2-\theta)^2] - \E_{P^{2}}[(Z_1Z_2-\mu^2)^2]\\
        &= 2(\mu^2-\theta)^2\E_{P^{2}}[Z_1Z_2] - 2\mu^2(\mu^2-\theta)^2 + (\theta-\mu^2)^2 
        \\&= (\theta-\mu^2)^2.
    \end{align*}
    Since $\widehat{\mathbb{M}}_2$ is an unbiased estimator, the MSE is given by variance. It then follows that 
    \begin{align*}
        \mathbb{V}_{P^{2}}(\theta) &=\mathrm{Var}_{P^{2}}\left[{n_2 \choose 2}^{-1} \sum_{n_1+1 \le i < j \le N} (Z_i Z_j - \theta)^2-(Z_i Z_j - \mu^2)^2\right]\\
        &= 4(\mu^2- \theta)^2\mathrm{Var}_{P^{2}}\left[{n_2 \choose 2}^{-1} \sum_{n_1+1 \le i < j \le N} Z_i Z_j \right].
    \end{align*}
    The variance of U-statistics with degree-2 Kernel is given by
    \begin{align*}
        \mathrm{Var}_{P^{2}}\left[{n_2 \choose 2}^{-1} \sum_{n_1+1 \le i < j \le N} Z_i Z_j \right] = \frac{2}{n_2(n_2-1)}[\zeta_2 + 2(n_2-2)\zeta_1]
    \end{align*}
    where $\zeta_1 = \mathrm{Cov}(Z_1Z_2, Z_1Z_3)$ and $\zeta_2 = \mathrm{Var}(Z_1Z_2)$. Furthermore, we have
    \begin{align*}
        \mathrm{Cov}(Z_1Z_2, Z_1Z_3) = \E[Z_1^2]\E[Z_2]\E[Z_3]-\mu^4 = \mu^2(\sigma^2 + \mu^2) - \mu^4 = \mu^2\sigma^2,
    \end{align*}
    and 
    \begin{align*}
        \mathrm{Var}(Z_1Z_2) = \E[Z_1^2]\E[Z_2^2]-\mu^4 = (\sigma^2 + \mu^2)^2 - \mu^4 = \sigma^4 + 2\sigma^2\mu^2.
    \end{align*}
    Hence, we obtain 
    \begin{align*}
        \mathbb{V}_{P^{2}}(\theta) = (\mu^2- \theta)^2\left(\frac{8\sigma^4}{n_2(n_2-1)} + \frac{16\mu^2\sigma^2}{n_2}\right).
    \end{align*}
    We conclude that 
    \begin{align*}
    \mathbb{C}_{P^{2}}^2(\theta)/\mathbb{V}_{P^{2}}(\theta) = (\mu^2- \theta)^2\left(\frac{8\sigma^4}{n_2(n_2-1)} + \frac{16\mu^2\sigma^2}{n_2}\right)^{-1}.
    \end{align*}
    The result for the constant estimator can be obtained by plugging in  $\theta = c$ to the display above. Next, we consider the standard U-statistics estimator based on $D_1$, that is, 
    \begin{align}\label{eq:sample-statistics}
        \widehat\theta_1 = {n_1 \choose 2}^{-1} \sum_{1 \le i < j \le n_1} Z_i Z_j = \frac{1}{n_1(n_1-1)}\left[\left(\sum_{i=1}^{n_1}Z_i\right)^2 - \sum_{i=1}^{n_1}Z_i^2\right].
    \end{align}
    Denote sample mean and variance as 
    \begin{align*}
        \overline{Z}_{n_1} = \frac{1}{n_1} \sum_{i=1}^{n_1} Z_i \quad \mbox{and} \quad S_{n_1}^2 = \frac{1}{n_1-1}\sum_{i=1}^{n_1}\left(Z_i -\overline{Z}_{n_1}\right)^2.
    \end{align*}
    Then we can write
    \begin{align*}
        \left(\sum_{i=1}^{n_1}Z_i\right)^2 - \sum_{i=1}^{n_1}Z_i^2 = n_1(n_1-1) \overline{Z}_{n_1}^2 - (n_1-1)S_{n_1}^2
    \end{align*}
    and hence 
    \begin{align*}
        \mu^2 -\widehat{\theta}_1 &= \mu^2 - \overline{Z}_{n_1}^2 + \frac{S_{n_1}^2}{n_1} \\
        &= (\mu - \overline{Z}_{n_1})(\mu +\overline{Z}_{n_1}) + \frac{S_{n_1}^2}{n_1}\\
        &\overset{d}{=} -\frac{2Z\sigma}{\sqrt{n_1}}-\frac{Z^2\sigma^2}{n_1}+ \frac{\sigma^2V}{n_1(n_1-1)} = \frac{\sigma}{\sqrt{n_1}}\left(-2Z -\frac{Z^2\sigma}{\sqrt{n_1}} + \frac{\sigma V}{\sqrt{n_1}(n_1-1)}\right)
    \end{align*}
    where $Z\sim \mathcal{N}(0,1)$ and $V \sim \chi^2_{n-1}$, with $Z$ and $V$ independent. Putting together, we have 
    \begin{align*}
    \mathbb{C}_{P^{2}}^2(\widehat\theta_1)/\mathbb{V}_{P^{2}}(\widehat\theta_1) \overset{d}{=} \left(-2Z -\frac{Z^2\sigma}{\sqrt{n_1}} + \frac{\sigma V}{\sqrt{n_1}(n_1-1)}\right)^2\left(\frac{8n_1\sigma^2}{n_2(n_2-1)} + \frac{16n_1\mu^2}{n_2}\right)^{-1}.
    \end{align*}

\subsection{Numerical study for \Cref{example-u-statistics}}\label{supp:example1-sim}
Below, we provide a numerical result to demonstrate how the behavior of the estimator affects the miscoverage probability of the proposed confidence set. We generate $200$ observations $Z_1, \ldots, Z_{200}$ from independent $\mathcal{N}(\mu, 1)$, where $\mu$ varies over the grid $\{0.01, 0.002,\ldots, 0.5\}$. The first $100$ observations are used to construct three estimators: (1) the constant estimator at zero, denoting $\widehat{\theta}_1^{(1)}$; (2) the U-statistics estimator \eqref{eq:sample-statistics}, denoting $\widehat{\theta}_1^{(2)}$. For each estimator and value of $\mu$, we estimate, based on $1000$ replications, the probability
\begin{align*}
    \mathbb{P} \left(\widehat \M_2(\mu) - \widehat \M_2(\widehat\theta_1^{(k)}) > 0\right) \quad \text{for $k \in \{1,2\}$},
\end{align*}
where 
\begin{equation}
    \widehat{\mathbb{M}}_2(\theta) = {100 \choose 2}^{-1} \sum_{101 \le i < j \le 2} (Z_i Z_j - \theta)^2.\nonumber
    \end{equation}
For each estimator and value of $\mu$, we can compute the distribution of $\ratio$ exactly. We then compute the upper bound of miscoverage, $\E[1/(1+\ratio^2)]$, by numerical integration. The results are displayed in \Cref{fig:example1}. The $X$-axis shows the values of true mean $\mu$ and the observed miscoverage based on $1000$ replication is shown in red. The theoretical upper bound of \Cref{thm:coverage-anti-conservative-confidence-set-empirical-risk} is provided in dashed line. We observe that \Cref{thm:coverage-anti-conservative-confidence-set-empirical-risk} is valid but a conservative upper bound of the miscoverage. In particular, \Cref{thm:coverage-anti-conservative-confidence-set-empirical-risk} confirms that constant estimator at zero yields asymptotically conservative confidence set with miscoverage of zero as the value of $\mu$ increases. This is reflected in diverging $\ratio$. For the estimator based on U-statistics, \Cref{thm:coverage-anti-conservative-confidence-set-empirical-risk} provides a conservative bound such that the miscoverage should be less than $\approx 80\%$ but in practice the miscoverage can be less than $50\%$.

\subsection{Derivation for \Cref{example:hodges-estimator}}\label{supp:example2}
Again, let $Z_1, \ldots, Z_{N}$ be IID observations following $\mathcal{N}(\mu, \sigma^2)$. We split data evenly such that $|D_1|=n_1$ and $|D_2| = n_2$. Recall
    \begin{equation}
        \theta(P^N) = \argmin_{\theta \in \mathbb{R}}\, \E_{P^2}[(Z_1-\theta)^2],\nonumber
    \end{equation}
    and 
    \begin{equation}
        \widehat{\mathbb{M}}_2(\theta) = \frac{1}{n_2} \sum_{n_1+1 \le i \le N} (Z_i - \theta)^2.\nonumber
    \end{equation}
    It is straightforward to show that 
    \begin{align*}\ratio^2 = \frac{n_2(\widehat{\theta}_1 - \mu)^2}{4\sigma^2} = \frac{n_1(\widehat{\theta}_1 - \mu)^2}{4\sigma^2} \cdot \frac{n_2}{n_1}.
    \end{align*}
    From this, the behavior of the constant estimator is immediate. For sample mean,
    \begin{align*}
        \ratio^2 = \frac{n_1(\widehat{\theta}_1 - \mu)^2}{4\sigma^2} \cdot \frac{n_2}{n_1} \overset{d}{=} \frac{Z^2}{4}\cdot \frac{n_2}{n_1} \quad \textrm{where} \quad Z \sim \mathcal{N}(0,1).
    \end{align*}
    Now consider Hodge's estimator based on $D_1$ such that $\widehat \theta_1 := \bar{X}_{n_1} \mathbf{1}\{|\bar{X}_{n_1} |\ge {n_1}^{-1/4}\}$ where $ \bar{X}_{n_1} = {n_1}^{-1}\sum_{i=1}^{n_1} X_i$. Then 
    \begin{align*}
        {n_1}(\widehat{\theta}_1 - \mu)^2 = {n_1}(\bar{X}_{n_1} - \mu)^2\mathbf{1}\{|\bar{X}_{n_1} |\ge {n_1}^{-1/4}\} + {n_1}\mu^2\mathbf{1}\{|\bar{X}_{n_1} | < {n_1}^{-1/4}\}.
    \end{align*}
    When $\mu=0$, we have 
    \begin{align*}
        {n_1}(\widehat{\theta}_1 - \mu)^2 &= \sigma^2\left(\frac{\sqrt{{n_1}}\bar{X}_{n_1}}{\sigma}\right)^2\mathbf{1}\left\{\left|\frac{\sqrt{{n_1}}\bar{X}_{n_1}}{\sigma}\right |\ge \sigma^{-1}{n_1}^{1/4}\right\} \\
        &\overset{d}{=}\sigma^2Z^2\mathbf{1}\left\{|Z|\ge \sigma^{-1}{n_1}^{1/4}\right\}\overset{p}{\to} 0,
    \end{align*}
    where $Z \sim\mathcal{N}(0,1)$.
    On the other hands, take $\mu = {n_1}^{-1/4}/2$. Then 
    \begin{align*}
        &{n_1}(\widehat{\theta}_1 - \mu)^2 \\
        &\qquad = \sigma^2\left(\sqrt{{n_1}}\frac{(\bar{X}_{n_1} - \mu)}{\sigma}\right)^2\mathbf{1}\left\{\sqrt{{n_1}}\frac{(\bar{X}_{n_1} - \mu)}{\sigma}\ge \frac{{n_1}^{1/4}}{2\sigma}\, \cup \,\sqrt{{n_1}}\frac{(\bar{X}_{n_1} - \mu)}{\sigma}\le \frac{-3{n_1}^{1/4}}{2\sigma}\right\} \\
        &\qquad\qquad + \frac{\sqrt{{n_1}}}{4}\mathbf{1}\left\{\sqrt{{n_1}}\frac{(\bar{X}_{n_1} - \mu)}{\sigma}< \frac{{n_1}^{1/4}}{2\sigma}\, \cap \,\sqrt{{n_1}}\frac{(\bar{X}_{n_1} - \mu)}{\sigma}> \frac{-3{n_1}^{1/4}}{2\sigma}\right\}  \\
        &\qquad \overset{d}{=} \sigma^2 Z^2\mathbf{1}\left\{Z\ge \frac{{n_1}^{1/4}}{2\sigma}\, \cup \,Z\le \frac{-3{n_1}^{1/4}}{2\sigma}\right\}  + \frac{\sqrt{{n_1}}}{4}\mathbf{1}\left\{Z< \frac{{n_1}^{1/4}}{2\sigma}\, \cap \,Z> \frac{-3{n_1}^{1/4}}{2\sigma}\right\} \\
        &\qquad \overset{p}{\to} \infty,
    \end{align*}
      where $Z \sim\mathcal{N}(0,1)$. In general for any $\mu$, Hodges' estimator yields
      \begin{align*}
          \ratio^2 &\overset{d}{=} \frac{n_2}{4n_1} \bigg(Z^2 \mathbf{1}\left\{Z\ge \frac{n_1^{1/4}-n_1^{1/2}\mu}{\sigma}\,\cup\, Z \le \frac{-n_1^{1/4}-n_1^{1/2}\mu}{\sigma}  \right\}\\
          &\qquad + \frac{n_1\mu^2}{\sigma^2}\mathbf{1}\left\{\frac{-n_1^{1/4}-n_1^{1/2}\mu}{\sigma} < Z < \frac{n_1^{1/4}-n_1^{1/2}\mu}{\sigma}  \right\}\bigg)
      \end{align*}
        
\subsection{Derivation for \Cref{example:hodge-revisited}}\label{supp:example3}
Consider the same setting as \Cref{example:hodges-estimator}. Then 
\begin{equation*}
    \begin{split}
        \widehat \xi_i &= (Z_i - \widehat\theta_1)^2 - (Z_i - \mu)^2- (\widehat\theta_1-\mu)^2 = 2(Z_i - \mu)(\mu - \widehat \theta_1), \\
        \widehat{\mathbb{V}}_{2} &= 4n_2(\mu - \widehat \theta_1)^2 \sigma^2, \quad \textrm{and}\quad
        |\widehat \xi_i|/\widehat{\mathbb{V}}^{1/2}_{2} = \frac{|Z_i-\mu|}{\sqrt{n_2}\sigma}.
    \end{split}
\end{equation*}
The remainder term of \Cref{thm:coverage-anti-conservative-confidence-set-empirical-risk} can be bounded as 
\begin{align*}
    &\mathbb{E}_{P^N}\left[\sum_{i\in I_2} \frac{|\widehat\xi_i|^2}{n_2^2\widehat{\mathbb{V}}_{2}(1 + \widehat{\mathbb{C}}_{2}/\widehat{\mathbb{V}}^{1/2}_{2})^2}\min\left\{1,\,\frac{|\widehat \xi_i|}{n_2\widehat{\mathbb{V}}_{2}^{1/2}(1 + \widehat{\mathbb{C}}_{2}/\widehat{\mathbb{V}}^{1/2}_{2})}\right\}\right] \\
    &\qquad \le \mathbb{E}_{P^N}\left[\sum_{i\in I_2} \frac{|\widehat\xi_i|^3}{n_2^3\widehat{\mathbb{V}}^{3/2}_{2}(1 + \widehat{\mathbb{C}}_{2}/\widehat{\mathbb{V}}^{1/2}_{2})^3}\right] = \frac{\E[|Z_1-\mu|^3]}{n_2^{1/2}\sigma^{3/2}(1 + \ratio)^3} = 2\sqrt{\frac{2}{\pi}}\frac{\sigma^{3/2}}{n_2^{1/2}(1 + \ratio)^3}. 
\end{align*}

\subsection{Numerical study for \Cref{example:hodge-revisited}}\label{sup:numerical-example3}
Below, we provide a numerical result to demonstrate how the behavior of the estimator affects the miscoverage probability of the proposed confidence set. We consider the following setting. First, we generate $200$ observations $Z_1,\ldots, Z_{200}$ from independent $\mathcal{N}(\mu, 1)$, where $\mu$ varies over the grid $\{0.01, 0.002,\ldots, 0.5\}$. The first $100$ observations are used to construct three estimators: (1) the constant estimator $\widehat\theta_1^{(1)} =0$; (2) the sample mean estimator $\widehat\theta_1^{(2)} =\overline{Z} =(100)^{-1}\sum_{i=1}^{100}Z_i$; (3) Hodges' estimator $\widehat\theta^{(3)}_1 =\overline{Z}\mathbf{1}\{|\overline{Z}| \ge 100^{-1/4}\}$.

For each estimator, we estimates the probability 
\begin{align*}
    \mathbb{P} \left(\frac{1}{100}\sum_{i=101}^{200} (Z_i - \mu)^2 -(Z_i - \widehat\theta_1^{(k)})^2 > 0\right) \quad \text{for $k \in \{1,2,3\}$}
\end{align*}
using $1000$ replications. This probability corresponds to the miscoverage probability of the confidence set \eqref{eq:anti-conservative-confidence-set} when the respective estimator is used. The results are shown in \Cref{fig:example3}. 

The simulation displays how closely the upper bound \Cref{thm:coverage-anti-conservative-confidence-set-empirical-risk} tracks the observed miscoverage. First, the miscoverage does not exceed $1/2$ (indicated by the dashed line) for any estimator across all values of $\mu$, hence the confidence set is valid at level $1/2$, which is evident from the fact that $\ratio \ge 0$ almost surely. The upper bound $\E[1-\Phi(\ratio)]$ gives sharper control on the miscoverage. When $\ratio$ is close to zero, the coverage is close to $1/2$, or the exact nominal level. This is observed for a constant estimator or Hodges' estimator near $\mu = 0$ where $\ratio \approx 0$. For these estimators, there are regions where the corresponding confidence set is more conservative than the set based on the sample mean estimator, for instance, the constant estimator for $\mu > 0$ or Hodges' estimator near $\mu = 100^{-1/4} \approx 0.31$. Finally, the sample mean estimator yields miscoverage that is approximately constant in $\mu$ as the behavior $\ratio$ does not depend on $\mu$. 

% \begin{figure}
% \centering
% \begin{subfigure}{0.8\textwidth}
%   \centering
%   \includegraphics[width=\linewidth]{fig/example3.pdf}
% \end{subfigure}
% \caption{Estimated miscoverage probability of the confidence set $\widehat{\mathrm{CI}}_N^\dagger$ defined in \eqref{eq:anti-conservative-confidence-set} as a function of the true mean $\mu$, for three choices of initial estimator. A total of $N=200$ are generated independently from $N(\mu,1)$ with $\mu$ ranging over a grid of values in $\{0.01, 0.02, \ldots, 0.5\}.$ The first $100$ observations are used to construct the initial estimators. For each estimator and each value of $\mu$, the miscoverage probability of the confidence set is estimated via $1000$ replications. The dashed horizontal line indicates the $1/2$ level. \Cref{thm:coverage-anti-conservative-confidence-set-empirical-risk} states that the miscoverage does not exceed $1/2$ asymptotically for any estimator. Hodges' estimator exhibits an interesting pattern: the miscoverage approaches $1/2$ near $\mu =0$ where we observe super-efficiency, and it drops near $100^{-1/4} \approx 0.31.$ which is where the estimator performs poorly, and finally recovers the level comparable to the confidence set based on the sample mean estimator.}
%   \label{fig:example3}
% \end{figure}
    
\section{Additional Numerical 
Results}\label{supp:num}
\subsection{High-dimensional Mean Inference}
The setup follows from \Cref{sec:mean}. For $N=300$ and $2 \le d \le 200$, observations $X_i \in \mathbb{R}^d$, $1 \le i \le N$ are generated independently as
\begin{equation}
    X_i \sim \mathcal{N}(0, \Sigma)\quad  \textrm{where}\quad \Sigma_{i,j} = 0.1^{|i-j|}.
\end{equation}
Five methods are compared. The first method, \texttt{Wald}, is based on the asymptotic distribution of the sample mean, with confidence set
\begin{equation}\label{eq:CI-wald}
\widehat{\mathrm{CI}}^{\mathtt{Wald}, \mathtt{Mean}}_{N,\alpha} := \left\{\theta \in \mathbb{R}^d : (\theta - \bar{X}_N)^\top \widehat{\Sigma}_N^{-1}(\theta - \bar{X}_N) \le N^{-1}\chi^2_{d,\alpha}\right\},
\end{equation}
where $\bar{X}_N$ is the sample mean and $\widehat{\Sigma}_N$
is the sample covariance matrix. Two proposed methods use an even split of the data: \texttt{CLT}, based on \eqref{eq:CI-CLT}, and \texttt{CLT+UCB}, which combines \eqref{eq:CI-CLT} with the upper confidence bound derived in \Cref{example:mean-UCB} as in \eqref{eq:LCB-set-UCB}. These are compared against nonparametric bootstrap (\texttt{Bootstrap}) and multiplier bootstrap (\texttt{Multiplier}); see Section~\ref{supp:baseline} for implementation details.

\paragraph{Validity} \Cref{fig:coverage-mean} displays the empirical coverage of $90\%$ and $50\%$ confidence sets. The $X$-axis displays the dimension $d$ and the $Y$-axis displays the estimated coverage over $500$ replications, with $N=300$ fixed. The $50\%$ nominal level corresponds to \eqref{eq:anti-conservative-confidence-set} and \eqref{eq:anti-conservative-confidence-set-UCB}, for which no variance estimation is required. For both levels, the coverage of the \texttt{Wald} method deteriorates rapidly as $d$ increases. The proposed CLT-based methods maintain validity across all dimensions examined. The improvement from incorporating the upper confidence bound is substantial: \texttt{CLT+UCB} achieves near-nominal coverage across all dimensions, whereas the coverage of \texttt{CLT} approaches one as $d$ increases. This conservativeness of \texttt{CLT} is captured by \Cref{thm:coverage-anti-conservative-confidence-set-empirical-risk}, which identifies large $\ratio$ as the mechanism driving the coverage toward one in high dimensions. Both bootstrap methods achieve near-nominal coverage with slight conservativeness at $\alpha=0.1$ when $d$ is large.

\begin{figure}
\centering
\begin{subfigure}{0.9\textwidth}
  \centering
  \includegraphics[width=\linewidth]{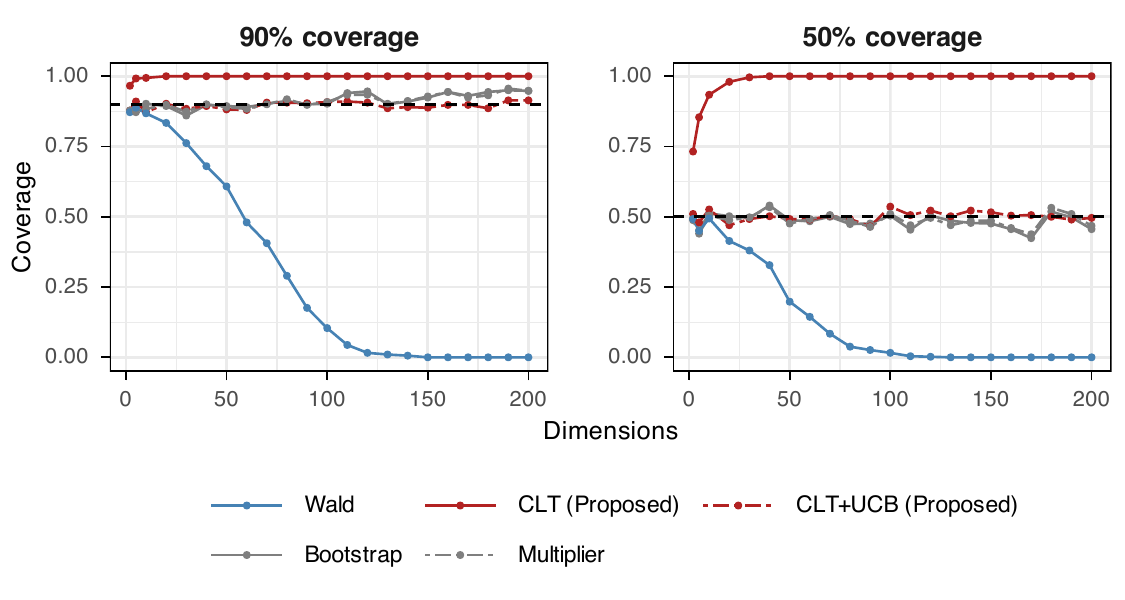}
\end{subfigure}
\caption{Estimated coverage of five confidence set methods for high-dimensional mean, targeted at the $90\%$ and $50\%$ nominal levels. The $X$-axis displays the dimension $d$ and the $Y$-axis displays the estimated coverage over $500$ replications, with $N=300$ fixed. The methods compared are the Wald interval (\texttt{Wald}), the CLT-based proposed method (\texttt{CLT}), the CLT-based method with upper confidence bound (\texttt{CLT+UCB}), nonparametric bootstrap (\texttt{Bootstrap}), and multiplier bootstrap (\texttt{Multiplier}). The coverage of \texttt{Wald} deteriorates rapidly as $d$ increases. Both \texttt{CLT} and \texttt{CLT+UCB} maintain validity across all dimensions. The \texttt{CLT+UCB} method in particular achieves near-nominal coverage throughout, while the coverage of \texttt{CLT} approaches one as $d$ increases. Both bootstrap methods achieve near-nominal coverage with slight conservativeness at $\alpha = 0.1$ and large $d$.}
\label{fig:coverage-mean}
\end{figure}
\begin{figure}
\centering
\begin{subfigure}{0.9\textwidth}
  \centering
  \includegraphics[width=\linewidth]{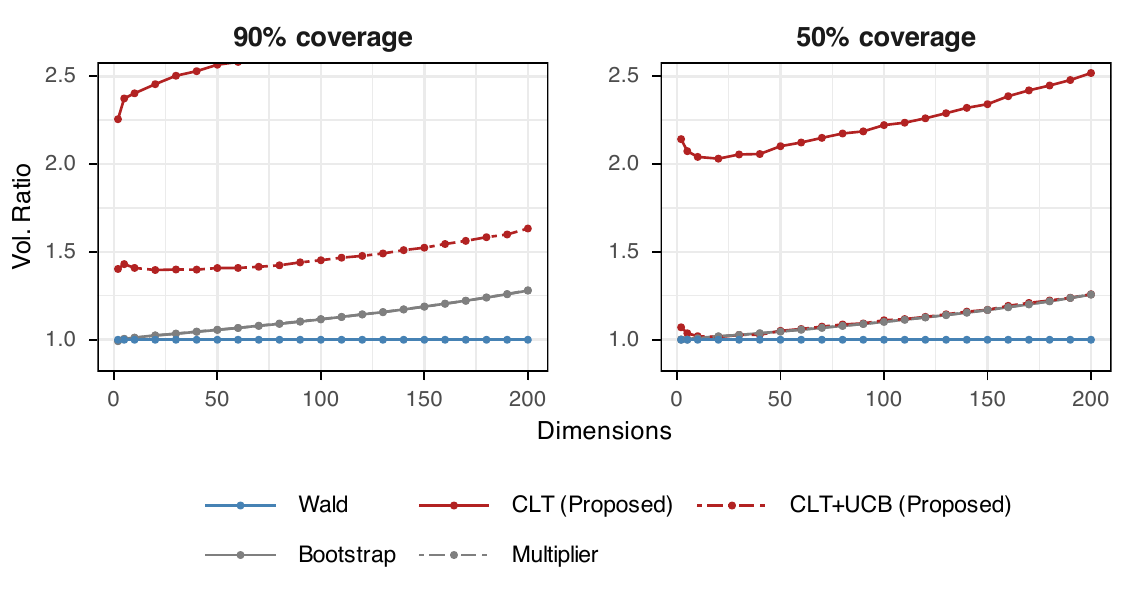}
\end{subfigure}
\caption{Average ratio of radii of each confidence set relative to the \texttt{Wald} interval for high-dimensional mean, computed over $500$ replications with $N=300$ fixed. The $X$-axis displays the dimension $d$ and the $Y$-axis displays the average radius ratio. For $\alpha=1$, the proposed methods use the three-way data split of \Cref{sec:computation}. The \texttt{CLT} method produces the largest sets, with radii more than twice those of \texttt{Wald}. Incorporating the upper confidence bound substantially reduces the size: \texttt{CLT+UCB} achieves radii approximately $1.5$ times those of \texttt{Wald} at $\alpha = 0.1$ and is considerably closer in size at $\alpha=0.5$. At $\alpha=0.5$, the methods \texttt{CLT+UCB}, \texttt{Bootstrap}, and \texttt{Multiplier} are comparable in size.}
\label{fig:volume-mean}
\end{figure}
% \begin{figure}
% \centering
% \begin{subfigure}{0.6\textwidth}
%   \centering
%   \includegraphics[width=\linewidth]{fig/coverage.pdf}
%   \label{fig:sub1}
% \end{subfigure}
% \caption{Comparison of the empirical coverages of the $95\%$ confidence sets; \textit{Asymptotic}, \textit{Sample split} and \textit{Sample split + Upper bound}. The empirical coverages are computed from $1000$ replications. The figure displays the poor coverage accuracy of the Wald interval in high-dimensional settings, with coverage dropping as low as $50\%$ in the analyses.  In contrast, the proposed sample-splitting procedures maintain robust validity as the dimension increases. The results further highlight the empirical conservativeness of the standard method, despite its theoretical optimality as claimed by Theorem~\ref{thm:mean-ci-width}. We show that the small modification to the confidence set restores the nominal coverage.}
% \label{fig:coverage}
% \end{figure}

\paragraph{Width Analysis} We next compare the size of the proposed confidence sets against the \texttt{Wald} interval. As established in \Cref{thm:geometry-mean}, all proposed confidence sets are Euclidean balls and can be computed in closed form. Since all methods yield
$d$-dimensional balls, size comparison reduces to comparing radii, which corresponds to the $d$th root of the volume ratio:
\begin{align*}
    \left(\frac{\mathrm{Vol}(B_d(c_1; r_1))}{\mathrm{Vol}(B_d(c_2; r_2))}\right)^{1/d} = r_1/r_2.
\end{align*}
\Cref{fig:volume-mean} the average ratio of radii relative to \texttt{Wald}, computed over $500$ replications across different dimensions. The $X$-axis displays the dimension $d$ and the $Y$-axis displays the average radius ratio. For $\alpha = 0.1$, we use the three-way data split procedure of \Cref{sec:computation}, which provides the closed-form expression for the corresponding confidence sets.

The \texttt{CLT} method yields the largest sets, with radii on average more than twice those of \texttt{Wald}. Incorporating the upper confidence bound substantially reduces the size: \texttt{CLT+UCB} achieves radii approximately $1.5$ times those of \texttt{Wald} at $\alpha = 0.1$, and is considerably closer in size at $\alpha = 1/2$. The bootstrap methods (\texttt{Bootstrap} and \texttt{Multiplier}) produce smaller confidence sets, and at $\alpha = 1/2$ all three methods \texttt{CLT+UCB}, \texttt{Bootstrap}, and \texttt{Multiplier} are comparable in size. For high-dimensional mean inference, the bootstrap achieves both valid coverage and competitive size. A practical advantage of \texttt{CLT+UCB} over the bootstrap procedures is that the confidence set has an exact closed-form expression given by \Cref{thm:geometry-mean}, requiring only a single pass through the data rather than repeated resampling.

Finally, we provide a qualitative illustration of different confidence sets by visualizing the confidence regions in a bivariate setting. We generate $X_1, \ldots, X_{100}$ from a bivariate normal distribution $\mathcal{N}(0, \Sigma)$ where
\begin{align*}
        \Sigma^2 = \begin{bmatrix}1 & 0.3 \\
        0.3 & 1\end{bmatrix}.
    \end{align*}
\Cref{fig:demo} displays the confidence regions three methods defined in Section~\ref{sec:num} for confidence levels of $95\%$, $85\%$, and $75\%$. This visualization highlights key qualitative differences between the methods. As shown, the proposed confidence sets (middle and right panels) are non-convex, whereas the Wald interval (left panel) forms an ellipse. Additionally, the proposed methods yield slightly enlarged confidence regions. A close inspection of the confidence set presented in the right panel reveals that it coincides with the adaptive confidence sets by \citet{Robins2006} and dimension-agnostic confidence sets based on cross U-statistics proposed by \citet{kim2020dimension} in their Appendix D (See the left panel of their Figure 3). For mean estimation, these methods become identical.

\begin{figure}
\centering
\begin{subfigure}{0.8\textwidth}
  \centering
  \includegraphics[width=\linewidth]{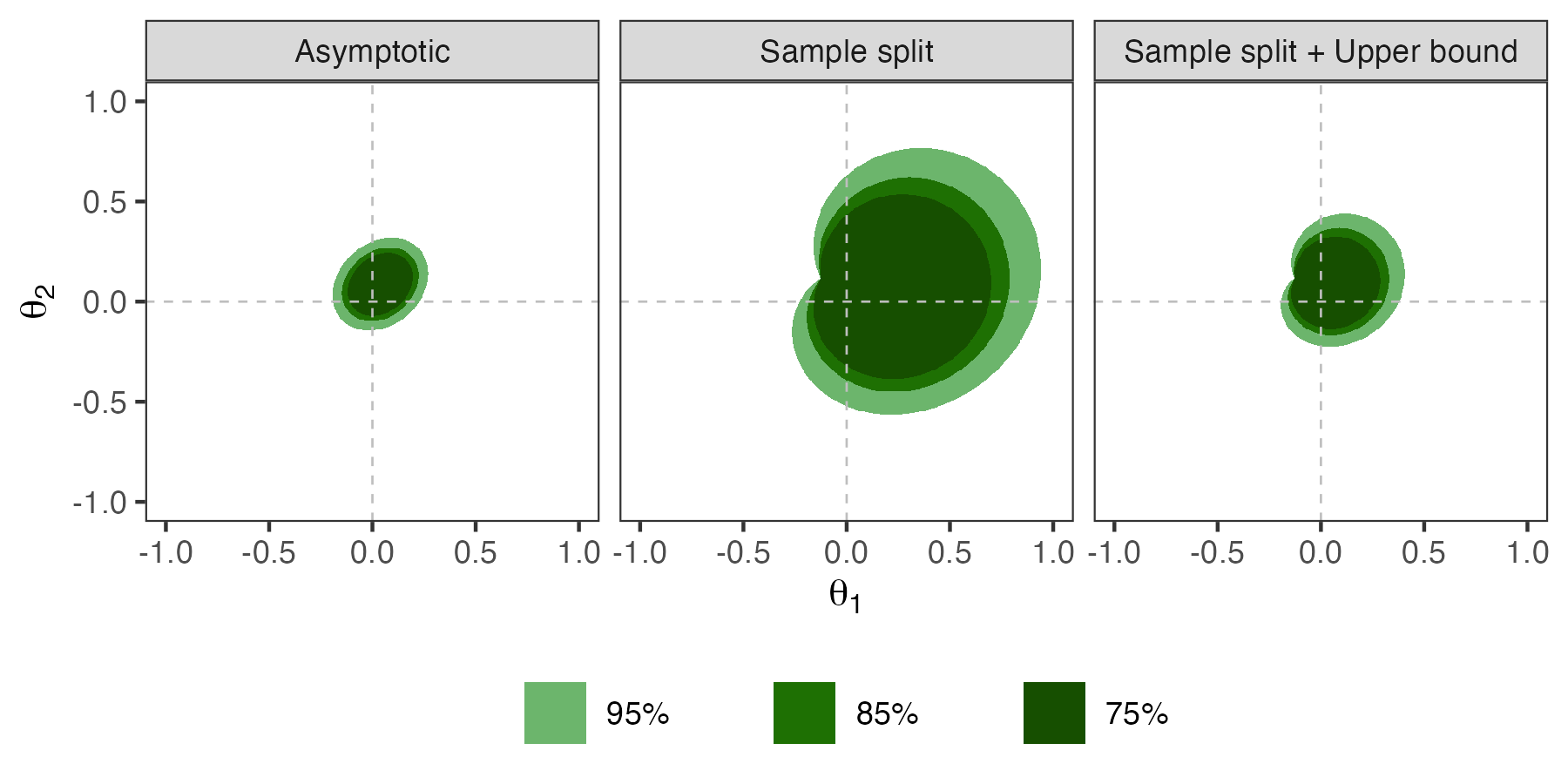}
\end{subfigure}
\caption{An illustration of confidence sets for the bivariate mean $\theta = (\theta_1, \theta_2)^\top$, where the true parameter corresponds to $\theta(P) = (0, 0)^\top$. Three confidence sets are shown at confidence levels of $95\%$, $85\%$, and $75\%$. The confidence set based on the asymptotic distribution (left) yields an elliptical region while the proposed confidence sets (middle and right) are non-convex. }\label{fig:demo}
\end{figure}
\subsection{High-dimensional linear regression}
The setup follows from \Cref{sec:ols}. For given sample size $N$ and $d$, observations $(X_i, Y_i)$ are generated independently as,
\begin{equation}
    X_i \sim \mathcal{N}(0, \Sigma) \quad \textrm{where} \quad \Sigma_{i,j} = 0.1^{|i-j|},
\end{equation}
with $\beta_0 = (1/\sqrt{d}, \ldots, 1/\sqrt{d})^\top$ and heteroskedastic errors
\begin{equation}
    Y_i = \beta_0^\top X_i + \varepsilon_i \quad \textrm{where} \quad \varepsilon_i \sim \mathcal{N}(0, |\beta_0^\top x| + 0.5).
\end{equation}

Seven methods are compared. The first method, \texttt{Wald}, is based on the asymptotic distributions of the ordinary least square with sandwich variance estimator:
\begin{equation}
    \widehat{\mathrm{CI}}^{\mathtt{Wald}, \mathtt{LR}}_{N,\alpha} := \left\{\theta \in \mathbb{R}^d : (\theta - \theta_{\mathtt{OLS}})^\top \widehat{\Sigma}_N^{-1}(\theta - \theta_{\mathtt{OLS}}) \le N^{-1}\chi^2_{d,\alpha}\right\},
\end{equation}
where $\mathtt{OLS}$ is ordinary least square and $\widehat{\Sigma}_N$ is the sandwich variance estimator. Two proposed methods use an even split of the data: \texttt{CLT}, based on \eqref{eq:CI-CLT}, and \texttt{CLT+UCB}, which combines \eqref{eq:CI-CLT} with the upper confidence bound derived in \Cref{example:LR-UCB} as in \eqref{eq:LCB-set-UCB}. These are compared against nonparametric bootstrap (\texttt{Bootstrap}), multiplier bootstrap (\texttt{Multiplier}), wild bootstrap (\texttt{Wild}) and residual bootstrap (\texttt{Resid}); see Section~\ref{supp:baseline} for implementation details. Confidence sets are constructed at levels $\alpha \in \{0.1, 0.5\}$. 

\paragraph{Validity} \Cref{fig:lr_coverage} displays the empirical coverage of $90\%$ and $50\%$ confidence sets over $500$ replications. The $X$-axis displays the dimension $d \in \{10, \ldots, 120\}$. The $Y$-axis displays the estimated coverage over $500$ replications, with $N=300$ fixed. As in the mean inference setting, \texttt{Wald} suffers from severe undercoverage as $d$ increases. Both proposed methods maintain validity across all dimensions examined. The coverage of \texttt{CLT} approaches one as $d$ increases, reflecting the resulf of \Cref{thm:coverage-anti-conservative-confidence-set-empirical-risk}. The \texttt{CLT+UCB} method stays closer to the nominal level but becomes moderately conservative at extreme dimensions, coming from overestimation of the variance term as discussed by \citet{takatsu2025precise}. At $\alpha=1/2$, where no variance estimation is required, the coverage of \texttt{CLT+UCB} remains at the nominal level across all dimensions. In contrast to the mean inference setting, aslmost all bootstrap variants undercover as $d$ increases, though less severely than \texttt{Wald}. The exception is \texttt{Bootstrap}, which remains above the nominal level; for both $\alpha = 1/2$
and $\alpha = 0.1$, it is more conservative than \texttt{CLT+UCB} but less conservative than \texttt{CLT}.

\paragraph{Width Analysis} We next compare the size of the proposed confidence sets against the \texttt{Wald} interval. As establish in \Cref{thm:geometry-LR}, all proposed confidence sets are $d$-dimensional ellipsoids with respect to the $\widehat{\Gamma}$-norm. For two ellipsoids with semi-axes, $\bar{r} = (r_1, \ldots, r_d)$ and $\bar{s} =(s_1, \ldots, s_d)$, size comparison reduces to the geometric mean of the axes ratio, which equals the $d$th root of the volume ratio:
\begin{align*}
    \left(\frac{\mathrm{Vol}(\mathcal{E}_d(c_1; \bar{r}))}{\mathrm{Vol}(\mathcal{E}_d(c_2; \bar{s}))}\right)^{1/d} = \left(\prod_{i=1}^d r_i\right)^{1/d}/\left(\prod_{i=1}^d s_i\right)^{1/d}.
\end{align*}
Since bootstrap procedures yield $d$-dimensional balls rather than ellipsoids, their radius is compared against the geometric mean of the semi-axes of the \texttt{Wald} ellipsoid. For $\alpha = 0.1$, the three-way data split of \Cref{sec:computation} is used. For \texttt{CLT} and \texttt{CLT+UCB}, volume ratios are not reported for $d>100$, though the method itself remains valid since the initial estimator can incorporate ridge regularization.

Results are displayed in \Cref{fig:lr_volume}. The \texttt{CLT} method yields the largest sets across all dimensions and both levels. The bootstrap variants \texttt{Wild}, \texttt{Multiplier}, and \texttt{Resid} produce smaller confidence sets, but, they fail to achieve nominal coverage and so the size comparison is not meaningful. The comparison between \texttt{Bootstrap} and \texttt{CLT+UCB} is more informative, as both achieve valid coverage. At $\alpha = 0.1$, where the three-way split is used, \texttt{CLT+UCB} is larger than \texttt{Bootstrap}; at $\alpha = 1/2$, where no variance estimation is required, \texttt{CLT+UCB} yields a smaller set than \texttt{Bootstrap}. As with the mean inference setting, a practical advantage of \texttt{CLT+UCB} over all bootstrap procedures is that Theorem~\ref{thm:geometry-LR} provides an exact closed-form expression for the confidence set, requiring no resampling.
\begin{figure}
\centering
\begin{subfigure}{0.9\textwidth}
  \centering
  \includegraphics[width=\linewidth]{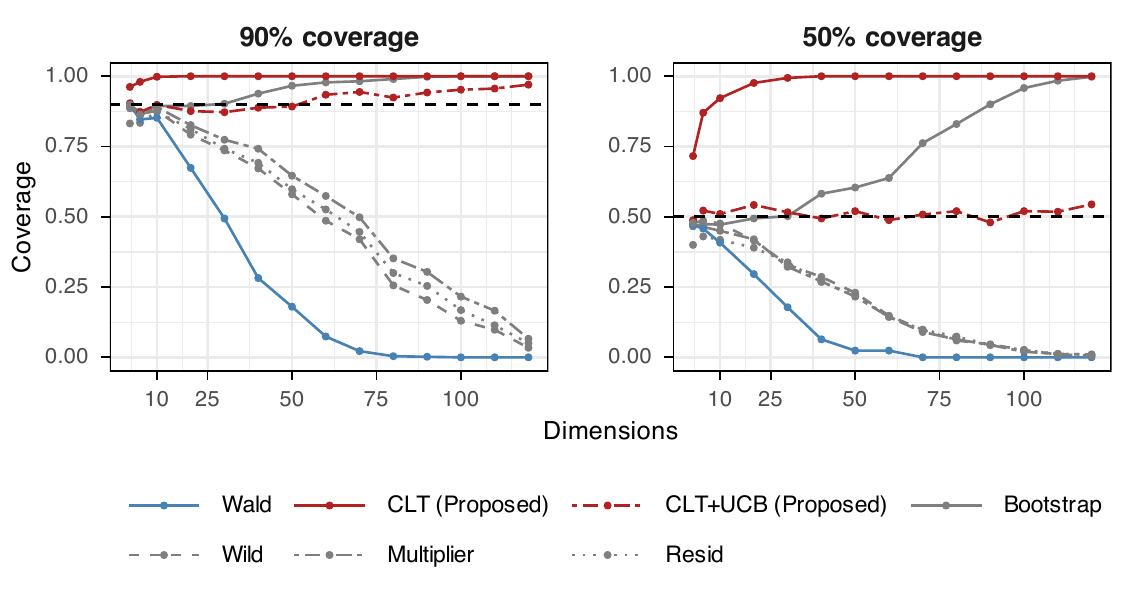}
\end{subfigure}
\caption{Estimated coverage of seven confidence set methods for high-dimensional linear regression under heteroskedastic errors, targeted at the $90\%$
and $50\%$ nominal levels. The $X$-axis displays the dimension $d \in \{10, \ldots, 120\}$ and the $Y$-axis displays the estimated coverage over $500$ replications, with $N=300$ fixed. The methods compared are the Wald interval with sandwich variance estimator (\texttt{Wald}), the CLT-based proposed method (\texttt{CLT}), the CLT-based method with upper confidence bound (\texttt{CLT+UCB}), nonparametric bootstrap (\texttt{Bootstrap}), multiplier bootstrap (\texttt{Multiplier}), wild bootstrap (\texttt{Wild}), and residual bootstrap (\texttt{Resid}). The \texttt{Wald} method undercovers severely as $d$ increases. Both proposed methods maintain validity across all dimensions; \texttt{CLT} becomes increasingly conservative as $d$ grows, while \texttt{CLT+UCB} remains close to the nominal level, with moderate conservativeness at extreme dimensions due to variance overestimation. At $\alpha = 0.5$, where no variance estimation is required, \texttt{CLT+UCB} achieves coverage at the nominal level across all dimensions. All bootstrap variants undercover as $d$ increases, with the exception of \texttt{Bootstrap}, which remains above the nominal level and is more conservative than \texttt{CLT+UCB} but less conservative than \texttt{CLT}.}
\label{fig:lr_coverage}
\end{figure}
\begin{figure}
\centering
\begin{subfigure}{0.9\textwidth}
  \centering
  \includegraphics[width=\linewidth]{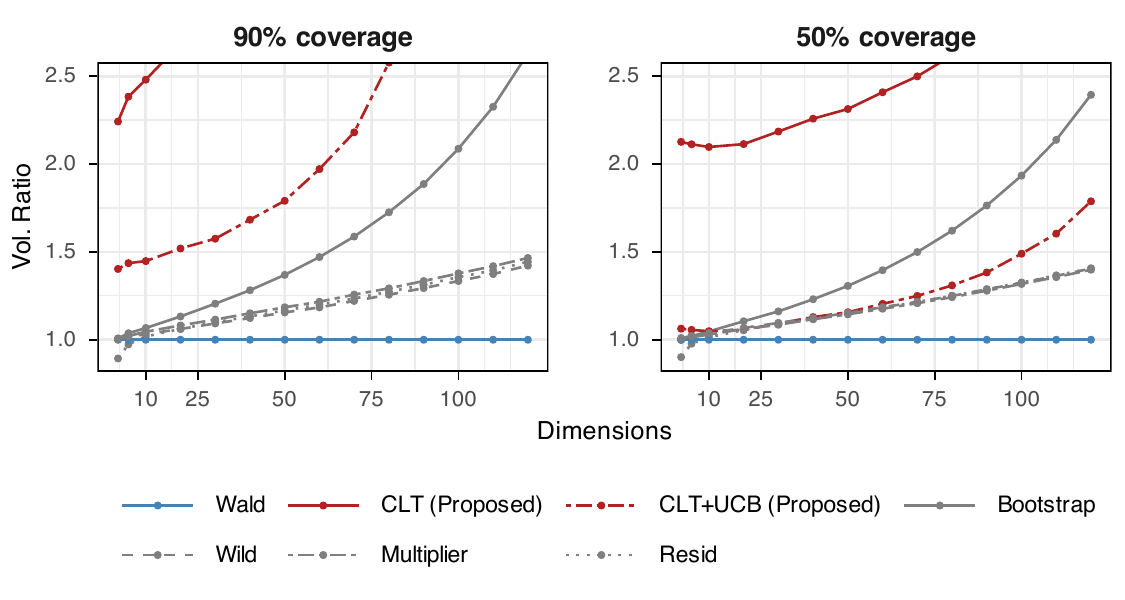}
\end{subfigure}
\caption{Average volume ratio of each confidence set relative to the \texttt{Wald} for high-dimensional linear regression computed over $500$ replications, with $N=300$ fixed. The $X$-axis displays the dimension $d \in \{10, \ldots, 100\}$ and the $Y$-axis displays the average $d$-th root of the volume ratio.For $\alpha = 0.1$, the proposed methods use the three-way data split of \Cref{sec:computation}. The \texttt{CLT} method produces the largest sets. Among valid methods, \texttt{CLT+UCB} is larger than \texttt{Bootstrap} at $\alpha = 0.1$ but smaller at $\alpha = 0.5$. The proposed methods \texttt{CLT} and \texttt{CLT+UCB} admit exact closed-form expressions by \Cref{thm:geometry-LR} and require no resampling.}
\label{fig:lr_volume}
\end{figure}

\subsection{Manski's Discrete Choice Model}\label{supp:manski-algorithm}
\paragraph{Diameter estimation for high-dimension}
We provide the approximation algorithm employed in the numerical study to estimate the diameter of the confidence set for Manski's problem. The algorithm is agnostic to the choice of confidence set among \eqref{eq:anti-conservative-confidence-set}, \eqref{thm:general_LCB} or their improvements considered in \Cref{sec:improvements}. We denote a generic confidence set by $\widehat{\mathrm{CI}}$. The key properties we use are (1) for Manski's problem, $\theta(P^N) \in \mathbb{S}^d$ and (2) for all confidence sets in this manuscript, the initial estimator $\widehat\theta_1$ is always contained in the set. 

The general idea is as follows. First, draw a random vector $u \in \mathbb{S}^{d-1}$ and project onto the tangent space at $\widehat\theta_1$:
\begin{equation}
    v = \frac{u - (u^\top \widehat\theta_1)\widehat\theta_1}{\|u - (u^\top \widehat\theta_1)\widehat\theta_1\|}.
\end{equation}
Then, define a geodesic through $\widehat\theta_1$ in the direction $v$, 
\begin{equation}
    \gamma(\phi) = \cos(\phi)\widehat\theta_1 + \sin(\phi)v \quad \textrm{where} \quad \phi \in [-2\pi, 2\pi].
\end{equation}
A univariate bisection search finds the largest angle $\phi^*$ such that $\gamma(\phi^*) \in \widehat{\mathrm{CI}}$. We repeat this procedure via Gram-Schmidt, which generates an orthonormal basis of the tangent space at $\widehat\theta_1$ with $v$ as the first component. We keep monitoring the maximum $\phi^*$ as we search across all directions. this is repeated until $\phi^*$ stabilizes. The diameter is estimated as then $2\phi^*$.

We do not perform a formal analysis of this algorithm. We expect that when $X$ is close to isotropic, $d-1$ evaluations suffice, as the confidence set is approximately a spherical cap and the boundary distance is nearly uniform across directions. With a more ill-conditioned $X$, additional repetitions may be needed. This procedure is well-suited to the confidence sets in this manuscript as it is fast to compute membership $\gamma(\phi) \in \widehat{\mathrm{CI}}$.
\subsection{Quantile without Positive Densities}
We report numerical results for median inference under non-standard rates of convergence. For a given sample size $N$, IID observations are generated as 
\begin{equation}
    X_i = \mathrm{sgn}(Z_i) \cdot |Z_i|^{1 / (\gamma + 1)} \quad \textrm{where} \quad Z_i \sim \mathcal{N}(0,1).
\end{equation}
The distribution of $X_i$ has a flat density near zero when $\gamma > 0$ so the sample median converges at rate $N^{-1/(2\gamma+2)}$. We consider $\gamma \in \{0, 1/2, 1\}$, corresponding to rates $N^{-1/2}, N^{-1/3}$, and $N^{-1/4}$, respectively. 
\paragraph{Validity}
Varying sample size over $N \in \{100, 200, \ldots, 1000\}$ is analyzed during the numerical study. Three methods are compared: the proposed CLT-based confidence set with $m_{\theta}(X) = |X-\theta|$ based on an even data split; nonparametric bootstrap; and subsampling with rate of convergence estimated via \citet{bertail1999subsampling}. Coverage is estimated over $500$ replications. The nominal level is $90\%$. 

Results are displayed in \Cref{fig:median-validity}, with sample size $N$ on the $X$-axis and estimated coverage on the $Y$-axis. Across all values of $\gamma$, the proposed method achieves coverage above $90\%$. A certain level of conservativeness is expected in view of \Cref{thm:coverage-anti-conservative-confidence-set-empirical-risk}. The performance of the resampling methods vary. Subsampling with estimated rate achieves nominal coverage only for $\gamma=0$ and large $N$, and fails for $\gamma = 1/2$ and $\gamma=1$. Nonparametric bootstrap fails for all values of $\gamma$. As its coverage falls below $60\%$, it is not visible in \Cref{fig:median-validity} at the reported scale. 

\paragraph{Width Analysis}
During the same replications, the diameter of each confidence set is recorded and averaged across replications. Results are displayed in \Cref{fig:median-width}, with sample size on a log scale on the $X$-axis and estimated average diameter on a log scale on the $Y$-axis. The slope of each line corresponds to the exponent in the rate of convergence, with theoretical values $-1/2, -1/3$ and $-1/4$ for $\gamma = 0,1/2,1$ respectively. The slope for the proposed method is also estimated via linear regression and reported in the figure. Although the proposed confidence sets are wider than those of the resampling methods, this comparison is not meaningful since the resampling methods fail to achieve the nominal coverage. the diameter of the proposed confidence set converges at a rate closely matching the theoretical values in each case, demonstrating that the method adapts to the unknown smoothness parameter $\gamma$ without requiring knowledge of it. 

\begin{figure}
\centering
\begin{subfigure}{0.9\textwidth}
  \centering
  \includegraphics[width=\linewidth]{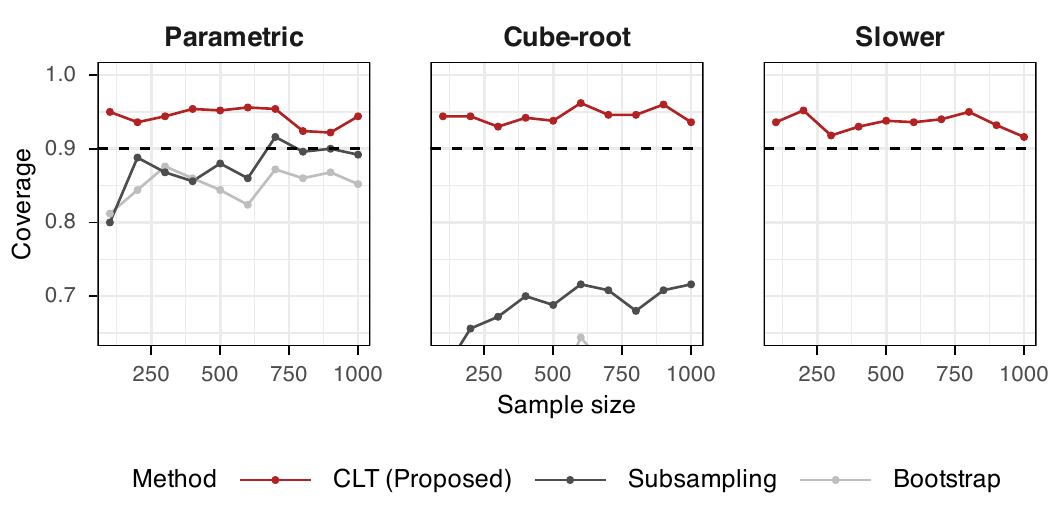}
\end{subfigure}
\caption{Estimated coverage of the proposed confidence set and two sampling methods, targeted at the $90\%$ nominal level. The $X$-axis displays the total sample size $N$ and the $Y$-axis displays the estimated coverage over $500$ replications. From left to right, the panels correspond to $\gamma=0,1/2$, and $1$, resulting in convergence rates $N^{-1/2}, N^{-1/3}$ and $N^{-1/4}$. The base estimator for the proposed method is the sample median, with an even split of observations. The proposed method achieves coverage above $90\%$ across all settings with a certain conservativeness agreeing with the theoretical result in \Cref{thm:coverage-anti-conservative-confidence-set-empirical-risk}. Subsampling with estimated rate achieves nominal coverage only for $\gamma=0$ and large $N$. Nonparametric bootstrap fails for all $\gamma$, with coverage below $60\%$ and outside of the visible range.}
\label{fig:median-validity}
\end{figure}
\begin{figure}
\centering
\begin{subfigure}{0.9\textwidth}
  \centering
  \includegraphics[width=\linewidth]{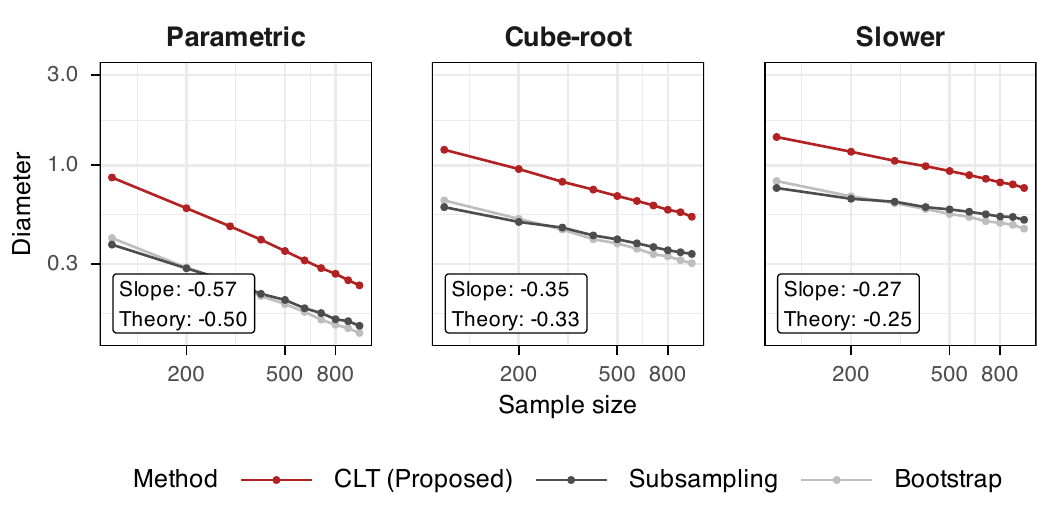}
\end{subfigure}
\caption{Average diameter of the proposed confidence set on a log--log scale. The $X$-axis displays sample size on a log scale the $Y$-axis displays the average diameter of the confidence sets on a log scale over $500$ replications. From left to right, the panels correspond to $\gamma = 0, 1/2$ and $1$, with theoretical rates correspond to $N^{-1/2}, N^{-1/3}$ and $N^{-1/4}$. The slope for the proposed method estimated by linear regression is reported in the figure. The observed slopes closely match the theoretical rates, demonstrating that the proposed confidence set adapts to the unknown smoothness parameter $\gamma$ without requiring prior knowledge of the convergence rate.}
\label{fig:median-width}
\end{figure}

\subsection{Implementation details on baseline methods}\label{supp:baseline}

This section describes the inferential methods implemented during the numerical studies for comparison. All methods take a parameter $B$, a bootstrap sample. For all simulations, we set $B=200$ for computational reasons, which might be considered small for some problems. 
\begin{enumerate}

\item \textbf{Bootstrap \citep{efron1979bootstrap}.} Given observations $D= \{Z_i\}_{i=1}^N$ and an estimator $\widehat\theta = \widehat\theta(D)$, draw $B$ bootstrap samples $D_b = \{Z_i^{(b)}\}_{i=1}^N$ for $1 \le b \le B$ by sampling with replacement from $D$, and recompute $\widehat\theta^{(b)}$ on each $D_b$. For a univariate parameter, bootstrap confidence interval is constructed as the empirical $(\alpha/2, 1-\alpha/2)$ quantile of $\{\widehat\theta^{(b)} - \widehat\theta\}_{b=1}^B$. For a multivariate parameter, bootstrap confidence set is constructed by inverting the $(1-\alpha)$ quantile of $\{\|\widehat\theta^{(b)} - \widehat\theta\|_2\}_{b=1}^B$. 

\item \textbf{Subsampling.} Given observations $D=\{Z_i\}_{i=1}^N$ and an estimator $\widehat\theta = \widehat\theta(D)$, draw $B$ subsamples $D_s$ of size $m = \lceil N^{2/3} \rceil$ without replacement for $1 \le s \le B$, and recompute $\widehat\theta^{(s)}$ on each $D_s$. The subsample size $m$ is an additional tuning parameter, but we fix $m = \lceil N^{2/3}\rceil$ for computational convenience. The convergence rate $\tau_N$ of the estimator is unknown in the settings we considered. Thus, we employ \citet{bertail1999subsampling} to estimate $\widehat\tau_m$. For a univariate parameter, confidence interval is constructed as the empirical $(\alpha/2, 1-\alpha/2)$ quantile of $\{\widehat\tau_m(\widehat\theta^{(s)} - \widehat\theta)\}_{s=1}^B$. For a multivariate parameter, the confidence set is constructed by inverting the $(1-\alpha)$ quantile of $\{\widehat\tau_m\|\widehat\theta^{(s)} - \widehat\theta\|_2\}_{s=1}^B$. 
\item \textbf{Multiplier Bootstrap.} For mean estimation and linear regression, we consider the following procedure. Let $\widehat\theta$ be an estimator; sample mean for mean estimation and OLS for linear regression. Define the estimated score $\widehat\psi_i$ as 
\begin{align}\nonumber
\widehat\psi_i=\begin{cases}
    Z_i-\widehat\theta & \textrm{mean estimation}\\
    (Y_i - X_i^\top\widehat\theta)X_i & \textrm{linear regression}
\end{cases}, \qquad
\widehat\Gamma=\begin{cases}
    I_d & \textrm{mean estimation}\\
    N^{-1}X^\top X & \textrm{linear regression.}
\end{cases}
\end{align}
For each $1 \le b \le B$, draw $W_1^{(b)}, \ldots, W_N^{(b)} \sim \mathcal{N}(0,1)$ independently, and compute 
\begin{align}\nonumber
    t^{(b)} = \widehat\Gamma^{-1} \cdot \frac{1}{\sqrt{N}}\sum_{i=1}^N W_i^{(b)} \widehat\psi_i.
\end{align}
The bootstrap quantile is taken as the empirical $(1-\alpha)$ quantile of $\{N^{-1/2}\|t^{(b)}\|_2\}_{b=1}^B$. 

\item \textbf{Residual Bootstrap.} For linear regression, first obtain the OLS estimator $\widehat\theta$ using the full data and compute the empirical residuals $\widehat\varepsilon_i = Y_i - X_i^\top\widehat\theta$. For each $1 \le b \le B$, draw a bootstrap sample $\{\widehat\varepsilon_i^{(b)}\}_{i=1}^N$ with replacement from $\{\widehat\varepsilon_i\}_{i=1}^N$, and construct bootstrap data $D_b = \{(X_i, X_i^\top \widehat\theta + \widehat\varepsilon_i^{(b)}\}_{i=1}^N$. A new OLS estimator $\widehat\theta^{(b)}$
is obtained from $D_b$, and the bootstrap quantile is taken as the empirical $(1-\alpha)$ quantile of $\{\|\widehat\theta^{(b)} - \widehat\theta\|_2\}_{b=1}^B$. 

\item \textbf{Wild Bootstrap \citep{wu1986jackknife}.} As in the residual bootstrap, compute the empirical residuals $\widehat\varepsilon_i = Y_i = X_i\widehat\theta$ from the full-data OLS estimator $\widehat\theta$. For each $1 \le b \le B$, draw $\epsilon^{(b)}_1, \ldots, \epsilon^{(b)}_N$ from independent Rademacher and construct bootstrap data by  $\{(X_i, X_i^\top \widehat\theta + \widehat\varepsilon_i \cdot \epsilon_i^{(b)}\}_{i=1}^N$. A new OLS estimator $\widehat\theta^{(b)}$ is obtained from $D_b$, and the bootstrap quantile is taken as the empirical $(1-\alpha)$ quantile of $\{\|\widehat\theta^{(b)} - \widehat\theta\|_2\}_{b=1}^B$.

\end{enumerate}

\section{Results based on Concentration Inequality}\label{suppsec:empirical-bernstein}
% \subsection{Lower Confidence Bounds I---Concentration Inequalities}

For some M-estimation problems, the loss difference $m_{\theta}-m_{\theta'}$ is uniformly bounded by definition, making concentration inequalities for bounded random variables directly applicable. For instance, when the observation is bounded uniformly, all applications, then $m_{\theta}-m_{\theta'}$ is also uniformly bounded for all applications considered in this manuscript. We emphasize that boundedness is required of the loss difference, not of the observations.
\begin{example}[Manski's Maximum Score Estimator]
    We consider independent observations $(X_1^\top, Y_1), \ldots, (X_N^\top, Y_N) \in \mathbb{R}^d \times \{-1, 1\}$ with loss function:
    \begin{equation}
        m_{\theta}(X, Y) = -Y\mathrm{sgn}(\theta^\top X) \quad \text{where} \quad \mathrm{sgn}(t) = 2\mathbf{1}\{t \ge 0\} - 1,
    \end{equation}
    corresponding to Manski's maximum score estimator \citep{manski1975maximum}. The loss difference $m_{\theta} - m_{\theta'}$ takes values in $\{-2, 0, 2\}$ and is therefore bounded by construction with $B=2$.
\end{example}
As a concrete illustration, we apply the one-sided empirical Bernstein inequality of \citet[Theorem~11]{maurer2009empirical}, which holds under independent but not necessarily identically distributed observations. Given $\{Z_i\, : \, i \in I_2\}$ with $n_2 = |I_2|$, suppose $\|m_{\theta} - m_{\theta'}\|_\infty \le B$. This yields the empirical Bernstein confidence set
\begin{equation}\label{eq:CI-EB}
    \widehat{\mathrm{CI}}^{\mathtt{EB}}_{N, \alpha} := \left\{\theta\in\Theta:\, \frac{1}{n_2}\sum_{i\in I_2}(m_{\theta}-m_{ \widehat{\theta}_1})(Z_i)\le \sqrt{\frac{2\widehat \sigma^2_{\theta, \widehat\theta_1} \log(2/\alpha)}{n_2}} + \frac{7B\log(2/\alpha)}{3(n_2-1)}\right\},
\end{equation}
where $\mathtt{EB}$ stands for empirical Bernstein. The following is an immediate consequence of Theorem~\ref{thm:general_LCB}.
\begin{theorem}\label{thm:lcb-empirical-bernstein}
    Let $Z_1, \ldots, Z_N$ be independent observations satisfying $\|m_{\theta}-m_{\theta'}\|_\infty \le B$ for all $\theta, \theta' \in \Theta$. Then for any $n_2\ge 2$,
    \begin{align*}
        \mathbb{P}_{P^N}\left(\theta(P^N) \not\in \widehat{\mathrm{CI}}^{\mathtt{EB}}_{N, \alpha}\right) \le \alpha.
    \end{align*}
\end{theorem}
\begin{proof}[\bfseries{Proof of \Cref{thm:lcb-empirical-bernstein}}]
    This result is a direct consequence of \Cref{thm:general_LCB} with $\beta(r)=0$ and Theorem 11 of \cite{maurer2009empirical}.
\end{proof}
\begin{theorem}\label{thm:eb-width}Assume $Z_1,\ldots, Z_N$ is independent and $B= \|m_{\theta} - m_{\theta'}\|_\infty$ for $\theta, \theta'\in\Theta$. Assume $\theta(P^N)$ is the unique solution of \eqref{eq:def-m-functional} that satisfies \ref{as:margin}, \ref{as:local-entropy}, \ref{as:square-process} and \ref{as:rate-initial-estimator3}. Define $r_{n_2}$ and $u_{n_2}$ as any values that satisfy
\begin{align}
    r_{n_2}^{-2} \phi_{n_2}(c_0^{-1/(1+\gamma)}r_{n_2}^{2/(1+\gamma)}) \le 1, \quad \textrm{and} \quad u_{n_2}^{-2} \omega_{\mathtt{pop}}(c_0^{-1/(1+\gamma)}u_{n_2}^{2/(1+\gamma)}) \le n_2^{1/2}. \label{rq:define-un-eb}
\end{align}
Then, for any $n_1\ge 1$, $n_2 \ge 2$ and $\varepsilon > 0$, 
\begin{align*}
\mathbb{P}^*_{P^N}\left(\mathrm{Diam}_{\|\cdot\|}\big(\widehat{\mathrm{CI}}^{\mathtt{EB}}_{N, \alpha}\big)\le C \left(\frac{1+\sqrt{\max\{1,B\}\log(2/\alpha)}}{\varepsilon}\right)^{1/(1+\gamma-q)} \mathrm{R}_N^{\mathtt{EB}}\right) \ge 1-\varepsilon - \widetilde \varepsilon_{\mathtt{init}},
\end{align*}
where 
\begin{align*}
    \mathrm{R}_N^{\mathtt{EB}} = c_0^{-1/(1+\gamma)}(r_{n_2}^{2/(1+\gamma)}  + u_{n_2}^{2/(1+\gamma)}+\widetilde s_{n_1, n_2}^{1/(1+\gamma)}+(B/n_2)^{1/(1+\gamma)}),
\end{align*}
and $C$ is a constant depending only on $\gamma, q$, and $\widetilde C_{\mathtt{init}}$. 
\end{theorem}
We remark that the term $u_{n_2}$ is only determined through $\omega_{n_2, \mathtt{pop}}$ while \Cref{thm:clt-width} involves both $\omega_{n_2, \mathtt{pop}}$ and $\omega_{n_2, \mathtt{emp}}$. When $\|m_{\theta}-m_{\theta'}\|_\infty$ is bounded, it follows that $\omega_{n_2, \mathtt{emp}} = \phi_{n_2}$ via the contraction inequality \citep[Theorem 4.12]{ledoux2013probability}. Consequently, the proof of \Cref{thm:eb-width} can be seen as a special case of \Cref{thm:clt-width}. Both $\widehat{\mathrm{CI}}_{N,\alpha}^{\mathtt{EB}}$ and $\widehat{\mathrm{CI}}_{N,\alpha}^{\mathtt{CLT}}$ can be used for bounded cases, but only $\widehat{\mathrm{CI}}_{N,\alpha}^{\mathtt{CLT}}$ extends to unbounded cases. 

While the diameter bound based on \Cref{thm:clt-width} is harder to establish due to $\omega_{n_2, \mathtt{emp}}$, it always results in a smaller confidence set. This is evident in the dependence on $\alpha$, made explicit in \Cref{thm:clt-width} and \Cref{thm:eb-width}. 

Although the convergence rates are comparable (up to the dependence on $\alpha$), the validity requirements differ: as discussed in \Cref{sec:coverage-slpha} and summarized in \Cref{tab:prescribed-level-confidence-sets}, the set  $\widehat{\mathrm{CI}}_{N,\alpha}^{\mathtt{CLT}}$ can fail for some bounded random variables. For instance, it is not valid for Bernoulli random variables with success probability $p$ such that $np\longrightarrow \lambda$.

Finally, we will not provide an analogous result to \Cref{thm:clt-width-local} where the relevant assumptions are imposed only on the neighborhood of $\theta_0$. This can be established without significant modification to the proof since the boundedness of $\|m_{\theta} - m_{\theta'}\|_\infty$ allows the peeling step in the proof to be over finite partitions, meaning that a remainder term is summable under weaker assumptions. While we do not write out the corresponding result, the almost identical proof as follows goes through. 
\begin{proof}[\bfseries{Proof of \Cref{thm:eb-width}}]
The proof is similar to that of \Cref{thm:clt-width}. First, observe that 
\begin{align*}
    \widehat t_{\alpha}(\theta, \widehat \theta_1) &\le 2\sqrt{2 \log(2/\alpha)} n_2^{-1/2} \sqrt{\left|\frac{1}{n_2}\sum_{i \in I_2} \{(m_{\theta}-m_{ \theta(P^N)})(Z_i)\}^2 - \E_{P_i}[\{(m_{\theta}-m_{\theta(P^N)})(Z)\}^2]\right|} \\
    &\quad + 2\sqrt{2 \log(2/\alpha)}  n_2^{-1/2}\sqrt{\left| \E_{P_i}[\{(m_{\theta}-m_{ \theta(P^N)})(Z)\}^2]\right|} \\
    &\quad + 2\sqrt{2 \log(2/\alpha)} n_2^{-1/2}\sqrt{\frac{1}{n_2}\sum_{i \in I_2} \{(m_{\theta(P^N)}-m_{ \widehat{\theta}_1})(Z_i)\}^2} + \frac{7/3B\log(2/\alpha)}{n_2-1}\\
    &= \mathfrak{R}_1 + \mathfrak{R}_2 + \mathfrak{R}_3.
\end{align*}
The proof follows analogously up to the definition of $\mathbf{I}$--$\mathbf{V}$. No modification is necessary for $\mathbf{I}$ and $\mathbf{II}$, and we obtain 
\begin{align*}
   \mathbf{I} + \mathbf{II} \lesssim  \widetilde C_{\mathtt{init}}2^{-M(1+\gamma)} + C_{q,\gamma}2^{-M(1+\gamma-q)} + \widetilde\varepsilon_{\mathtt{init}}. 
\end{align*}
For $\mathbf{III}$, we note that 
\begin{align*}
    &\left|\frac{1}{n_2}\sum_{i \in I_2} \{(m_{\theta}-m_{ \theta(P^N)})(Z_i)\}^2 - \E_{P_i}[\{(m_{\theta}-m_{\theta(P^N)})(Z)\}^2]\right|\\
    &\quad = 4B^2\left|\frac{1}{n_2}\sum_{i \in I_2} \left\{\frac{(m_{\theta}-m_{ \theta(P^N)})(Z_i)}{2B}\right\}^2 - \E_{P_i}\left[\left\{\frac{(m_{\theta}-m_{ \theta(P^N)})(Z)}{2B}\right\}^2\right]\right|
\end{align*}
where $(m_{\theta}-m_{ \theta(P^N)})(Z)/2B \in [-1,1]$ be the boundedness and $t \mapsto t^2$ is $2$-Lipschitz on $t \in [-1,1]$. 
Now let $\epsilon$ be independent Rademacher random variables. By symmetrization (see for instance, Lemma 2.3.1 of \citet{van1996weak}), the contraction inequality (for instance, Theorem 4.12 of \citet{ledoux2013probability} or Corollary 3.2.2 of \cite{gine2021mathematical}), and desymmetrization (Lemma 2.3.6 of \citet{van1996weak} for mean-zero processes), we arrive for $S \subseteq \Theta$
\begin{align*}
    &\E_{P^2}\left[\sup_{\theta \in S}\left|\frac{1}{n_2}\sum_{i \in I_2} \left\{\frac{(m_{\theta}-m_{ \theta(P^N)})(Z_i)}{2B}\right\}^2 - \E_{P_i}\left[\left\{\frac{(m_{\theta}-m_{ \theta(P^N)})(Z)}{2B}\right\}^2\right]\right|\right] \\
    &\quad \le 2\E_{P^2 \times \epsilon}\left[\sup_{\theta \in S}\left|\frac{1}{n_2}\sum_{i \in I_2} \epsilon_i\left\{\frac{(m_{\theta}-m_{ \theta(P^N)})(Z_i)}{2B}\right\}^2\right|\right]\\
    &\quad \le \frac{16}{B}\E_{P^2 \times \epsilon}\left[\sup_{\theta \in S}\left|\frac{1}{n_2}\sum_{i \in I_2} \epsilon_i(m_{\theta}-m_{ \theta(P^N)})(Z_i)\right|\right]\\
    &\quad \le \frac{32}{B}\E_{P^2 }\left[\sup_{\theta \in S}\left|(m_{\theta}-m_{ \theta(P^N)})(Z_i) - \E_{P_i}[m_{\theta}-m_{ \theta(P^N)}]\right|\right].
\end{align*} 
Hence, we have $ \omega_{n_2,\mathtt{emp}}^2(\delta) \le 128B \cdot \phi_{n_2}(\delta)$. Take 
\begin{align*}
    \bar{u}_{n_2}^{2/(1+\gamma)} =u^{2/(1+\gamma)}_{n_2} +  \left(\frac{B}{n_2}\right)^{1/{(1+\gamma)}}.
\end{align*}
By AM-GM inequality, we have 
\begin{align*}
    &u^{2/(1+\gamma)}_{n_2} +  \left(\frac{B}{n_2}\right)^{1/{(1+\gamma)}} \ge 2\sqrt{u^{2/(1+\gamma)}_{n_2}\left(\frac{B}{n_2}\right)^{1/{(1+\gamma)}}} \\
    &\quad \implies u^{4}_{n_2} \ge 2^{2(1+\gamma)}u^{2}_{n_2}\left(\frac{B}{n_2}\right).
\end{align*}
Using these results, we have 
\begin{align*}
    \mathbf{III} &\le \mathbb{P}^*_{P^2|\widetilde P^1}\left(c_0\|\theta - \theta(P^N)\|^{1+\gamma} \le 5\mathfrak{R}_1\right.\\
    &\quad\quad\quad\quad\bigg. \mbox{ for } \|\theta-\theta(P^N)\| \ge 2^{M} c_0^{-1/(1+\gamma)}\bar{u}_{n_2}^{2/(1+\gamma)}  \bigg)\\
    &=\mathbb{P}^*_{P^2|\widetilde P^1}\left(c_0^2\|\theta - \theta(P^N)\|^{2+2\gamma} \le 25\mathfrak{R}_1^2\right.\\
    &\quad\quad\quad\quad\bigg. \mbox{ for } \|\theta-\theta(P^N)\| \ge 2^{M} c_0^{-1/(1+\gamma)}\bar{u}_{n_2}^{2/(1+\gamma)}  \bigg)\\
    &\le 25 \cdot 4(2\log(2/\alpha)) n_2^{-1} \sum_{j=M}^\infty 2^{-2j(1+\gamma)} \bar{u}_{n_2}^{-4}\omega^2_{n_2}(2^{j+1}c_0^{-1/(1+\gamma)}\bar{u}_{n_2}^{2/(1+\gamma)})\\
    &\le 25 \cdot 4(2\log(2/\alpha)) n_2^{-1} \sum_{j=M}^\infty 2^{-2j(1+\gamma)} 2^{2q(j+1)}\bar{u}_{n_2}^{-4}128B \cdot \phi_{n_2}(c_0^{-1/(1+\gamma)}u_{n_2}^{2/(1+\gamma)})\\
    &\le 25 \cdot 8(2B\log(2/\alpha))\sum_{j=M}^\infty 2^{-2j(1+\gamma)} 2^{2q(j+1)} \le 200 (2\log(2/\alpha)) C_{q,\gamma}^2 2^{-2M(1+\gamma-q)}. 
\end{align*}
Meanwhile a similar derivation for $\mathbf{IV}$ gives 
\begin{align*}
    \mathbf{IV} &\le \mathbb{P}^*_{P^2|\widetilde P^1}\left(c_0\|\theta - \theta(P^N)\|^{1+\gamma} \le 5\mathfrak{R}_2\right.\\
    &\quad\quad\quad\bigg. \mbox{ for } \|\theta-\theta(P^N)\| \ge 2^{M} c_0^{-1/(1+\gamma)}u_{n_2}^{2/(1+\gamma)}  \bigg)\lesssim \log(2/\alpha)  \cdot C_{q,\gamma}^2 2^{-2M(1+\gamma-q)}. 
\end{align*}
Finally, for $\mathbf{V}$, we have 
\begin{align*}
    \mathbf{V} &= \mathbb{P}^*_{P^2|\widetilde P^1}\left(c_0\|\theta - \theta(P^N)\|^{1+\gamma} \le 5\mathfrak{R}_3 \cap B^c\right) \\
   &\le \mathbb{P}^*_{P^2|\widetilde P^1}\left(c_0\|\theta - \theta(P^N)\|^{1+\gamma} \le  \frac{20 \sqrt{2 \log(2/\alpha)}}{n_2}\sqrt{\sum_{i \in I_2} \{(m_{\theta(P^N)}-m_{ \widehat{\theta}_1})(Z_i)\}^2}\cap B^c\right) \\
   &\quad  + \mathbb{P}^*_{P^2|\widetilde P^1}\left(c_0\|\theta - \theta(P^N)\|^{1+\gamma} \le  \frac{7/3B\log(2/\alpha)}{n_2-1}\cap B^c\right) \\
   & \lesssim \widetilde C_{\mathtt{init}}\log(2/\alpha) 2^{-2M(1+\gamma)} + (B\log(2/\alpha)) 2^{-M(1+\gamma)} + \widetilde\varepsilon_{\mathtt{init}}.
\end{align*}
Hence, we obtain 
\begin{align*}
&\mathbb{P}^*_{P^N}\left(\mathrm{Diam}_{\|\cdot\|}\big(\widehat{\mathrm{CI}}^{\mathtt{EB}}_{N, \alpha}\big)> 2^{M}c_0^{-1/(1+\gamma)}\left\{r_{n_2}^{2/(1+\gamma)}  +u_{n_2}^{2/(1+\gamma)} +{s'}_{n_1,n_2}^{1/(1+\gamma)}+ \left(\frac{B}{n_2}\right)^{1/(1+\gamma)}\right\}\right) \\
&\quad \lesssim  \widetilde C_{\mathtt{init}} 2^{-M(1+\gamma)} + C_{q,\gamma}2^{-M(1+\gamma-q)} + BC_{q,\gamma}^2 \log(2/\alpha) 2^{-2M(1+\gamma-q)} \\
&\quad\quad+ \widetilde C_{\mathtt{init}}\log(2/\alpha) 2^{-2M(1+\gamma)} + (B\log(2/\alpha)) 2^{-M(1+\gamma)} + \widetilde\varepsilon_{\mathtt{init}}\\
&\quad \lesssim \mathfrak{C}(1+\max\{1,B\}\log(2/\alpha)2^{-M(1+\gamma-q)}) 2^{-M(1+\gamma-q)} + \widetilde\varepsilon_{\mathtt{init}},
\end{align*}
where $\mathfrak{C}$ only depends on $\gamma, q, \widetilde C_{\mathtt{init}}$. We thus conclude the claim by choosing $M$ to be 
\begin{align*}
    M &= \frac{\log ((1+\sqrt{\max\{1,B\}\log(2/\alpha)})\mathfrak{C}/\varepsilon)}{(1+\gamma-q) \cdot \log 2} \quad\textrm{and} \quad \\
    2^{M} &= \left(\frac{\mathfrak{C}(1+\sqrt{\max\{1,B\}\log(2/\alpha)})}{\varepsilon}\right)^{1/(1+\gamma-q)}.
\end{align*}
\end{proof}

We provide two remarks regarding statistical applications. 
\begin{remark}[Manski’s Discrete Choice Model]
For the Manski's problem introduced in \Cref{sec:manski}, the loss function is uniformly bounded with $B = 2$. Hence, the result in this section applies. \Cref{thm:manski-validity-consistent} requires $\widehat\theta_1$ to converge not too rapidly. The validity for the set developed in this section can be established without any assumptions. 
\end{remark}

\begin{remark}[Discrete Argmin Inference]
For the discrete inference problem introduced in \Cref{sec:argmin-inference}, the result in this section applies if it can be assumed that $e_j^\top X_i$ is almost surely uniformly bounded for all $1\le j \le d$ and $1 \le i \le n$ for some constant. The validity follows without the control on the remainder term \eqref{as:pairwise}.
\end{remark}

% \end{appendices}
% \newpage
% \begin{adjustwidth}{-.25in}{-.25in}
% \begin{appendices}
% \begin{small}
% % \appendix
% \singlespace
% % \clearpage
% % \input{supp/supp-concentration}
% % \input{supp/supp-clt}
% \input{supp/convex}
% \input{supp/example}
% \end{small}
% \end{appendices}
% \end{adjustwidth}
\end{document}